\documentclass{conm-p-l}

\usepackage{amssymb}

\usepackage{pictex}

\newtheorem{theorem}{Theorem}[subsubsection]
\newtheorem{lemma}[theorem]{Lemma}
\newtheorem{prop}[theorem]{Proposition}
\newtheorem{cor}[theorem]{Corollary}
\newtheorem{no-text}[theorem]{}
\newtheorem{addendum}[theorem]{Addendum}

\theoremstyle{definition}

\theoremstyle{remark}

\newtheorem{note}{\bf N}[section]

\numberwithin{equation}{section}

\def\ssize{\scriptstyle}
\def\sssize{\scriptstyle}
\def\rmk{\small}
\def\arr#1#2{\arrow <1.5mm> [0.25,0.75] from #1 to #2}

\def\subfactor{\ \beginpicture
  \setcoordinatesystem units <0.1cm,0.2cm>
  \multiput{} at 0 -.7  2 1 /
  \plot 2 1.2  0 1.2  0 0.2  2 0.2 /
  \plot 2 -.1  0 -0.1 /\endpicture\ }

\def\above{\beginpicture
  \setcoordinatesystem units <0.1cm,0.3cm>
  \multiput{} at 0 0.1  2 1 /
  \plot 0 0  1 0  1 1  2 1 /\endpicture}

\def\car{\beginpicture
  \setcoordinatesystem units <.2cm,.2cm>
  \multiput{} at .5 -.5  3.5 1.5 /
  \plot 0 0  -.5 0  -.5 .5  0 1.5  2 1.5  2.5 1  3.5 .7  3.5 0  3 0 /
  \plot 1 0  2 0 /
  \circulararc 360 degrees from 0 0  center at 0.5 0 
  \circulararc 360 degrees from 2 0  center at 2.5 0 
  \endpicture}

\def\LT{\operatorname{\bold{LT}}}
\def\A{\operatorname{\bold A}}
\def\P{\operatorname{\bold P}}

\def\M{\operatorname{\bold M}}
\def\NC{\operatorname{\bold{Nc}}}

\def\E{\operatorname{\bold E}}
\def\F{\operatorname{\bold F}}
\def\B{\operatorname{\bold B}}
\def\H{\operatorname{\bold H}}

\def\a{\operatorname{\bold a}}
\def\c{\operatorname{\bold c}}
\def\d{\operatorname{\bold d}}
\def\f{\operatorname{\bold f}}
\def\t{\operatorname{\bold t}}

\def\diff{\text{\rm d}}
\def\op{{\text{\rm op}}}
\def\subfactor{\ \beginpicture
  \setcoordinatesystem units <0.1cm,0.2cm>
  \multiput{} at 0 -.7  2 1 /
  \plot 2 1.2  0 1.2  0 0.2  2 0.2 /
  \plot 2 -.1  0 -0.1 /\endpicture\ }

\DeclareMathOperator{\cox}{cox}
\DeclareMathOperator{\pdeg}{pdeg}
 \DeclareMathOperator{\Tr}{Tr}
\DeclareMathOperator{\rank}{rank}
\DeclareMathOperator{\soc}{soc}
\DeclareMathOperator{\topp}{top}
\DeclareMathOperator{\mo}{mod}
\DeclareMathOperator{\ind}{ind}
\DeclareMathOperator{\sinc}{sinc}
\DeclareMathOperator{\Der}{Der}

\DeclareMathOperator{\GL}{GL}

\DeclareMathOperator{\Hom}{Hom}
\DeclareMathOperator{\End}{End}
\DeclareMathOperator{\Aut}{Aut}
\DeclareMathOperator{\Sub}{Sub}
\DeclareMathOperator{\Quot}{Quot}
\DeclareMathOperator{\Ext}{Ext}
\DeclareMathOperator{\rad}{rad}
\DeclareMathOperator{\add}{add}
\DeclareMathOperator{\Ker}{Ker}

\DeclareMathOperator{\Imm}{Im}
\DeclareMathOperator{\bdim}{\mathbf{dim}}
  \def\ss{\ssize }

\newcommand\Rahmen[1]%
   {\centerline{\vbox{\hrule\hbox%
                  {\vrule%
                       \hskip0.5cm%
                            \vbox{\vskip0.3cm\relax%
                               \hbox{{#1}}%
                                  \vskip0.28cm}%
                       \hskip0.5cm%
                  \vrule}%
           \hrule}}}

\begin{document}

\title{The Catalan Combinatorics of the Hereditary Artin Algebras}

\author{Claus Michael Ringel}
\address{
Fakult\"at f\"ur Mathematik\\
Universit\"at Bielefeld\\
POBox 100 131\\
D-33501 Bielefeld, Germany\\
and Department of Mathematics\\
Shanghai Jiao Tong University\\
Shanghai 200240, P. R. China.
}

\email{ringel at math.uni-bielefeld.de}
\thanks{}


\date{August 8, 2015}

\begin{abstract}
The Catalan numbers are 
{\it  one of the most ubiquitous and fascinating sequences 
of enumerative combinatorics} (Stanley), in particular they count the
number of non-crossing partitions of a finite set.
In the appendix of these notes
we will try to outline in which way the Catalan combinatorics could be seen
as the heart of the theory of finite sets, starting with the subsets of
cardinality two. If we fix a finite set $C$ of cardinality $n+1$, the subsets of
cardinality two may be considered as the positive roots of a root system
(in the sense of Lie theory) of Dynkin type $\Bbb A_n$, and 
there are recent proposals
to work with generalized non-crossing partitions, starting with any
root system (of Dynkin type $\Bbb A_n,
\Bbb B_n, \dots, \Bbb G_2)$. 
The Catalan combinatorics looks for sets of partitions of $C$
which are of relevance and relates them  
to subsets of the automorphism group $S_{n+1} = \Aut(C)$, this 
is the Weyl group of type $\mathbb A$.
The generalized Cartan combinatorics starts directly with a 
suitable subset of $G$, where $G$ is any Weyl (or, more generally, any Coxeter) group. 
It turns out that the representation theory of representation-finite hereditary
artin algebras $\Lambda$ can be used in order to categorify these generalized non-crossing 
partitions in the Weyl group case. 
In particular, for the case $\mathbb A_n$, one may use the ring
$\Lambda_n$ of all upper triangular $(n\times n)$-matrices with coefficients in a 
field. 
\end{abstract}

\maketitle

\section*{\bf Introduction}

A {\bf root system} is a finite set of vectors in a Euclidean vector space satisfying some strong symmetry conditions. Root systems are used as convenient index sets when dealing with semi-simple complex Lie algebras or algebraic groups, but play an important role also in other parts of mathematics. The 
(crystallographic) root systems have been classified by Killing and Cartan at the end of the 19th century, the different types of irreducible root systems are labeled by the Dynkin diagrams 
$\mathbb A_n, \mathbb B_n,..., \mathbb G_2.$ 
As we have mentioned, 
the definition of the root systems refers to symmetry properties, but it turns out that there are further hidden symmetries which are not at all apparent at first sight. They have been discovered only quite recently and extend the use of root systems considerably. 

Always, $\Lambda$ will be a hereditary artin algebra. 
If $\Lambda$ is of finite representation type, 
it is well-known that the indecomposable $\Lambda$-modules correspond bijectively 
to the positive roots of a root system. The positive roots form in a natural way a poset, 
these posets are called the {\bf root posets.} 
In the setting of $\Lambda$-modules, the ordering is given by looking at subfactors. 
Root posets play a decisive role in many parts of mathematics: of course in Lie theory, 
in geometry (hyperplane arrangements) and group theory (reflection groups), but also say in 
singularity theory, in topology, and even in free probability theory (non-crossing partitions). 
The aim of this survey will be to report on combinatorial properties of the root posets which 
have been found in recent years by various mathematicians, often in view of these applications. 
Several of the results which we will discuss have been generalized to the Kac-Moody root systems.
We will focus the attention to the relevance of these properties in the 
representation theory of hereditary artin algebras and the use of this categorification.

\subsection*{Outline} 
A root system $\Phi$ is a finite subset of a Euclidean space $V$.
If $x$ is a root, we denote by $H_x$ the hyperplane orthogonal to $x$, and by $\rho_x$ the
reflection at $H_x$. In this way, we attach to $\Phi$ (or $\Phi_+$)
a finite set $\H(\Phi)$ of hyperplanes in $V$,
such sets are called hyperplane arrangements. The reflections $\rho_x$ generate the 
corresponding Weyl group $W$. Using the reflections, one defines on $W$ a partial ordering, the
so-called absolute ordering $\le_a$. Given a Coxeter element $c$ in $W$, the set $\NC(W,c)$
of all element $w\in W$ with $w\le_a c$ is called the lattice of generalized non-crossing
partitions. In the case of a root system of type $\mathbb A$, one just obtains the
usual lattice of non-crossing partitions, as introduced by Kreweras and now used
for example in free probability theory. 

As we have mentioned, we always will consider a hereditary artin algebra $\Lambda$. 
Let $\mo\Lambda$ be the category of all (left) $\Lambda$-modules
of finite length. Recall that a full subcategory $\mathcal C$ of $\mo\Lambda$ is said to be
{\it thick} provided it is closed under kernels, cokernels, and extensions, thus it is an
abelian exact subcategory, and we say that a thick subcategory $\mathcal C$ of $\mo\Lambda$
is {\it exceptional} provided it is categorically equivalent to the module category $\mo\Lambda'$
where $\Lambda'$ is also a (necessarily hereditary) artin algebra, or, equivalently, 
provided $\mathcal C$
has a generator. We denote by $\A(\mo\Lambda)$ the poset of
exceptional subcategories of $\mo\Lambda$, this will be the main object of interest. The
central result to be shown asserts that 
	\smallskip 

\Rahmen{$\A(\mo\Lambda) \simeq \NC(W(\Lambda),c(\Lambda)),$}

\noindent
where $W(\Lambda)$ and $c(\Lambda)$ are the Weyl group and the
Coxeter element, respectively, corresponding to $\Lambda$, see Theorem \ref{IST}.
	\medskip 

{\bf Chapter 1} is devoted to {\bf numbers}
which arise from counting problems dealing with a representation-finite
hereditary artin algebra $\Lambda$. The numbers we are interested in will depend just
on the Dynkin type of $\Lambda$ (and not on the orientation). Thus,
here we deal with what we  call Dynkin functions:  A Dynkin function $\f$
attaches to any Dynkin diagram $\Delta$ an integer, or more
generally a real number, sometimes even a set
or a sequence of real numbers (for example the sequence of exponents);
thus a Dynkin function $\f$ consists of four sequences of numbers, namely 
$\f(\mathbb A_n),\ \f(\mathbb B_n),\ \f(\mathbb C_n),\ \f(\mathbb D_n)$
as well as five additional single values $
 \f(\mathbb E_6),\ \f(\mathbb E_7),\ \f(\mathbb E_8),\ \f(\mathbb F_4),\ \f(\mathbb G_2).$
Typical Dynkin functions are the number of indecomposable modules,
the number of tilting modules, the number of complete exceptional sequences.
We will analyze some of these Dynkin functions, of particular interest 
seem to be the prime factorizations of their values. 

As we will see, there is a unified, but quite mysterious way to deal with 
some of these Dynkin functions, namely to invoke the so-called exponents of $\Delta$.
Usually, the exponents just fall from heaven: either by looking at the 
invariant theory of the action of the Weyl group on the ambient space 
of the root system (Chevalley 1955),
or by dealing with the eigenvalues of a Coxeter element (Coxeter, 1951).
As Shapiro and Kostant (1959) have shown, there is a third
possibility to obtain the exponents, namely looking at the root poset:
if $r_t$ is the number of roots of height $t$, then
$(r_1,r_2,\dots)$  is a Young 
partition and the dual partition is the partition of the exponents.
It is of interest that one may determine the exponents
inductively, going up step by step in a chain of poset ideals $I$ of $\Phi_+$.
A recent result of Sommers-Tymoczko (and 
Abe-Barakat-Cuntz-Hoge-Terao) based on old investigations 
of Arnold and Saito (1979)
asserts that the set $\H = \H(I)$ of hyperplanes orthogonal to the roots in $I$
is a so-called free hyperplane arrangement. This means that the corresponding
module $D(\H)$ of $\H$-derivations is free, thus one may consider 
the degrees of a free generating system of $D(\H)$. We obtain in this way an
increasing sequence of Young partitions which terminates in the partition of the
exponents. 
	\smallskip

{\bf Chapter 2} concerns the classical {\bf tilting theory}, the study of (finitely generated)
tilting modules for a hereditary artin algebra. 
As we will see in this chapter, already the basic setting of tilting theory can be refined,
replacing the usually considered torsion pair
by a torsion triple or even a torsion quadruple. In this way, tilting theory is put into
the realm of the stability theory of King. The study of tilted algebras turns out 
to be just the study of sincere exceptional subcategories. 
We also will study perpendicular pairs of exceptional
subcategories. Altogether we obtain a wealth of bijections (the Ingalls-Thomas bijections) 
between sets of modules and subcategories. These bijections explain why 
we obtain the same Dynkin functions when dealing with quite different counting problems.
	\smallskip 
		
{\bf Chapter 3} presents the poset $\A(\mo\Lambda)$ of all {\bf 
exceptional antichains} in $\mo\Lambda$, or, equivalently, of all exceptional
 subcategories of $\mo\Lambda$. Using the results of Chapter 2,
it will be shown that this poset is self-dual
(it has a self-duality whose square is essentially the Auslander-Reiten translation).
Also, any interval in $\A(\mo\Lambda)$ is again of the form $\A(\mo\Lambda')$ for some
hereditary artin algebra $\Lambda'$, 
and the maximal chains in $\A(\mo\Lambda)$ correspond bijectively to the 
complete exceptional sequences of $\Lambda$-modules. On the other hand, we will see that
$\A(\mo\Lambda)$  can be identified with 
the poset $\NC(W(\Lambda),c(\Lambda))$. In this way the 
theory of generalized non-crossing partitions can be seen as part of the 
representation theory of hereditary artin algebras.
I should stress that the main results outlined in Chapters 2 and 3
are due to Ingalls and Thomas \cite{[IT]}, 
and a subsequent paper by Igusa and Schiffler \cite{[IS]}.
	\smallskip 

{\bf Chapter 4} deals with the special case of the {\bf Dynkin types $\mathbb A$}. We denote by $\Lambda_n$
the path algebra of the linearly oriented quiver of type $\mathbb A_n$.
We will show that the lattice $\A(\mo\Lambda_n)$ may be identified in a canonical
way with the lattice of non-crossing partitions as introduced by Kreweras; this is
now an important tool in several parts of mathematics, for example in free
probability theory. Thus, one may consider the module categories 
$\mo\Lambda_n$ as a natural frame for a categorification of the lattices 
of non-crossing partitions. In particular, here we deal with the Catalan numbers,
{\it  one of the most ubiquitous and fascinating sequences 
of enumerative combinatorics} (Stanley in \cite{[S1]}). 
Also, at the end of Chapter 4, 
we review some classical problems which are related to the maximal chains
in $\A(\mo\Lambda_n)$: namely, to count labeled trees (Sylvester, 1857, Borchardt, 1860,
Cayley, 1889), as well as parking functions (Pyke, 1959, Konheim-Weiss, 1966, Stanley, 1997).
	\smallskip 
		
The {\bf Appendix.} 
As an after-thought we will try to discuss the nature of Catalan combinatorics: 
it seems to us that it should be considered as the heart of the theory of finite sets. 
	\smallskip 

This report concerns the {\bf Catalan combinatorics} and the
corresponding \linebreak Narayana numbers. One may also
say that it is about the {\bf cluster complex.} 
Actually, I will mention the cluster complex only in passing by, but one should be aware that 
the cluster combinatorics in the Dynkin case
is really the combinatorics of the representation-finite hereditary artin algebras as
discussed in these lectures. 
Of course, we deal with the {\bf categorification of combinatorial data,}
this is the essence of our considerations. An axiomatic account of this categorification
can be found in a recent paper by Hubery and Krause \cite{[HK]}. 

As we have mentioned, Chapters 2 and 3 deal with hereditary artin algebras in general, whereas
Chapter 4 and the Appendix
restrict the attention to the special case $\mathbb A_n$, or better just to $\mathbb A_n$
endowed with the linear orientation (the corresponding path algebra will be denoted
by $\Lambda_n$ and will be used as the standard example throughout these lectures). 
In this way, we present first the general theory and
specialize afterwards in order to capture the classical theory of non-crossing
partitions in terms of the representation theory of artin algebras. Our account
should also allow the interested reader to go the opposite way:
to start with Chapter 4 in order to
see in which way $\Lambda_n$-modules are used for the categorification of partitions
(see Section 4.2)
and only afterwards to immerge into the general representation theory of artin algebras. 
	\smallskip

{\bf Too late?} This report comes late, very late, maybe too late.  
It concerns objects  which have been in the mainstream of
representation theory 40 years ago, now they seem to be standard and well understood.
The first chapter will focus the attention to a lot of numbers; 
such numbers had been calculated in the early days of representation theory,
but as it seems, never systematically, and only few records are available
(by Gabriel-de la Pe\~na and Bretscher-L\"aser-Riedtmann, as well as by  
Seidel, a student of Happel).
As Assem wrote to me: there should be many student theses at various universities 
devoted to such calculations, but one did not dare to publish them. 
The mathematicians working in the representation theory of algebras
felt that there would not be an independent interest in these numbers, 
the only exception may have been Gabriel \cite{[G2]}:
he pointed out that here the Catalan numbers play a role --- but as far as I know never in lectures
to a mathematical audience, just in a text written for amateurs and enthusiasts.
To repeat: a survey similar to the first parts of these notes may (and should) have been
given in the seventies or early eighties of the last century.

Actually, the numbers presented have been discussed, but usually outside of representation
theory. We should stress that Chapoton \cite{[C1]} presented already in 2002, 
thus more than 10 years ago,
the numbers of clusters, positive clusters and exceptional sequences on his web page, and 
there is a corresponding survey by Fomin and Reading \cite{[FR]} written 2005. 
Some of the numerology can be traced much further back, namely to 
considerations concerning singularity theory by Brieskorn and Deligne in the seventies.

Of course, there is an advantage of a late presentation: we are 
able to present a rather complete picture. But 
be aware: There are still many open questions. 
In particular, one misses an interpretation of the numerical data 
in terms of the exponents (see Chapter 1). Also, given a
hereditary artin algebra of Dynkin type $\Delta$, it is not clear how
to relate the antichains in the category $\mo\Lambda$  and the
antichains in the poset $\Phi_+(\Delta)$, thus to relate
non-crossing and non-nesting partitions in a satisfactory way.

Our survey is quite long, but unfortunately it is in no way complete.
There are many related topics which we do not touch at all; for example the geometrical
realizations of lattices and posets using polyhedra, or important hyperplane
arrangements such as the Shi arrangements; even the cluster approach (and 
the use of cluster categories) is not mentioned explicitly.  
On the other hand, the topics considered here are restricted to a very narrow 
setting: a general report should start with hereditary artinian
rings, not just artin algebras, in order to cover also non-crystallographic
Coxeter groups; it should avoid the restriction to hereditary rings by
looking at $\tau$-tilting modules instead of tilting modules; and it should
consider generalized (not necessarily finitely generated)
tilting modules in order to take into account thick subcategories without covers.
Concerning these general settings, many  satisfactory results are already
known, but a unified theory is still out of reach. 
Thus, one may say that 
it really is {\bf too early (not too late)} for a general presentation. 

As we have mentioned, the appendix outlines in which way the Catalan combinatorics 
can be seen as the heart of the theory of finite sets, starting with the subsets of
cardinality two, thus with the positive roots of a root system
of type $\mathbb A$. We do not know which kind of categories could replace the
category of finite sets in order to deal with the remaining root systems.
Also here, our considerations are open-ended. 
	\smallskip 

{\bf The approach.}
I will try to be as elementary as possible. I will prefer to consider individual modules
in contrast to subcategories (thus, instead of dealing with thick subcategories, I usually will
work with antichains: a thick subcategory $\mathcal C$ of $\mo\Lambda$ 
is an abelian exact subcategory closed under extensions, the corresponding antichain 
is given by the simple objects of $\mathcal C$, and $\mathcal C$
is obtained back from the antichain as its extension closure). Given an artin algebra $\Lambda$,
I will prefer to work with its module category $\mo\Lambda$ and will not touch the
corresponding derived category $D^b(\mo\Lambda)$. I know that triangulated
categories are now well-known and well-appreciated, but they will not be needed in an essential way.
	
The survey is based on lectures which were
addressed to mathematicians working in the representation theory of finite-dimensional 
algebras, and they deal with a topic all participants were familiar with, namely
the representation theory of hereditary artin algebras: first we consider just
representation-finite ones, say corresponding to 
quivers of finite type, or, more generally, to species of finite type,
later then hereditary artin algebras in general.
The literature usually restricts to quivers, and avoids species. As I mentioned, I want to be
as elementary as possible, but nevertheless we will take into account species.
The reason is the following: There is the division between the series $\mathbb A_n, \mathbb B_n, 
\mathbb C_n, \mathbb D_n$, and the exceptional cases. Let me look at the series: 
always, the case $\mathbb A_n$ is considered as the basic case, the three remaining cases 
$\mathbb B_n, \mathbb C_n, \mathbb D_n$ are deviations of $\mathbb A_n$ (note that for such a diagram 
$\mathbb B_n, \mathbb C_n, \mathbb D_n,$  a large portion is of type $\mathbb A$, 
there is a difference only at one of the ends). For many
counting problems as discussed  
in the first chapter, the cases $\mathbb B_n$ and $\mathbb C_n$ yield the same answer, 
and the formulas which one obtains are really neat, condensed and 
surprisingly easy to remember, whereas the formulas for $\mathbb D_n$ usually 
look much more complicated, and often they may be considered as variations of the
$\mathbb B_n$ formulas. Thus, in order to understand the formulas for $\mathbb D_n$ properly,
it seems to be advisable to look first at the numbers for $\mathbb A$ and $\mathbb B$ 
and only afterwards to the case $\mathbb D$.
This is one of the reasons why we definitely want to include the cases $\mathbb B$ (and $\mathbb C$), 
although this requires to work not only with path algebras of quivers, 
but with hereditary artin algebras in general.
	\smallskip 

{\bf References.} 
This is a survey dealing with contributions by a large number of mathematicians,
see also the note N\,\ref{new} at the end of the introduction. 
I will try to indicate the main sources, but to name all contributors seems to be a 
nearly hopeless task.
The material to be covered is vast and I am not at all an expert in several of the
topics, thus sometimes I have to be vague, and provide just some indications.
I am grateful to many mathematicians for introducing me to various questions, see the 
acknowledgments at the end of the introduction. 

In Chapter 1, we will deal with a large set of counting problems, and it will turn out that 
several of these problems yield the same answer. This is of course of great interest
and asks for some explanation: to provide natural bijections between the objects in
questions. However, this also tends to be a source for priority fights: just think of
say 100 equivalent counting problems (see the note N\,\ref{number}). 
Now any problem can be solved 
individually (so there should be 100 different 
proofs), or else one can show the equivalence to a similar problem where the answer is known 
(there are $\binom {100}2 = 4950$ equivalence proofs).
But the situation may be even more complex: in case we deal with a Dynkin function, one may
have to consider it case by case, or one can find a unified proof; one may need to rely
on computer calculations or find a conceptual proof. And the answer may be given 
by a magic formula, say in terms of the exponents, and a final proof should explain this!
	\smallskip

{\bf Acknowledgments.} This is the written account 
for a sequence of four lectures given at the ICRA Workshop August 2014 in Sanya, Hainan, and 
a related series of lectures September 2014 at  SJTU, Shanghai.
The author thanks the organizers of the ICRA conference and Zhang Pu from SJTU
for the invitations to give these lectures. We have omitted here part of the general considerations
on hyperplane arrangements (this was lecture 2), only the final result of Sommers-Tymoczko 
will be outlined in Section 1.5. The third lecture at Sanya was
devoted to the new vision of tilting theory following Ingalls and Thomas; this is now Chapter 2.
The material
which was presented in lecture 4 has been expanded considerably and now forms Chapters 3 and 4.

A central part of this report is based on a joint project with M.~Obaid, K.~Nauman, 
W.~Al Shammakh and W.~Fakieh 
from KAU, Jeddah, which was devoted to review and to complement known 
counting results for support-tilting modules and complete exceptional sequences, 
see the papers \cite{[ONAFR], [ONFR1], [ONFR2]}.
	
Of course, all my considerations concerning representations of hereditary
artin algebras rely on the old collaboration with V.~Dlab, those on tilted algebras
on the collaboration with D.~Happel. I am grateful to H.~Krause and L.~Hille for stressing the 
relevance of thick subcategories, and of King's stability theory, respectively. 
I have learned from G.~R\"ohrle the basic induction principle for hyperplane arrangements;
H.~Thomas made me aware of the Shapiro-Kostant relation between the exponents and the
height partition of the positive roots.
	
But as the main driving force I have to mention F.~G\"otze, the
chairman of the Bielefeld CRC 701. He advised me already several years ago to
study non-crossing partitions. He 
organized joint study groups of the Bielefeld research groups in probability theory
and in representation theory in order to raise the mutual interest --- for a long time, 
this seemed to be a hopeless endeavor. One of the topics he always stressed 
were the parking functions, but I realized only now, when writing up these notes, the
direct bijection between the parking functions and 
complete exceptional sequences for the linearly oriented quiver of type $\mathbb A_n$ as outlined at
the end of Chapter 4: I had been working on exceptional sequences without being aware
of such a relationship (but he seemed to know).
Thus, I have to thank the Bielefeld CRC 701 who has supported me in this way (see also the note 
N\,\ref{DFG}). 
I should add that the presentation has gained from the Bielefeld workshop 
on {\it Non-crossing Partitions in Representation Theory}
organized in June 2014 
by B.~Baumann, A.~Hubery, H.~Krause, Chr.~Stump, and the Bielefeld CRC has to be praised 
for providing the financial support. 
	
My earlier drawings of the various root posets have been improved by A.~Beineke, in addition I 
have to thank him for his permission to include in the first chapter 
some of his observations concerning
the cubical structure of the root posets. 
	
I am grateful to many mathematicians for answering questions and for 
helpful comments concerning the presented material, 
in particular to Th\.~Br\"ustle, F.~Chapoton, X.~W.~Chen, 
W.~Crawley-Boevey, M.~Cuntz, S.~Fomin, L.~Hille, H.~Krause, 
G.~R\"ohrle, and H.~Thomas.
	\bigskip

{\bf Notes.}
	\smallskip

\begin{note}\label{new}
We hope that our presentation adds some small improvements to the
present knowledge. 

This concerns in Chapter 2 the unified treatment of torsion and 
perpendicular pairs, invoking the stability theory of King; it relies on a systematic use
of normalizations of modules (see Section 2.2). 
Following \cite{[ONFR1]}, the artin algebras to be used in order to categorify the posets of
generalized non-crossing partitions are not assumed to be path algebras of quivers,
thus we cover all the symmetrizable Cartan matrices, not just the symmetric ones. 
There are new drawings of the root posets in Section \ref{comb}
which may be helpful for the reader. Our discussion of relevant Dynkin functions and
their prime factors reveals the strange prime factor 4759, see Table 2.
In Chapter 4 we outline a categorical interpretation of the bijection 
between non-crossing partitions and binary trees,
using perpendicular pairs  in $\mo\Lambda_n$, see Theorem \ref{tors-perp}. 
And there is an interpretation of Stanley's bijection (between maximal chains of
non-crossing partitions and parking functions) in terms of the representation
theory of the algebras $\Lambda_n$, see Theorem \ref{soc}.
Throughout the survey, we try to stress the relevance of antichains in additive
categories.
\end{note}

\begin{note}\label{number}
This is not an exaggeration: there is the famous list by R.~Stanley \cite{[S1]}
on problems which yield the Catalan numbers. There, he exhibits 66 different problems, 
and many additional ones can be found in his Catalan Addendum \cite{[S2]}.
\end{note}	

\begin{note}\label{DFG}
The author was project leader at the CRC 701 until June 2013, thus he
wants to thank the DFG for the corresponding financial support. 
\end{note}	
\vfill\eject

\section{\bf Numbers}

\subsection{The Setting} On the one hand, there is the combinatorial setting
of the (finite) root systems and the corresponding hyperplane arrangements. 
On the other hand, there is the representation theoretical
setting of dealing with a hereditary artin algebra $\Lambda$ and its module
category $\mo\Lambda$. 
	\medskip

\subsubsection{\bf The combinatorial setting.}\label{comb} 
We consider a (finite) root system $\Phi = \Phi(\Delta)$
in the Euclidean space $V$,
it is the disjoint union of irreducible root systems. 
A connected root system is of type $\Delta$,
where 
$\Delta = \Delta_n$ is a Dynkin diagram with $n$ vertices,
thus one of $\mathbb A_n, \dots, \mathbb G_2$. 
Given a root system $\Phi$ in a vector space $V$, 
we choose a root basis $\Delta$, this is a basis of $V$ which consists of elements of $\Phi$ 
(they are called {\it simple roots})
such that all elements of $\Phi$ are linear combinations of these basis elements with integer
coefficients which are either non-negative (the {\it positive roots}) or non-positive 
(the {\it negative roots}). We recall the relevant
definitions in note N\,\ref{system} at the end of the chapter.

We denote by $\Phi_+$ the set of positive roots, this is a poset
with respect to the ordering $x\le y$ iff $y-x$ is a non-negative linear combination of positive
roots (or of simple roots). 
The objects to be considered in these lectures are the root posets 
	\smallskip 

\Rahmen{$ \Phi_+ = (\Phi_+,\le)$}
	\medskip
	
If we write a positive root $\alpha$ 
as a linear combination of simple roots, the sum of the
coefficients is called the {\it height} of $\alpha$. For any pair $t\le t'$ of
natural numbers, we denote by $\Phi_{t,t'}$ the subposet of $\Phi_+$ given by
the positive roots with height $s$, where $t\le s \le t'.$  

The root posets
$\Phi_+(\mathbb A_n)$ may be identified with the poset of all intervals
of the form $[i,j]$ where $0 \le i < j \le n$ are integers, using as
ordering the set-theoretical inclusion, see the note N\,\ref{A}.
	\medskip 

Let us exhibit the Hasse diagrams of some of the root posets, namely 
the cases $\mathbb A_5, \mathbb B_5, \mathbb C_5, \mathbb D_6$, as well as
all the exceptional cases.
Let us mention some of the properties of a root system which can be seen quite
nicely by looking at these Hasse diagrams.
\begin{itemize}
\item The lowest row consists of the simple roots, there are 
 precisely $n$ simple roots.
\item The row above the lowest row consists of the roots of height $2$,
  there are  precisely $n-1$ roots of height $2$,
  they correspond bijectively to the edges of the Dynkin diagram. In fact, the
  subposet $\Phi_{1,2}$ given by the two lowest rows may be considered as 
  the incidence graph
  of the Dynkin diagram. 
\item There is a unique maximal element in $\Phi_+$: the highest
  root of $\Phi$.
\item The numbers $r_t$ of roots of height $t$ form a decreasing
 sequence, this will be discussed in detail in Section 1.4.
\item The Hasse diagram of $\Phi_+$ may be considered as the 2-dimensional
  projection of a 3-dimensional object formed by cubes, squares and intervals.
  This will be discussed in the note N\,\ref{3-dim}.
\end{itemize}

	\smallskip 

\noindent 
Note that the poset $\Phi_+(\mathbb B_n)$ and  $\Phi_+(\mathbb C_n)$ are isomorphic
for any $n$ (this is no longer the case, if we take into account
the different root lengths, see the note N\,\ref{long}).
  
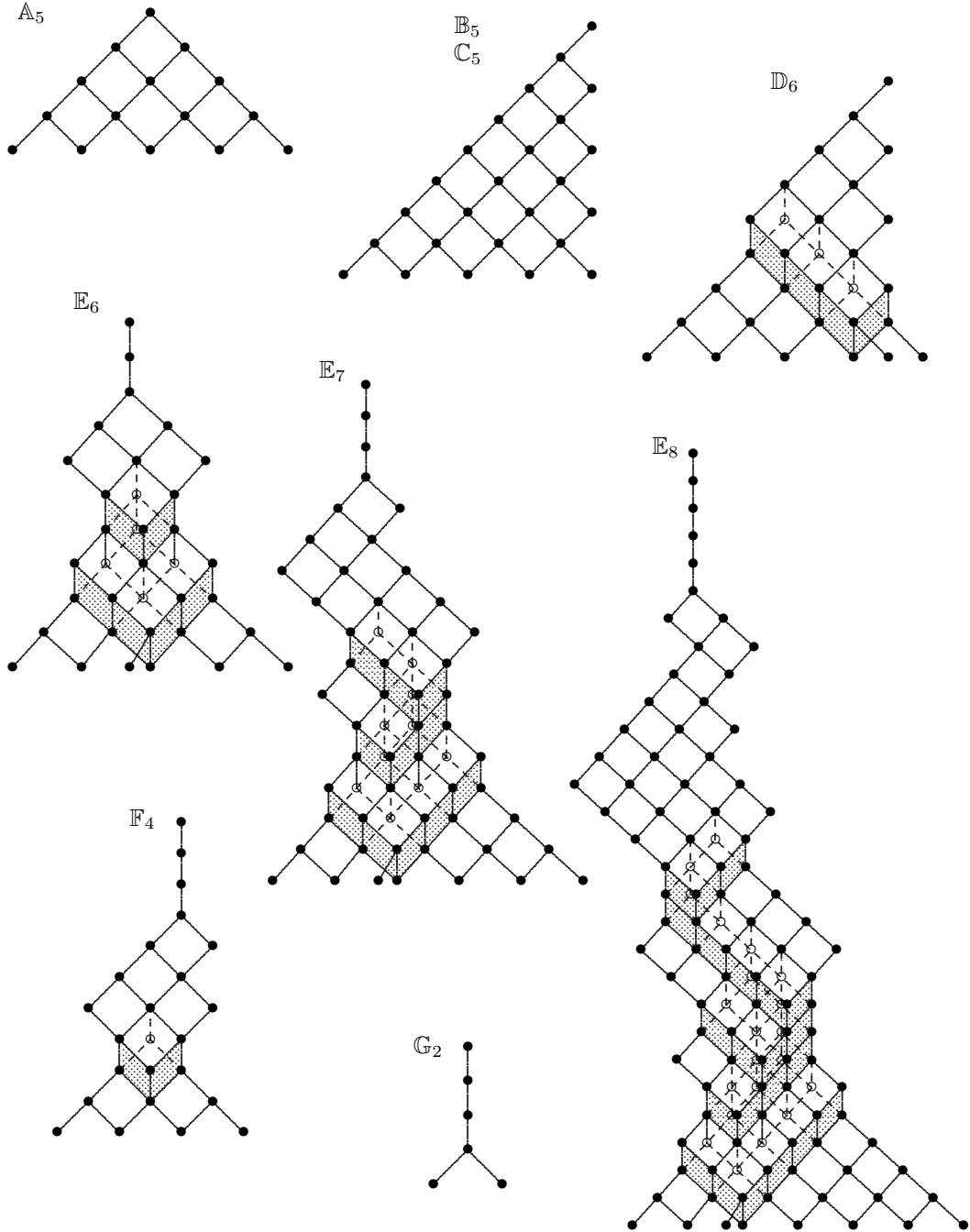
$$
\hbox{\beginpicture
\setcoordinatesystem units <1.15cm,1cm>
\put{\beginpicture
\setcoordinatesystem units <.5cm,.5cm>
\multiput{$\bullet$} at  
  0 0  2 0  4 0  6 0  8 0 
  1 1  3 1  5 1  7 1
  2 2  4 2  6 2
  3 3  5 3
  4 4 /
\plot 0 0  4 4  8 0 /
\plot 1 1  2 0  5 3 /
\plot 2 2  4 0  6 2 /
\plot 3 3  6 0  7 1 /    
\endpicture} at 0 1
\put{\beginpicture
\setcoordinatesystem units <.45cm,.45cm>
\multiput{$\bullet$} at  
   0 0  2 0  4 0  6 0  8 0  
      1 1   3 1  5 1  7 1  
        2 2  4 2  6 2  8 2
            3 3  5 3  7 3
             4 4  6 4  8 4
                5 5  7 5
                  6 6  8 6
                     7 7
                       8 8 /
\plot 0 0  8 8 /
\plot 1 1  2 0  8 6  7 7 /
\plot 2 2  4 0  8 4  6 6 /
\plot 3 3  6 0  8 2  5 5 /
\plot 4 4  8 0 /
\endpicture} at 4 0 
\put{\beginpicture
\setcoordinatesystem units <.5cm,.5cm>
\multiput{$\bullet$} at 0 0  2 0  4 0  6 0  8 0  
     1 1  3 1  5 1  7 1
     2 2  4 2 
     3 3 
      6 1
     5 2  7 2
     4 3  6 3   
     3 4  5 4  7 4 
     4 5  6 5 
     5 6  7 6 
      6 7 
     7 8 
     7 0 /
               
\plot 0 0  3 3  6 0  7 1  8 0 /
\plot 1 1  2 0  4 2 /
\plot 2 2  4 0  5 1 /

 \plot 7 0  3 4  7 8 /
\plot 6 1  7 2  4 5 /
\plot 5 2  7 4  5 6 /
\plot 4 3  7 6  6 7 /
\plot 3 3  3 4 /
\plot 4 2  4 3 /
\plot 5 1  5 2 /
\plot 6 0  6 1 /
\plot 7 1  7 2 /

\setdashes <1mm>
\plot  3 3  4 4  7 1 /
\plot 4 4  4 5 /
\plot 4 2  5 3  5 4 /
\plot 5 1  6 2  6 3 /
   
\multiput{$\circ$} at 4 4 5 3  6 2 /
\setshadegrid span <.4mm>
\vshade 3 3 4  <,z,,> 6 0 1  <z,,,> 7 1 2 /

\endpicture} at 8 -1 
\put{\beginpicture
\setcoordinatesystem units <.5cm,.5cm>
\multiput{$\circ$} at 2.7 3  4.7 3  3.8 2  3.6 5  3.6 4 /
\multiput{$\bullet$} at
  0 0  2 0  4 0  6 0  8 0
    0.9 1  2.9 1  4.9 1  6.9 1
      1.8 2  5.8 2  / 
      \plot 0 0  1.8 2  /
      \plot 5.8 2   8 0 /
      \plot 0.9 1  2 0  2.8 1 /
      \plot 1.8 2  4 0  5.8 2 /
      \plot 4.9 1  6 0 6.9 1 /

      \setdashes <1mm>
      \plot 1.8 2  3.6 4  5.8 2 /
      \plot 2.8 1  4.7 3 /
      \plot  2.7 3  4.9 1 /
      \plot 3.8 2   3.8 3  /
      \plot 3.6 4  3.6 5 /
      \plot 2.7 4   3.6 5  4.7 4 /
      \plot 3.6 5  3.6 6  /

      \setsolid
      \multiput{$\bullet$} at 4 1
        2.9 2  4.9 2
	  1.8 3  3.8 3  5.8 3
	    2.7 4  4.7 4  / %
	    \plot 4 0  4 1 /
	    \plot   2.9 1   2.9 2 /
	    \plot 4.9 1  4.9 2 /
	    \plot  1.8 2    1.8 3  /
	    \plot 5.8 2  5.8 3 /
	    \plot  2.7 3    2.7 4  /
	    \plot 4.7 3  4.7 4 /
	    \plot   4 1  5.8 3  4.7 4 /
	    \plot 2.7 4
	        1.8 3  4 1 / 
		\plot 2.9 2  4.7 4 /
		\plot  4.9 2   2.7 4  /

		\multiput{$\bullet$} at 3.8 4
		    2.7 5  4.7 5
		        3.6 6 /
			\plot  3.8 3  3.8 4 /
			\plot 2.7 4  2.7 5  /
			\plot 4.7 4  4.7 5 /
			\plot 3.8 4
			    2.7 5   3.6 6  4.7 5
			       3.8 4 /

			       \setshadegrid span <.4mm>
			       \vshade 1.8 2 3  <,z,,> 4 0 1  <z,,,> 5.8 2 3 /
			       \vshade 2.7 4 5  <,z,,> 3.8 3 4  <z,,,> 4.7 4 5 /

			       \multiput{$\bullet$} at 1.6 6  5.6 6  2.5 7  4.5 7  3.4 8  /
			       \plot 2.7 5  1.6 6  3.4 8  5.6 6  4.7 5 /
			       \plot 2.5 7  3.6 6  4.5 7 /

			       \multiput{$\bullet$} at 3.4 9  3.4 10  /
			       \plot 3.4 8  3.4 10 /

			       \put{$\bullet$} at 3.4 0
			       \plot 3.4 0  4 1 /
			       
\endpicture} at 0 -5
\put{\beginpicture
\setcoordinatesystem units <.45cm,.45cm>

\multiput{$\circ$} at 3.8 2
       2.7 3  4.7 3
       3.6 4  5.6 4
       4.5 5   4.5 6 
       3.6 5 
       4.5 7  3.4 8 /

\multiput{$\bullet$} at
  0 0  2 0  4 0  6 0  8 0  10 0  
  0.9 1  2.9 1  4.9 1  6.9 1  8.9 1
  1.8 2  5.8 2  7.8 2
  6.7 3 /  
\plot 0 0  1.8 2  /
\plot 5.8 2   8 0  8.9 1 /
\plot 0.9 1  2 0  2.8 1 /
\plot 1.8 2  4 0  5.8 2 /
\plot 4.9 1  6 0  7.8 2 /
\plot 5.8 2  6.7 3  10 0 /

\setdashes <1mm>
\plot 1.8 2  3.6 4  5.8 2 / 
\plot 3.6 4  4.5 5  6.7 3 /
\plot 2.8 1  5.6  4 /
\plot  2.7 3  4.9 1 /
\plot 3.8 2   3.8 3  /
\plot 3.6 4  3.6 5 /
\plot 2.7 4   3.6 5  4.7 4 /
\plot 3.6 5  3.6 6  /
\plot 5.6 4  5.6 5 /
\plot 4.5 5  4.5 6 /
\plot 3.6 5  4.5 6  5.6 5 / 
\plot  4.5 6  4.5 7 /
\plot 2.5 7  3.4 8  5.6 6 /
\plot 3.6 6  4.5 7 / 
\plot 4.5 7  4.5 8 /

\plot 3.4 8  3.4 9 /

\setsolid 
\multiput{$\bullet$} at 4 1
  2.9 2  4.9 2 
  1.8 3  3.8 3  5.8 3
  2.7 4  4.7 4 
   6.7 4  5.6 5 / 
\plot 4 0  4 1 /
\plot 2.9 1   2.9 2 /
\plot 4.9 1  4.9 2 /
\plot 1.8 2    1.8 3  /
\plot 5.8 2  5.8 3 /
\plot 2.7 3    2.7 4  /
\plot 4.7 3  4.7 4 /
\plot 4 1  5.8 3  4.7 4 /
\plot 2.7 4  
    1.8 3  4 1 / 
\plot 2.9 2  4.7 4 /
\plot  4.9 2   2.7 4  /
\plot 6.7 3  6.7 4  5.8 3  /
\plot 6.7 4  5.6 5  4.7 4 /
\plot 5.6 5  5.6 6 /

\multiput{$\bullet$} at 3.8 4
    2.7 5  4.7 5
    3.6 6 /
\plot  3.8 3  3.8 4 /
\plot 2.7 4  2.7 5  /
\plot 4.7 4  4.7 5 /
\plot 3.8 4
    2.7 5   3.6 6  4.7 5
   3.8 4 /

\plot  5.6 6  4.7 5 /

\multiput{$\bullet$} at   0.3 10  1.4 9  2.5 8   3.6 7  4.7 6  
     1.2 11   2.3 10  3.4 9  4.5 8   5.6 7 
     2.1 12  3.2 11  4.3 10  5.4 9  6.5 8   
     3. 13  4.1 12
 /
\plot 2.1 12  3 13  4.1 12  3.2 11 /
\plot 0.3 10  4.7 6  6.5 8  2.1 12  0.3 10 /
\plot 1.2 11  5.6 7 /

\plot 1.4 9  3.2 11 /
\plot 2.5 8  4.3 10 /
\plot 3.6 7  5.4 9 /

\plot 4.7 5  4.7 6 /
\plot 3.6 6  3.6 7 /
\plot 5.6 6  5.6 7 /
\plot 2.5 7  2.5 8 /

\multiput{$\bullet$} at  3  14  3 15  3 16  /
\plot 3  13  3 16 /

\setshadegrid span <.4mm>
\vshade 1.8 2 3  <,z,,> 4 0 1  <z,,,> 6.7 3 4 /
\vshade 2.7 4 5  <,z,,> 3.8 3 4  <z,,,> 5.6  5 6 /
\vshade 2.5 7 8  <,z,,> 4.7 5 6  <z,,,> 5.6   6 7 /

\multiput{$\bullet$} at 1.6 6   5.6 6  2.5 7  /

\plot 2.7 5  1.6 6   2.5 7 /

\plot 2.5 7  3.6 6 /

\put{$\bullet$} at 3.4 0 
\plot 3.4 0  4 1 /

\endpicture} at 3.5 -7
\put{\beginpicture
\setcoordinatesystem units <.4cm,.4cm>
\multiput{$\circ$} at 3.8 2
       2.7 3  4.7 3
       3.6 4  5.6 4
       4.5 5  6.5 5
       5.4 6  
       3.6 5  
       4.5 6   
       5.4 7 
       4.5 7  3.4 8 
        5.4 8 
       4.3 9 
      2.1 12  3.2 11  4.3 10  5.4 9   /
\multiput{$\bullet$} at
  0 0  2 0  4 0  6 0  8 0  10 0  
  0.9 1  2.9 1  4.9 1  6.9 1  8.9 1
  1.8 2  5.8 2  7.8 2
  6.7 3 
  7.6 4  8.7 3  9.8 2  10.9 1  12 0 
   7.6 5  6.5 6  6.5 7 
/  
\plot 0 0  1.8 2  /
\plot 5.8 2   8 0  9.8 2 /
\plot 0.9 1  2 0  2.8 1 /
\plot 1.8 2  4 0  5.8 2 /
\plot 4.9 1  6 0  8.7 3 /
\plot 5.8 2  6.7 3  10 0  10.9 1 /
\plot  7.6 4  12 0 /
\plot 6.7 3  7.6 4  7.6 5  6.5 6  5.6 5 /
\plot 6.5 7  5.6 6 /
\plot 6.5 6  6.5 8 /
\plot 6.7 4 7.6 5 /

\setdashes <1mm>
\plot 1.8 2  3.6 4  5.8 2 / 
\plot 3.6 4  4.5 5  6.7 3 /
\plot 2.8 1  5.6  4 /
\plot  2.7 3  4.9 1 /
\plot 3.8 2   3.8 3  /
\plot 3.6 4  3.6 5 /
\plot 2.7 4   3.6 5  4.7 4 /
\plot 3.6 5  3.6 6  /
\plot 5.6 4  5.6 5 /
\plot 4.5 5  4.5 6 /
\plot 3.6 5  4.5 6  5.6 5 / 
\plot  4.5 6  4.5 7 /
\plot 2.5 7  3.4 8  5.6 6 /
\plot 3.6 6  4.5 7 / 
\plot 4.5 7  4.5 8 /
\plot 3.4 8  3.4 9 /

\plot 7.6 4  5.4 6  4.5 5 /
\plot 6.5 6  6.5 5  5.6 4 /
\plot 4.5 6  5.4 7  6.5 6 /
\plot 5.4 7  5.4 6 /

\plot 5.4 7  5.4 9 /
\plot 4.5 7  5.4 8   6.5 7 /
\plot 3.4 8  4.3 9  5.4 8 /
\plot 4.3 9  4.3 10 /

\setsolid 
\multiput{$\bullet$} at 4 1
  2.9 2  4.9 2 
  1.8 3  3.8 3  5.8 3
  2.7 4  4.7 4 
   6.7 4  5.6 5 / 
\plot 4 0  4 1 /
\plot 2.9 1   2.9 2 /
\plot 4.9 1  4.9 2 /
\plot 1.8 2    1.8 3  /
\plot 5.8 2  5.8 3 /
\plot 2.7 3    2.7 4  /
\plot 4.7 3  4.7 4 /
\plot 4 1  5.8 3  4.7 4 /
\plot 2.7 4  
    1.8 3  4 1 / 
\plot 2.9 2  4.7 4 /
\plot  4.9 2   2.7 4  /
\plot 6.7 3  6.7 4  5.8 3  /
\plot 6.7 4  5.6 5  4.7 4 /
\plot 5.6 5  5.6 6 /

\multiput{$\bullet$} at 3.8 4
    2.7 5  4.7 5
    3.6 6 /
\plot  3.8 3  3.8 4 /
\plot 2.7 4  2.7 5  /
\plot 4.7 4  4.7 5 /
\plot 3.8 4
    2.7 5   3.6 6  4.7 5
   3.8 4 /

\plot  5.6 6  4.7 5 /

\multiput{$\bullet$} at   0.3 10  1.4 9  2.5 8   3.6 7  4.7 6  
     1.2 11   2.3 10  3.4 9  4.5 8   5.6 7  6.5 8  /
\plot   1.2 11 0.3 10  4.7 6  6.5 8 /

\plot 1.2 11  5.6 7 /

\plot 1.4 9  2.3 10 /
\plot 2.5 8  3.4 9  /
\plot 3.6 7  4.5 8  /

\plot 4.7 5  4.7 6 /
\plot 3.6 6  3.6 7 /
\plot 5.6 6  5.6 7 /
\plot 2.5 7  2.5 8 /
\setdashes <1mm>
\plot  1.2 11  2.1 12  6.5 8  /
\plot  2.3 10  3.2 11  3.2 12  /
\plot  3.4 9   4.3 10  4.3 11 /
\plot  4.5 8   5.4 9   5.4 10 /
\plot 2.1 12  2.1 13 /
\plot 1.2 12  2.1 13  3.2 12 /

\plot 3 14  3 15 /
\plot 2.1 14  2.1 13  3 14  4.1 13 /
\setsolid 

\multiput{$\circ$} at 2.1 13  3.0 14  /
\multiput{$\bullet$} at    
    1.2 12  2.3 11  3.4 10  4.5 9  5.6 8 
    3.2 12  4.3 11  5.4 10  6.5 9 
    4.1 13  5.2 12  6.3 11  7.4 10 /
\plot 1.2 13  1.2 12  5.6 8  7.4 10  4.1 13  2.3 11  2.3 12  /
\plot 3.2 13  3.2 12  6.5 9 /
\plot  3.4 10  5.2 12 /
\plot 4.5 9  6.3 11 /

\plot 1.2 12  1.2 11 /
\plot 2.3 11  2.3 10 /
\plot 3.4 10  3.4 9 /
\plot 4.5  9  4.5 8 /
\plot 5.6  8  5.6 7 /
\plot 6.5  9  6.5 8 /

\plot 4.1 13  4.1 14 /

\setshadegrid span <.4mm>
\vshade 1.8 2 3  <,z,,> 4 0 1  <z,,,>  7.6 4 5 /
\vshade 2.7 4 5  <,z,,> 3.8 3 4  <z,,,> 6.5   6 7 /
\vshade 2.5 7 8  <,z,,> 4.7 5 6  <z,,,> 6.5   7 8 /
\vshade 1.2 11 12  <,z,,> 5.6 7 8  <z,,,> 6.5 8 9 /
\vshade 1.2 12 13  <,z,,> 2.3 11 12  <z,,,> 4.1 13 14 /

\multiput{$\bullet$} at 1.6 6   5.6 6  2.5 7  / 
\plot 2.7 5  1.6 6   2.5 7 /
\plot 2.5 7  3.6 6 /

\put{$\bullet$} at 3.4 0 
\plot 3.4 0  4 1 /

\multiput{$\bullet$} at 
  2.3 12  3.2 13  4.1 14  5 15 
  1.2 13  2.1 14  3 15  3.9 16 
  0.1 14  1 15  1.9 16  2.8 17  3.7 18 
  -1 15  -.1 16    .8 17  1.7 18  2.6 19  3.5 20  4.4 21 
  -2.1 16 -1.2 17 -.3 18  0.6 19  1.5 20    2.4 21  3.3 22               
                             1.3 22  2.2 23 /
\plot 3.3 22  -2.1 16  2.3 12  5 15  0.6 19 /

\multiput{$\bullet$} at  2.2 24  2.2  25  2.2 26  2.2 27  2.2 28 /
\plot 2.2 23  2.2 28  /
\plot 3.5 20  1.3 22  2.2 23  4.4 21  -1 15 /
\plot 0.1 14  3.7 18  1.5 20 /
\plot 1.2 13  3.9 16 /
\plot 3.2 13 -1.2 17 /
\plot 4.1 14  -.3 18 /

\endpicture} at 7.8 -10
\put{\beginpicture
\setcoordinatesystem units <.45cm,.45cm>

\multiput{$\bullet$} at 
     0 0  2 0  4 0  6 0 
     1 1  3 1  5 1
     2 2  3 2  4 2 
     2 3  4 3
     1 4  3 4  5 4
     2 5  4 5
     3 6  5 6
     4 7  4 8  4 9  4 10  /
\multiput{$\circ$} at 3 3 /
\plot  0 0  2 2  4 0 5 1  /
\plot 1 1  2 0  4 2  6 0 /
\plot 3 1  3 2 /
\plot 2 2  2 3  4 5 /
\plot 4 2  4 3  2 5 /
\plot 3 2  5 4  3 6  1 4  3 2 /
\plot 3 6  4 7  5 6  4 5 /
\plot 4 7  4 10 /

\setdashes <1mm>
\plot 2 2  3 3  4 2 /
\plot 3 3  3 4 /

\setshadegrid span <.4mm>
\vshade 2 2 3  <,z,,> 3 1 2  <z,,,>  4 2 3  /
\endpicture}  at 0 -12
\put{\beginpicture
\setcoordinatesystem units <.5cm,.5cm>

\multiput{$\bullet$} at 0 0  2 0
  1 1   1 2  1 3  1 4 /
  \plot 0 0  1 1  2 0 /
  \plot 1 1  1 4 /
  \endpicture} at 4 -14
\put{$\mathbb A_5$} at -1.5 2
\put{$\mathbb B_5$} at 4 1.8
\put{$\mathbb C_5$} at 4 1.4
\put{$\mathbb D_6$} at 8 1
\put{$\mathbb E_6$} at -.8 -2.2
\put{$\mathbb E_7$} at 2.3 -3.2
\put{$\mathbb E_8$} at 6.5 -4.3
\put{$\mathbb F_4$} at -.1 -9.7
\put{$\mathbb G_2$} at 3.5 -13
\endpicture}
$$
	\medskip 

\subsubsection{\bf Roots and hyperplanes.}
A root system $\Phi$ is a subset of a Euclidean space $V$.
If $x$ is a root, we denote by $H_x$ the hyperplane orthogonal to $x$, and by $\rho_x$ the
reflection at $H_x$. 
We denote by 
	 
\Rahmen{$\H(\Phi)$}
	\smallskip

\noindent
the set of these hyperplanes. These sets will be considered in detail in Section 1.5.
Note that for any root $x$, we have $H_x = H_{-x}$, and we have
$H_x\neq H_y$ 
in case $y \neq \pm x$. This shows that $\H(\Phi)$ can be indexed by the set
$\Phi_+$ of positive roots. 

We will denote by $W$ the Weyl group, it is the subgroup of $\GL(V)$ generated by the 
reflections $\rho_x$ with $x\in \Delta$. The Weyl group contains all the 
reflections $\rho_x$ with
$x\in \Phi$, these are just all the reflections in $W$ (a 
{\it reflection} in $\GL(V)$ is
an element of finite order which fixes pointwise precisely a hyperplane), see 
the note N\,\ref{reflections}.

If we consider the set $x(1),\dots,x(n)$ of simple roots with a fixed ordering,
the corresponding Weyl group element $c = \rho_{x(n)}\cdots \rho_{x(2)}\rho_{x(1)}$
is called a Coxeter element in $W$. Actually, instead of fixing such an ordering, 
we usually work with an orientation $\Omega$ of the diagram $\Delta$
(but in the more general setting of dealing with arbitrary finite graphs, we only
will be interested in orientations without oriented cyclic paths).
To choose an {\it orientation} of a graph
means to replace any edge $\{x,y\}$ of the graph by one of the ordered pairs $(x,y)$,
$(y,x)$; if it is replaced by $(x,y)$, we will indicate this by drawing an 
arrow $x\to y$. If there are no oriented cyclic paths, we can order 
the vertices $x(1),\dots,x(n)$ in such a way that
the existence of an arrow $i \leftarrow j$ implies that $i < j$, and then we attach to it
the Coxeter element $c_\Omega = \rho_{x(n)}\cdots \rho_{x(2)}\rho_{x(1)}$;
note that $c_\Omega$
only depends on the orientation $\Omega$ and not on the actual ordering. 

Since a Dynkin diagram does not have
cycles, we see that  
{\it there is a bijection between the orientations $\Omega$ 
of $\Delta$ and the Coxeter elements $c_\Omega$ in $W$.} For example, the Dynkin diagram
$\mathbb A_3$ has four orientations (thus four Coxeter elements):
$$
{\beginpicture
\setcoordinatesystem units <.9cm,1cm> 
\put{\beginpicture
\setcoordinatesystem units <1cm,1cm> 
\multiput{$\circ$} at 0 0  1 0  2 0 /
\put{$\ssize 1$} at 0 -.3
\put{$\ssize 2$} at 1 -.3
\put{$\ssize 3$} at 2 -.3
\arr{0.2 0}{0.8 0}
\arr{1.8 0}{1.2 0}
\endpicture} at 3.5 0
\put{$\sigma_1\sigma_3\sigma_2 = \sigma_3\sigma_1\sigma_2 $} at 3.5 -.7
\put{\beginpicture
\setcoordinatesystem units <1cm,1cm> 
\multiput{$\circ$} at 0 0  1 0  2 0 /
\put{$\ssize 1$} at 0 -.3
\put{$\ssize 2$} at 1 -.3
\put{$\ssize 3$} at 2 -.3
\arr{0.8 0}{0.2 0}
\arr{1.2 0}{1.8 0}
\endpicture} at 7 0
\put{$\sigma_2\sigma_1\sigma_3 = \sigma_2\sigma_3\sigma_1 $} at 7 -.7
\put{\beginpicture
\setcoordinatesystem units <1cm,1cm> 
\multiput{$\circ$} at 0 0  1 0  2 0 /
\put{$\ssize 1$} at 0 -.3
\put{$\ssize 2$} at 1 -.3
\put{$\ssize 3$} at 2 -.3
\arr{0.8 0}{0.2 0}
\arr{1.8 0}{1.2 0}
\endpicture} at 0 0
\put{$\sigma_3\sigma_2\sigma_1$} at 0 -.7
\put{\beginpicture
\multiput{$\circ$} at 0 0  1 0  2 0 /
\put{$\ssize 1$} at 0 -.3
\put{$\ssize 2$} at 1 -.3
\put{$\ssize 3$} at 2 -.3
\arr{0.2 0}{0.8 0}
\arr{1.2 0}{1.8 0}
\endpicture} at 10.5 0
\put{$\sigma_1\sigma_2\sigma_3$} at 10.5  -.7
\put{} at 2.5 -1
\endpicture}
$$
We should stress that using Dynkin graphs with orientations, there are
two kinds of arrows which should not be confused: namely, besides the arrow
heads describing the orientation, the non-simply-laced Dynkin graphs $\mathbb B_n,
\mathbb C_n, \mathbb F_4, \mathbb G_2 $ also use an arrow (usually drawn in the middle of the
edge) which indicates the relative length of the corresponding roots. Thus, for example,
there are two different Dynkin quivers of type $\mathbb B_2$, namely
$$
{\beginpicture
\setcoordinatesystem units <1.5cm,.7cm> 
\put{\beginpicture
\multiput{$\circ$} at 0 0  1 0 /
\put{$\ssize 1$} at 0 -.3
\put{$\ssize 2$} at 1 -.3
\plot 0.2 0.075  0.8 0.075 /
\plot 0.2 -.075  0.8 -.075 /
\plot  0.6 0.2  0.4 0  0.6 -.2 /
\plot  0.3 0.2  0.1 0  0.3 -.2 /
\endpicture} at 0 0
\put{$\sigma_2\sigma_1$} at 0 -.7

\put{\beginpicture
\multiput{$\circ$} at 0 0  1 0 /
\put{$\ssize 1$} at 0 -.3
\put{$\ssize 2$} at 1 -.3
\plot 0.2 0.075  0.8 0.075 /
\plot 0.2 -.075  0.8 -.075 /
\plot  0.6 0.2  0.4 0  0.6 -.2 /
\plot  0.7 0.2  0.9 0  0.7 -.2 /
\endpicture} at 2.5 0
\put{$\sigma_1\sigma_2$} at 2.5 -.7
\put{} at 2.5 -1
\endpicture}
$$
In both cases, the basis root with label $1$ is a short root, the root with
label $2$ a long root (as indicated by the arrow head in the middle of the edge). 
	\medskip 

\subsubsection{\bf The representation theoretical setting.}
We start with a hereditary artin algebra $\Lambda$
and consider its module category $\mo\Lambda$ (modules are left $\Lambda$-modules of finite length).
 The {\it quiver} $Q(\mathcal C)$ of an abelian category $\mathcal C$ has as vertices
the isomorphism classes $[S]$ of the simple objects $S$, and there is an arrow 
$[T]\rightarrow [S]$ provided $\Ext^1(T,S) \neq 0$. 
Any arrow carries a valuation which records the dimensions of $\Ext^1(T,S)$
as a module over the endomorphism rings of $S$ and $T$, respectively. 
For the category $\mathcal C = \mo\Lambda$, we just
write $Q(\Lambda)$ instead of $Q(\mo\Lambda)$.    
If $\Lambda$ is 
representation-finite, the underlying valued graph of $Q(\Lambda)$
turns out to be the disjoint union of Dynkin diagrams and all Dynkin diagram 
arise in this way. For further details see the note N\,\ref{quiver}.
In Chapter 1, we usually will assume that $\Lambda$ is representation-finite.
	
The representation theoretical objects to be considered are the categories 
$\mo\Lambda$ where $\Lambda$ is a hereditary artin algebra (in Chapters 1 and 4, they will be 
assumed to be
in addition representation-finite).
We denote by $K_0(\Lambda)$ the Grothendieck group of $\Lambda$ (of all $\Lambda$-modules modulo 
exact sequences); this is the free abelian group with basis the isomorphism classes $[S]$ of the simple
modules; if $M$ is a module, $\bdim M$ is the corresponding element in $K_0(\Lambda)$, thus
$\bdim S = [S]$, and $\bdim M = \sum [M:S][S],$ where $[M:S]$ is the Jordan-H\"older multiplicity
of $S$ in $M$.

\begin{theorem} 
{\bf (Gabriel 1972, Dlab-Ringel 1973).} Let $\Lambda$ be a hereditary
artin algebra. Then $\Lambda$ is representation-finite if and only if the valued
quiver $Q(\Lambda)$ is the disjoint union of quivers of 
Dynkin type. If $\Lambda$ is of Dynkin type $\Delta$, then the map $\bdim$ 
provides a bijection between the
isomorphism classes of the indecomposable
$\Lambda$-modules and the positive roots of the root system of type $\Delta$.
\end{theorem}

We denote by $\ind\Lambda$ a set of indecomposable $\Lambda$ modules, one from each isomorphism
class. If $X,Y$ are in $\ind \Lambda$, we write $X \subfactor Y$ provided $X$ is a subfactor of $Y$,
thus provided there are submodules $Y'' \subseteq Y' \subseteq Y$ with $X$ isomorphic to $Y'/Y''$.

\begin{theorem} 
{\bf (Dlab-Ringel 1979).} 
Let $k$ be a field with at least $3$ elements. Let
$\Lambda$ be a hereditary finite-dimensional $k$-algebra of Dynkin type $\Delta$.
Then the map $\bdim$ 
provides  an isomorphism of posets
		\smallskip

\Rahmen{$\bdim\!: (\ind \Lambda,\subfactor)\longrightarrow  (\Phi_+,\le)$}
\end{theorem}

It should be stressed that the assumption $|k| \ge 3$ is necessary, see N\,\ref{why}.
	\medskip 

The theorem shows that the root poset $\Phi_+(\Delta)$ can be categorified by  
the indecomposable $\Lambda$-modules, where $\Lambda$ is a hereditary
artin algebra of Dynkin type $\Delta$. We should stress the following: whereas
the category $\ind\Lambda$ strongly depends on the orientation
of the Dynkin quiver $Q(\Lambda)$, the poset $(\ind \Lambda,\subfactor)$ does not
depend on the orientation (as the poset isomorphism  
$\bdim\!: (\ind \Lambda,\subfactor)\ \longrightarrow  (\Phi_+,\le)$ shows).
	\medskip

\subsubsection{} The aim of this survey is to exhibit 
combinatorial data which can be derived from the category $\ind\Lambda$, 
and we are going to emphasize
those which do not depend on the orientation of $Q(\Lambda)$ (at least in case
$Q(\Lambda)$ is a tree). Of particular importance seems to be the set of all  
antichains in $\mo\Lambda$: 
An {\it antichain} in an additive category $\mathcal C$ is a set of pairwise orthogonal bricks
(a brick is an object whose endomorphism ring is a division ring); the note
N\,\ref{antichain} will provide an explanation for the terminology. 
In case
$\mathcal C$ is abelian, starting with an antichain $A$, we can consider the full subcategory 
$\mathcal E(A)$ 
of all objects with a filtration with factors in the antichain: this is a {\it thick} subcategory 
(an exact abelian subcategory which is closed under extensions), again let us refer to some
comments in the note N\,\ref{thick}.
Conversely, given a thick subcategory of $\mathcal C$, the simple objects
in this subcategory form an antichain (this is just Schur's Lemma).
Thus, there is an obvious bijection between antichains and
thick subcategories. 
An antichain is said to be 
{\it exceptional} provided the quiver of $\mathcal E(A)$ has no oriented cyclic paths. 
For hereditary artin algebras $\Lambda$, every antichain is exceptional 
if and only if $\Lambda$ is
representation-finite.  Chapter 3 will be devoted to the study of the poset
	\medskip

\Rahmen{$\A(\mo\Lambda)$}
	\medskip

\noindent 
of exceptional subcategories of $\mo\Lambda.$

\subsection{Dynkin functions}

As the title of part 1 indicates, this part is devoted to numbers, to numbers
which arise from counting problems dealing mainly with representation-finite
hereditary artin algebras $\Lambda$. The numbers we are interested in will depend just
on the Dynkin type of $\Lambda$, but not on the orientation of $Q(\Lambda)$. Thus,
here we deal with what we want to call Dynkin functions.   
	\medskip

\subsubsection{\bf Definition.} 
A {\it Dynkin function} $\f$ attaches to any Dynkin diagram an integer (or more
generally a real number, sometimes even a set
or a sequence of real numbers, for example the sequence of exponents);
thus we get four sequences of numbers, namely 
$$
 \f(\mathbb A_n),\ \f(\mathbb B_n),\ \f(\mathbb C_n),\ \f(\mathbb D_n)
$$
as well as five additional single values 
$$
 \f(\mathbb E_6),\ \f(\mathbb E_7),\ \f(\mathbb E_8),\ \f(\mathbb F_4),\ \f(\mathbb G_2).
$$
(Sometimes, such a function is also defined for the remaining Coxeter diagrams $\mathbb I_2(t),
\mathbb H_3, \mathbb H_4.$)  

Let us draw the attention to Sloane's OEIS \cite{[Sl]}, the {\it Online Encyclopedia of Integer
Sequences.} This is a marvelous tool when dealing with integer sequences, however in
our context it would be nice to be able to use a similar data bank, an OEDF
({\it Online Encyclopedia of Dynkin Functions}), so that the integer sequences which arise
say in case $\mathbb A_n$ immediately refer to corresponding sequences which
arise in the cases $\mathbb B_n, \mathbb C_n, \mathbb D_n$, and also to the numbers which
occur for the exceptional cases. If we look at integral sequences $a = (a_1,a_2,a_3,\dots)$,
it is usually difficult to predict whether for a finite set of indices $n$, the values
$a_n$ are relevant to determine the whole sequence $a$. In the case of a Dynkin
function $\f$, we are in another realm: here the three numbers $\f(\mathbb E_n)$ with
$n = 6,7,8$ seem always to be exciting numbers.
	\medskip 

\subsubsection{\bf Examples.}
Here are some Dynkin functions which one may consider starting with 
hereditary artin algebras $\Lambda$ of Dynkin type $\Delta_n$.
Let us stress that formulations concerning counting of modules are meant as a short
form for counting isomorphism classes of modules.
	\medskip 

$r(\Delta_n)$ the number of indecomposable modules (thus, the number of positive roots).
	\medskip
 
$r_t(\Delta_n)$ the number of indecomposable modules of length $t$ (thus, the number of
positive roots of height $t$).
	\medskip 

$\sinc(\Delta_n)$ the number of sincere indecomposable modules (thus, 
   the number of sincere positive roots). We recall that the {\it support}
   of a module $M$ is the set of simple modules which occur as subfactor of $M$, and
   $M$ is called
   {\it sincere} provided any simple module belongs to its support. 
	\medskip 

$\d(\Delta_n)= x_1\cdots x_n\cdot x_{\text{th}}$, where $x = (x_1,\dots,x_n)$ 
    is the highest root
    (note that $x = \bdim M$, where $M$ is the indecomposable module of maximal length) 
     and $x_{\text{th}}$ is equal to  $1$ plus the number of indices $i$ with $x_i = 1$
    (it is well-known that $x_{\text{th}}$ is just the determinant of the 
    Cartan matrix of type $\Delta_n$).
   For example, the highest root for $\mathbb E_6$ is 
$$
 x = \smallmatrix & & 2\cr 1&2&3&2&1\endsmallmatrix,
$$
thus $\d(\mathbb E_6) = 72,$ namely $1\cdot2\cdot3\cdot2\cdot1\cdot2\cdot x_{\text{th}}$ and $x_{\text{th}} = 
1+2 = 3.$ Similarly,  the highest root for $\mathbb A_n$ is $x = \smallmatrix 1&1&\dots&1\endsmallmatrix$,
thus $x_{\text{th}} = 1+n$, therefore $\d(\mathbb A_n) = n+1.$ 
	\medskip 

$\c(\Delta_n)$ the number of complete exceptional sequences. Recall that a sequence
$(E_1,\dots,E_t)$ of indecomposable $\Lambda$-modules is said to be {\it exceptional}
provided \linebreak 
$\Ext^1(E_i,E_j) = 0$ for $i \ge j$ and $\Hom(E_i,E_j) = 0$ for $i > j$
(in case $t = 2$, one calls it an {\it exceptional pair}). 
An exceptional sequence $(E_1,\dots,E_t)$ is
said to be {\it complete} provided $t$ is equal to the number of simple $\Lambda$-modules.
	\medskip 

$\a(\Delta_n)$ the number of antichains in $\mo\Lambda$.
	\medskip 

$\a_t(\Delta_n)$ the number of antichains in $\mo\Lambda$ of cardinality $t$.
	\medskip 

$\t(\Delta_n)$ the number of (multiplicity-free) support-tilting modules. We recall that a 
    multiplicity-free module $T = \bigoplus_{i=1}^t$ with
    indecomposable direct summands $T_i$ is said to be a {\it support-tilting} module, provided
    $\Ext^1(T,T) = 0$ and $t$ is the cardinality of its support.
	\medskip
 
$\t_s(\Delta_n)$ the number of (multiplicity free) support-tilting modules which are direct sums
    of $s$ indecomposable modules (thus with support of cardinality $s$). 
    In particular:
	\medskip 

$\t_n(\Delta_n)$ is the number of (multiplicity free) tilting modules (since a sincere
    support-tilting module is just a tilting module)
	\bigskip

Of course, there are further important Dynkin functions, for example:
	\medskip 

$|W|$ the order of the Weyl group $W$.
	\medskip 

$|W_i|$ where $W_i$ is the set of elements 
$w\in W$ with fixed point space of dimension $n-i.$ In particular
$W_0$ consists just of the identity element, $W_1$ is the set of reflections.
	\medskip 

\subsubsection{\bf Verification.} 
In order to see that we deal with Dynkin functions, it is necessary to check in
any case that the numbers in question are independent of the orientation. For example, let us
do this for the number $\a_t(\Delta)$ of antichains of cardinality $t$
(or even for the number $\a_t^s(\Delta)$ of antichains of cardinality $t$ with support of
cardinality $s$), since finally these
antichains will be our main concern.
	
\begin{lemma}\label{BGP}
The number $\a_r^s$ of antichains of cardinality $r$
with support of cardinality $s$ does not depend on the orientation.
\end{lemma}

\begin{proof} 
We consider the set $\A_r^s(\mo\Lambda)$ of antichains of cardinality $r$
with support of cardinality $s$.
Let $x$ be a sink and $\rho_x$ the
BGP-reflection functor for $x$, let $\Lambda' =  \rho_x\Lambda.$
Let $S$ be the simple $\Lambda$-module with support $x$, and $S'$ the
simple $\Lambda'$-module with support $x$.

$$
{\beginpicture
\setcoordinatesystem units <1cm,1cm> 
\put{\beginpicture
\multiput{} at 0 -1  2 1 /
\plot .9 -1 2 -1  2 1  .9 1  .9 -1 /
\put{$x$} at 0 0
\arr{0.7 0.6}{0.3 0.3}
\arr{0.7 0}{0.3 0}   
\arr{0.7 -.6}{0.3 -.3}
\plot -.2 -.2  .2 -.2  .2 .2  -.2 .2  -.2 -.2 /
\put{$\Lambda$} at -.4 .9
\put{$\Lambda''$} at 1.45 0
\endpicture} at 0 0
\put{\beginpicture
\multiput{} at 0 -1  2 1 /
\plot .9 -1 2 -1  2 1  .9 1  .9 -1 /
\put{$x$} at 0 0
\arr{0.3 0.3}{0.7 0.6}
\arr{0.3 0}{0.7 0}   
\arr{0.3 -.3}{0.7 -.6}
\plot -.2 -.2  .2 -.2  .2 .2  -.2 .2  -.2 -.2 /
\put{$\Lambda'$} at -.4 .9
\put{$\Lambda''$} at 1.45 0
\endpicture} at 5 0
\endpicture}
$$
	\medskip 

We claim that there is a bijection
$$
  \eta\!: \A_r(\mo\Lambda) \to \A_r(\mo\Lambda').
$$
Thus, let $A$ be an antichain in $\mo\Lambda$ of cardinality $r$. 

Case 1: If $S$ belongs to $A$, 
then for the remaining modules $A_i$ in $A$, 
we have $(A_i)_x = \Hom(S,A_i) = 0$, thus $A_i$ may be considered as a 
$\Lambda'$-module and
$A' = A\setminus \{S\}\cup \{S'\}$ is an antichain  
in $\mo\Lambda'$ which contains $S'$. 

Case 2. Now assume that 
$S$ does not belong to $A$, then $A' = \{\rho_x(A_i)\mid
A_i\in A\}$ is an antichain in $\mo\Lambda'$ 
which does not contain $S'$. 
	\medskip

Now, let us refine this to deal with $\A_r^s.$

Case 1 is as before: If $S$ belongs to the antichain
$A$, then for the remaining modules $A_i$ in $A$, 
we have $(A_i)_x = \Hom(S,A_i) = 0$, thus $A_i$ may be considered 
as a $\Lambda'$-module and
$A' = A\setminus \{S\}\cup \{S'\}$ is an antichain  
in $\mo\Lambda'$ which contains $S'$. Clearly, both $A$ and $A'$ 
have the same support.

Case 2. Now assume that 
$S$ does not belong to $A$, then $A' = \{\rho_x(A_i)\mid
A_i\in A\}$ is an antichain in $\mo\Lambda'$ which does 
not contain $S'$. 

There are four possibilities, whether $x$ belongs to the support of $A$ or to the support
of $A'$. The cardinality of the support changes in case $x$ does not belong to the support of $A$
but to the support of $A'$ or the other way round.  
Let $\bold B$ be the antichains $A$ such that
$x$ does not belong to the support of $A$, but to the support of $A'$. 
Then these are the
antichains in $\mo\Lambda''$ such that at least one of the elements lives on a vertex which is a neighbor
of $x$. But then the antichains in $\bold B$ (considered as antichains 
of $\Lambda'$-modules)
are precisely the antichains $\rho_x(A)$, where $A$ is an antichain of cardinality $r$
with support of cardinality $s$ such that  $\rho_x(A)$ has support of cardinality $s-1$.
\end{proof}
	
In order to obtain $A'$ from $A$, we have distinguished two cases: we used the
functor $\rho_x$ whenever possible, otherwise we see that the support 
of any element $A_i$ of $A$
is either $\{x\}$ or does not involve $x$, so that $A_i$ is already a $\Lambda'$-module. Thus we can
write: $A' = A$. 
If we look at the corresponding dimension vectors, we see that we deal with {\bf piecewise linear}
functions: the dimension vectors of $A'$ are obtained from those of $A'$ partly
by using the reflection $\rho_x$, partly by taking the identity map.
	\medskip

\subsubsection{\bf Table 1.}

 Here are the values for some of these Dynkin functions.
$$
{\beginpicture
\setcoordinatesystem units <1.6cm,1cm> 
\put{\beginpicture
\multiput{} at -.2 2  9 2 /
\put{$\Delta_n$} at -0.2 2 
\put{$\mathbb A_n$} at 1 2 
\put{$\mathbb B_n$} at 2 2 
\put{$\mathbb D_n$} at 3 2 
\put{$\mathbb E_6$} at 4 2 
\put{$\mathbb E_7$} at 5 2 
\put{$\mathbb E_8$} at 6 2 
\put{$\mathbb F_4$} at 7 2 
\put{$\mathbb G_2$} at 8 2 
\endpicture} at 0 -.7
\plot -5 -1.2  4.2 -1.2 /
\setcoordinatesystem units <1.6cm,1cm> 
\put{\beginpicture
\multiput{} at -.5 0  9 0 /
\put{$h$} at -.2 0

\put{$n+1$} at 1 0
\put{$2n$} at 2 0
\put{$2(n-1)$} at 3 0
\put{$12$} at 4 0
\put{$18$} at 5 0
\put{$30$} at 6 0
\put{$12$} at 7 0
\put{$6$} at 8 0
\put{$\ssize 2^2\cdot 3$} at 4 -0.35
\put{$\ssize 2\cdot 3^2$} at 5 -0.35
\put{$\ssize 2\cdot 3\cdot 5$} at 6 -0.35
\put{$\ssize 2^2\cdot 3$} at 7 -0.35
\put{$\ssize 2\cdot 3^{\phantom2}$} at 8 -0.35
\endpicture} at 0 -2

\setcoordinatesystem units <1.6cm,1cm> 
\put{\beginpicture
\multiput{} at -.2 0  9 0 /

\put{$\d(\Delta_n)$} at -.2 -0
\put{$n+1$} at 1 -0
\put{$2^n$} at 2 -0
\put{$2^{n-1}$} at 3 -0
\put{$72$} at 4 0
\put{$576$} at 5 0
\put{$16740$} at 6 0
\put{$48$} at 7 0
\put{$6$} at 8 0
\put{$\ssize 2^3\cdot 3^2$} at 4 -0.35
\put{$\ssize 2^6\cdot 3^2$} at 5 -0.35
\put{$\ssize 2^7\cdot 3^3\cdot 5$} at 6 -0.35
\put{$\ssize 2^4\cdot 3$} at 7 -0.35
\put{$\ssize 2\cdot 3^{\phantom2}$} at 8 -0.35
\endpicture} at 0 -3

\setcoordinatesystem units <1.6cm,1cm> 
\put{\beginpicture
\multiput{} at -.2 0  9 0 /

\put{$\c(\Delta_n)$} at -.2 0
\put{$(n\!+\!1)^{n-1}$} at 1 0
\put{$n^n$} at 2 0
\put{$2(n\!-\!1)^n$} at 3 0
\put{$\ssize 2^9\cdot 3^4$} at 4 -.35
\put{$\ssize 2\cdot 3^{12}$} at 5 -.35
\put{$\ssize 2\cdot 3^5\cdot 5^7$} at 6 -.35
\put{$\ssize 2^4\cdot 3^3$} at 7 -.35
\put{$\ssize 2\cdot 3$} at 8 -.35

\put{$41472$} at 4 0
\put{$1062882$} at 5 0
\put{$37968750$} at 6 0
\put{$432$} at 7 0
\put{$6$} at 8 0
\endpicture} at 0 -4

\setcoordinatesystem units <1.6cm,1cm> 
\put{\beginpicture
\multiput{} at -.4 0  9 0 /

\put{$|W|$} at -.2 0
\put{$(n+1)!$} at 1 0
\put{$2^n\cdot n!$} at 2 0
\put{$2^{n-1}n!$} at 3 0
\put{$51840$} at 4 0
\put{$2903040$} at 5 0
\put{$696729600$} at 6.07 0
\put{$1152$} at 7.08 0
\put{$12$} at 8 0
\put{$\ssize 2^7\cdot 3^4\cdot 5$} at 4 -.35
\put{$\ssize 2^{10}\cdot 3^4\cdot 5 \cdot 7$} at 5 -.35
\put{$\ssize 2^{14}\cdot 3^5\cdot 5^2\cdot 7$} at 6 -.35
\put{$\ssize 2^7\cdot 3^2$} at 7 -.35
\put{$\ssize 2^2\cdot 3$} at 8 -.35
\endpicture} at 0 -5

\setcoordinatesystem units <1.6cm,1cm> 
\put{\beginpicture
\multiput{} at -.4 0  9 0 /
\put{$|\Phi_+|$} at -.2 0
\put{$\binom{n+1}2$} at 1 0
\put{$n^2_{\phantom g}$} at 2 0
\put{$n(n-1)$} at 3 0
\put{$36$} at 4 0
\put{$63$} at 5 0
\put{$120$} at 6 0
\put{$24$} at 7 0
\put{$6$} at 8 0
\put{$\ssize 2^2\cdot 3^2$} at 4 -0.35
\put{$\ssize 3^2\cdot 7$} at 5 -0.35
\put{$\ssize 2^3\cdot 3\cdot 5$} at 6 -0.35
\put{$\ssize 2^3\cdot 3$} at 7 -0.35
\put{$\ssize 2\cdot 3$} at 8 -0.35
\endpicture} at 0 -6

\endpicture}
$$
Observation concerning the prime factors which appear:
With the exception of $\mathbb A_n$, all the prime factors
are bounded by $n$, for $\mathbb A_n$ the bound is $n+1$.
	\medskip 

\subsubsection{\bf Table 2.} 
Here are the values for further Dynkin functions.
$$
{\beginpicture
\setcoordinatesystem units <1.4cm,.8cm> 
\put{\beginpicture
\multiput{} at -2 2  9 2 /
\put{$\Delta_n$} at -1 2 
\put{$\mathbb A_n$} at .8 2 
\put{$\mathbb B_n$} at 1.9 2 
\put{$\mathbb D_n$} at 3 2 
\put{$\mathbb E_6$} at 4 2 
\put{$\mathbb E_7$} at 5 2 
\put{$\mathbb E_8$} at 6 2 
\put{$\mathbb F_4$} at 7 2 
\put{$\mathbb G_2$} at 8 2 
\endpicture} at 0 -6.3
\plot -5.2 -6.8  4.8 -6.8 /

\setcoordinatesystem units <1.4cm,.9cm> 
\put{\beginpicture
\multiput{} at -2 0  9 0 /

\put{$\sinc(\Delta_n)$} at -1 0
\put{$1$} at .8 0
\put{$n$} at 1.9 0
\put{$n-2$} at 3 0
\put{$7$} at 4 0
\put{$16$} at 5 0
\put{$44$} at 6 0
\put{$10$} at 7 0
\put{$4$} at 8 0
\put{$\ssize 7$} at 4 -.35
\put{$\ssize 2^{4}$} at 5 -.35
\put{$\ssize 2^{2}\cdot 11$} at 6 -.35
\put{$\ssize 2\cdot 5$} at 7 -.35
\put{$\ssize 2^2$} at 8 -.35
\endpicture} at 0 -7.5

\setcoordinatesystem units <1.4cm,.9cm> 
\put{\beginpicture
\multiput{} at -2 0  9 0 /

\put{$\a(\Delta_n) = \t(\Delta_n)$} at -1 0
\put{$\frac1{n+2}\binom{2n+2}{n+1}$} at .8 0
\put{\small Catalan numbers} at 0.7 -.6
\put{$\binom{2n}{n}$} at 1.9 0
\put{$\frac{3n-2}n\binom{2n-2}{n-1}$} at 3 0
\put{833} at 4 0
\put{4160} at 5 0
\put{25080} at 6 0
\put{105} at 7 0
\put{8} at 8 0
\put{$\ssize 7^2\cdot17$} at 4 -.35
\put{$\ssize 2^6\cdot5\cdot13$} at 5 -.35
\put{$\ssize 2^3\cdot3\cdot5\cdot11\cdot19$} at 6 -.35
\put{$\ssize 3\cdot5\cdot 7$} at 7 -.35
\put{$\ssize 2^3$} at 8 -.35
\endpicture} at 0 -9

\setcoordinatesystem units <1.4cm,.9cm> 
\put{\beginpicture
\multiput{} at -2 0  9 0 /

\put{$\a_t(\Delta_n)$} at -1 0
\put{$\tfrac1{n+1}\binom{n+1}t\binom{n+1}{t+1}\qquad$} at .9 0
\put{$\binom nt^2_{\phantom g}$} at 1.95 0
\put{$\ssize \a_t(\mathbb D_n)$} at 3 0
\put{$\smallmatrix 1\cr 36\cr 204\cr 351 \cr 204 \cr 36\cr 1 \endsmallmatrix$} at 4 0
\put{$\smallmatrix 1\cr 63 \cr 546\cr  1470 \cr 1470\cr 546\cr 63 \cr 1
   \endsmallmatrix$} at 5 0
\put{$\smallmatrix 1\cr 120\cr 1540\cr 6120\cr 9518\cr 6120\cr 1540\cr 120 \cr 1 \endsmallmatrix$} at 6 0
\put{$\smallmatrix 1\cr 24\cr 55 \cr 24 \cr 1\endsmallmatrix$} at 7 0
\put{$\smallmatrix 1\cr 6\cr 1\endsmallmatrix$} at 8 0
\put{\small Narayana numbers} at 0.7 -.6
\endpicture} at 0 -11

\setcoordinatesystem units <1.4cm,.9cm> 
\put{\beginpicture
\multiput{} at -2 0  9 0 /

\put{$\t_s(\Delta_n)$} at -1 0
\put{$\tfrac{n-s+1}{n+1}\binom{n+s}s$\qquad} at 0.82 0
\put{$\binom{n+s-1}s$} at 1.85 0
\put{$\ \ \ssize \t_s(\mathbb D_n)$} at 2.95 0
\put{$\smallmatrix 1\cr 6\cr 20\cr 50 \cr 110 \cr 228\cr 418 \endsmallmatrix$} at 4 0
\put{$\smallmatrix 1\cr 7 \cr 27\cr  77 \cr 187\cr 429\cr 1001 \cr 2431
   \endsmallmatrix$} at 5 0
\put{$\smallmatrix 1\cr 8\cr 35\cr 112\cr 299\cr 728\cr 1771\cr 4784 \cr 17342 \endsmallmatrix$} at 6 0
\put{$\smallmatrix 1\cr 4\cr 10 \cr 24 \cr 66\endsmallmatrix$} at 7 0
\put{$\smallmatrix 1\cr 2\cr 5\endsmallmatrix$} at 8 0
\endpicture} at 0 -14

\setcoordinatesystem units <1.4cm,.9cm> 
\put{\beginpicture
\multiput{} at -2 0  9 0 /

\put{$\t_n(\Delta_n)$} at -1 0
\put{$\frac1{n+1}\binom{2n}{n}$} at .8 0
\put{\small Catalan numbers} at 0.7 -.6
\put{$\binom{2n-1}{\ n-1}$} at 1.8 0
\put{$\frac{3n-4}{2n-2}\binom{2n-2}{n-2}$} at 3 0
\put{418} at 4 0
\put{2431} at 5 0
\put{17342} at 6 0
\put{66} at 7 0
\put{5} at 8 0
\put{$\ssize 2\cdot11 \cdot19$} at 4 -.35
\put{$\ssize 11\cdot13\cdot17$} at 5 -.35
\put{$\ssize 2\cdot13\cdot23\cdot29$} at 6 -.35
\put{$\ssize 2\cdot3\cdot11$} at 7 -.35
\put{$\ssize 5$} at 8 -.35
\endpicture} at 0 -16

\endpicture}
$$
with 
$$
{\beginpicture
\setcoordinatesystem units <1cm,.6cm> 
\put{$ \a_t(\mathbb D_n)$} [r] at -.2 0
\put{$= \binom nt^2 - \frac n{n-1}\binom{n-1}{t-1}\binom{n-1}t$} [l] at 0 0
\put{$= \binom nt^2 - \binom{n}{t}\binom{n-2}{t-1}$} [l] at 0 -1
\put{$= \binom nt \binom{n-2}{t-1}\frac{n(n-1)-t(n-t)}{t(n-t)},$} [l] at 0 -2
\endpicture}
$$
for $1\le t \le n-1$ and
$$
  \t_s(\mathbb D_n) = \frac{n+2s-2}{n+s-2}\binom{n+s-2}s
$$
for $0\le s \le n$.
	\medskip 

Here are the missing factorizations for  $\a_t(\Delta_n)$ and $\t_s(\Delta_n)$:
$$
{\beginpicture
\setcoordinatesystem units <2.5cm,.9cm> 
\put{$\mathbb E_6$} at 1 -.2
\put{$\mathbb E_7$} at 2 -.2
\put{$\mathbb E_8$} at 3 -.2 
\put{$\mathbb F_4$} at 4 -.2
\put{$\a_t(\Delta_n)$} at 0 -1
\put{$\smallmatrix 
       204 = 2^2\cdot3\cdot17\cr 
       351 = 3^3\cdot 13 \endsmallmatrix$} at 1 -1
\put{$\smallmatrix  
       546 = 2\cdot3\cdot7\cdot 13 \cr  
       1470 = 2\cdot3\cdot5\cdot7^2\cr 
         \endsmallmatrix$} at 2 -1
\put{$\smallmatrix 
      1540 = 2^2\cdot 5\cdot7\cdot 11\cr 
      6120 = 2^3 \cdot 3^2\cdot 5\cdot 17\cr 
      9518 = 2\cdot 4759
      \endsmallmatrix$} at 3 -1
\put{$\smallmatrix  55 = 5 \cdot 11 
      \endsmallmatrix$} at 4 -1

\put{$\t_s(\Delta_n)$} at 0 -3
\put{$\smallmatrix 
       20 = 2^2\cdot5\cr 
       50 = 2\cdot 5^5 \cr
       110 = 2\cdot 5\cdot 11 \cr
       228 = 2^2\cdot3\cdot 19 \cr
       418 = 2\cdot 11\cdot 19 
       \endsmallmatrix$} at 1 -3
\put{$\smallmatrix  
       27 = 3^3 \cr  
       77 = 7\cdot11\cr 
       187 = 11\cdot17\cr 
       429 = 3\cdot11\cdot13\cr 
       1001 = 7\cdot11\cdot13\cr 
       2431 = 11\cdot13\cdot17
         \endsmallmatrix$} at 2 -3
\put{$\smallmatrix 
       35 = 5\cdot 7\cr 
       112 = 2^4 \cdot 7\cr 
       299 = 13\cdot 23\cr 
       728 = 2^3 \cdot 7\cdot13\cr 
       1771 = 7 \cdot 11\cdot 23\cr 
       4784 = 2^4 \cdot 13\cdot 23\cr 
       17342 = 2\cdot 13\cdot23\cdot29
      \endsmallmatrix$} at 3 -3
\put{$\smallmatrix  
       10 = 2 \cdot 5 \cr
       24 = 2^3 \cdot 3 \cr
       66 = 6 \cdot 11 
      \endsmallmatrix$} at 4 -3
\endpicture}
$$
\vfill\eject
		\medskip

Observations concerning the prime factors which appear in table 2 as well as
in the subsequent material:
	
The numbers $\a(\Delta)$ and $\a_t(\Delta)$:
{\it Always, at most one prime factor $p$ is greater than $h$} and with the
exception of the central coefficient for $\mathbb E_8$, one has $p < n(n-1)$.
(The central coefficient for $\mathbb E_8$ has the surprising prime factor $\bold {4759}$.)
This is clear for $\mathbb A_n$ and $\mathbb B_n$. In case $\mathbb D_n$, the prime factors
$p > h$ must divide $n(n-1)-t(n-t)$, thus $p < n(n-1)$, and only one such prime factor is
possible. For the exceptional cases, the factorizations are listed above.
	
The numbers $\t_s(\Delta_n)$: Here only primes bounded by $h$ play a role.
	\medskip 

\subsubsection{\bf Table 3.}
Let us draw the attention to the values $f(\Delta)$ 
which occur for the infinite sequences
$\Delta = \mathbb A_n, \mathbb B_n, \mathbb C_n, \mathbb D_n$. Actually, we prefer to order the columns
differently, namely, first $\mathbb A_n$, then $\mathbb D_n$, and finally the common values
for $\mathbb B_n, \mathbb C_n$.  
$$
{\beginpicture
\setcoordinatesystem units <2.5cm,.8cm> 
\put{\beginpicture
\multiput{} at -1.5 2  9 2 /
\put{$\Delta_n$} at -1 2 
\put{$\mathbb A_n$} at .75 2 
\put{$\mathbb B_n, \mathbb C_n$} at 3 2 
\put{$\mathbb D_n$} at 2 2 
\endpicture} at 0 -.7
\plot -5.2 -1.2  0 -1.2 /
\put{\beginpicture
\multiput{} at -1.5 0  9 0 /

\put{$h$} at -1 0
\put{$n+1$} at .75 0
\put{$2n$} at 3 0
\put{$2(n-1)$} at 2 0
\endpicture} at 0 -2

\put{\beginpicture
\multiput{} at -1.5 0  9 0 /

\put{$\d(\Delta_n) = x_1\cdots x_nx_{\text{th}}$} at -.9 -0
\put{$n+1$} at .75 -0
\put{$2^n$} at 3 -0
\put{$2^{n-1}$} at 2 -0
\endpicture} at 0 -3

\put{\beginpicture
\multiput{} at -1.5 0  9 0 /

\put{$\c(\Delta_n)$} at -1 0
\put{$(n\!+\!1)^{n-1}$} at .8 0
\put{$n^n$} at 3 0
\put{$2(n\!-\!1)^n$} at 2 0
\endpicture} at 0 -4

\put{\beginpicture
\multiput{} at -1.5 0  9 0 /

\put{$|W|$} at -1 0
\put{$(n+1)!$} at .75 0
\put{$2^n\cdot n!$} at 3 0
\put{$2^{n-1}n!$} at 2 0
\endpicture} at 0 -5

\put{\beginpicture
\multiput{} at -1.5 0  9 0 /
\put{$|\Phi_+|$} at -1 0
\put{$\binom{n+1}2$} at .75 0
\put{$n^2_{\phantom g}$} at 3 0
\put{$n(n-1)$} at 2 0
\endpicture} at 0 -6

\put{\beginpicture
\multiput{} at -1.5 0  9 0 /

\put{$\sinc(\Delta_n)$} at -1 0
\put{$1$} at .75 0
\put{$n$} at 3 0
\put{$n-2$} at 2 0
\endpicture} at 0 -7.5

\put{\beginpicture
\multiput{} at 0 0  9 0 /

\put{$\a(\Delta_n) = \t(\Delta_n)$} at -1 0
\put{$\left]\smallmatrix 2n+2\cr n+1\endsmallmatrix\right[$} at .75 0
\put{\rmk Catalan numbers} at 0.75 -.6
\put{$\left[\smallmatrix 2n-1\cr n-1\endsmallmatrix\right]$} at 2 0
\put{$\left(\smallmatrix 2n\cr n\endsmallmatrix\right)$} at 3 0
\endpicture} at 0 -9

\put{\beginpicture
\multiput{} at -1.5 0  9 0 /

\put{$\a_t(\Delta_n)$} at -1 0
\put{$\tfrac1{n+1}\binom{n+1}t\binom{n+1}{t+1}\qquad$} at .85 0
\put{$\binom nt^2_{\phantom g}$} at 3.05 0
\put{$\binom nt \binom{n-2}{t-1}\frac{n(n-1)-t(n-t)}{t(n-t)}$} at 2.05 0
\put{\rmk Narayana numbers} at 0.75 -.6
\endpicture} at 0 -11

\put{\beginpicture
\multiput{} at -1.5 0  9 0 /

\put{$\t_s(\Delta_n)$} at -1 0
\put{$\ssize 0 \le s < n$} at -1 -.5

\put{$\left]\smallmatrix n+s\cr s\endsmallmatrix\right[ $} at .75 0
\put{$\left[\smallmatrix n+s-2\cr s\endsmallmatrix\right]$} at 2 0
\put{$\left(\smallmatrix n+s-1\cr s\endsmallmatrix\right)$} at 3.03 0

\endpicture} at 0 -12.5

\put{\beginpicture
\multiput{} at -1.5 0  9 0 /

\put{$\t_n(\Delta_n)$} at -1 0
\put{$\left]\smallmatrix 2n\cr n\endsmallmatrix\right[ $} at .75 0
\put{\rmk Catalan numbers} at 0.75 -.6
\put{$\left[\smallmatrix 2n-2\cr n-2\endsmallmatrix\right]$} at 2 0
\put{$\left(\smallmatrix 2n-1\cr n-1\endsmallmatrix\right)$} at 3.03 0
\endpicture} at 0 -14

\endpicture}
$$
Here, in three rows, namely the rows for $\a(\Delta_n)= \t(\Delta_n)$ and
for all $\t_s(\Delta_n)$, we have used square-bracket notations in order to
indicate the parallelity to the binomial coefficients. The binomial
coefficients themselves arise in the cases $\mathbb B$ and $\mathbb C$ 
(thus in the last column). The notation 
$\left[\smallmatrix t\cr s \endsmallmatrix\right] = \frac{s+t}t\binom t s$ (used in the 
third column) was proposed by 
Bailey \cite{[Ba]}: this concerns the cases $\mathbb D$.
It has been suggested in \cite{[ONFR2]} to write similarly
$\left]\smallmatrix t\cr s\endsmallmatrix\right[ = \frac{t-2s+1}{t-s+1}\binom t s$.
This is done in the second column and concerns the classical case $\mathbb A$.

The reader should observe that for $\Delta_n$ equal to 
$\mathbb A_n$ or $\mathbb B_n$, the formula given for $\t_s(\Delta_n)$ and $0\le s < n$
works also for $s = n$. This is not the case for $\mathbb D_n$: 
Whereas $\binom{2n-2}{n-2} = \binom{2n-2}n$, the numbers 
$\left[{\smallmatrix 2n-2\cr n-2\endsmallmatrix}\right]$ and 
$\left[{\smallmatrix 2n-2\cr n\endsmallmatrix}\right]$ 
are different.
	\bigskip 

Four observations should be mentioned. 
	\medskip 

{\bf First} of all, it seems that often the numbers which arise for $\mathbb B$ (and $\mathbb C$)
are given by very simple expressions, whereas those for $\mathbb D$ tend to be similar
to those obtained for $\mathbb B$, but sometimes much 
more complicated (see for example the row $\a_t(\Delta_n)$). There has been a tendency in
representation theory to avoid working with non-simply-laced Dynkin diagrams,
thus restricting the attention to the cases $\mathbb A_n, \mathbb D_n, \mathbb E_6, \mathbb E_7$
and $\mathbb E_8$. But to avoid the really nice cases $\mathbb B_n$ seems to be a mistake!
The slight increase of difficulties when looking at species and not just quivers
is definitely honored by the unified numerical picture which one obtains.  

{\bf Second,} the numbers presented in the table are increasing from left to right
as soon as we rearrange the columns (as we have done): first the case $\mathbb A_n,$ then
the case $\mathbb D_n$ and finally the cases $\mathbb B_n$ and $\mathbb C_n$.
(Actually, there are further reasons to prefer the sequence $\mathbb A_n, \mathbb D_n,
\mathbb B_n, \mathbb C_n$ over the alphabetical order: This ordering corresponds to the
quite natural ordering of geometries, starting with affine geometry, followed
by the orthogonal geometry and ending with symplectic geometry --- this is
the sequence geometries are usually taught, with symplectic geometry as the
climax, or as an afterthought which is left to the students as an exercise.) 

{\bf Third.} We see that the Catalan numbers 
$$
 C_n = \tfrac1{n+1}\tbinom{2n}n= \left]\smallmatrix 2n\cr n\endsmallmatrix\right[ 
$$ 
appear as values for the Dynkin diagrams of type $\mathbb A$, looking
at two different counting problems, namely we have 
$$
 \a(\mathbb A_{n-1})=C_n, \quad\text{as well as}\quad \t_n(\mathbb A_n)= C_n
$$ 
(of course, there is an index shift). 
For the remaining Dynkin diagrams $\mathbb B_n,\dots,\mathbb G_2$, 
these Dynkin functions $\a$ and $\t_n$ take completely different values. 
Thus, we deal with two different generalizations of the Catalan combinatorics:
to look at the set of antichains in $\mo \Lambda$ (this yields the function $\a$)
and to look at the tilting modules in $\mo \Lambda$ (this yields the function $\t_n$).

{\bf Fourth.} The non-zero numbers $\t_s(\Delta_n)$ with $\Delta = \mathbb A, \mathbb B, \mathbb D$ 
yield three triangles which have similar properties, see \cite{[ONFR2]}.
The triangle for $\mathbb A$ is the Catalan triangle itself
(this is Sloane's sequence A009766). The triangle for $\mathbb B$ is the triangle A059481, it 
corresponds to the increasing part of the Pascal triangle
(thus it consists of the binomial coefficients $\binom ts$ with $2s\le t+1$).
The triangle for $\mathbb D$ is  an expansion of the increasing part of
the Lucas triangle  A029635: taking the increasing parts of the rows
in the Lucas triangle (thus the numbers $\left[\smallmatrix t\cr s \endsmallmatrix\right]$ with
$2s\le t+1$), one obtains numbers which occur in the triangle $\mathbb D$, 
namely the numbers $\t_s(\mathbb D_n)$
with $0 \le s < n$. The numbers $\t_n(\mathbb D_n)$ on the diagonal have to be treated separately:
recall that  $\t_s(\mathbb D_n) = \left[{\smallmatrix n+s-2\cr s\endsmallmatrix}\right]$,
whereas $\t_n(\mathbb D_n) = \left[{\smallmatrix 2n-2\cr n-2\endsmallmatrix}\right]$.
	\medskip

\subsubsection{\bf Why counting?} The result may give an indication about structural
similarities. If for two different counting problems we obtain 
the same numbers, one may ask whether there is a natural bijection
between the sets in question. 
And if we find one, it may turn out 
such a bijection exists in similar situations looking at sets which are
no longer finite. 

For example, if the answer to a counting
problem turns out to be $(n+1)^{n-1}$, as it is
for counting the number of complete exceptional sequences in case $\mathbb A_n$, 
one may try
to find a bijection between the complete exceptional sequences and
labeled trees. In Section 2.4 we will present a bijection between the
support-tilting modules and the 
normal partial tilting modules. One could have asked for such a bijection
as soon as one had observed that in the representation-finite cases
the numbers
of multiplicity-free modules which are 
support-tilting or normal coincide.  

We should stress that 
the appearance of the same Dynkin functions in different parts of mathematics
has been a great stimulus to look for corresponding relations. 
As examples, let us mention here questions in singularity theory \cite{[B1],[B2],[D],[Lo]} and
the study of ideals in Lie theory \cite{[CP1], [CP2], [Pa], [Sm]}.

\subsection{The Exponents}
	\medskip 

There is a unified, but quite mysterious way to deal with some of the Dynkin functions,
namely to invoke the so-called exponents.
If $\Delta = \Delta_n$ is a root system of rank $n$, there is attached a sequence 
$\epsilon_1 \le \epsilon_2 \le \cdots\le \epsilon_n$ of positive integers, the {\it exponents.}
	\medskip

\subsubsection{}
The relevance of the exponents is revealed by the classical definitions:
\begin{itemize}

\item The {\bf eigenvalues of a Coxeter transformation} $c$ are of the form $\zeta^{\epsilon_i}$,
  where $\zeta$ is a primitive $h$-th root of unity 
 (here, $h$ is the Coxeter number, this is the order of $c$),
  and $1\le i \le n.$  
\item The degrees of a basic set of
   {\bf invariants} of $W$ acting on $V$
   are $\epsilon_1\!+\!1,$ $\dots,\epsilon_n\!+\!1.$
\item The degrees of a basic set of {\bf $\H$-derivations} (see Section 1.5) 
  are $\epsilon_1,\dots,\epsilon_n.$
\end{itemize}

Some properties of the sequence of the exponents $(\epsilon_1,\epsilon_2,\cdots,\epsilon_n)$:

\begin{itemize} 

\item These are $n$ positive integers.
\item If $\epsilon_i$ is an exponent, also $h-\epsilon_i$ is an exponent 
     (in particular, $\epsilon_i < h$ for all $i$).
\item The exponents are usually pairwise different, the only exceptions are the
    cases $\mathbb D_{2n}$: here $2n-1$ is twice an exponent. 
\item Actually, the exponents are easy to remember for the series: for $\mathbb A_n$, 
    these are just the first $n$ numbers
    $1,2,\dots,n$; for $\mathbb B_n$ and $\mathbb C_n$, take the first 
    $n$ odd numbers: $1,3,5,\dots,2n-1$,
    finally
    for $\mathbb D_n$, take the first $n-1$ odd numbers as well as $n-1$ itself, thus
    $1,3,5,\dots,2n-3,n-1.$ (As we have mentioned above, often the cases $\mathbb D_n$ 
    turn out to be more complicated than the
    cases $\mathbb A_n$ and $\mathbb B_n$).
\item {\it If $p < h$ is a prime number with $(p,h)=1,$ then $p$ is an exponent} 
    (see for example \cite{[H2]}, 3.20). This result, together with the facts that
    the number of exponents is $n$, that $1$ is an exponent, and that with $\epsilon$ 
    also $h-\epsilon$ is an exponent,
    determines the exponents uniquely: For $\mathbb E_7$, one needs precisely one exponent which is
    different from $1$ and not a prime, this can be only $9$; for $\mathbb E_6$ 
    one needs precisely two exponents which are
    different from $1$ and not primes, these can be only the numbers $4$ and $8$
    (altogether we see that the exponents for the exceptional cases are prime powers).
\end{itemize}
	\medskip 

\subsubsection{\bf Formulas, using the exponents.}
$$
{\beginpicture
\setcoordinatesystem units <1cm,.8cm> 
\put{$(1)$} at -5 0 
\put{$h$} [r] at -.2 0
\put{$= \epsilon_n+1 $} [l] at 0 0 

\put{$(2)$} at -5 -1 
\put{$\d(\Delta_n)= x_1\cdots x_nx_{\text{th}} $} [r] at -.2 -1
\put{$= \prod \frac{\epsilon_i+1}i$} [l] at 0 -1

\put{$(3)$} at -5 -2 
\put{$\c(\Delta_n)$} [r] at -.2 -2
\put{$= \frac1{|W|} n!h^n  =
  \frac{h^n}{x_1\cdots x_nx_{\text{th}}}= h^n\prod \frac{i}{\epsilon_i+1}$} [l] at 0 -2

\put{$(4)$} at -5 -3 
\put{$|W|$} [r] at -.2 -3
\put{$= \prod (\epsilon_i+1) $} [l]  at 0 -3

\put{$(4$'$)$} at -5 -4
\put{$\sum _{i=1}^n |W_i|t^i$} [r] at -.2 -4
\put{$= \prod (1+\epsilon_it) $} [l] at 0 -4

\put{$(5)$} at -5 -5 
\put{$ |\Phi_+|$} [r] at -.2 -5
\put{$ = \tfrac12nh = \sum \epsilon_i,$}   [l] at 0 -5 

\put{$(6)$} at -5 -6
\put{$  \sinc(\Delta_n)$} [r] at -.2 -6
\put{$=nh\prod\nolimits_{i\ge 2}\frac{\epsilon_i\!-\!1}{\epsilon_i+1}$}  [l] at 0 -6

\put{$(7)$} at -5 -7
\put{$ \a(\Delta_n) $} [r] at -.2 -7
\put{$= \frac1{|W|}\prod (h+\epsilon_i+1) = \prod\frac{h+\epsilon_i+1}{\epsilon_i+1}$}  [l] at 0 -7

\put{$(8)$} at -5 -8
\put{$\t_n(\Delta_n))$} [r] at -.2 -8
\put{$= \frac1{|W|}\prod (h+\epsilon_i-1) = \prod\frac{h+\epsilon_i-1}{\epsilon_i+1}$}  [l] at 0 -8
\endpicture}
$$
	\medskip 

It is not completely clear to us, who first found these formulas.
Formula (4$'$) may have been the first one obtained, it
is due to Shephard-Todd \cite{[ST]}, 1954. 
Of course, formula (4) is a special case of (4$'$), namely $t=1.$ 
Some of the other formulas seem to be due to Chapoton. 

The formulas for $\a(\Delta_n)$ and $\t_n(\Delta_n)$
are taken from Fomin and Zelevinsky \cite{[FZ]}, 2003. 
They showed that $\a_n(\Delta)$ is the number of clusters for a cluster algebra of type
$\Delta_n$. They write:  
{\it we are grateful to Fr\'ed\'eric Chapoton who observed that
these expressions, which we obtained on a case by case basis, can be replaced by the
unifying formula $a(\Delta_n) = \prod\frac{h+\epsilon_i+1}{\epsilon_i+1}$. F.~Chapoton brought to our attention
that these numbers appear in the study of non-crossing partitions by V.~Reiner, C.~Athanasiadis and 
A.~Postnikov.} And they add:  
{\it The appearance of the exponents in the
formula $a(\Delta_n) = \prod\frac{h+\epsilon_i+1}{\epsilon_i+1}$ for the number of clusters is a mystery 
for us at the moment. To add to this mystery, a similar expression can be given for the number of 
positive clusters,} namely $a'_n(\Delta_n) = \prod\frac{h+\epsilon_i-1}{\epsilon_i+1}$.
It seems that this mystery has not been resolved until now.
So it is a challenge for the readers.

\subsection{The height partition}

It seems to be worthwhile to gather as much information as possible on the exponents,
in particular about the relationship between
the exponents and the positive roots. We saw already:
\begin{itemize} 
\item the number of exponents is the number of simple roots.
\item the sum of the exponents is just the number of positive roots.
\end{itemize} 

\noindent
But there is a more intense interrelation between the positive roots
and the exponents: the exponents may seem to be mysterious, but they are easily obtained
from the root poset!
Akyildiz and Carrell wrote a paper \cite{[AC]} with the  title:
{\it Betti numbers of smooth Schubert varieties 
and the remarkable formula of Kostant, Macdonald, Shapiro, and Steinberg.}
The result in question has been presented by Humphreys in his book on Reflection Groups,
we will recall it next.
The first published proof is by Kostant (1959),  with a reference to unpublished investigations
of Shapiro, further proofs were given by Macdonald (1972) and Steinberg.
	 \medskip 

\subsubsection{}
We say that a sequence of numbers $r_1,r_2,\dots,r_t$ is a {\it Young partition}
provided $r_1\ge r_2 \ge \cdots r_t \ge 0$ (see the note N\,\ref{Young}). 
Given a Young partition $r = (r_1,r_2,\dots,r_t)$
its {\it dual} partition $r' = (r'_1,r'_2,\dots)$ is defined as follows: $r'_j$ is the number of indices 
$i$ with $r_i \ge j.$
	\medskip 

\begin{theorem}
{\bf (Shapiro, Kostant).} Let $r_t$ be the number of roots of height $t$. 
Then $r = (r_1,r_2,\dots)$  is a Young
partition, called the height partition of $\Phi_+$. 
The dual partition of the height partition of $\Phi_+$
is the Young partition of the exponents.
\end{theorem}
	\bigskip 

As a consequence, we see

\Rahmen{$
 r_t+r_{h+1-t} = n.
$}
	\medskip 

This follows directly from the  well-known symmetry condition for the
exponents: $\epsilon_t+\epsilon_{n-t+1} = h.$  
	\medskip 

For example: $r_1 = n,\ r_h = 0.$ Next, for a connected root system: $r_2 = n-1,\ r_{h-1} = 1$
(this means that there is a unique highest root).
	\medskip

For a similar result concerning the roots of fixed length, see
the note N\,\ref{fixed-height}.
. 
	\bigskip

Let us draw the corresponding Young diagrams. Actually, we will draw the Young diagram 
$Y = Y(\epsilon)$ 
of the
partition $\epsilon = (\epsilon_1,\dots, \epsilon_n)$ 
of the exponents, using shaded boxes. 
Thus as the dual partition we see the height partition 
$(r_1,r_2,\dots).$ 
We will draw the partition of exponents inside a rectangle $R$ with $n$ rows and $h$
columns (thus, in $R$, there are altogether $nh$ square boxes), and 
we may consider the square boxes in
$R\setminus Y$ as corresponding bijectively to the negative roots.


$$
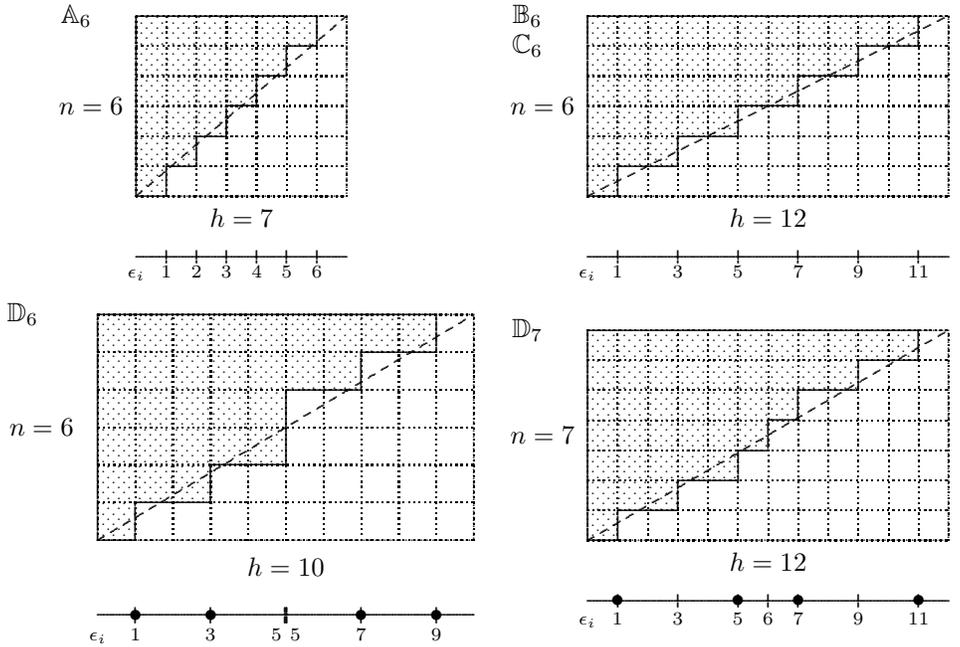
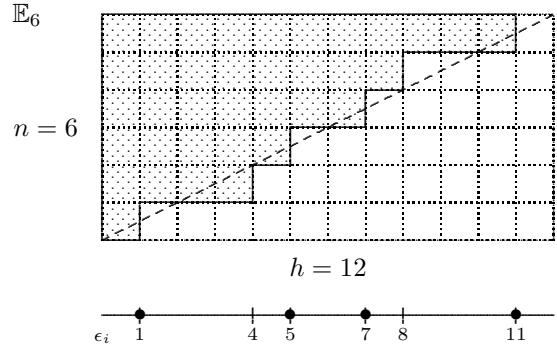
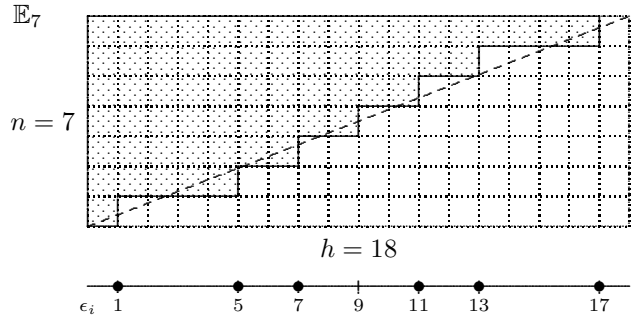
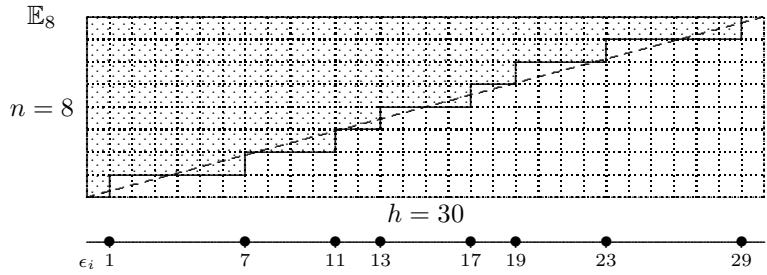
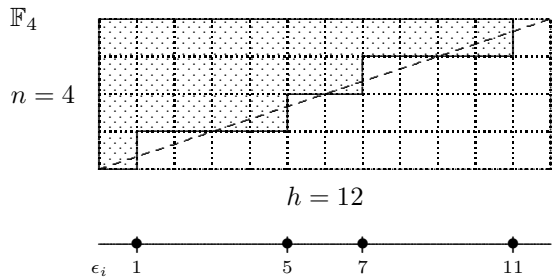
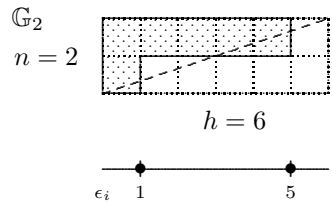
\hbox{\beginpicture
\setcoordinatesystem units <.4cm,.4cm>

\put{\beginpicture
\put{\beginpicture
\multiput{} at 0 0  12 6 /

\put{$\mathbb A_6$} at -2 6
\setdots <.7mm>
\setplotarea x from 0 to 7, y from 0 to 6
\grid {7} {6} 
\setsolid
\plot 0 0  0 6  6 6  6 5  5 5  5 4  4 4  4 3  3 3  3 2  2 2  2 1  1 1  1 0  0 0  /
\put{$n=6$} at -1.5 3
\put{$h = 7$} at 3.5 -.7
\plot 0 -2  7 -2 /
\plot 1 -1.8  1 -2.2 /
\plot 2 -1.8  2 -2.2 /
\plot 3 -1.8  3 -2.2 /
\plot 4 -1.8  4 -2.2 /
\plot 5 -1.8  5 -2.2 /
\plot 6 -1.8  6 -2.2 /
\put{$\ssize \epsilon_i\strut$} at 0 -2.5
\put{$\ssize 1$} at 1 -2.5
\put{$\ssize 2$} at 2 -2.5
\put{$\ssize 3$} at 3 -2.5
\put{$\ssize 4$} at 4 -2.5
\put{$\ssize 5$} at 5 -2.5
\put{$\ssize 6$} at 6 -2.5
\setshadegrid span <.7mm>
\vshade 0 0 6 <,z,,>  1 0 6  <z,z,,> 
      1.01 1 6  <z,z,,> 2 1 6  <z,z,,> 
      2.01 2 6  <z,z,,> 3 2 6  <z,z,,> 
      3.01 3 6  <z,z,,> 4 3 6  <z,z,,> 
      4.01 4 6  <z,z,,> 5 4 6  <z,z,,> 
      5.01 5 6  <z,,,> 6 5 6  /

\setdots <.4mm>
\plot 6 6  7 6  7 0  1 0 /
\setdashes <1mm>
\plot 0 0  7 6 /
\endpicture} at 0 0
\endpicture} at 0 0

\put{\beginpicture  
\setcoordinatesystem units <.4cm,.4cm>
\put{\beginpicture
\multiput{} at 0 0  12 6 /

\put{$\mathbb B_6$} at -2 6
\put{$\mathbb C_6$} at -2 5
\setdots <.7mm>
\setplotarea x from 0 to 12, y from 0 to 6
\grid {12} {6} 
\setsolid
\plot 0 0  0 6  11 6  11 5  9 5  9 4  7 4  7 3  5 3  5 2  3 2  3 1  1 1  1 0  0 0  /
\put{$n=6$} at -1.5 3
\put{$h = 12$} at 6 -.7
\plot 0 -2  12 -2 /
\plot 1 -1.8  1 -2.2 /
\plot 3 -1.8  3 -2.2 /
\plot 5 -1.8  5 -2.2 /
\plot 7 -1.8  7 -2.2 /
\plot 9 -1.8  9 -2.2 /
\plot 11 -1.8  11 -2.2 /


\put{$\ssize \epsilon_i\strut$} at 0 -2.5
\put{$\ssize 1$} at 1 -2.5
\put{$\ssize 3$} at 3 -2.5
\put{$\ssize 5$} at 5 -2.5
\put{$\ssize 7$} at 7 -2.5
\put{$\ssize 9$} at 9 -2.5
\put{$\ssize 11$} at 11 -2.5

\setshadegrid span <.7mm>
\vshade 0 0 6 <,z,,>  1 0 6  <z,z,,> 
      1.01 1 6  <z,z,,> 3 1 6  <z,z,,> 
      3.01 2 6  <z,z,,> 5 2 6  <z,z,,> 
      5.01 3 6  <z,z,,> 7 3 6  <z,z,,> 
      7.01 4 6  <z,z,,> 9 4 6  <z,z,,> 
      9.01 5 6  <z,,,> 11 5 6  /

\setdashes <1mm>
\plot 0 0  12 6 /

\setdots <.4mm>
\plot 11 6  12 6  12 0  1 0 /

\endpicture} at 0 0
\endpicture} at 15 0 

\put{\beginpicture
\setcoordinatesystem units <.5cm,.5cm>
\put{\beginpicture
\multiput{} at 0 0  12 6 /

\put{$\mathbb D_6$} at -2 6
\setdots <.7mm>
\setplotarea x from 0 to 10, y from 0 to 6
\grid {10} {6} 
\setsolid
\plot 0 0  0 6  9 6  9 5  7 5  7 4  5 4   5 2  3 2  3 1  1 1  1 0  0 0  /
\put{$n=6$} at -1.5 3
\put{$h = 10$} at 5 -.7
\plot 0 -2  10 -2 /
\plot 1 -1.8  1 -2.2 /
\plot 3 -1.8  3 -2.2 /
\plot 4.97 -1.8  4.97 -2.2 /
\plot 5 -1.8  5 -2.2 /
\plot 5.03 -1.8  5.03 -2.2 /
\plot 7 -1.8  7 -2.2 /
\plot 9 -1.8  9 -2.2 /

\multiput{$\bullet$} at 1 -2  3 -2  7 -2  9 -2 /

\put{$\ssize \epsilon_i\strut$} at 0 -2.5
\put{$\ssize 1$} at 1 -2.5
\put{$\ssize 3$} at 3 -2.5
\put{$\ssize 5$} at 4.75 -2.5
\put{$\ssize 5$} at 5.25 -2.5
\put{$\ssize 7$} at 7 -2.5
\put{$\ssize 9$} at 9 -2.5
\setshadegrid span <.7mm>
\vshade 0 0 6 <,z,,>  1 0 6  <z,z,,> 
      1.01 1 6  <z,z,,> 3 1 6  <z,z,,> 
      3.01 2 6  <z,z,,> 5 2 6  <z,z,,> 
      5.01 4 6  <z,z,,> 7 4 6  <z,z,,> 
      7.01 5 6  <z,,,> 9 5 6 /

\setdashes <1mm>
\plot 0 0  10 6 /

\setdots <.4mm>
\plot 9 6  10 6  10 0  1 0 /

\endpicture} at 0 0
\endpicture} at 0 -11

\put{\beginpicture
\setcoordinatesystem units <.4cm,.4cm>
\put{\beginpicture
\multiput{} at 0 0  18 7 /
\put{$\mathbb D_7$} at -2 7
\setdots <.7mm>
\setplotarea x from 0 to 12, y from 0 to 7
\grid {12} {7} 
\setsolid
\plot 0 0  0 7  11 7  11 6  9 6  9 5  7 5  7 4  6 4  6 3  5 3  5 2  3 2  3 1  1 1  1 0  0 0  /

\put{$n=7$} at -1.5 3.5
\put{$h = 12$} at 6 -.7
\plot 0 -2  12 -2 /
\plot 1 -1.8  1 -2.2 /
\plot 3 -1.8  3 -2.2 /
\plot 5 -1.8  5 -2.2 /
\plot 6 -1.8  6 -2.2 /
\plot 7 -1.8  7 -2.2 /
\plot 9 -1.8  9 -2.2 /
\plot 11 -1.8  11 -2.2 /

\multiput{$\bullet$} at 1 -2  5 -2  7 -2  11 -2 /

\put{$\ssize \epsilon_i\strut$} at 0 -2.6
\put{$\ssize 1$} at 1 -2.6
\put{$\ssize 3$} at 3 -2.6
\put{$\ssize 5$} at 5 -2.6
\put{$\ssize 6$} at 6 -2.6
\put{$\ssize 7$} at 7 -2.6
\put{$\ssize 9$} at 9 -2.6
\put{$\ssize 11$} at 11 -2.6
\setshadegrid span <.7mm>
\vshade 0 0 7 <,z,,>  1 0 7  <z,z,,> 
      1.01 1 7  <z,z,,> 3 1 7  <z,z,,> 
      3.01 2 7  <z,z,,> 5 2 7  <z,z,,> 
      5.01 3 7  <z,z,,> 6 3 7  <z,z,,> 
      6.01 4 7  <z,z,,> 7 4 7  <z,z,,> 
      7.01 5 7  <z,z,,> 9 5 7  <z,z,,> 
      9.01 6 7  <z,,,> 11 6 7  /

\setdashes <1mm>
\plot 0 0  12 7 /

\setdots <.4mm>
\plot 11 7  12 7  12 0  1 0 /

\endpicture} at 0 0
\endpicture} at  18  -11
\endpicture}
$$
$$
\hbox{\beginpicture
\setcoordinatesystem units <.5cm,.5cm>
\put{\beginpicture
\multiput{} at 0 0  12 6 /

\put{$\mathbb E_6$} at -2 6
\setdots <.7mm>
\setplotarea x from 0 to 12, y from 0 to 6
\grid {12} {6} 
\setsolid
\plot 0 0  0 6  11 6  11 5  8 5  8 4  7 4  7 3  5 3  5 2  4 2  4 1  1 1  1 0  0 0  /
\put{$n=6$} at -1.5 3
\put{$h = 12$} at 6 -.7
\plot 0 -2  12 -2 /
\plot 1 -1.8  1 -2.2 /
\plot 4 -1.8  4 -2.2 /
\plot 5 -1.8  5 -2.2 /
\plot 7 -1.8  7 -2.2 /
\plot 8 -1.8  8 -2.2 /
\plot 11 -1.8  11 -2.2 /

\multiput{$\bullet$} at 1 -2  5 -2  7 -2  11 -2 /

\put{$\ssize \epsilon_i\strut$} at 0 -2.5
\put{$\ssize 1$} at 1 -2.5
\put{$\ssize 4$} at 4 -2.5
\put{$\ssize 5$} at 5 -2.5
\put{$\ssize 7$} at 7 -2.5
\put{$\ssize 8$} at 8 -2.5
\put{$\ssize 11$} at 11 -2.5

\setshadegrid span <.7mm>
\vshade 0 0 6 <,z,,>  1 0 6  <z,z,,> 
      1.01 1 6  <z,z,,> 4 1 6  <z,z,,> 
      4.01 2 6  <z,z,,> 5 2 6  <z,z,,> 
      5.01 3 6  <z,z,,> 7 3 6  <z,z,,> 
      7.01 4 6  <z,z,,> 8 4 6  <z,z,,> 
      8.01 5 6  <z,,,> 11 5 6  /

\setdashes <1mm>
\plot 0 0  12 6 /

\setdots <.4mm>
\plot 11 6  12 6  12 0  1 0 /

\endpicture} at 0 0
\endpicture}
$$
 \bigskip

$$
\hbox{\beginpicture
\setcoordinatesystem units <.4cm,.4cm>
\put{\beginpicture
\multiput{} at 0 0  18 7 /
\put{$\mathbb E_7$} at -2 7
\setdots <.7mm>
\setplotarea x from 0 to 18, y from 0 to 7
\grid {18} {7} 
\setsolid
\plot 0 0  0 7  17 7  17 6  13 6  13 5  11 5  11 4  9 4  9 3  7 3  7 2  5 2  5 1  1 1  1 0  0 0  /

\put{$n=7$} at -1.5 3.5
\put{$h = 18$} at 9 -.7
\plot 0 -2  18 -2 /
\plot 1 -1.8  1 -2.2 /
\plot 5 -1.8  5 -2.2 /
\plot 7 -1.8  7 -2.2 /
\plot 9 -1.8  9 -2.2 /
\plot 11 -1.8  11 -2.2 /
\plot 13 -1.8  13 -2.2 /
\plot 17 -1.8  17 -2.2 /

\multiput{$\bullet$} at 1 -2  5 -2  7 -2  11 -2  13 -2  17 -2 /

\put{$\ssize \epsilon_i\strut$} at 0 -2.6
\put{$\ssize 1$} at 1 -2.6
\put{$\ssize 5$} at 5 -2.6
\put{$\ssize 7$} at 7 -2.6
\put{$\ssize 9$} at 9 -2.6
\put{$\ssize 11$} at 11 -2.6
\put{$\ssize 13$} at 13 -2.6
\put{$\ssize 17$} at 17 -2.6

\setshadegrid span <.7mm>
\vshade 0 0 7 <,z,,>  1 0 7  <z,z,,> 
      1.01 1 7  <z,z,,> 5 1 7  <z,z,,> 
      5.01 2 7  <z,z,,> 7 2 7  <z,z,,> 
      7.01 3 7  <z,z,,> 9 3 7  <z,z,,> 
      9.01 4 7  <z,z,,> 11 4 7  <z,z,,> 
      11.01 5 7  <z,z,,> 13 5 7  <z,z,,> 
      13.01 6 7  <z,,,> 17 6 7  /
\setdashes <1mm>
\plot 0 0  18 7 /

\setdots <.4mm>
\plot 17 7  18 7  18 0  1 0 /

\endpicture} at 0 0
\endpicture}
$$
\bigskip

$$
\hbox{\beginpicture
\setcoordinatesystem units <.3cm,.3cm>
\put{\beginpicture
\multiput{} at 0 0  30 8 /
\put{$\mathbb E_8$} at -2 8
\setdots <.7mm>
\setplotarea x from 0 to 30, y from 0 to 8
\grid {30} {8} 
\setsolid
\plot 0 0  0 8  29 8   29 7  23 7  23 6  
  19 6  19 5  17 5  17 4  13 4  13 3  11 3  11 2  7 2  7 1  1 1  1 0  0 0  /

\put{$n=8$} at -2 4
\put{$h = 30$} at 15 -.7
\plot 0 -2  30 -2 /
\plot 1 -1.8  1 -2.2 /
\plot 7 -1.8  7 -2.2 /
\plot 11 -1.8  11 -2.2 /
\plot 13 -1.8  13 -2.2 /
\plot 17 -1.8  17 -2.2 /
\plot 19 -1.8  19 -2.2 /
\plot 23 -1.8  23 -2.2 /
\plot 29 -1.8  29 -2.2 /

\multiput{$\bullet$} at 1 -2  7 -2  11 -2  13 -2
   17 -2  19 -2  23 -2  29 -2  /

\put{$\ssize \epsilon_i\strut$} at 0 -2.8
\put{$\ssize 1$} at 1 -2.8
\put{$\ssize 7$} at 7 -2.8
\put{$\ssize 11$} at 11 -2.8
\put{$\ssize 13$} at 13 -2.8
\put{$\ssize 17$} at 17 -2.8
\put{$\ssize 19$} at 19 -2.8
\put{$\ssize 23$} at 23 -2.8
\put{$\ssize 29$} at 29 -2.8

\setshadegrid span <.7mm>
\vshade 0 0 8 <,z,,>  1 0 8  <z,z,,> 
      1.01 1 8  <z,z,,> 7 1 8  <z,z,,> 
      7.01 2 8  <z,z,,> 11 2 8  <z,z,,> 
      11.01 3 8  <z,z,,> 13 3 8  <z,z,,> 
      13.01 4 8  <z,z,,> 17 4 8  <z,z,,> 
      17.01 5 8  <z,z,,> 19 5 8  <z,z,,> 
      19.01 6 8  <z,z,,> 23 6 8  <z,z,,> 
      23.01 7 8  <z,,,> 29 7 8  /

\setdashes <1mm>
\plot 0 0  30 8 /

\setdots <.4mm>
\plot 29 8  30 8  30 0  1 0 /

\endpicture} at 0 0
\endpicture}
$$
	\bigskip 

$$
\hbox{\beginpicture
\setcoordinatesystem units <.5cm,.5cm>
\put{\beginpicture
\multiput{} at 0 0  6 4 /
\put{$\mathbb F_4$} at -2 4
\setdots <.7mm>
\setplotarea x from 0 to 12, y from 0 to 4
\grid {12} {4} 
\setsolid
\plot 0 0  0 4  11 4  11 3  7 3  7 2  5 2  5 1  1 1  1 0  0 0  /

\put{$n=4$} at -1.5 2
\put{$h = 12$} at 6 -.7
\plot 0 -2  12 -2 /
\plot 1 -1.8  1 -2.2 /
\plot 5 -1.8  5 -2.2 /
\plot 7 -1.8  7 -2.2 /
\plot 11 -1.8  11 -2.2 /

\multiput{$\bullet$} at 1 -2  5 -2  7 -2  11 -2 /

\put{$\ssize \epsilon_i\strut$} at 0 -2.6
\put{$\ssize 1$} at 1 -2.6
\put{$\ssize 5$} at 5 -2.6
\put{$\ssize 7$} at 7 -2.6
\put{$\ssize 11$} at 11 -2.6

\setshadegrid span <.7mm>
\vshade 0 0 4 <,z,,>  1 0 4  <z,z,,> 
      1.01 1 4  <z,z,,> 5 1 4  <z,z,,> 
      5.01 2 4  <z,z,,> 7 2 4  <z,z,,> 
      7.01 3 4  <z,,,> 11 3 4  /

\setdashes <1mm>
\plot 0 0  12 4 /

\setdots <.4mm>
\plot 11 4  12 4  12 0  1 0 /

\endpicture} at 0 0
\put{\beginpicture
\setcoordinatesystem units <.5cm,.5cm>
\multiput{} at 0 0  6 2 /
\put{$\mathbb G_2$} at -2 2
\setdots <.7mm>
\setplotarea x from 0 to 6, y from 0 to 2
\grid {6} {2} 
\setsolid
\plot 0 0  0 2  5 2   5 1  1 1  1 0  0 0  /

\put{$n=2$} at -1.5 1
\put{$h = 6$} at 3.5 -.7
\plot 0 -2  6 -2 /
\plot 1 -1.8  1 -2.2 /
\plot 5 -1.8  5 -2.2 /
\multiput{$\bullet$} at 1 -2  5 -2 / 

\put{$\ssize \epsilon_i\strut$} at 0 -2.6
\put{$\ssize 1$} at 1 -2.6
\put{$\ssize 5$} at 5 -2.6

\setshadegrid span <.7mm>
\vshade 0 0 2 <,z,,>  1 0 2  <z,z,,> 
      1.01 1 2  <z,,,> 5 1 2  /

\setdashes <1mm>
\plot 0 0  6 2 /

\setdots <.4mm>
\plot 5 2  6 2  6 0  1 0 /

\endpicture} at 14 0
\endpicture}
$$
	\medskip 

\subsubsection{\bf Solid subchains of $\Phi_+.$}
Let us have a detailed look at
the poset $\Phi_+$. Given a poset $P$, we call a subposet $P'$ a {\it solid} subposet
provided neighbors in $P'$ are neighbors in $P$ (if $x<y$ are neighbors in the subposet $P'$,
this interval cannot be refined in $P$).

\begin{prop}
Any root poset $\Phi_+$ is the disjoint union of solid subchains
which contain a minimal element of $\Phi_+$.
\end{prop}

This strengthens the assertion that the sequence $r_1,r_2,\dots,r_t$ is
a Young partition. Namely, write $\Phi_+$ as the disjoint union of solid subchains $C_i$
such that any $C_i$ contains a minimal element of $\Phi_+$. The minimal elements of $\Phi_+$ 
are just the simple roots, thus the number of subchains $C_i$ is $r_1$. It follows that
$r_j$ is the number of the subchains $C_i$ which have length at least $j$. It follows that
$r_1\ge r_2\ge r_3 \ge \cdots.$

Note that the Young partition property does not imply the solid subchain property.
The first example of a connected poset with the Young partition property but without
the solid subchain property is as follows:
$$
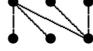
{\beginpicture
\setcoordinatesystem units <.5cm,.5cm> 
\multiput{$\bullet$} at 0 0  1 0  2 0  0 1  1 1  2 1 /
\plot 0 0  0 1  2 0  2 1 /
\plot 1 0  0 1 /
\plot 1 1  2 0 /
\endpicture}
$$
	\bigskip

The assertion of the proposition is obvious for the cases $\mathbb A_n$ and $\mathbb B_n$.
	\medskip 

Below we show solid subchains in the cases $\mathbb D_6$ and $\mathbb E_6,$ drawn by solid
lines. The largest elements of the solid subchains are always encircled (thus
these are suitable roots of height $\epsilon_i$, where $(\epsilon_1,\dots,\epsilon_n)$
is the exponent partition). The case $\mathbb D_6$ indicates a general rule how to obtain
a solid subchain decomposition in the cases $\mathbb D_n$ in general. To find such
a decomposition for the cases $\mathbb E_7, \mathbb E_8$ is more challenging. 
$$
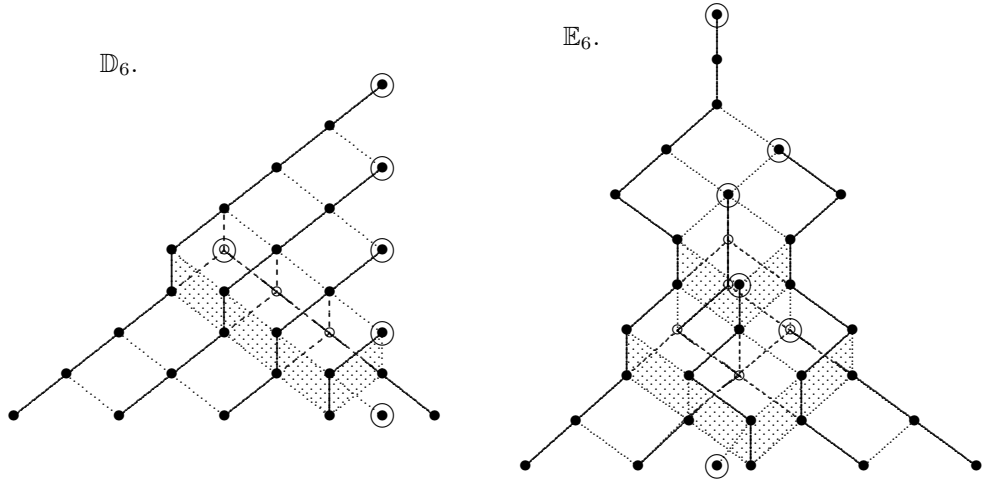
\hbox{\beginpicture
\setcoordinatesystem units <.7cm,.55cm>
\put{\beginpicture
\setsolid
\plot 0 0  3 3  3 4  7 8 /
\plot 2 0  4 2  4 3  7 6 /
\plot 4 0  5 1  5 2  7 4 /
\plot 6 0  6 1  7 2 /
\plot 8 0  6.7 1.3 /
\plot 6.3 1.7  5.7 2.3 /
\plot 5.3 2.7  4.7 3.3 /
\plot 4 4  4.2 3.8 /
\multiput{$\bigcirc$} at 7 8  7 6  7 4  7 2  4 4  7 0 /

\put{$\mathbb D_6.$} at 2 8.5
\setdots <.7mm> 
\multiput{$\bullet$} at 0 0  2 0  4 0  6 0  8 0  
     1 1  3 1  5 1  7 1
     2 2  4 2 
     3 3 
      6 1
     5 2  7 2
     4 3  6 3   
     3 4  5 4  7 4 
     4 5  6 5 
     5 6  7 6 
      6 7 
     7 8 
     7 0 /
               
\plot 0 0  3 3  6 0  7 1  8 0 /
\plot 1 1  2 0  4 2 /
\plot 2 2  4 0  5 1 /

 \plot 7 0  3 4  7 8 /
\plot 6 1  7 2  4 5 /
\plot 5 2  7 4  5 6 /
\plot 4 3  7 6  6 7 /
\plot 3 3  3 4 /
\plot 4 2  4 3 /
\plot 5 1  5 2 /
\plot 6 0  6 1 /
\plot 7 1  7 2 /

\setdashes <.7mm>
\plot  3 3  4 4  7 1 /
\plot 4 4  4 5 /
\plot 4 2  5 3  5 4 /
\plot 5 1  6 2  6 3 /
   
\multiput{$\circ$} at 4 4 5 3  6 2 /
\setshadegrid span <.6mm>
\vshade 3 3 4  <,z,,> 6 0 1  <z,,,> 7 1 2 /

\endpicture} at 0 0
\put{\beginpicture
\setcoordinatesystem units <.75cm,.6cm>
\put{\bf $\mathbb E_6.$} at 1 9.5 
\multiput{$\bigcirc$} at 3.4 10  4.5 7  3.6 6  3.8 4  4.7 3  3.4 0   / 
\setsolid
\plot 3.4 10  3.4 8  1.6 6  2.7 5  2.7 4  1.8 3  1.8 2  0 0 /
\plot  4.5 7  5.6 6  4.7 5  4.7 4  5.8 3  4.9 2  4.9 1  6 0 /
\plot 3.6 6  3.6 4  2.7 3   3 2.7 /
\plot 3.5 2.3  3.8 2  3.5 1.75  /
 
\plot 2 0  3.1 1.3  /

\plot 3.8 4  3.8 3  2.9 2  4 1  4 0 /
\plot 4.7 3  5.1 2.65 /
\plot 5.5 2.3  8 0 /

\setdots <.5mm> 

\multiput{$\circ$} at 3.8 2
       2.7 3  4.7 3
       3.6 4  
       3.6 5 /

\multiput{$\bullet$} at
  0 0  2 0  4 0  6 0  8 0
  0.9 1  2.9 1  4.9 1  6.9 1
  1.8 2  5.8 2  / 
\plot 0 0  1.8 2  /
\plot 5.8 2   8 0 /
\plot 0.9 1  2 0  2.8 1 /
\plot 1.8 2  4 0  5.8 2 /
\plot 4.9 1  6 0 6.9 1 /

\setdashes <.5mm>
\plot 1.8 2  3.6 4  5.8 2 / 
\plot 2.8 1  4.7 3 /
\plot  2.7 3  4.9 1 /
\plot 3.8 2   3.8 3  /
\plot 3.6 4  3.6 5 /
\plot 2.7 4   3.6 5  4.7 4 /
\plot 3.6 5  3.6 6  /

\setdots <.5mm> 

\multiput{$\bullet$} at 4 1
  2.9 2  4.9 2 
  1.8 3  3.8 3  5.8 3
  2.7 4  4.7 4  / %
\plot 4 0  4 1 /
\plot   2.9 1   2.9 2 /
\plot 4.9 1  4.9 2 /
\plot  1.8 2    1.8 3  /
\plot 5.8 2  5.8 3 /
\plot  2.7 3    2.7 4  /
\plot 4.7 3  4.7 4 /
\plot   4 1  5.8 3  4.7 4 /
\plot 2.7 4  
    1.8 3  4 1 / 
\plot 2.9 2  4.7 4 /
\plot  4.9 2   2.7 4  /

\multiput{$\bullet$} at 3.8 4
    2.7 5  4.7 5
    3.6 6 /
\plot  3.8 3  3.8 4 /
\plot 2.7 4  2.7 5  /
\plot 4.7 4  4.7 5 /
\plot 3.8 4
    2.7 5   3.6 6  4.7 5
   3.8 4 /

\setshadegrid span <.6mm>
\vshade 1.8 2 3  <,z,,> 4 0 1  <z,,,> 5.8 2 3 /
\vshade 2.7 4 5  <,z,,> 3.8 3 4  <z,,,> 4.7 4 5 /

\multiput{$\bullet$} at 1.6 6  5.6 6  2.5 7  4.5 7  3.4 8  / 
\plot 2.7 5  1.6 6  3.4 8  5.6 6  4.7 5 /
\plot 2.5 7  3.6 6  4.5 7 / 

\multiput{$\bullet$} at 3.4 9  3.4 10  / 
\plot 3.4 8  3.4 10 /
  
\put{$\bullet$} at 3.4 0 
\plot 3.4 0  4 1 /

\endpicture} at 10 0 
\endpicture}
$$

	\medskip 

Actually, the existence of a solid subchain decomposition 
concerns a local property, namely it concerns the bipartite
subposets $\Phi_{t-1,t}$ of all roots of height $t-1$ and $t$: As we know, we have
$|\Phi_{t}| \le |\Phi_{t-1}|$ (this just means that the height partition is a Young partition).
The assertion is that there is a {\bf matching} which involves all the vertices of
$\Phi_{t}$.

In particular, if we want to construct a solid subchain decomposition, 
we should start at the top of the poset $\Phi_+$ and go down. If the 
subchains have reached the layer $\Phi_t$, we have to look at
$\Phi_{t-1,t}$ and we have to continue the path downwards inside a matching.
For example, in case $\mathbb B_3$, starting with the maximal element $z$, the next two
choices for a solid chain containing $z$ are arbitrary, but then in $\Phi_{2,3}$ we have to
be careful:
$$
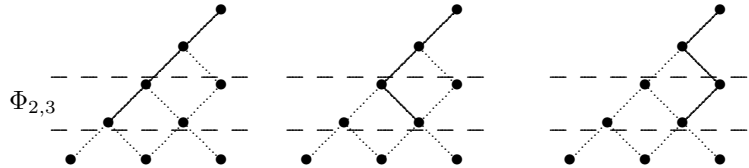
{\beginpicture
\setcoordinatesystem units <.5cm,.5cm> 
\put{\beginpicture
\multiput{$\bullet$} at 0 0  1 1  2 2  3 3  4 4  2 0  3 1  4 2  4 0  /
\setdots <.6mm>
\plot 0 0  4 4 /
\plot 1 1  2 0  4 2  3 3 /
\plot 2 2  4 0 /
\setsolid 
\plot 4 4  1 1 /
\setdashes <2mm>
\plot -.5 0.8  5 0.8 /
\plot -.5 2.2  5 2.2 /
\put{$\Phi_{2,3}$} at -1 1.5
\endpicture} at 0 0
\put{\beginpicture
\multiput{$\bullet$} at 0 0  1 1  2 2  3 3  4 4  2 0  3 1  4 2  4 0  /
\setdots <.6mm>
\plot 0 0  4 4 /
\plot 1 1  2 0  4 2  3 3 /
\plot 2 2  4 0 /
\setsolid 
\plot 4 4  2 2  3 1 /
\setdashes <2mm>
\plot -.5 0.8  5 0.8 /
\plot -.5 2.2  5 2.2 /
\endpicture} at 7 0
\put{\beginpicture
\multiput{$\bullet$} at 0 0  1 1  2 2  3 3  4 4  2 0  3 1  4 2  4 0  /
\setdots <.6mm>
\plot 0 0  4 4 /
\plot 1 1  2 0  4 2  3 3 /
\plot 2 2  4 0 /
\setsolid 
\plot 4 4  3 3  4 2  3 1 /
\setdashes <2mm>
\plot -.5 0.8  5 0.8 /
\plot -.5 2.2  5 2.2 /
\endpicture} at 14 0
\endpicture}
$$
The choice in the middle does not work, since $\Phi_{2,3}$ has just one
matching namely:
$$
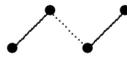
{\beginpicture
\setcoordinatesystem units <.5cm,.5cm> 
\multiput{$\bullet$} at 1 1  2 2  3 1  4 2 /
\setdots <.6mm>
\plot 2 2  3 1 /
\setsolid 
\plot 1 1  2 2 /
\plot 3 1  4 2 /
\endpicture}
$$

\subsection{Inductive determination of the exponents}
 
We have seen that there is a unified, but quite mysterious way to deal with 
some of the Dynkin functions: to invoke the exponents of $\Delta$.
Usually, the exponents seem to fall from heaven: either by looking at the 
invariant theory of the action of the Weyl group on the ambient space 
of the root system (Chevalley 1955),
or by dealing with the eigenvalues of a Coxeter element (Coxeter, 1951).
As we have mentioned, Shapiro and Kostant (1959) and later also Macdonald (1972)
have shown that there is a third
possibility to obtain the exponents, namely as the conjugate partition of
the height partition of the root poset.
It is of interest that in this way one may determine the exponents
inductively, going up step by step in a chain of poset ideals $I$ of $\Phi_+$.
A recent result of Sommers-Tymoczko \cite{[SoT]} (and Abe-Barakat-Cuntz-Hoge-Terao \cite{[A]})
based on old investigations 
of Arnold and Saito (1979)
asserts that the set $\H = \H(I)$ of hyperplanes orthogonal to the roots in $I$
is a so-called free hyperplane arrangement. This means that the corresponding
module $D(\H)$ of $\H$-derivations is free, thus one may consider 
the degrees of a free homogeneous generating system of $D(\H)$. We obtain in this way an
increasing sequence of Young partitions which terminates in the partition of the
exponents. 
	\medskip

\subsubsection{\bf Hyperplane arrangements.}
We consider a finite-dimensional vector space $V$ over the field of real numbers, say of dimension $n$.
A finite set $\H$ of (pairwise different) subspaces of dimension $n-1$ will be called
a {\it hyperplane $n$-arrangement,} or just an arrangement. As a general reference for 
hyperplane arrangements, we refer to the book \cite{[OT]} by Orlik and Terao, see also the
note N\,\ref{arrangements}.

An element $g\in \GL(V)$ is called a {\it reflection} provided $g$ is an involution (this means
$g^2 = 1 \neq g$) and fixes
pointwise a hyperplane.
A finite subgroup $G$ of $\GL(V)$  generated by reflections is called a (real) {\it reflection group.}

A hyperplane arrangement in $V$ is called a {\it Coxeter arrangement} if the hyperplanes
are those given by the reflections in a (real) reflection group $G\subseteq \GL(V)$.
The {\it Weyl arrangements} are those given by a
Weyl group, thus by a finite root system. Weyl arrangements are of course Coxeter arrangements.
	\medskip 

\subsubsection{\bf  The module $D(\H)$ of $\H$-derivations.} 
Let $S = \mathbb R[V]$ be the ring of regular functions $V \to \mathbb R$, this is the symmetric algebra
of $V^* = \Hom(V,\mathbb R)$. If we choose $x_1,\dots,x_n$ in $V^*$, then $S$ can be identified with
the ring $\mathbb R[x_1,\dots,x_n]$ of polynomials in the variables $x_1,\dots, x_n$.
We always consider $S$ as a $\mathbb Z$-graded ring with all variables having degree $1$.

Instead of looking at a hyperplane $H$, we may look at non-zero linear polynomials $\alpha
= \alpha_H$ (this is an element of $S$ of degree $1$); the corresponding hyperplane is just the kernel
of $\alpha_H$). Note that $\alpha_H$ is determined  by $H$ only up to
a non-zero scalar (but in the following, this usually will not matter).

Recall the definition of a {\it derivation} $\theta\!:S \to M$, where $M$ is an $S$-$S$-bimodule:
it is an $\mathbb R$-linear map such that 
$\theta(fg) = \theta(f)g+f\theta(g).$
The set $\Der(S)$ of derivations $S \to S$ is an $S$-module, using the following operations: 
Let $\theta, \theta'$ belong to $\Der(S),$ and $h\in S$.
Then $\theta+\theta'$ is defined by $(\theta+\theta')(s) = \theta(s)+\theta'(s)$; and
$h\theta$ is defined by $(h\theta)(f) = h\cdot \theta(f)$. 
Actually, the following is true (and easy to
verify): {\it The set $\Der(S)$ of derivations $S\to S$ is a free $S$-module 
with basis $D_i = \partial/\partial x_i$} (where we have chosen a basis $e_1,\dots,e_n$ of $V$).  
A derivation of the form $\theta = \sum f_iD_i$ 
with homogeneous polynomials $f_i$ of fixed degree $p$ is said to be {\it homogeneous
of (polynomial) degree $p = \pdeg\theta$}. Of course, any element of $\Der(S)$ can be written
as a sum of homogeneous polynomials. For $v\in V$, let $D_v$ be the derivation defined by
$$
 D_v(\alpha) = \alpha(v) \quad \text{for}\ \ \alpha\in S_1
$$
(and extended to all of $S$ by the derivation rule). In particular, $D_i = D_{e_i}$, where 
$e_1,\dots, e_n$ is the chosen basis of $V$.
Of course, the map
$$ 
  S\otimes_{\mathbb R} V \to \Der(S) \quad\text{defined by}\ s\otimes v \mapsto s D_v
$$
for $s\in S$ and $v\in V$
is an isomorphism of graded $S$-modules. 
	
A typical example of a derivation $S\to S$ is the {\it Euler derivation} 
$\theta_E = \sum x_i D_i$; its polynomial degree is $1$ and it has the
following important property: If $f\in S$ has degree $t$, then $\theta_E(f) = t\cdot f$,
in particular, if $\alpha$ is linear, then $\theta_E(\alpha) = \alpha.$
	
Let $\H$ be a hyperplane arrangement. 
The {\it module $D(\H)$ of $\H$-derivations}
is the set of all derivations $\theta$ of $S$ such that
$\theta(\alpha_H) \subseteq \alpha_HS$ for all $H\in \H$, note that this is 
an $S$-submodule of $\Der(S)$. 
For example, $\theta_E$ always belongs to $D(\H)$. 
	\medskip 

\subsubsection{\bf Free arrangements and their exponents.}
An arrangement $\H$ is said to be {\it free} provided the $S$-module $D(\H)$ is free.
If $D(\H)$ is free, then there is a free set of homogeneous generators
and {\it the weakly increasing sequence of degrees of  
a minimal set of homogeneous generators is uniquely determined} (see \cite{[OT]}, 4.18). 
We call a free set of homogeneous generators of $D(\H)$ a
{\it basic set of $\H$-derivations} and the weakly increasing sequence of degrees of the set
the {\it exponent partition for $\H.$}
	\medskip

Let $\H$ be a hyperplane arrangement and $H$ an element of $\H$.
If both $\H$ and 
$\H' = \H\setminus\{H\}$ are free arrangements, then there are basic sets for $\H$ and
$\H'$ which are nicely related to each other, namely:  
{\it There 
is a basic set $\{\theta_1,\dots,\theta_l\}$ of $\H'$-derivations and an index $i$ such that
$$
  \{\theta_1,\dots,\theta_{i-1},\alpha_H\theta_i,\theta_{i+1}\dots,\theta_l\}
$$ 
is a basic set of $\H$-derivations} (see  \cite{[OT]}, 4.46).
	\bigskip 

We call an ordering $\{H_1,\dots,H_m\}$ of the elements of $\H$ a {\it free ordering} provided 
all the subarrangements $\{H_1,\dots,H_t\}$ are free for $1\le t \le m$.
 	\bigskip

\subsubsection{\bf Ideal subarrangements of Weyl arrangements.}
An {\it ideal subarrangement} is given by the hyperplanes $H_x$, where $x$ belongs to
some fixed ideal $I$ in a root poset $\Phi_+$. 
If $I$ is an ideal in a root poset $\Phi_+$, let $r_I(t)$ be the number of roots in $I$ of height $t$.
We claim that {\it $r_I$ is a Young partition.} This follows directly from the fact that 
the root poset can be covered by $n$ solid subchains, see Section 1.4.
	
\begin{theorem} {\bf (Sommers-Tymoczko).} 
Every ideal of a root poset yields a free hyperplane arrangement whose exponent partition
is the dual partition of the height partition $r_I$.
\end{theorem} 
	
The proof of Sommers-Tymoczko \cite{[SoT]} covered all cases but $\mathbb F_4, \mathbb E_6, 
\mathbb E_7$ and $\mathbb E_8$; the missing cases were verified by Barakat using a computer. 
A unified proof was later given by Abe-Barakat-Cuntz-Hoge-Terao \cite{[A]}.

\begin{cor} Let $\H$ be a Weyl arrangement.
Any total ordering which refines the partial ordering of $\Phi_+$ is a free ordering.
\end{cor}
	\medskip 

\subsubsection{\bf Basic sets of invariants, of differential $1$-forms, of derivations.}
We assume that $G$ is a (real) reflection group and $\H$  the corresponding hyperplane
arrangement. 
It remains to be seen that the exponents for $\H$ are just the exponents for $G$ (and we
will see again that $\H$ is free). We outline here only the main steps.
We start with the main theorem of the invariant theory of reflection groups:

\begin{theorem} {\bf (Chevalley, 1955)}. Let $G\subseteq \GL(V)$ be a reflection group, where $V$
is an $n$-dimensional vector space. Let $S$ be the ring of regular functions on $V$ and $R = S^G$.
Then there are homogeneous polynomials $f_1,\dots,f_n$ which generate $S^G$.
And there exists a finite-dimensional $G$-stable graded subspace $U$ of $S$ such that 
$S = R\otimes_{\mathbb R} U$
and the $G$-module $U$ is the regular representation.
\end{theorem} 

The polynomials $f_1,\dots,f_n$ which generate $S^G$ have to be
algebraically independent, thus $R$ is isomorphic to a polynomial ring with free generators 
$f_1,\dots,f_n$. We call $f_1,\dots,f_n$ a {\it basic set of invariants} and 
the weakly increasing sequence of degrees of such a set
the sequence of {\it degrees for $G.$}
	 
\begin{proof} Bourbaki \cite{[B]} V,5.3. Theor\`emes 1, 2.
\end{proof} 
	
The group $G$ operates on $\Der(S)$ as follows 
$$
 (g\theta)(v) = g\theta(g^{-1}v).
$$ 
Let
$\Der(S)^G$ be the set of $G$-invariant derivations of $S$, thus the set of derivations
$\theta$ of $S$ with $g\theta = \theta$. 
	\bigskip 

{\bf (1)} {\it The graded $S$-modules $S\otimes_R \Der(S)^G$ and $D(\H)$ are isomorphic.}
	\medskip

Proof: It is shown in \cite{[OT]}, 6.59 that $\Der(S)^G \subseteq D(\H)$, this yields a map
$$
  S\otimes_R \Der(S)^G \to D(\H).
$$
One uses the tensor factorization $S = U\otimes R$
in order to show that the map is an embedding. Further calculations in \cite{[OT]} (p.237/8) show the
surjectivity. 
	\bigskip 

{\bf (2)  }{\it The $R$-module  $\Der(S)^G$
is a free $R$-module of rank $n$.} 
	\medskip

For the proof see \cite{[OT]}, 6.48. We note that the 
(bijective) map $S\otimes V \to \Der(S)$ given by $s\otimes v \mapsto sD_v$ is
$G$-equivariant, thus $\Der(S)^G \simeq (S\otimes V)^G,$
therefore
$$
(S\otimes_{\mathbb R} V)^G = 
(R\otimes_{\mathbb R} U \otimes_{\mathbb R} V)^G = R\otimes_{\mathbb R}(U\otimes_{\mathbb R}V)^G
$$
shows that $(S\otimes_{\mathbb R} V)^G$ is a free $R$-module. 
Since $U$ is the regular representation of $G$, the dimension of the space
$(U\otimes_{\mathbb R}V)^G$ is equal to the dimension of $V$.
	\bigskip 

Let $\Omega(S) = \bigoplus S\diff x_i$ be the $S$-module of all differential $1$-forms, this
is the free $S$-module with basis $\diff x_1,\dots,\diff x_n$ 
(note that $\Omega(S)$ is often denoted by $\Omega^1[V]$;
for the general setting, see \cite{[OT]}, Section 3.5 and  Section 4.6, p.123).
As usual, given $f\in S$, we write $\diff f = \sum D_i(f)\diff x_i$, and obtain in this way 
a derivation $\diff\!:S \to \Omega(S).$ 
The map
$$ 
  S\otimes_{\mathbb R} V^* \to \Omega(S) \quad\text{defined by}\ s\otimes \alpha \mapsto s \diff\alpha
$$
for $s\in S$ and $\alpha\in V^*$ is an isomorphism of graded $S$-modules. 
Using this identification, we see that we may consider the elements $\omega\in \Omega(S)$ 
as maps $\Der(S) \to S$.

The group $G$ operates on $\Omega(S)$ as follows 
$$
 (g\omega)(\theta) = g\omega(g^{-1}\theta),
$$ 
where $\omega\in \Omega(S)$ and $\theta\in \Der(S).$
Let
$\Omega(S)^G$ be the set of $G$-invariant differential $1$-forms, thus the set of all 
$\omega$ in $\Omega(S)$ with $g\omega = \omega$. 
	\bigskip 

{\bf (3)} {\it The $R$-modules $\Omega(S)^G$ and $\Der(S)^G$
are isomorphic.} 
	\medskip

In order to prove (3), we first note: 
{\it For a (real) reflection group $G \subseteq \GL(V)$, the $G$-modules $V$ and $V^*$
are isomorphic,}  therefore also the $G$-modules $S\otimes V$ and $S\otimes V^*$ are
isomorphic. It follows that the $R$-modules $(S\otimes V)^G $ and $(S\otimes V^*)^G$
are isomorphic. 

The map $S\otimes V \to \Der(S)$ defined by $s\otimes v \mapsto s\cdot D_v$ 
and $S\otimes V^* \to \Omega(S)$ defined by $s\otimes v \mapsto s\cdot \diff v$ are bijective and
$G$-equivariant, thus we get isomorphisms of $R$-modules
$$
 (S\otimes V)^G \to \Der(S)^G,\quad 
 (S\otimes V^*)^G \to \Omega(S)^G.
$$
	\bigskip 

The assertions (1) and (3) provide isomorphisms of graded $S$-modules:
	\smallskip 

\Rahmen{$ 
 D(\H)
   \underset{(1)}\simeq  
 S\otimes_R\Der(S)^G 
   \underset{(3)}\simeq
 S\otimes_R\Omega(S)^G 
 $}
	\medskip 

\noindent
and (2) asserts that the module in the middle is free, thus $D(\H)$ is a free
$S$-module.
		\bigskip

{\bf (4) (Solomon 1964).} {\it If $f_1,\dots,f_n$ is a basic set of invariants for $G$, 
then $\diff f_1,\dots,\diff f_n$ is basic set for the $R$-module $\Omega(S)^G$,}
thus a basic set for the $S$-module $S\otimes_R\Omega(S)^G$.
	\bigskip\medskip 

\begin{cor} {\bf (V.I.Arnold, K.Saito, 1979).} 
Let $G$ be a reflection group and $\H$ the
corresponding reflection arrangement. Then $\H$ is a free arrangement.
Let $d_1 \ge d_2 \ge \cdots \ge d_n$ be the degrees for $G$, 
and let $\epsilon_1 \ge \epsilon_2 \ge \cdots \ge \epsilon_n$ be the exponents for the
corresponding hyperplane arrangement $\H$. Then $\epsilon_i = d_i-1$ for $1\le i \le n.$
\end{cor}

{\bf Remark.} Solomon's result explains nicely why the degrees and the exponents for a (real)
reflection group differ by 1.
	\bigskip\bigskip

{\bf Notes.}

\begin{note}\label{system} {\bf 
Root systems, root bases, and the Dynkin diagrams.} 
We start with a finite-dimensional Euclidean vector space $V$, say of dimension $n$.
Given a non-zero element $\alpha\in V$, 
we may consider the hyperplane $H_\alpha$ orthogonal to
$\alpha$ and we denote by $\rho_\alpha\!:V \to V$ the reflection with respect
to $H_\alpha$: it is the identity map on $H_\alpha$ and it 
sends $\alpha$ to $-\alpha$. Note that if $\alpha\neq 0$ (so that $\rho_\alpha$
is defined), the difference vector $\rho_\alpha(\beta)-\beta$ is a multiple of $\alpha$,
for any $\beta \in V$.

A {\it root system} $\Phi$ in $V$ is a finite set of non-zero elements of $V$
(the {\it roots}) such that the following conditions are satisfied: The elements
of $\Phi$ generate $V$. If $\alpha,\beta$ belong to $V$ and generate the same real
subspace, then $\beta = \pm \alpha$. Finally (and this is the decisive condition):
If $\alpha, \beta$ belong to $\Phi$, then $\rho_\alpha(\beta)$ belongs 
again to $\Phi$ and the difference vector $\rho_\alpha(\beta)-\beta$ is an integral
multiple of $\alpha$.

By the very definition, we see that root systems are finite subsets of
Euclidean vector spaces with strong symmetry conditions. The subgroup of  $\GL(V)$
generated by the reflections $\rho_\alpha$ with $\alpha\in \Phi$ is called
the {\it Weyl group} for $\Phi$. It consists of orthogonal transformations of $V$
which map $\Phi$ into itself.

Root system play an important role in many parts of mathematics. In particular,
they have been used to obtain a structure theory for finite-dimensional semisimple 
complex Lie-algebras. Thus, any book which introduces such Lie algebras (for
example \cite{[H1]})  
provides a lot of information about root systems.   

Of importance is the following theorem:
{\it Given a root system $\Phi$ in $V$, there is a basis $\Delta$ of $V$ which consists of
elements on $\Phi$, such that any element of $\Phi$ is a linear combination of 
the elements of $\Delta$ using integral coefficients which are either all
only non-negative or all non-positive.} Such a basis $\Delta$ is called a {\it root
basis} of $\Phi$, it is unique up to an automorphism of $V$ given by an element of
the Weyl group $W$. Given a root basis $\Delta$, the elements of $\Phi$ which are
linear combinations of the elements of $\Delta$ with non-negative coefficients are
said to be the {\it positive} roots, the remaining ones the {\it negative} roots. 

Given a root basis $\Delta$, the possible angles between 
its elements are very restricted: Either the roots are orthogonal, or else the angle is
$120^0, 135^0$ or $150^0$. One uses the set $\Delta$ as the vertices of a graph,
connecting
two vertices of $\Delta$ by an edge provided they are not orthogonal. If the angle 
between $\alpha,\beta\in \Delta$ is
$135^0$ or $150^0$, then these vectors must have different length: in case of $135^0$
one draws not a single edge, but a double edge \; {\beginpicture 
\setcoordinatesystem units <.6cm,.7cm>
\multiput{$\circ$} at 0 0  1 0  /
\put{} at 0 0
\plot 0.2 0.075  0.8 0.075 /
\plot 0.2 -.075  0.8 -.075 /
\plot  0.65 0.2  0.35 0  0.65 -.2 /
\endpicture},\ decorated in the middle by an arrow head pointing
to the shorter root. Similarly, 
in case of $150^0$
one draws a triple edge \; {\beginpicture \setcoordinatesystem units <.6cm,.7cm>
\multiput{$\circ$} at 0 0  1 0  /
\put{} at 0 0
\plot 0.2 0.075  0.8 0.075 /
\plot 0.2 0  0.8 0 /
\plot 0.2 -.075  0.8 -.075 /
\plot  0.65 0.2  0.35 0  0.65 -.2 /
\endpicture}, 
such that the arrow head again points to the shorter root. One obtains in this way
the so-called {\it Dynkin diagram} $\Delta(\Phi)$ of $\Phi$. We exhibit below all the
connected Dynkin diagrams which arise in this way. 
	\medskip 
\vfill\eject 
{\bf The connected Dynkin diagrams (left) and the highest roots (right).}
$$
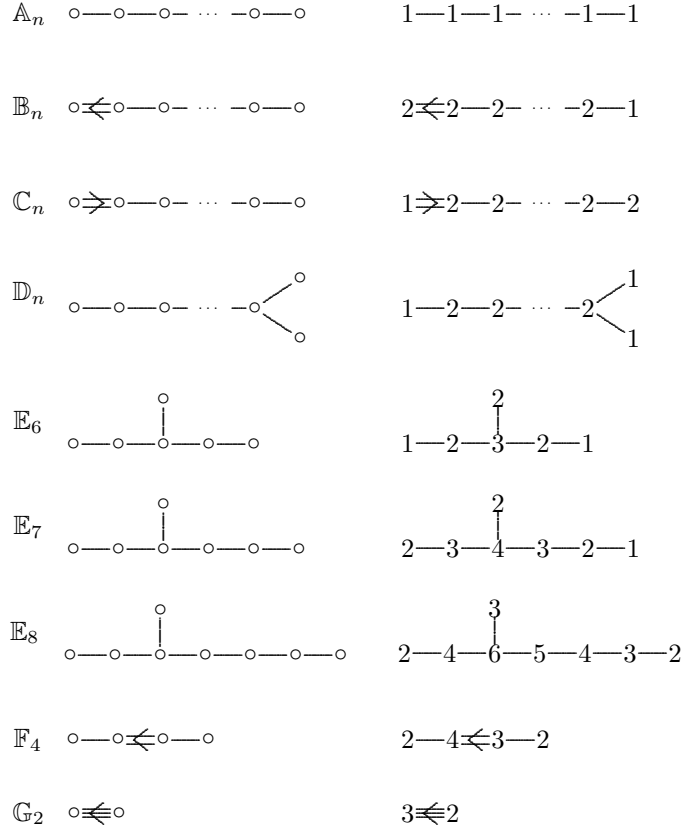
{\beginpicture
\setcoordinatesystem units <.6cm,.7cm>
\put{\beginpicture
\multiput{$\circ$} at 0 0  1 0  2 0  4 0  5 0 /
\plot 0.2 0  0.8 0 /
\plot 1.2 0  1.8 0 /
\plot 2.2 0  2.5 0 /
\plot 3.5 0  3.8 0 /
\plot 4.2 0  4.8 0 /
\setdots <1mm>
\plot 2.8 0  3.2 0 /
\multiput{} at  6 0 /
\put{$\mathbb A_n$} at -1 0
\endpicture} at 0 -.4 
\put{\beginpicture
\multiput{$\circ$} at 0 0  1 0  2 0  4 0  5 0 /
\plot 0.2 0.075  0.8 0.075 /
\plot 0.2 -.075  0.8 -.075 /
\plot  0.65 0.2  0.35 0  0.65 -.2 /
\plot 1.2 0  1.8 0 /
\plot 2.2 0  2.5 0 /
\plot 3.5 0  3.8 0 /
\plot 4.2 0  4.8 0 /
\setdots <1mm>
\plot 2.8 0  3.2 0 /
\multiput{} at  6 0 /
\put{$\mathbb B_n$} at -1 0
\endpicture} at 0 -2.2
\put{\beginpicture
\multiput{$\circ$} at 0 0  1 0  2 0  4 0  5 0 /
\plot 0.2 0.075  0.8 0.075 /
\plot 0.2 -.075  0.8 -.075 /
\plot  0.35 0.2  0.65 0  0.35 -.2 /

\plot 1.2 0  1.8 0 /
\plot 2.2 0  2.5 0 /
\plot 3.5 0  3.8 0 /
\plot 4.2 0  4.8 0 /
\setdots <1mm>
\plot 2.8 0  3.2 0 /
\multiput{} at  6 0 /
\put{$\mathbb C_n$} at -1 0
\endpicture} at 0 -4 
\put{\beginpicture
\setcoordinatesystem units <.6cm,.4cm>
\multiput{$\circ$} at 0 0  1 0  2 0  4 0  5 1  5 -1 /
\plot 0.2 0  0.8 0 /
\plot 1.2 0  1.8 0 /
\plot 2.2 0  2.5 0 /
\plot 3.5 0  3.8 0 /
\plot 4.2 0.2  4.8 0.8 /
\plot 4.2 -.2  4.8 -.8 /
\setdots <1mm>
\plot 2.8 0  3.2 0 /
\multiput{} at  6 0 /
\put{$\mathbb D_n$} at -1 0.5
\endpicture} at 0 -6

\put{\beginpicture
\setcoordinatesystem units <.6cm,.6cm>
\multiput{$\circ$} at 0 0  1 0  2 0  3 0  4 0  2 1  /
\plot 0.2 0  0.8 0 /
\plot 1.2 0  1.8 0 /
\plot 2.2 0  2.8 0 /
\plot 3.2 0  3.8 0 /
\plot 2 0.2  2 0.8 /
\multiput{} at 2 1.5  6  0 /
\put{${\mathbb E}_6$} at -1 0.5
\endpicture} at 0 -8

\put{\beginpicture
\setcoordinatesystem units <.6cm,.6cm>
\multiput{$\circ$} at 0 0  1 0  2 0  3 0  4 0  5 0  2 1  /
\plot 0.2 0  0.8 0 /
\plot 1.2 0  1.8 0 /
\plot 2.2 0  2.8 0 /
\plot 3.2 0  3.8 0 /
\plot 4.2 0  4.8 0 /
\plot 2 0.2  2 0.8 /
\put{} at 2 1.5
\multiput{} at 6 0 /
\put{${\mathbb E}_7$} at -1 0.5
\endpicture} at 0 -10
\put{\beginpicture
\setcoordinatesystem units <.6cm,.6cm>
\multiput{$\circ$} at 0 0  1 0  2 0  3 0  4 0  5 0  6 0  2 1  /
\plot 0.2 0  0.8 0 /
\plot 1.2 0  1.8 0 /
\plot 2.2 0  2.8 0 /
\plot 3.2 0  3.8 0 /
\plot 4.2 0  4.8 0 /
\plot 5.2 0  5.8 0 /
\plot 2 0.2  2 0.8 /
\put{} at 2 1.5
\multiput{} at 6 0 /
\put{${\mathbb E}_8$} at -1 0.5
\endpicture} at 0 -12
\put{\beginpicture
\multiput{$\circ$} at 0 0  1 0  2 0  3 0  /
\plot 0.2 0  0.8 0 /
\plot 1.2 0.075  1.8 0.075 /
\plot 1.2 -.075  1.8 -.075 /
\plot  1.65 0.2  1.35 0  1.65 -.2 /

\plot 2.2 0  2.8 0 /
\setdots <1mm>
\multiput{} at  6 0 /
\put{$\mathbb F_4$} at -1 0
\endpicture} at 0 -14.2

\put{\beginpicture
\multiput{$\circ$} at 0 0  1 0  /
\plot 0.2 0.075  0.8 0.075 /
\plot 0.2 0  0.8 0 /
\plot 0.2 -.075  0.8 -.075 /
\plot  0.65 0.2  0.35 0  0.65 -.2 /
\multiput{} at  6 0 /
\put{$\mathbb G_2$} at -1 0
\endpicture} at 0 -15.6

\put{\beginpicture
\put{$1$} at 0 0 
\put{$1$} at 1 0 
\put{$1$} at 2 0 
\put{$1$} at 4 0 
\put{$1$} at 5 0 
\plot 0.2 0  0.8 0 /
\plot 1.2 0  1.8 0 /
\plot 2.2 0  2.5 0 /
\plot 3.5 0  3.8 0 /
\plot 4.2 0  4.8 0 /
\setdots <1mm>
\plot 2.8 0  3.2 0 /
\multiput{} at  6 0 /
\endpicture} at 8 -.4 
\put{\beginpicture
\put{$2$} at 0 0 
\put{$2$} at 1 0 
\put{$2$} at 2 0 
\put{$2$} at 4 0 
\put{$1$} at 5 0 
\plot 0.2 0.075  0.8 0.075 /
\plot 0.2 -.075  0.8 -.075 /

\plot  0.65 0.2  0.35 0  0.65 -.2 /
\plot 1.2 0  1.8 0 /
\plot 2.2 0  2.5 0 /
\plot 3.5 0  3.8 0 /
\plot 4.2 0  4.8 0 /
\setdots <1mm>
\plot 2.8 0  3.2 0 /
\multiput{} at  6 0 /
\endpicture} at 8 -2.2
\put{\beginpicture
\put{$1$} at 0 0 
\put{$2$} at 1 0 
\put{$2$} at 2 0 
\put{$2$} at 4 0 
\put{$2$} at 5 0 
\plot 0.2 0.075  0.8 0.075 /
\plot 0.2 -.075  0.8 -.075 /
\plot  0.35 0.2  0.65 0  0.35 -.2 /

\plot 1.2 0  1.8 0 /
\plot 2.2 0  2.5 0 /
\plot 3.5 0  3.8 0 /
\plot 4.2 0  4.8 0 /
\setdots <1mm>
\plot 2.8 0  3.2 0 /
\multiput{} at  6 0 /
\endpicture} at 8 -4 
\put{\beginpicture
\setcoordinatesystem units <.6cm,.4cm>
\put{$1$} at 0 0 
\put{$2$} at 1 0 
\put{$2$} at 2 0 
\put{$2$} at 4 0 
\put{$1$} at 5 1 
\put{$1$} at 5 -1
\plot 0.2 0  0.8 0 /
\plot 1.2 0  1.8 0 /
\plot 2.2 0  2.5 0 /
\plot 3.5 0  3.8 0 /
\plot 4.2 0.2  4.8 0.8 /
\plot 4.2 -.2  4.8 -.8 /
\setdots <1mm>
\plot 2.8 0  3.2 0 /
\multiput{} at  6 0 /
\endpicture} at 8 -6

\put{\beginpicture
\setcoordinatesystem units <.6cm,.6cm>
\put{$1$} at 0 0 
\put{$2$} at 1 0 
\put{$3$} at 2 0 
\put{$2$} at 3 0 
\put{$1$} at 4 0 
\put{$2$} at 2 1
\plot 0.2 0  0.8 0 /
\plot 1.2 0  1.8 0 /
\plot 2.2 0  2.8 0 /
\plot 3.2 0  3.8 0 /
\plot 2 0.2  2 0.8 /
\multiput{} at 2 1.5  6  0 /
\endpicture} at 8 -8

\put{\beginpicture
\setcoordinatesystem units <.6cm,.6cm>
\put{$2$} at 0 0 
\put{$3$} at 1 0 
\put{$4$} at 2 0 
\put{$3$} at 3 0 
\put{$2$} at 4 0 
\put{$1$} at 5 0
\put{$2$} at 2 1 
\plot 0.2 0  0.8 0 /
\plot 1.2 0  1.8 0 /
\plot 2.2 0  2.8 0 /
\plot 3.2 0  3.8 0 /
\plot 4.2 0  4.8 0 /
\plot 2 0.2  2 0.8 /
\put{} at 2 1.5
\multiput{} at 6 0 /
\endpicture} at 8 -10
\put{\beginpicture
\setcoordinatesystem units <.6cm,.6cm>
\put{$2$} at 0 0 
\put{$4$} at 1 0 
\put{$6$} at 2 0 
\put{$5$} at 3 0 
\put{$4$} at 4 0 
\put{$3$} at 5 0
\put{$2$} at 6 0 
\put{$3$} at 2 1 
\plot 0.2 0  0.8 0 /
\plot 1.2 0  1.8 0 /
\plot 2.2 0  2.8 0 /
\plot 3.2 0  3.8 0 /
\plot 4.2 0  4.8 0 /
\plot 5.2 0  5.8 0 /
\plot 2 0.2  2 0.8 /
\put{} at 2 1.5
\multiput{} at 6 0 /
\endpicture} at 8 -12
\put{\beginpicture
\put{$2$} at 0 0 
\put{$4$} at 1 0 
\put{$3$} at 2 0 
\put{$2$} at 3 0 
\plot 0.2 0  0.8 0 /
\plot 1.2 0.075  1.8 0.075 /
\plot 1.2 -.075  1.8 -.075 /
\plot  1.65 0.2  1.35 0  1.65 -.2 /
\plot 2.2 0  2.8 0 /
\setdots <1mm>
\multiput{} at  6 0 /
\endpicture} at 8 -14.2
\put{\beginpicture
\put{$3$} at 0 0 
\put{$2$} at 1 0 
\plot 0.2 0.075  0.8 0.075 /
\plot 0.2 0  0.8 0 /
\plot 0.2 -.075  0.8 -.075 /
\plot  0.65 0.2  0.35 0  0.65 -.2 /
\multiput{} at  6 0 /
\endpicture} at 8 -15.6
\put{} at  8 -16.5
\endpicture}
$$
Recall that the vertices of a Dynkin diagram $\Delta = \Delta(\Phi)$ are the
elements of a root basis of $\Phi$. If we attach to these vertices $\alpha$
numbers $c_\alpha$ (as we do it on the right), we obtain as $\sum_\alpha c_\alpha\alpha$ 
an element of $V$.
The {\it highest root} displayed on the right is an element of $\Phi$ 
which is uniquely
determined by the property that all the coefficients $c_\alpha$ are maximal.
\end{note}

\begin{note}\label{A}
{\bf The root system $\mathbb A_n$.} The root poset $\Phi_+(\mathbb A_n)$ may be
identified with the set of intervals $[i,j]$ where $0 \le i < j\le n$ are integers,
using as partial ordering the (set-theoretical) inclusion of intervals. 
For example, here is the Hasse diagram of the interval poset for $n=4.$ 
$$
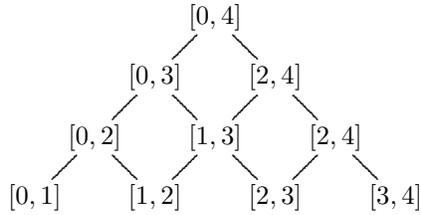
\hbox{\beginpicture
\setcoordinatesystem units <.8cm,.8cm>
\put{$[0,1]$} at 0 0 
\put{$[1,2]$} at 2 0
\put{$[2,3]$} at 4 0 
\put{$[3,4]$} at 6 0 
\put{$[0,2]$} at 1 1 
\put{$[1,3]$} at 3 1 
\put{$[2,4]$} at 5 1 
\put{$[0,3]$} at 2 2 
\put{$[2,4]$} at 4 2 
\put{$[0,4]$} at 3 3
\plot 0.3 0.3  0.7 0.7 /
\plot 1.3 0.7  1.7 0.3 /
\plot 2.3 0.3  2.7 0.7 /
\plot 3.3 0.7  3.7 0.3 /
\plot 4.3 0.3  4.7 0.7 /
\plot 5.3 0.7  5.7 0.3 /
\plot 1.3 1.3  1.7 1.7 /
\plot 2.3 1.7  2.7 1.3 /
\plot 3.3 1.3  3.7 1.7 /
\plot 4.3 1.7  4.7 1.3 /
\plot 2.3 2.3  2.7 2.7 /
\plot 3.3 2.7  3.7 2.3 /
\endpicture}
$$ 

A usual realization of the root system of type $\mathbb A_n$ (see for example
Humphreys \cite{[H1]}) is to start
with the Euclidean vector space $\mathbb R^{n+1}$ with orthonormal 
basis $e_0,e_1,\dots,e_n$. Let $V$ be the subspace of $\mathbb R^{n+1}$
orthogonal to the vector $\sum_i e_i.$
The vectors in $V$ with integer coefficients and of length $\sqrt 2$ 
are just the vectors of the form $e_i-e_j$
with $i\neq j$. The set of these vectors is a root system $\Phi$ and the 
vectors $e_{i-1}-e_{i}$ with $1\le i\le n$ 
form a root basis. One obtains a bijection between 
the interval poset and $\Phi_+$ by sending the interval $[i,j]$ with 
$0\le i < j \le n$ to the vector $e_i-e_j.$
\end{note}

\begin{note}\label{3-dim}
{\bf The root posets as 2-dimensional projections of 3-dimensional objects.}
The root posets may be considered as projections of
some 3-dimensional objects formed by cubes, squares and
intervals. For example, the root poset $\Phi_+(\mathbb E_7)$
shown below on the left may be thought of being obtained
from the constellation of ten cubes shown on the right
by adding 1- and 2-dimensional pieces (squares and intervals),
thus, on the right, we see its 3-dimensional ``core'':
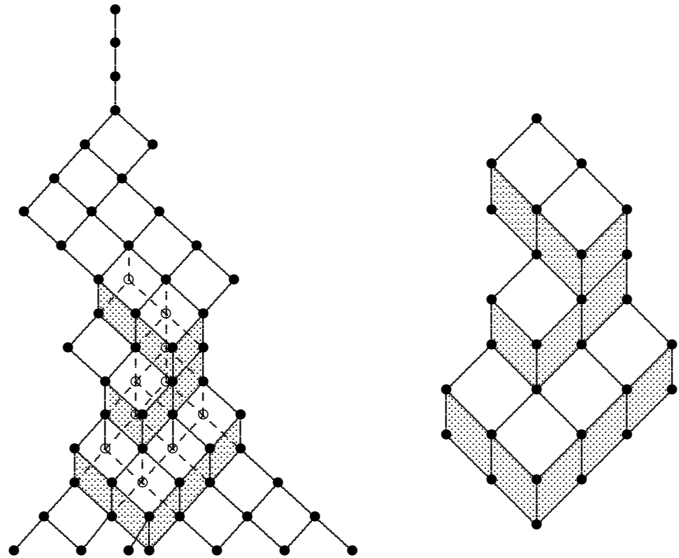
$$
\hbox{\beginpicture
\setcoordinatesystem units <1cm,1cm>
\put{\beginpicture
\setcoordinatesystem units <.45cm,.45cm>

\multiput{$\circ$} at 3.8 2
       2.7 3  4.7 3
       3.6 4  5.6 4
       4.5 5   4.5 6 
       3.6 5 
       4.5 7  3.4 8 /

\multiput{$\bullet$} at
  0 0  2 0  4 0  6 0  8 0  10 0  
  0.9 1  2.9 1  4.9 1  6.9 1  8.9 1
  1.8 2  5.8 2  7.8 2
  6.7 3 /  
\plot 0 0  1.8 2  /
\plot 5.8 2   8 0  8.9 1 /
\plot 0.9 1  2 0  2.8 1 /
\plot 1.8 2  4 0  5.8 2 /
\plot 4.9 1  6 0  7.8 2 /
\plot 5.8 2  6.7 3  10 0 /

\setdashes <1mm>
\plot 1.8 2  3.6 4  5.8 2 / 
\plot 3.6 4  4.5 5  6.7 3 /
\plot 2.8 1  5.6  4 /
\plot  2.7 3  4.9 1 /
\plot 3.8 2   3.8 3  /
\plot 3.6 4  3.6 5 /
\plot 2.7 4   3.6 5  4.7 4 /
\plot 3.6 5  3.6 6  /
\plot 5.6 4  5.6 5 /
\plot 4.5 5  4.5 6 /
\plot 3.6 5  4.5 6  5.6 5 / 
\plot  4.5 6  4.5 7 /
\plot 2.5 7  3.4 8  5.6 6 /
\plot 3.6 6  4.5 7 / 
\plot 4.5 7  4.5 8 /

\plot 3.4 8  3.4 9 /

\setsolid 
\multiput{$\bullet$} at 4 1
  2.9 2  4.9 2 
  1.8 3  3.8 3  5.8 3
  2.7 4  4.7 4 
   6.7 4  5.6 5 / 
\plot 4 0  4 1 /
\plot 2.9 1   2.9 2 /
\plot 4.9 1  4.9 2 /
\plot 1.8 2    1.8 3  /
\plot 5.8 2  5.8 3 /
\plot 2.7 3    2.7 4  /
\plot 4.7 3  4.7 4 /
\plot 4 1  5.8 3  4.7 4 /
\plot 2.7 4  
    1.8 3  4 1 / 
\plot 2.9 2  4.7 4 /
\plot  4.9 2   2.7 4  /
\plot 6.7 3  6.7 4  5.8 3  /
\plot 6.7 4  5.6 5  4.7 4 /
\plot 5.6 5  5.6 6 /

\multiput{$\bullet$} at 3.8 4
    2.7 5  4.7 5
    3.6 6 /
\plot  3.8 3  3.8 4 /
\plot 2.7 4  2.7 5  /
\plot 4.7 4  4.7 5 /
\plot 3.8 4
    2.7 5   3.6 6  4.7 5
   3.8 4 /

\plot  5.6 6  4.7 5 /

\multiput{$\bullet$} at   0.3 10  1.4 9  2.5 8   3.6 7  4.7 6  
     1.2 11   2.3 10  3.4 9  4.5 8   5.6 7 
     2.1 12  3.2 11  4.3 10  5.4 9  6.5 8   
     3. 13  4.1 12
 /
\plot 2.1 12  3 13  4.1 12  3.2 11 /
\plot 0.3 10  4.7 6  6.5 8  2.1 12  0.3 10 /
\plot 1.2 11  5.6 7 /

\plot 1.4 9  3.2 11 /
\plot 2.5 8  4.3 10 /
\plot 3.6 7  5.4 9 /

\plot 4.7 5  4.7 6 /
\plot 3.6 6  3.6 7 /
\plot 5.6 6  5.6 7 /
\plot 2.5 7  2.5 8 /

\multiput{$\bullet$} at  3  14  3 15  3 16  /
\plot 3  13  3 16 /

\setshadegrid span <.4mm>
\vshade 1.8 2 3  <,z,,> 4 0 1  <z,,,> 6.7 3 4 /
\vshade 2.7 4 5  <,z,,> 3.8 3 4  <z,,,> 5.6  5 6 /
\vshade 2.5 7 8  <,z,,> 4.7 5 6  <z,,,> 5.6   6 7 /

\multiput{$\bullet$} at 1.6 6   5.6 6  2.5 7  /

\plot 2.7 5  1.6 6   2.5 7 /

\plot 2.5 7  3.6 6 /

\put{$\bullet$} at 3.4 0 
\plot 3.4 0  4 1 /

\endpicture} at 0 0
\put{\beginpicture
\setcoordinatesystem units <.6cm,.6cm>
\multiput{$\bullet$} at 0 2  1 1  2 0  3 1  4 2  5 3 
   0 3  1 4  1 5  2 6  1 7  1 8 
   1 2  2 3  3 4  4 5  
   2 4  3 5  4 6  3 6  4 7
   2 7  3 8  2 9 
   2 1  3 2  4 3  5 4 /
\plot 2 0  0 2  0 3  1 4  1 5  2 6  1 7  1 8  3 6  4 7 /
\plot 2 0  2 1 /
\plot 1 1  1 2  4 5 /
\plot 2 0  5 3  5 4  2 1  0 3 /
\plot 3 1  3 2  1 4 /
\plot 4 2  4 3  3 4 /
\plot 5 4  4 5  4 7  2 9  1 8  /
\plot 1 5  2 4  4 6 /
\plot 2 3  2 4 /
\plot 3 4  3 6 /
\plot 3 5  2 6  2 7  3 8 /
\put{} at 0 11
\setshadegrid span <.4mm>
\vshade 0 2 3  <,z,,> 2 0 1  <z,,,> 5 3 4 /
\vshade 1 4 5  <,z,,> 2 3 4  <z,,,> 4 5 6 /
\vshade 1 7 8  <,z,,> 3 5 6  <z,,,> 4 6 7 /
\endpicture} at 5 0
\endpicture}
$$
If $P$ is a finite poset, we may call an interval $[x,y]$ in $P$ a {\it cuboid}
provided it is the product of three proper chains, say
$C_1\times C_2\times C_3$;
its  volume is $c_1c_2c_2$, where $c_i+1$ is the number of vertices of the chain $C_i$.
A {\it cube} in $P$ is a cuboid of volume 1.

For the Dynkin graphs $\Delta$ of type $\mathbb A, \mathbb D, \mathbb E$, 
Beineke \cite{[Bn]} has observed that the map
$$
            [x,y] \mapsto x+y
$$
furnishes a bijection between
the cuboids in the root poset and the non-thin positive roots,
as well as a bijection between the cubes
and the non-thin positive roots with precisely 3 odd coefficients
(a root
is said to be {\it thin} provided its coefficients are bounded by $1$).
	
We may use the cubes in $\Phi_+$ in oder to describe different levels
inside the poset.
The root posets of type $\mathbb A_n, \mathbb B_n, \mathbb G_2$
have no cubes, thus there is just one level. Those of type
$\mathbb D_n$ and $\mathbb F_4$ have two levels, the
number of levels for $\mathbb E_6, \mathbb E_7, \mathbb E_8$
is 3, 4 and 6, respectively. It should be mentioned that
the set of roots belonging to a fixed level can be described by inequalities
concerning the coefficients, and for the higher levels, these sets are always intervals.
For example, in case $\mathbb E_7$, here is the description of the levels:
$$
\hbox{\beginpicture
\setcoordinatesystem units <.85cm,.7cm>
\multiput{} at 1 1  13 0 /
\put{\bf Level} at 1 1
\put{ 1} at 1 0
\put{ 2} at 1 -1
\put{ 3} at 1 -2
\put{ 4} at 1 -3
\put{Conditions}  at 3.5 1
\put{$u = 0$}  at 3.5 0
\put{$u = 1,\ c\le 1$}  at 3.5 -1
\put{$c=2,\ d=1$}  at 3.5 -2
\put{$d\ge 2$}  at 3.5 -3
\put{Minimal elements}  at 7 1
\put{Maximal element}  at 10.7 1 
\put{6 simple roots} at 7 0
\put{$\smallmatrix   &        &  0  \cr
                   1 & 1 & 1 & 1 & 1 & 1 \endsmallmatrix$} at 10.7 0 
\put{$\smallmatrix   &   & 1  \cr
                   0 & 0 & 0 & 0 & 0 & 0 \endsmallmatrix$} at 7  -1 
\put{$\smallmatrix   &   & 1  \cr
                   1 & 1 & 1 & 1 & 1 & 1\endsmallmatrix$} at 10.7  -1 
\put{$\smallmatrix   &   & 1  \cr
                   0 & 1 & 2 & 1 & 0 & 0 \endsmallmatrix$} at 7  -2 
\put{$\smallmatrix   &   & 1  \cr
                   1 & 2 & 2 & 1 & 1 & 1 \endsmallmatrix$} at 10.7  -2 
\put{$\smallmatrix   &   & 1  \cr
                   0 & 1 & 2 & 2 & 1 & 0 \endsmallmatrix$} at 7  -3 
\put{$\smallmatrix   &   & 2  \cr
                   2 & 3 & 4 & 3 & 2 & 1 \endsmallmatrix$} at 10.7  -3 
\plot 0.6 0.5  14.5 0.5 /
\plot 2 1.3  2 -3.4 /
\plot 5 1.3  5 -3.4 /
\plot 8.8 1.3  8.8 -3.4 /
\plot 12.5 1.3  12.5 -3.4 /

\put{number}  at 13.5 1.3
\put{of roots}  at 13.5 .85
\put{$21$} [r] at 13.6 0
\put{$13$} [r] at 13.6 -1
\put{$9$} [r] at 13.6 -2
\put{$20$} [r] at 13.6 -3
\endpicture}
$$
where we denote the coefficients as follows:
$$
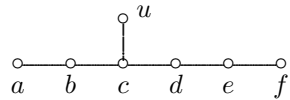
\hbox{\beginpicture
\setcoordinatesystem units <.7cm,.6cm>
\multiput{$\circ$} at 0 0  1 0  2 0  3 0  4 0  5 0 
   2 1 /
\plot 0.1 0  0.9 0 /
\plot 1.1 0  1.9 0 /
\plot 2.1 0  2.9 0 /
\plot 3.1 0  3.9 0 /
\plot 4.1 0  4.9 0 /
\plot 2 0.1  2 0.9 /
\put{$a$\strut} at 0 -.5 
\put{$b$\strut} at 1 -.5 
\put{$c$\strut} at 2 -.5 
\put{$d$\strut} at 3 -.5 
\put{$e$\strut} at 4 -.5 
\put{$f$\strut} at 5 -.5 
\put{$u$\strut} at 2.4 1.2
\endpicture} 
$$

Let us add that the root posets of type $\mathbb D_n$ have precisely two
levels, each level consists of $\binom n2$ roots: thus, here we obtain 
a separation of the positive roots into two classes of equal cardinality. 
\end{note}

\begin{note}\label{long}
{\bf The root posets with the long roots being marked by circles.}

$$
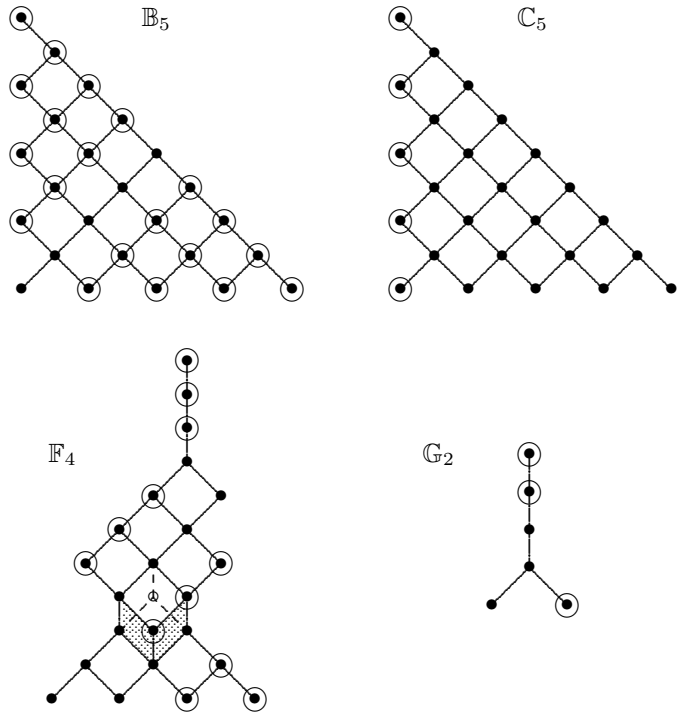
\hbox{\beginpicture
\setcoordinatesystem units <1.25cm,1cm>
\put{\beginpicture
\setcoordinatesystem units <.45cm,.45cm>
\multiput{$\bullet$} at  
   0 0  -2 0  -4 0  -6 0  -8 0  
      -1 1   -3 1  -5 1  -7 1  
        -2 2  -4 2  -6 2  -8 2
            -3 3  -5 3  -7 3
             -4 4  -6 4  -8 4
                -5 5  -7 5
                  -6 6  -8 6
                     -7 7
                       -8 8 /
\multiput{$\bigcirc$} at  
   0 0  -2 0  -4 0  -6 0  
      -1 1   -3 1  -5 1 
        -2 2  -4 2  -8 2
            -3 3   -7 3
             -6 4  -8 4
                -5 5  -7 5
                  -6 6  -8 6
                     -7 7
                       -8 8 /
\plot 0 0  -8 8 /
\plot -1 1  -2 0  -8 6  -7 7 /
\plot -2 2  -4 0  -8 4  -6 6 /
\plot -3 3  -6 0  -8 2  -5 5 /
\plot -4 4  -8 0 /
\endpicture} at 0 0 
\put{\beginpicture
\setcoordinatesystem units <.45cm,.45cm>
\multiput{$\bullet$} at  
   0 0  -2 0  -4 0  -6 0  -8 0  
      -1 1   -3 1  -5 1  -7 1  
        -2 2  -4 2  -6 2  -8 2
            -3 3  -5 3  -7 3
             -4 4  -6 4  -8 4
                -5 5  -7 5
                  -6 6  -8 6
                     -7 7
                       -8 8 /
\multiput{$\bigcirc$} at -8 8  -8 6  -8 4  -8 2  -8 0  /
\plot 0 0  -8 8 /
\plot -1 1  -2 0  -8 6  -7 7 /
\plot -2 2  -4 0  -8 4  -6 6 /
\plot -3 3  -6 0  -8 2  -5 5 /
\plot -4 4  -8 0 /
\endpicture} at 4 0 
\put{\beginpicture
\setcoordinatesystem units <.45cm,.45cm>

\multiput{$\bullet$} at 
     0 0  2 0  4 0  6 0 
     1 1  3 1  5 1
     2 2  3 2  4 2 
     2 3  4 3
     1 4  3 4  5 4
     2 5  4 5
     3 6  5 6
     4 7  4 8  4 9  4 10  /
\multiput{$\bigcirc$} at  4 8  4 9  4 10 
     1 4  2 5  3 6 
     3 2  4 3  5 4
     4 0  5 1  6 0 /
\multiput{$\circ$} at 3 3 /
\plot  0 0  2 2  4 0 5 1  /
\plot 1 1  2 0  4 2  6 0 /
\plot 3 1  3 2 /
\plot 2 2  2 3  4 5 /
\plot 4 2  4 3  2 5 /
\plot 3 2  5 4  3 6  1 4  3 2 /
\plot 3 6  4 7  5 6  4 5 /
\plot 4 7  4 10 /

\setdashes <1mm>
\plot 2 2  3 3  4 2 /
\plot 3 3  3 4 /

\setshadegrid span <.4mm>
\vshade 2 2 3  <,z,,> 3 1 2  <z,,,>  4 2 3  /
\endpicture}  at 0 -5
\put{\beginpicture
\setcoordinatesystem units <.5cm,.5cm>

\multiput{$\bullet$} at 0 0  2 0
  1 1   1 2  1 3  1 4 /
  \plot 0 0  1 1  2 0 /
  \plot 1 1  1 4 /
\multiput{$\bigcirc$} at 2 0  1 3  1 4 /
  \endpicture} at 4 -5
\put{$\mathbb B_5$} at 0 1.8
\put{$\mathbb C_5$} at 4 1.8
\put{$\mathbb F_4$} at -1 -4
\put{$\mathbb G_2$} at 3 -4
\endpicture}
$$
Always, the lowest row consists of the simple roots. Here, the 
sequence of the simple roots is the same as in the pictures above
showing the Dynkin diagrams. 

\end{note}

\begin{note}\label{reflections}
{\bf Reflections in the Weyl group $W$.} 
Here we refer to Humphreys \cite{[H2]}, 1.2. It asserts: 
if $g$ is an orthogonal transformation, then $gs_x g^{-1}$ maps
$gx$ to $-gx$ and fixes the hyperplane $H_{gx}$ orthogonal to $x$ pointwise. This shows that
for $w\in W,$ we have $s_{wx} = ws_xw^{-1}$, thus $s_{wx}$ belongs to $W$.
\cite{[H2]} 1.14 asserts that any reflection $w\in W$ is $w = s_x$ for some root $x$.
\end{note}

\begin{note}\label{quiver}
{\bf The quiver $Q(\Lambda)$ of a hereditary artin algebra.}
We may assume that $\Lambda$ is connected, then $\Lambda$ is a finite-dimensional
$k$-algebra for some field $k$ (namely for $k$ the center of $\Lambda$).

The vertices of $Q(\Lambda)$ 
are the isomorphism classes $[S]$ of the simple $\Lambda$-modules
$S$ and there is an arrow $[T] \to [S]$ provided $\Ext^1(T,S) \neq 0.$
For $S$ a simple module, let $\Gamma(S) = \End(S)^{\text{op}}.$ Let $v$ be the
product of $\dim_{\Gamma(T)}\Ext^1(T,S)$ and 
$\dim\Ext^1(T,S)_{\Gamma(S)}$. Taking into account the function $v\!: Q(\Lambda)_1 \to 
\Bbb N_1$, the quiver $Q(\Lambda)$ is called the {\it valued quiver} of $\Lambda$.

If we assume that $\Lambda$ is representation-finite, then $v\le 3$ and we draw
the edge between $S$ and $T$ as a double edge, in case $v = 2$, and as
a triple edge, in case $v = 3.$ In the case of a double or triple edge between the
vertices $[S_1], [S_2]$, we
endow it in the middle with an arrow pointing to $[S_1]$ 
provided $\dim_k \Gamma(S_1)
< \dim_k \Gamma(S_2)$ (let us stress again that these middle arrows should
not be confused with the arrows given by the orientation).
	\medskip 

It is easy to construct hereditary artin algebras with simply-laced quivers,
say of type  $\mathbb A_n, \mathbb D_n, \mathbb E_6,\mathbb E_7,\mathbb E_8$: just take the
path algebra of a corresponding quiver. For constructing hereditary artin algebras with 
quiver of type $\mathbb B_n, \mathbb C_n$ and $\mathbb F_4$, we need a field extension
$K:k$ of degree 2 (or, more generally, division rings $D_1\subset D_2$ which are
finite-dimensional $k$-algebras such that $\dim_k D_2 = 2\dim_k D_1).$ 
For example, the matrix algebra 
 $\left[\smallmatrix \mathbb R & \mathbb C & \mathbb C \cr 
                                                 & \mathbb C & \mathbb C \cr
                                                 &        & \mathbb C \endsmallmatrix \right]$
is a hereditary artin algebra of type $\mathbb B_3$, whereas 
$\left[\smallmatrix \mathbb C & \mathbb C & \mathbb C \cr 
                                                 & \mathbb R & \mathbb R \cr
                                                 &        & \mathbb R \endsmallmatrix \right]$
is a hereditary artin algebra of type $\mathbb C_3$.
In order to realize $\mathbb G_2$, let $\Lambda = 
\left[\smallmatrix k & K \cr 
                            & K \endsmallmatrix \right]$,
where  $K:k$ is a field extension of degree 3. In general, for dealing with
non-simply-laced quivers, Gabriel has introduced the notion of a ``species'':
whereas for simply-laced quivers, it is sufficient to work with vertices, arrows
and one field $k$, a species is given by a set of division rings (indexed by the
vertices of a quiver) as well as bimodules (indexed by the arrows). For an outline
of the representation theory of species, 
we may refer to several joint papers with Dlab.
\end{note}

\begin{note}\label{why}
{\bf  Why $|k| > 2$ ?}
Here is an example to have in mind. Consider the 3-subspace quiver $Q$ with sink $0$,
and its $k$-representation, where $k$ is a field. 
Let $X = I(0)$ be the indecomposable injective module of length $4$ and $Y$ the maximal indecomposable
module. Then $\bdim X < \bdim Y$. Now $|X|$ and $|Y|$ differ by $1$, thus, if $X$ is a subfactor of
$Y$, it is a submodule or a factor module of $Y$. Since $X$ is injective, it cannot be a submodule of $Y$.
In order to analyze whether $X$ is a factor module of $Y$, we consider $\Hom(Y,X)$, this is a 
2-dimensional $k$-space. For any field $k$, there are 3 maps
$Y \to X$ which are non-zero and not surjective, all other non-zero maps $Y \to X$ are surjective.
Thus, if $k$ is the field with two elements, then there is no surjective map $Y \to X$.

A second example of interest: Consider an artin algebra of type $\mathbb B_2$, say with Auslander-Reiten quiver
$$
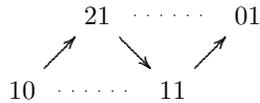
{\beginpicture
\setcoordinatesystem units <1cm,1cm> 
\put{$10$} at 0 0 
\put{$21$} at 1 1 
\put{$11$} at 2 0 
\put{$01$} at 3 1
\arr{0.3 0.3}{0.7 0.7} 
\arr{1.3 0.7}{1.7 0.3} 
\arr{2.3 0.3}{2.7 0.7} 
\setdots
\plot 0.5 0 1.5 0 /
\plot 1.5 1 2.5 1 /
\endpicture}
$$
Let $X = 11$ and $Y = 21$. Then $X$ is a factor module of $Y$, thus $X \subfactor Y$. 
But since $X$ is
the injective envelope of $10$ and the socle of $21$ is the direct sum of two copies of $10$,
we see that $Y$ is a submodule of the direct sum of two copies of $11$. This shows that 
$Y \subfactor X^2$. 
\end{note}

\begin{note}\label{antichain}
{\bf Antichains in a poset, antichains in an additive category.} 
Let us motivate the definition. We recall the following: 
Given a poset $P$, a chain in $P$ is a subset of pairwise comparable elements, 
whereas an antichain in $P$ is a subset of pairwise incomparable elements. 
Now consider the linearization $kP$ of $P$, were $k$ is a field: 
this is an additive $k$-category whose indecomposable
objects are the elements of $P$ such that $\Hom_{kP}(x,y) = k$ 
provided $x \le y$ in $P$ and $\Hom_{kP}(x,y) = 0$
otherwise, such that the composition of maps in  $kP$ 
is given  by the multiplication in $k$, and, finally, such that any object in
$kP$ is the direct sum of indecomposable objects. Of course, a subset
$A$ of $P$ is an antichain in $P$ if and only if  $A$ (considered as a set of objects in
 $kP$) consists of pairwise orthogonal bricks 
(thus, is an antichain in the additive category $kp$).
\end{note}

\begin{note}\label{thick} {\bf Thick subcategories of abelian categories.}
If $mathcal C$ is an abelian category, we say that a full subcategory $\mathcal U$ is a
{\it thick} subcategory, provided it is closed under kernels, cokernels, and extensions.
Such a thick subcategory is again an abelian category, and the embedding functor is
exact. On the other hand, 
assume that $\mathcal U$ is a full subcategory which is an abelian subcategory.
Then $\mathcal U$ is a thick subcategory of $\mathcal C$ if and only if the embedding functor
is exact and $\mathcal U$ is closed under extensions. Thick subcategories had been called
`wide'' by Hovey \cite{[Ho]}, the denomination ``thick'' seems to be due to Krause \cite{[K1]}. 
\end{note}

\begin{note}\label{Young}
{\bf Young partition.} In this survey, 
a sequence of numbers $(r_1,r_2,\dots,r_t)$ with $r_1\ge r_2 \ge \cdots r_t \ge 0$ is 
called a {\it Young partition}. In algebra and number theory, such 
sequences are usually just called ``partitions'', but we need also the concept of a 
partition as used in set theory (see Chapter 4 and the Appendix): 
a {\it partition of a set $M$} is a set of disjoint non-empty subsets $M_i$ of $M$
such that $M = \bigcup M_i$. The non-crossing partitions in type $\mathbb A_n$ considered 
in Chapter 4
are just certain (set-theoretical) partitions of the set $\{1,2,\dots,n\}$. Note that
the word {\it partition} in the formulation {\it non-crossing partitions}
refers to (set-theoretical) partitions. Of course, Young partitions may be considered as
special set-theoretical partitions:
namely, the Young partition $\lambda = (\lambda_1, \lambda_2, \dots,\lambda_t)$ with $n = \sum \lambda_i$
may be considered as the partition of the set $M=\{1,2,\dots,n\}$ with parts 
$M_j = \{x\in \mathbb Z\mid \sum_{i<j}\lambda_i < x \le \sum_{i\le j}\lambda_i\}$ for $1\le j\le t$.
\end{note}

\begin{note}\label{fixed-height}
{\bf The short roots of fixed height.} 
Let $r^s_t$ be the number of short roots of height $t$. We should stress that
this is usually not a Young partition. But still we have the following symmetry condition:
	\medskip 

\Rahmen{$r^s_t + r^s_{h+1-t} = r^s_1$}
	\medskip

\noindent
and $r^s_1$ is equal to $n-1$ in case $\mathbb B_n$, to $1$ in case $\mathbb C_n$ and $\mathbb G_2$,
and finally to $2$ in case $\mathbb F_4.$ 
	\bigskip

Here are the Young diagrams $Y$ for the exponent partitions in the cases $\mathbb B_6, \mathbb C_6,$
$\mathbb F_4, \mathbb G_2,$ always 
inserted into a rectangle $R$ with $n$ rows and $h$ columns. 
The boxes of the Young diagram correspond bijectively to the 
positive roots, we have shaded these boxes in two different ways: the darker shading
marks long roots. Of course, the square boxes in $R\setminus Y$ may be interpreted
as corresponding to the negative roots, here we use again a shading, but this time
we shade only one kind of the roots: the short ones in case $\mathbb B_6$  and the long
 ones in the cases $\mathbb C_6, \mathbb F_4, \mathbb G_2.$

$$
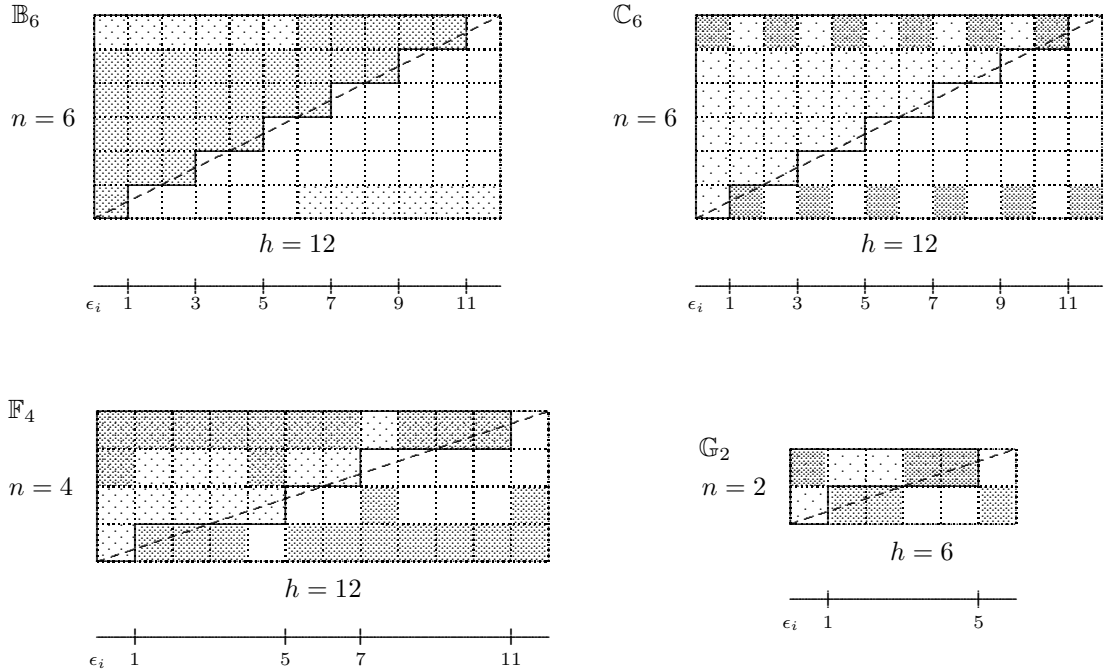
\hbox{\beginpicture
\setcoordinatesystem units <1cm,1cm>

\put{\beginpicture
\setcoordinatesystem units <.45cm,.45cm>
\put{\beginpicture
\multiput{} at 0 0  12 6 /

\put{$\mathbb B_6$} at -2 6

\setdots <.7mm>
\setplotarea x from 0 to 12, y from 0 to 6
\grid {12} {6} 
\setsolid
\plot 0 0  0 6  11 6  11 5  9 5  9 4  7 4  7 3  5 3  5 2  3 2  3 1  1 1  1 0  0 0  /
\put{$n=6$} at -1.5 3
\put{$h = 12$} at 6 -.7
\plot 0 -2  12 -2 /
\plot 1 -1.8  1 -2.2 /
\plot 3 -1.8  3 -2.2 /
\plot 5 -1.8  5 -2.2 /
\plot 7 -1.8  7 -2.2 /
\plot 9 -1.8  9 -2.2 /
\plot 11 -1.8  11 -2.2 /


\put{$\ssize \epsilon_i\strut$} at 0 -2.5
\put{$\ssize 1$} at 1 -2.5
\put{$\ssize 3$} at 3 -2.5
\put{$\ssize 5$} at 5 -2.5
\put{$\ssize 7$} at 7 -2.5
\put{$\ssize 9$} at 9 -2.5
\put{$\ssize 11$} at 11 -2.5

\setshadegrid span <.8mm>
\vshade 0 5 6  6 5 6  / 

\setshadegrid span <.4mm>
\vshade 0 0 5 <,z,,>  1 0 5  <z,z,,> 
      1.01 1 5  <z,z,,> 3 1 5  <z,z,,> 
      3.01 2 5  <z,z,,> 5 2 5  <z,z,,> 
      5.01 3 5  <z,z,,> 6 3 5  <z,z,,> 
      6.01 3 6  <z,z,,> 7 3 6  <z,z,,> 
      7.01 4 6  <z,z,,> 9 4 6  <z,z,,> 
      9.01 5 6  <z,,,> 11 5 6  /

\setshadegrid span <.7mm>
\vshade 6 0 1  12 0 1  / 

\setdashes <1mm>
\plot 0 0  12 6 /

\setdots <.4mm>
\plot 11 6  12 6  12 0  1 0 /

\endpicture} at 0 0
\endpicture} at 0 0 

\put{\beginpicture
\setcoordinatesystem units <.45cm,.45cm>
\put{\beginpicture
\multiput{} at 0 0  12 6 /

\put{$\mathbb C_6$} at -2 6

\setdots <.7mm>
\setplotarea x from 0 to 12, y from 0 to 6
\grid {12} {6} 
\setsolid
\plot 0 0  0 6  11 6  11 5  9 5  9 4  7 4  7 3  5 3  5 2  3 2  3 1  1 1  1 0  0 0  /
\put{$n=6$} at -1.5 3
\put{$h = 12$} at 6 -.7
\plot 0 -2  12 -2 /
\plot 1 -1.8  1 -2.2 /
\plot 3 -1.8  3 -2.2 /
\plot 5 -1.8  5 -2.2 /
\plot 7 -1.8  7 -2.2 /
\plot 9 -1.8  9 -2.2 /
\plot 11 -1.8  11 -2.2 /


\put{$\ssize \epsilon_i\strut$} at 0 -2.5
\put{$\ssize 1$} at 1 -2.5
\put{$\ssize 3$} at 3 -2.5
\put{$\ssize 5$} at 5 -2.5
\put{$\ssize 7$} at 7 -2.5
\put{$\ssize 9$} at 9 -2.5
\put{$\ssize 11$} at 11 -2.5

\setshadegrid span <1mm>
\vshade 0 0 6 <,z,,>  1 0 6  <z,z,,> 
      1.01 1 6  <z,z,,> 3 1 6  <z,z,,> 
      3.01 2 6  <z,z,,> 5 2 6  <z,z,,> 
      5.01 3 6  <z,z,,> 7 3 6  <z,z,,> 
      7.01 4 6  <z,z,,> 9 4 6  <z,z,,> 
      9.01 5 6  <z,,,> 11 5 6  /

\setshadegrid span <.3mm>
\vshade 0 5 6  1 5 6  /
\vshade 2 5 6  3 5 6  /
\vshade 4 5 6  5 5 6  /
\vshade 6 5 6  7 5 6  /
\vshade 8 5 6  9 5 6  /
\vshade 10 5 6  11 5 6  /

\setshadegrid span <.3mm>
\vshade 1 0 1  2 0 1  /
\vshade 3 0 1  4 0 1  /
\vshade 5 0 1  6 0 1  /
\vshade 7 0 1  8 0 1  /
\vshade 9 0 1  10 0 1  /
\vshade 11 0 1  12 0 1  /

\setdashes <1mm>
\plot 0 0  12 6 /

\setdots <.4mm>
\plot 11 6  12 6  12 0  1 0 /

\endpicture} at 0 0
\endpicture} at 8 0 

\put{\beginpicture
\setcoordinatesystem units <.5cm,.5cm>
\multiput{} at 0 0  6 4 /
\put{$\mathbb F_4$} at -2 4
\setdots <.7mm>
\setplotarea x from 0 to 12, y from 0 to 4
\grid {12} {4} 
\setsolid
\plot 0 0  0 4  11 4  11 3  7 3  7 2  5 2  5 1  1 1  1 0  0 0  /

\put{$n=4$} at -1.5 2
\put{$h = 12$} at 6 -.7
\plot 0 -2  12 -2 /
\plot 1 -1.8  1 -2.2 /
\plot 5 -1.8  5 -2.2 /
\plot 7 -1.8  7 -2.2 /
\plot 11 -1.8  11 -2.2 /


\put{$\ssize \epsilon_i\strut$} at 0 -2.6
\put{$\ssize 1$} at 1 -2.6
\put{$\ssize 5$} at 5 -2.6
\put{$\ssize 7$} at 7 -2.6
\put{$\ssize 11$} at 11 -2.6

\setshadegrid span <.8mm>
\vshade 0 0 4 <,z,,>  1 0 4  <z,z,,> 
      1.01 1 4  <z,z,,> 5 1 4  <z,z,,> 
      5.01 2 4  <z,z,,> 7 2 4  <z,z,,> 
      7.01 3 4  <z,,,> 11 3 4  /

\setshadegrid span <.35mm>
\vshade 0 2 4 <,z,,>  1 2 4  <z,z,,> 
      1.01 3 4  <z,z,,> 4 3 4 <z,z,,>
      4.01 2 4  <z,z,,> 5 2 4  <z,z,,> 
      5.01 3 4  <z,z,,> 7 3 4  /
\vshade 8 3 4 <z,z,,>  11 3 4 /

\setshadegrid span <.35mm>
\vshade 1 0 1  4 0 1 /
\vshade 5 0 1 <,z,,>  7 0 1  <z,z,,> 
      7.01 0 2  <z,z,,> 8 0 2  <z,z,,> 
      8.01 0 1  <z,z,,> 11 0 1  <z,z,,> 
      11.01 0 2  <z,,,> 12 0 2  /

\setdashes <1mm>
\plot 0 0  12 4 /

\setdots <.4mm>
\plot 11 4  12 4  12 0  1 0 /

\endpicture} at .3 -5
\put{\beginpicture
\setcoordinatesystem units <.5cm,.5cm>
\multiput{} at 0 0  6 2 /
\put{$\mathbb G_2$} at -2 2
\setdots <.7mm>
\setplotarea x from 0 to 6, y from 0 to 2
\grid {6} {2} 
\setsolid
\plot 0 0  0 2  5 2   5 1  1 1  1 0  0 0  /

\put{$n=2$} at -1.5 1
\put{$h = 6$} at 3.5 -.7
\plot 0 -2  6 -2 /
\plot 1 -1.8  1 -2.2 /
\plot 5 -1.8  5 -2.2 /

\put{$\ssize \epsilon_i\strut$} at 0 -2.6
\put{$\ssize 1$} at 1 -2.6
\put{$\ssize 5$} at 5 -2.6

\setshadegrid span <1mm>
\vshade 0 0 2 <,z,,>  1 0 2  <z,z,,> 
      1.01 1 2  <z,,,> 5 1 2  /

\setshadegrid span <.3mm>
\vshade 0 1 2  1 1 2 /
\vshade 3 1 2  5 1 2 /

\setshadegrid span <.35mm>
\vshade 1 0 1  3 0 1 /
\vshade 5 0 1  6 0 1 /

\setdashes <1mm>
\plot 0 0  6 2 /

\setdots <.4mm>
\plot 5 2  6 2  6 0  1 0 /

\endpicture} at 8 -5 
\endpicture}
$$
\end{note}
\medskip 

\begin{note}\label{arrangements}
{\bf Hyperplane arrangements.} 
A  hyperplane arrangement is a finite set of pairwise different 
hyperplanes in a fixed vector space $V$. If $V$ is of dimension $n$, one calls it
an $n$-arrangement. Such a set is called {\it real} provided the base field is the field of
real numbers and {\it central} provided 
all the hyperplanes are subspaces (that means: affine hyperplanes which contain the zero vector).
In the lectures, all the hyperplane arrangements considered are real and central. 

What is an ``arrangement''? Just a fancy word for a finite set. Of course,
any central hyperplane arrangement $\H$ may be considered as a representation of
a quiver, namely of the $t$-subspace quiver, where $t$ is the cardinality of $\H$,
the discussion in Section 1.5 should also be seen as an attempt 
to propagate a method to deal with such quiver representations, which 
has been found useful outside of representation theory, namely to look at modules of derivations. 
	
Let us repeat: Hyperplane $n$-arrangements of cardinality $t$ are representations 
of the $t$-subspace quiver with
dimension vector $(n;n\!-\!1,\dots,n\!-\!1)$. Using duality and reflections, they correspond
bijectively to the representations with dimension vector $(n;1,\dots,1)$, thus to
$t$-tuples of elements of the projective space $\mathbb P^{n-1}.$ 
A hyperplane arrangement is called {\it irreducible} provided the corresponding 
representation is indecomposable. 
	
Let $k$ be an arbitrary base field.
Let $Q$ be the $t$-subspace quiver (it has $t+1$ vertices and $t$ arrows; one vertex, say with label 
$0$ is a sink, the
remaining vertices $1,2\,dots, t$ are sources; the arrows are of the form $\alpha_i\!:i \to 0$
with $1\le i \le n$). 
A subspace representation of $Q$ is a representation of $Q$
which uses only inclusion maps (or at least monomorphisms). Let $M$ be a subspace representation of $Q$
such that all the vector spaces $M_i$ are one-dimensional, for $1\le i \le t.$ Then $M$
is a direct sum of indecomposable subspace representation $N$ of $Q$ such that all the vector spaces
$N_i$ with $1\le i \le t$ are at most one-dimensional. Thus, it seems reasonable for us to 
look for a moment at the full subcategory $\mathcal U$ of $\mo kQ$ of direct sums of 
indecomposable subspace representation $N$ of $Q$ such that all the vector spaces
$N_i$ with $1\le i \le t$ are at most one-dimensional. The indecomposable objects in $\mathcal U$
are of two different kinds: first of all, there is the simple representation $S(0)$ corresponding to the
vertex $0$;
let us denote by $\emptyset_t$ the direct sum of $t$ copies of $S(0)$ 
(it is denoted by $\Phi_t$ in \cite{[OT]}). The remaining indecomposable objects $N$ in $\mathcal U$
have the property that the sum of the images of the maps $\alpha_i$ is the total space $N_0$.
It is not difficult to see that {\it all the indecomposable objects in $\mathcal U$ have
endomorphism ring $k$.} 
This is clear for the simple representation $S(0) = \emptyset_1.$ 
We can assume that we deal with an indecomposable representation $X$
with $\bdim X = (n;1,\dots,1)$. 
Let $f\neq 0$ be an endomorphism with $f^2 = 0$. Then $\bdim X =
2\bdim \Ker(f) + \bdim \Imm(f)$. But then $\Ker(f)$ is of the form $\emptyset_t$ for some $t$, impossible. 
This shows that the endomorphism ring of $X$ is a division ring. 
The restriction map $f \mapsto f|X_1$ is a ring homomorphism from
$\End(X)$ onto $\End_k(X_1) = k$,  thus $\End(X) = k.$ 

Altogether we see: 
{\it Any representation in $\mathcal U$ is the direct sum of a representation $\emptyset_t$, 
and pairwise non-isomorphic non-simple representations with endomorphism ring $k$.} 
	\medskip 

{\bf Example.} Already the four subspace 
quiver shows that the category $\mathcal U$  may be quite complicated:
in particular, let us stress that 
there may be non-isomorphic indecomposable 
representations $X, X'$ in $\mathcal U$ 
with $\Hom(X,X') \neq 0$ and $\Hom(X',X) \neq 0.$
	\medskip 

$$
{\beginpicture
\setcoordinatesystem units <1cm,1cm> 

\arr{0.3 0.3}{0.7 1.3}
\arr{0.3 0.1}{0.7 .3}
\arr{0.3 -.1}{0.7 -.3}
\arr{0.3 -.3}{0.7 -1.3}

\arr{2.3 0.3}{2.7 1.3}
\arr{2.3 0.1}{2.7 .3}
\arr{2.3 -.1}{2.7 -.3}
\arr{2.3 -.3}{2.7 -1.3}

\arr{1.3 1.3}{1.7 0.3}
\arr{1.3 0.3}{1.7 .1}
\arr{1.3 -.3}{1.7 -.1}
\arr{1.3 -1.3}{1.7 -.3}

\arr{11.3 1.3}{11.7 0.3}
\arr{11.3 0.3}{11.7 .1}
\arr{11.3 -.3}{11.7 -.1}
\arr{11.3 -1.3}{11.7 -.3}
\put{\beginpicture
\setcoordinatesystem units <.2cm,.2cm> 
\put{$\ssize 1$} at 0 0
\put{$\ssize 1$} at 1 1.5
\put{$\ssize 1$} at 1 0.5
\put{$\ssize 0$} at 1 -.5
\put{$\ssize 0$} at 1 -1.5
\endpicture} at  5 1 
\put{\beginpicture
\setcoordinatesystem units <.2cm,.2cm> 
\put{$\ssize 2$} at 0 0
\put{$\ssize 1$} at 1 1.5
\put{$\ssize 1$} at 1 0.5
\put{$\ssize 1$} at 1 -.5
\put{$\ssize 1$} at 1 -1.5
\endpicture} at  6 2 
\put{\beginpicture
\setcoordinatesystem units <.2cm,.2cm> 
\put{$\ssize 1$} at 0 0
\put{$\ssize 0$} at 1 1.5
\put{$\ssize 0$} at 1 0.5
\put{$\ssize 1$} at 1 -.5
\put{$\ssize 1$} at 1 -1.5
\endpicture} at  7 1 
\put{\beginpicture
\setcoordinatesystem units <.2cm,.2cm> 
\put{$\ssize 2$} at 0 0
\put{$\ssize 1$} at 1 1.5
\put{$\ssize 1$} at 1 0.5
\put{$\ssize 1$} at 1 -.5
\put{$\ssize 1$} at 1 -1.5
\endpicture} at  8 2 
\put{\beginpicture
\setcoordinatesystem units <.2cm,.2cm> 
\put{$\ssize 1$} at 0 0
\put{$\ssize 1$} at 1 1.5
\put{$\ssize 1$} at 1 0.5
\put{$\ssize 0$} at 1 -.5
\put{$\ssize 0$} at 1 -1.5
\endpicture} at  9 1 

\arr{5.3 1.3}{5.7 1.7}
\arr{6.3 1.7}{6.7 1.3}
\arr{7.3 1.3}{7.7 1.7}
\arr{8.3 1.7}{8.7 1.3}
\setdashes <1mm>
\plot 5 1.4  5 2.5 /
\plot 9 1.4  9 2.5 /
\setdots <1mm>
\plot 5.5 1  6.5 1 /
\plot 7.5 1  8.5 1 /
\setsolid

\put{and 2 similar cycles} at 7 0
\put{\beginpicture
\setcoordinatesystem units <.2cm,.2cm> 
\put{$\ssize 2$} at 0 0
\put{$\ssize 1$} at 1 1.5
\put{$\ssize 1$} at 1 0.5
\put{$\ssize 1$} at 1 -.5
\put{$\ssize 1$} at 1 -1.5
\endpicture} at  7 -1
\put{parametrized by $k\setminus\{0,1\}$} at 7 -1.8
\plot 4.5 -.35 9.5 -.35 /

\put{\beginpicture
\setcoordinatesystem units <.2cm,.2cm> 
\put{$\ssize 1$} at 0 0
\put{$\ssize 0$} at 1 1.5
\put{$\ssize 0$} at 1 0.5
\put{$\ssize 0$} at 1 -.5
\put{$\ssize 0$} at 1 -1.5
\endpicture} at  0 0 
\put{\beginpicture
\setcoordinatesystem units <.2cm,.2cm> 
\put{$\ssize 1$} at 0 0
\put{$\ssize 1$} at 1 1.5
\put{$\ssize 0$} at 1 0.5
\put{$\ssize 0$} at 1 -.5
\put{$\ssize 0$} at 1 -1.5
\endpicture} at  1 1.5
\put{\beginpicture
\setcoordinatesystem units <.2cm,.2cm> 
\put{$\ssize 1$} at 0 0
\put{$\ssize 0$} at 1 1.5
\put{$\ssize 1$} at 1 0.5
\put{$\ssize 0$} at 1 -.5
\put{$\ssize 0$} at 1 -1.5
\endpicture} at  1 0.5
\put{\beginpicture
\setcoordinatesystem units <.2cm,.2cm> 
\put{$\ssize 1$} at 0 0
\put{$\ssize 0$} at 1 1.5
\put{$\ssize 0$} at 1 0.5
\put{$\ssize 1$} at 1 -.5
\put{$\ssize 0$} at 1 -1.5
\endpicture} at  1 -.5
\put{\beginpicture
\setcoordinatesystem units <.2cm,.2cm> 
\put{$\ssize 1$} at 0 0
\put{$\ssize 0$} at 1 1.5
\put{$\ssize 0$} at 1 0.5
\put{$\ssize 0$} at 1 -.5
\put{$\ssize 1$} at 1 -1.5
\endpicture} at  1 -1.5
\put{\beginpicture
\setcoordinatesystem units <.2cm,.2cm> 
\put{$\ssize 3$} at 0 0
\put{$\ssize 1$} at 1 1.5
\put{$\ssize 1$} at 1 0.5
\put{$\ssize 1$} at 1 -.5
\put{$\ssize 1$} at 1 -1.5
\endpicture} at  2 0
\put{\beginpicture
\setcoordinatesystem units <.2cm,.2cm> 
\put{$\ssize 2$} at 0 0
\put{$\ssize 0$} at 1 1.5
\put{$\ssize 1$} at 1 0.5
\put{$\ssize 1$} at 1 -.5
\put{$\ssize 1$} at 1 -1.5
\endpicture} at  3 1.5
\put{\beginpicture
\setcoordinatesystem units <.2cm,.2cm> 
\put{$\ssize 2$} at 0 0
\put{$\ssize 1$} at 1 1.5
\put{$\ssize 0$} at 1 0.5
\put{$\ssize 1$} at 1 -.5
\put{$\ssize 1$} at 1 -1.5
\endpicture} at  3 0.5
\put{\beginpicture
\setcoordinatesystem units <.2cm,.2cm> 
\put{$\ssize 2$} at 0 0
\put{$\ssize 1$} at 1 1.5
\put{$\ssize 1$} at 1 0.5
\put{$\ssize 0$} at 1 -.5
\put{$\ssize 1$} at 1 -1.5
\endpicture} at  3 -.5
\put{\beginpicture
\setcoordinatesystem units <.2cm,.2cm> 
\put{$\ssize 2$} at 0 0
\put{$\ssize 1$} at 1 1.5
\put{$\ssize 1$} at 1 0.5
\put{$\ssize 1$} at 1 -.5
\put{$\ssize 0$} at 1 -1.5
\endpicture} at  3 -1.5

\put{\beginpicture
\setcoordinatesystem units <.2cm,.2cm> 
\put{$\ssize 1$} at 0 0
\put{$\ssize 0$} at 1 1.5
\put{$\ssize 1$} at 1 0.5
\put{$\ssize 1$} at 1 -.5
\put{$\ssize 1$} at 1 -1.5
\endpicture} at  11 1.5
\put{\beginpicture
\setcoordinatesystem units <.2cm,.2cm> 
\put{$\ssize 1$} at 0 0
\put{$\ssize 1$} at 1 1.5
\put{$\ssize 0$} at 1 0.5
\put{$\ssize 1$} at 1 -.5
\put{$\ssize 1$} at 1 -1.5
\endpicture} at  11 0.5
\put{\beginpicture
\setcoordinatesystem units <.2cm,.2cm> 
\put{$\ssize 1$} at 0 0
\put{$\ssize 1$} at 1 1.5
\put{$\ssize 1$} at 1 0.5
\put{$\ssize 0$} at 1 -.5
\put{$\ssize 1$} at 1 -1.5
\endpicture} at  11 -.5
\put{\beginpicture
\setcoordinatesystem units <.2cm,.2cm> 
\put{$\ssize 1$} at 0 0
\put{$\ssize 1$} at 1 1.5
\put{$\ssize 1$} at 1 0.5
\put{$\ssize 1$} at 1 -.5
\put{$\ssize 0$} at 1 -1.5
\endpicture} at  11 -1.5
\put{\beginpicture
\setcoordinatesystem units <.2cm,.2cm> 
\put{$\ssize 1$} at 0 0
\put{$\ssize 1$} at 1 1.5
\put{$\ssize 1$} at 1 0.5
\put{$\ssize 1$} at 1 -.5
\put{$\ssize 1$} at 1 -1.5
\endpicture} at  12 0

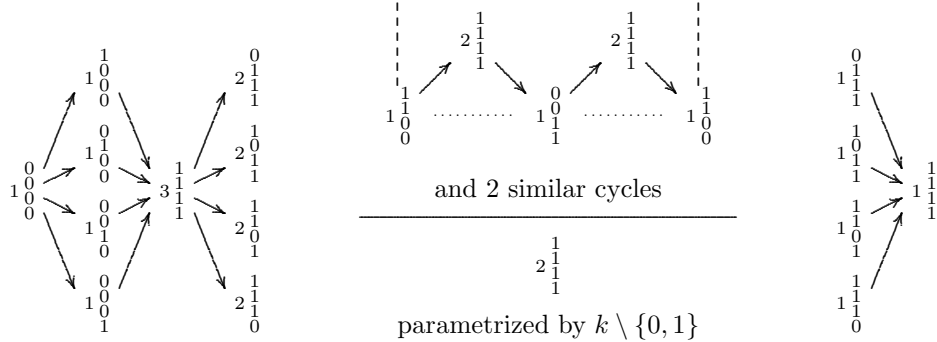
\endpicture}
$$
On the left we see the 10 indecomposable preprojective modules which belong to $\mathcal U$
(the module farthest left is the simple module $S(0) = \emptyset_1$),
on the right the 5 indecomposable preinjective modules which belong to $\mathcal U$,  the middle
part exhibits indecomposable regular modules. Each of the 3 tubes of rank $2$ hosts 
4 indecomposable modules which belong to $\mathcal U$, namely the modules of regular length at most $2$,
here we see the cycles mentioned above. And there are the remaining indecomposable modules 
with dimension vector $(2;1,1,1,1)$, they are modules on the mouth of homogeneous tubes;
the number of such modules is $|k|-2$. 

\end{note}

\vfill\eject
\section{\bf Tilting Theory}

This chapter concerns the classical tilting theory, the study of (finitely generated)
tilting modules for a hereditary artin algebra. 
In my appendix to the 
{\it Handbook of Tilting Theory} (2007) I wrote: {\it At the time the Handbook was conceived} (= 2002) {\it there was a common  feeling that the tilted algebras (as the core of tilting theory)
were understood well and that this part of the theory had reached a sort of final shape.
But in the meantime 
this has turned out to be wrong: the tilted algebras have to be seen as factor algebras of the so called cluster tilted algebras, and it may very well be, that in future the
cluster tilted algebras and the cluster categories will topple the tilted algebras}
(\cite{[R6]}, p.~446). 

Actually, as we will see in this chapter, already the 
basic setting of tilting theory should be refined,
replacing the usually considered torsion pair
$(\mathcal F(T),\mathcal G(T))$ by a torsion triple $(\mathcal F(T),\mathcal N(T),\mathcal Q(T))$, 
a torsion triple which is defined by 
a linear form $\alpha_T\!:K_0(\Lambda) \to \mathbb Z$. Here, $\mathcal N(T)$ will be an
arbitrary sincere exceptional subcategory, and $\mathcal N(T)$ and $(\mathcal F(T),\mathcal G(T))$
determine each other. 
In this way, the study of tilted algebras turns out 
to be just the study of sincere exceptional subcategories. 
This refinement is due to Ingalls and Thomas (2009), it puts tilting theory into
the realm of the stability theory of King. Whereas the impetus for this refinement
came from the theory of cluster algebras and cluster tilted algebras, there is now no 
further need to refer to cluster theory. 

In this chapter, we usually will deal with an arbitrary hereditary artin algebra $\Lambda$, 
and do not restrict the attention to the representation-finite ones. 
The modules which we will consider are left $\Lambda$-modules of finite length. 
The chapter is essentially independent from Chapter 1 and is mainly devoted to advertise 
findings of Ingalls and Thomas. 
	\bigskip

Let us recall the basic definitions of tilting theory. 
A module $T$ will be called a {\it partial tilting} module
provided it is multiplicity-free (this means that it is the direct sum of pairwise
non-isomorphic indecomposable modules) and has no self-extensions 
(this means that $\Ext^1(T,T) = 0$).
A {\it tilting} module is a partial
tilting module which is the direct sum of $n$ indecomposable modules, where 
$n$ is the number of simple $\Lambda$-modules (thus the rank of the
Grothendieck group $K_0(\Lambda)$).

A {\it torsion pair} $(\mathcal F,\mathcal G)$ in $\mo\Lambda$ is a pair of full subcategories
$\mathcal F,\mathcal G$ which satisfies the following two properties: first,
$\Hom(G,F)= 0$ for all modules $F\in \mathcal F$ and $G\in \mathcal G$ (we just will
write $\Hom(\mathcal G,\mathcal F) = 0$ and use a similar convention also elsewhere). 
Second, any module
$X$ with $\Hom(X,\mathcal F) = 0$ belongs to $\mathcal G$ and every module
$Y$ with $\Hom(\mathcal G,Y) = 0$ belongs to $\mathcal F$.
Alternatively, one may replace the second condition by requiring that every module
$M$ has a submodule $M'$ which belongs to $\mathcal G$ such that $M/M'$ belongs to $\mathcal F$,
this submodule $M'$ is called the {\it torsion submodule} of $M$ and is a uniquely
determined submodule.  The class $\mathcal G$
is called a torsion class, the class $\mathcal F$ a torsionfree class, the torsion classes
are just the full subcategories closed under factor modules and extensions,
the torsionfree classes are the full subcategories closed under submodules and
extensions. A torsion pair $(\mathcal F,\mathcal G)$ is said to be {\it split} provided
any indecomposable module belongs to $\mathcal F$ or to $\mathcal G$ (or, equivalently,
provided the torsion submodule of any module is a direct summand).
  
A tilting module $T$ determines a torsion pair $(\mathcal F(T),\mathcal G(T))$
as follows: $\mathcal G(T)$ consists of the
modules which are generated by $T$ (thus the factor modules of modules of the form
$T^m$ for some natural number $m$), or equivalently, the modules $G$ satisfying
$\Ext^1(T,G) = 0.$ The modules in $\mathcal F(T)$ are the modules $F$
with $\Hom(T,F) = 0.$ 

If $T$ is a tilting module, one is interested in the ring $\Gamma(T) = 
\End(T)^{\text{op}}$; this is called a {\it tilted} artin algebra.
The indecomposable $\Gamma(T)$-modules are obtained from the indecomposable
$\Lambda$-modules as follows: The functor $\Hom(T,-)\!:\mo\Lambda \to 
\mo \Gamma(T)$ yields an equivalence between $\mathcal G(T)$ and its image category
$\mathcal Y(T)$, the functor $\Ext^1(T,-)\!:\mo \Lambda \to 
\mo \Gamma(T)$ yields an equivalence between $\mathcal F(T)$ and its image category
$\mathcal X(T)$, and the pair $(\mathcal Y(T),\mathcal X(T))$ is a torsion pair in $\mo 
\Gamma(T)$ which is split. The functors  $\Hom(T,-)$ and $\Ext^1(T,-)$
are called the {\it tilting functors} given by $T$. 

Let us stress that here we deal with torsion pairs in $\mo \Lambda$ and
$\mo \Gamma(T)$, namely with $(\mathcal F(T),\mathcal G(T))$ and $(\mathcal Y(T),\mathcal X(T))$,
such that the two given subcategories of $\mo \Lambda$ are equivalent to
the two given subcategories in $\mo \Gamma(T)$, but it is a cross-over: the torsion
class of $\mo \Lambda$ is equivalent to the torsionfree class of $\mo \Gamma(T)$, and
the torsionfree
class of $\mo \Lambda$ is equivalent to the torsion class of $\mo \Gamma(T)$.
One illustrates this as follows:
$$
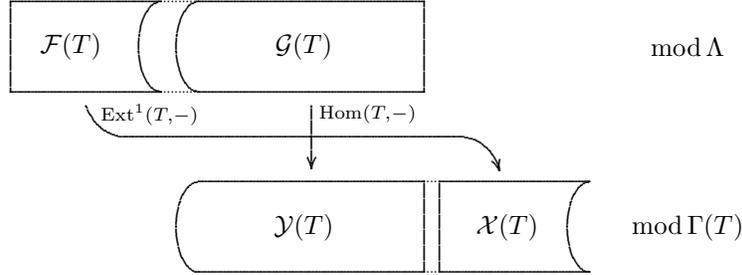
{\beginpicture
\setcoordinatesystem units <1cm,.6cm>
\put{\beginpicture
\multiput{} at 0 0   8 2 /
\plot 2 0  0 0  0 2  2 2 /
\ellipticalarc axes ratio 1:2  -180 degrees from 2 0 center at 2 1    

\ellipticalarc axes ratio 1:2  -180 degrees from 2.5 0 center at 2.5 1    
\plot 2.5 0  5.5 0  5.5 2  2.5 2 /
\put{$\mathcal F(T)$} at 0.8 1
\put{$\mathcal G(T)$} at 3.9 1
\setdots <.5mm>
\plot 2 0  2.5 0 / 
\plot 2 2  2.5 2 /
\endpicture} at 0 0 
\put{\beginpicture
\multiput{} at 0 0   8 2 /

\ellipticalarc axes ratio 1:2  -180 degrees from 2.5 0 center at 2.5 1    
\plot 2.5 0  5.5 0  5.5 2  2.5 2 /
\put{$\mathcal X(T)$} at 6.6 1
\put{$\mathcal Y(T)$} at 3.9 1
\plot 7.7 0  5.7 0   5.7 2  7.7 2 /
\ellipticalarc axes ratio 1:2  -180 degrees from 7.7 0 center at 7.7 1    
\setdots <.5mm>
\plot 5.5 0  5.7 0 / 
\plot 5.5 2  5.7 2 /
\endpicture} at 0 -4 

\arr{0 -1.3}{0 -2.7}
\setquadratic
\plot -3 -1.3 -2.8 -1.8 -2.5 -2  0 -2
  2 -2  2.3 -2.2 2.5 -2.7 /
\arr{2.48 -2.65}{2.5 -2.75}
\put{$\ssize \Hom(T,-)$} [l] at 0.1 -1.5
\put{$\ssize \Ext^1(T,-)$} [l] at -2.8 -1.5
\put{$\mo \Lambda$} at 5 0 
\put{$\mo \Gamma(T)$}  at 5 -4 
\endpicture}
$$
It should be observed that usually, the torsion pair 
$(\mathcal F(T),\mathcal G(T))$ is not split (whereas $(\mathcal Y(T),\mathcal X(T))$ always splits).
Thus, going from $\mo \Lambda$ to $\mo \Gamma(T)$ via the tilting functors,
one ``loses'' some modules (the indecomposable $\Lambda$-modules which are neither
torsion, nor torsionfree). 
	\medskip 

\subsection{Linearity of tilting torsion pairs}
Following Ingalls and Thomas, we are going to show that the torsion pair 
$(\mathcal F(T),\mathcal G(T))$ for any tilting module $T$ is defined by 
a linear form $\alpha_T\!:K_0(\Lambda) \to \mathbb Z$.
	\bigskip 

\subsubsection{\bf First formulation.}
\begin{lemma}\label{lin1}  If $\alpha\!:K_0(\Lambda)
\to \mathbb R$ is a linear form, let
$$
{\beginpicture
\setcoordinatesystem units <1cm,.7cm>
\put{$\mathcal F(\alpha)$} [r] at -.2 0
\put{$\{M\mid \alpha(M') < 0\quad\text{for all submodules $0 \neq M'$ of $M$}\}\ $} [l] at 0 0
\put{$\mathcal G(\alpha)$} [r] at -.2 -1
\put{$\{M\mid \alpha(M'') \ge 0\quad\text{for all factor modules $M''$ of $M$}\}\ $} [l] at 0 -1
\endpicture}
$$
Then the pair $(\mathcal F(\alpha),\mathcal G(\alpha))$ is a torsion pair.
\end{lemma}
Such a torsion pair
will be said to be a {\it linear} torsion pair, or also the {\it torsion pair defined by} $\alpha$.
The proof of the lemma 
is not difficult, we will outline it later in a more general setting (see Lemma \ref{lin} and the 
note N\,\ref{linearity}).
	\medskip

Here is a first version of Theorem \ref{version4}, one of the  main results of the chapter.

\begin{no-text}\label{version1} Let $\Lambda$ be a hereditary artin
algebra and $T$ a tilting module. Then 
$(\mathcal F(T),\mathcal G(T))$ is a linear
torsion pair, it is defined by a linear form $\alpha\!:K_0(\Lambda)
\to \mathbb Z$.
\end{no-text} 
	
The proof will be given in Section \ref{N(T)}. It provides a clear recipe for
$\alpha$; in general, as we will see,  there will be many different linear forms on $K_0(\Lambda)$
which define the torsion pair  $(\mathcal F(T),\mathcal G(T))$.

The result may look quite innocent, but it has striking consequences. Namely,
such a linear form $\alpha$ is just a {\bf stability condition} 
as considered by King, and it determines not only a torsion {\bf pair,} but even 
a torsion {\bf triple}. Here is the definition of a torsion triple:
	
A {\it torsion triple} $(\mathcal F, \mathcal N, \mathcal Q)$ consists of three full subcategories
such that we have 
$\Hom(\mathcal N, \mathcal F) = \Hom(\mathcal Q,\mathcal F) = \Hom(\mathcal Q, \mathcal N) = 0$
and such that any module $M$ has a filtration $0 \subseteq M'' \subseteq M' \subseteq M$
such that $M''$ belongs to $\mathcal Q$, $M'/M''$ to $\mathcal N$ and $M/M'$ to $\mathcal F.$
	
If $\mathcal A, \mathcal B$ are subcategories of $\mo  \Lambda$ (or classes of 
$\Lambda$-modules), we write
$\mathcal A\above\mathcal B$ for the full subcategory of all modules 
$M$ with a submodule $M'$ in $\mathcal B$ such that
$M/M'$ is in $\mathcal A.$
       
\begin{lemma} 
If $(\mathcal F, \mathcal N, \mathcal Q)$ is a torsion triple, then 
$(\mathcal F, \mathcal N \above \mathcal Q)$ and
$(\mathcal F \above \mathcal N, \mathcal Q)$ are torsion pairs
and $\mathcal N = (\mathcal F \above \mathcal N)\cap (\mathcal N \above \mathcal Q).$
In particular, all three classes $\mathcal F, \mathcal N, \mathcal Q$ are closed under extensions, 
$\mathcal F$ is closed under submodules and $\mathcal Q$ is closed under factor modules. 
\end{lemma}

Thus, starting with a torsion triple $(\mathcal F, \mathcal N, \mathcal Q)$, we obtain two torsion pairs
$(\mathcal F, \mathcal G) = (\mathcal F, \mathcal N \above \mathcal Q)$ and 
$(\mathcal F',\mathcal G') = (\mathcal F \above \mathcal N, \mathcal Q)$ with $\mathcal F \subseteq \mathcal F'$.
{\it Conversely, let $(\mathcal F, \mathcal G)$ and $(\mathcal F',\mathcal G')$ be torsion pairs with $\mathcal F \subseteq \mathcal F'$.
Then $(\mathcal F, \mathcal G\cap \mathcal F', \mathcal G')$ is a torsion triple.}

\begin{lemma} 
Let $(\mathcal F, \mathcal N, \mathcal Q)$ and $(\mathcal F', \mathcal N', \mathcal Q')$ be torsion triples.
If $\mathcal F \subseteq \mathcal F',\ \mathcal N \subseteq \mathcal N',\ \mathcal Q \subseteq \mathcal Q',$  
then $\mathcal F = \mathcal F',\ \mathcal N = \mathcal N',\ \mathcal Q = \mathcal Q'$.
\end{lemma} 

Let $\alpha\!:K_0(\Lambda ) \to \mathbb R$ be a linear form. We have already
defined $\mathcal F(\alpha)$ (as well as $\mathcal G(\alpha))$. Let us 
define in a corresponding way  $\mathcal N(\alpha)$ and $\mathcal Q(\alpha)$.

Let  $\mathcal N(\alpha)$ be the full
subcategory of all modules $M$ with $\alpha(M) = 0$, such that $\alpha(M') \le 0$ for any submodule $M'$ of $M$
(or, equivalently, such that $\alpha(M'') \ge 0$ for any factor module $M''$ of $M$).
Following King, the modules in $\mathcal N(\alpha)$
are usually said to be {\it $\alpha$-semistable.}

Let $\mathcal Q(\alpha)$ be the full subcategory of all
the modules $M$ with $\alpha(M') > 0$ for any non-zero factor module $M''$ of $M$
(in particular, $\alpha(M) > 0$ provided $M \neq 0$).
   
\begin{lemma}\label{lin}
 Let $\alpha\!:K_0(\Lambda ) \to \mathbb R$ be a linear form. Then
$(\mathcal F(\alpha), \mathcal N(\alpha), \mathcal Q(\alpha))$ is a torsion triple such that 
$\mathcal N(\alpha)$ is a thick subcategory. {\rm 

We call $(\mathcal F(\alpha), 
\mathcal N(\alpha), \mathcal Q(\alpha))$ the} torsion triple defined by $\alpha$.
\end{lemma} 

	\medskip 

Recall that a subcategory of an abelian category is said to be {\it thick}
provided it is closed under kernels, cokernels and extensions. 
Any thick subcategory
is an abelian category, and the embedding functor is exact.
	\medskip 

For the proof of \ref{lin} (and thus also of Lemma \ref{lin1}) see the note N\,\ref{linearity} 
at the end of the
chapter. The note N\,\ref{examples} draws the attention to some examples of torsion pairs: torsion
pairs usually are not linear, and 
linear torsion pairs are not necessarily tilting (or support-tilting) torsion
pairs. 
	\bigskip

{\bf Consequences.}
Let $(\mathcal F,\mathcal G)$ be the torsion pair defined by some tilting module $T$.
Thus, there is a torsion triple $(\mathcal F,\mathcal N, \mathcal Q)$. Whereas the subcategories
$\mathcal F, \mathcal G$ are not symmetric in nature, the categories 
$\mathcal F(\alpha),\mathcal Q(\alpha)$ defined by a linear from $\alpha$ are defined in 
a symmetric way!
	\smallskip 

As we have mentioned, the subcategories $\mathcal F, \mathcal G$ of $\mo \Lambda$ 
are usually drawn as follows
$$
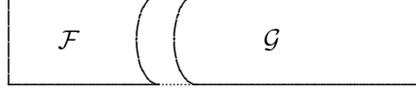
{\beginpicture
\setcoordinatesystem units <1cm,.6cm>
\multiput{} at 0 0   7 2 /
\plot 2 0  0 0  0 2  2 2 /
\ellipticalarc axes ratio 1:2  -180 degrees from 2 0 center at 2 1    

\ellipticalarc axes ratio 1:2  -180 degrees from 2.5 0 center at 2.5 1    
\plot 2.5 0  5.5 0  5.5 2  2.5 2 /
\put{$\mathcal F$} at 0.8 1
\put{$\mathcal G$} at 3.5 1
\setdots <.5mm>
\plot 2 0  2.5 0 / 
\plot 2 2  2.5 2 /
\endpicture}
$$
Now we see that $\mathcal G = \mathcal N\above\mathcal Q$, thus we deal with the following
subcategories:
$$
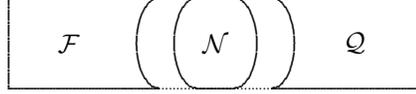
{\beginpicture
\setcoordinatesystem units <1cm,.6cm>
\multiput{} at 0 0   7 2 /
\plot 2 0  0 0  0 2  2 2 /
\ellipticalarc axes ratio 1:2  -180 degrees from 2 0 center at 2 1    

\ellipticalarc axes ratio 1:2  -180 degrees from 2.5 0 center at 2.5 1    
\plot 2.5 0  3 0  /
\plot 2.5 2  3 2  /
\ellipticalarc axes ratio 1:2   180 degrees from 3 0 center at 3 1    

\ellipticalarc axes ratio 1:2   170 degrees from 3.5 0 center at 3.5 1    
\put{$\mathcal F$} at 0.8 1
\put{$\mathcal N$} at 2.75 1
\put{$\mathcal Q$} at 4.6 1
\plot 3.5 0  5.5 0  5.5 2  3.5 2 /
\setdots <.5mm>
\plot 2 0  2.5 0 / 
\plot 2 2  2.5 2 /
\plot 3.5 0  2.5 0 / 
\plot 3.5 2  2.5 2 /
\endpicture}
$$

The illustration concerning the tilting functors 
$\phi = \Hom(T,-)$ and $\Ext^1(T,-)$ has now to be refined as follows:
$$
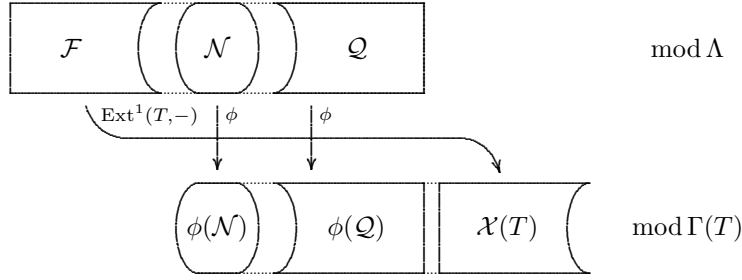
{\beginpicture
\setcoordinatesystem units <1cm,.6cm>
\put{\beginpicture
\multiput{} at 0 0   8 2 /
\plot 2 0  0 0  0 2  2 2 /
\ellipticalarc axes ratio 1:2  -180 degrees from 2 0 center at 2 1    

\ellipticalarc axes ratio 1:2  -180 degrees from 2.5 0 center at 2.5 1    
\plot 2.5 0  3 0  /
\plot 2.5 2  3 2  /
\ellipticalarc axes ratio 1:2   180 degrees from 3 0 center at 3 1    

\ellipticalarc axes ratio 1:2   170 degrees from 3.5 0 center at 3.5 1    
\put{$\mathcal F$} at 0.8 1
\put{$\mathcal N$} at 2.75 1
\put{$\mathcal Q$} at 4.6 1
\plot 3.5 0  5.5 0  5.5 2  3.5 2 /

\setdots <.5mm>
\plot 2 0  2.5 0 / 
\plot 2 2  2.5 2 /
\plot 3.5 0  5.5 0  5.5 2  3.5 2 /
\plot 3.5 0  2.5 0 / 
\plot 3.5 2  2.5 2 /
\endpicture} at 0 0 
\put{\beginpicture
\multiput{} at 0 0   8 2 /

\put{$\mathcal X(T)$} at 6.6 1

\ellipticalarc axes ratio 1:2  -180 degrees from 2.5 0 center at 2.5 1    
\plot 2.5 0  3 0  /
\plot 2.5 2  3 2  /
\ellipticalarc axes ratio 1:2   180 degrees from 3 0 center at 3 1    

\ellipticalarc axes ratio 1:2   170 degrees from 3.5 0 center at 3.5 1    
\put{$\phi(\mathcal N)$} at 2.75 1
\put{$\phi(\mathcal Q)$} at 4.6 1
\plot 3.5 0  5.5 0  5.5 2  3.5 2 /

\plot 7.7 0  5.7 0   5.7 2  7.7 2 /
\ellipticalarc axes ratio 1:2  -180 degrees from 7.7 0 center at 7.7 1    
\setdots <.5mm>
\plot 5.5 0  5.7 0 / 
\plot 5.5 2  5.7 2 /

\setdots <.5mm>
\plot 3.5 0  2.5 0 / 
\plot 3.5 2  2.5 2 /

\endpicture} at 0 -4 
\arr{0 -1.3}{0 -2.7}
\arr{-1.25 -1.3}{-1.25 -2.7}
\setquadratic
\plot -3 -1.3 -2.8 -1.8 -2.5 -2  0 -2
  2 -2  2.3 -2.2 2.5 -2.7 /
\arr{2.48 -2.65}{2.5 -2.75}
\put{$\ssize \phi$} [l] at -1.15 -1.5
\put{$\ssize \phi$} [l] at 0.1 -1.5
\put{$\ssize \Ext^1(T,-)$} [l] at -2.8 -1.5
\put{$\mo \Lambda$} at 5 0 
\put{$\mo \Gamma(T)$}  at 5 -4 
\endpicture}
$$

	\bigskip 

{\bf Remark.} The name ''tilting'' was chosen by Brenner and Butler in view of 
the tilting of the coordinate axes which one observes when going from 
$K_0(\Lambda)$ to $K_0(\Gamma(T))$. The linearity assertion supports this
observation.
	\medskip

\subsubsection{\bf The support of a module, sincere modules, 
support tilting modules}\label{support-tilting-def}
As we have mentioned in Chapter 1, 
the {\it support} of a module $M$ consists of all the simple modules $S$ (or better, their
isomorphism classes) which occur as composition factors of $M$. Also, if all simple
modules occur as composition factors of $M$, then $M$ is said to be {\it sincere}.
Given a module $M$, its {\it support algebra} $\Lambda_M$ is the factor algebra
$\Lambda/I_M$, where $I_M$ is the two-sided ideal generated by all idempotents $e\in \Lambda$
with $eM = 0.$ Since $I_M$ annihilates the module $M$, one may consider $M$ as a
$\Lambda_M$-module; of course, $M$ considered as a $\Lambda_M$-module is sincere.
A module $T$ is called a {\it support-tilting} module, provided $T$ considered
as a $\Lambda_T$-module is a tilting module. 
	\bigskip 

{\bf Support tilting torsion pairs.} If $T$ is a support tilting module,
then we denote again by $\mathcal G(T)$ the full subcategory of all modules generated by
$T$ and by $\mathcal F(T)$ the class of all modules $Y$ with $\Hom(T,Y) = 0.$  
Then this is again a torsion pair (note that if $\bold S$ is the support of $T$,
then all the modules in $\mathcal G(T)$ have support in $\bold S$, whereas $\mathcal F(T)$
will contain modules whose support is not contained in $\bold S$ (provided $\bold S$
is not the set of all simple modules).
  
As we will see, \ref{version1} above holds true in this more general
situation (thus, here we have a second preliminary version of Theorem \ref{version4}): 
	
\begin{no-text}\label{version2}  Let $\Lambda$ be a hereditary artin
algebra and $T$ a support-tilting module. Then 
$(\mathcal F(T),\mathcal G(T))$ is a linear
torsion pair, it is defined by a linear form $\alpha\!:K_0(\Lambda)
\to \mathbb Z$.
\end{no-text} 

\subsection{Exceptional antichains and normal partial tilting modules}
Our main interest will be devoted to the exceptional antichains in $\mo\Lambda$.
We will show that the exceptional antichains can be obtained by normalizing a partial
tilting module. 
	\medskip 

\subsubsection{\bf Normality.}
Here we go back to the beginning of modern representation theory, namely to Roiter's
proof \cite{[Ro]} of the first Brauer-Thrall conjecture. The first topic which he discusses in his
paper is the normalization of a module. Actually, this concept was reinvented several times,
for example by Auslander-Smal\o{} when dealing with minimal covers of additive
subcategories (a normalization of a module $M$ is a minimal cover of the
subcategory $\add M$, where 
$\add M$ denotes the subcategory of all direct summands of direct sums of copies of $M$). 
	\medskip

We call a module $M$ {\it normal} provided no proper direct summand of $M$ generates $M$;
this means that $M = M'\oplus M''$ with  $M'$
generating $M''$ (thus even $M$) implies that
$M'' = 0$.
	\medskip

\begin{lemma}[\bf Roiter's Normalization-Lemma]
Let $\Lambda$ be an artin algebra. 
Any module $M$ can be decomposed $M = M'\oplus M''$ such that $M'$ is
normal and generates $M$ (or, equivalently, $M''$).
Such a decomposition is unique up to isomorphism.
\end{lemma}

For a proof, see the note N\,\ref{Roiter}. 
The module $M' = \nu(M)$ is called a {\it normalization} of $M$. If $M$
is multiplicity-free, we write $M = \nu(M) \oplus \nu'(M)$, where $\nu(M)$ 
is the normalization of $M$ (the assumption on $M$ to be multiplicity-free
assures that  $\nu(M)$ and $\nu'(M)$ have no
indecomposable direct summands which are isomorphic).
		\medskip 

\subsubsection{\bf Antichains.}
Let us recall from Chapter 1 the definition of an antichain (see also the note
N\,\ref{antichain}):
An {\it antichain} $A$ in an additive category
$\mathcal C$ is a family $A = \{A_i\}_{i}$ of pairwise orthogonal bricks $A_i$.
Given an antichain in an abelian category $\mathcal C$, we may consider 
its extension closure $\mathcal E(A)$, this is the full subcategory of all
objects in $\mathcal C$ having a filtration with factors in $A$. 
It has been shown in \cite{[R2]} that $\mathcal E(A)$
is a thick subcategory. Conversely, if $\mathcal C$ is a length category (an abelian category
in which any object has finite length), then any thick subcategory arises in this way:
we denote by $S (\mathcal C)$ the set of simple objects in $\mathcal C$; the
Schur Lemma asserts that this is an antichain and we have $\mathcal C = \mathcal E(S (\mathcal C))$. 
Antichains (and thick subcategories) have been considered in several papers. For example, antichains 
appear in \cite{[GP]} under the name {\it $\Hom$-free sets.}
	\medskip

Starting with an antichain 
$A = \{A_1,\dots,A_t\}$, its 
{\it $\Ext$-quiver} has $t$ vertices, and there is an arrow $j\to i$ provided 
$\Ext^1(A_j,A_i) \neq 0.$ We say that the antichain $A$ is 
{\it exceptional} provided its $\Ext$-quiver 
has no oriented cyclic path (thus, provided we may number the modules $A_i$ in such
a way that $\Ext^1(A_i,A_j) = 0$ for all $i\le j$). Clearly, {\it $A$ is exceptional
if and only if $\mathcal E(A)$ is exceptional.} 

	\bigskip 
\subsubsection{} The following bijections will be the first ones of a long list,
see Theorem \ref{IT-bijections}.

\begin{no-text}
There are bijections between: 
\begin{enumerate}
\item[\rm(1)] Exceptional antichains.
\item[\rm(1$'$)] Exceptional subcategories.
\item[\rm(1$''$)] Normal partial tilting modules.
\end{enumerate}
\end{no-text}

The bijection between (1) and (1$'$) has already been mentioned (with reference
to \cite{[R2]}, 1976):
$$
{\beginpicture
\setcoordinatesystem units <2cm,1cm> 
\put{\{antichains\}} at -.8 0 
\arr{0.1 0.1}{0.9 0.1}
\arr{0.9 -.1}{0.1 -.1}
\put{$\mathcal E(-)$} at 0.5 .4
\put{$S (-)$} at 0.5 -.4
\put{\{thick subcategories\}} at 2 0 
\endpicture}
$$
Given an antichain $A$, the subcategory $\mathcal E(A)$ is abelian and closed under
extensions and the simple objects in $\mathcal E(A)$ are the elements of $A$, thus $S (\mathcal E(A)) = A$.
Conversely, given a thick subcategory $\mathcal C$, the set $S (\mathcal C)$ consists of 
the simple objects in $\mathcal C.$ This is an antichain, and $\mathcal C = \mathcal E(S (\mathcal C)).$

Of course, this bijection induces a bijection of the following subsets:
$$
{\beginpicture
\setcoordinatesystem units <2cm,1cm> 
\put{antichains} at -1 -.165
\put{exceptional} at -1 .165
\multiput{$\Big\{$} at -1.55 0  1.7 0 /
\multiput{$\Big\}$} at -.45 0  2.9 0 /
\arr{0.1 0.1}{0.9 0.1}
\arr{0.9 -.1}{0.1 -.1}
\put{$\mathcal E(-)$} at 0.5 .4
\put{$S (-)$} at 0.5 -.4
\put{exceptional} at 2.3 .165 
\put{subcategories} at 2.3 -.165 
\endpicture}
$$
But here we can add a third class, namely the normal partial tilting modules:
$$
{\beginpicture
\setcoordinatesystem units <2cm,1cm> 
\put{antichains} at -1 -.165
\put{exceptional} at -1 .165
\multiput{$\Big\{$} at -1.55 0  1.7 0 /
\multiput{$\Big\}$} at -.45 0  2.9 0 /
\arr{0.1 0.1}{0.9 0.1}
\arr{0.9 -.1}{0.1 -.1}
\put{$\mathcal E(-)$} at 0.5 .4
\put{$S (-)$} at 0.5 -.4
\put{exceptional} at 2.3 .165 
\put{subcategories} at 2.3 -.165 
\put{\{normal partial tilting modules\}} at .5 -2 
\arr{1.5 -0.6}{.7 -1.6}
\setdashes <1mm>
\arr{.3 -1.6}{-.5 -0.6}
\put{$\Delta$} at -.3 -1.2
\put{minimal generator} at 1.95 -1.2
\endpicture}
$$
If $\mathcal C$ is an exceptional subcategory, then $\mathcal C$ has
a minimal generator $P$ (namely, $\mathcal C$ is equivalent to the module category of an artin
algebra $\Lambda'$, thus it has a progenerator). Clearly, 
$P$, considered as a $\Lambda$-module, is a normal partial tilting module. 
	\medskip 

What about the dashed arrow with the label $\Delta$? This will be discussed now.
	\medskip 

\subsubsection{}
We refer to considerations which have been useful in the study of quasi-hereditary
algebras (see \cite{[DR3]}): an exceptional antichain $A$ is a standardizable set, thus  
there is a quasi-hereditary algebra $B$ such that the 
subcategory $\mathcal E(A)$ is equivalent to the category of $\Delta$-filtered $B$-modules. Since the
standardizable set $A$ consists of pairwise orthogonal modules, the same is true 
for the $\Delta$-modules
of $B$, and consequently the $\Delta$-modules of $B$ are just the simple $B$-modules. This shows that
the category of $\Delta$-filtered $B$-modules is the whole category $\mo  B.$
In this way, we obtain the converse of the dashes arrow:
We start with an exceptional
antichain $A$, the paper \cite{[DR3]} shows how to  
construct indecomposable objects in $\mathcal E(A)$ which are the 
indecomposable projective objects in $\mathcal E(A)$. In this way, we obtain
a minimal generator for the
abelian category $\mathcal E(A)$ (thus a normal partial tilting module). 
	\medskip

Now let us discuss the dashed arrow itself.
Let $N = \bigoplus_{i=1}^t N_i$ be a normal partial tilting module with
indecomposable direct summands $N_i$. 
Then non-zero homomorphisms  between indecomposables in $\add N$ are monomorphisms
(namely, since $\Ext^1(N_j,N_i) = 0$, any non-zero homomorphism $N_i \to N_j$ has to be
either injective or surjective). 
Thus, the indecomposables in $\add N$ form a poset. 

Given an indecomposable direct summand $N_i$ of $N$, the
sink map for $N_i$ in $\add N$
cannot be surjective (since $N$ is normal), 
thus it has to be injective (since $\Ext^1(N,N) = 0$).
Let $\Delta(i)$ be its cokernel. Thus, there is an exact sequence
$$
{\beginpicture
\setcoordinatesystem units <1cm,1cm> 
\put{$0 \to N_i' \to N_i \to \Delta(i) \to 0$} [l] at 0 0
\put{$(*)$} [r] at -3 0
\put{} at 9 0 
\endpicture}
$$
with $N'_i\in \add N$. It is easy to see that these modules $\Delta(i)$ form
an antichain and by induction one sees that $N_i$ has a filtration with factors of the
form $\Delta(j)$. Of course, this antichain 
$\Delta(N) = \{\Delta(1),\dots,\Delta(t)\}$ is exceptional.
	\medskip 

\subsubsection{} It is worthwhile to look at the composition of $\Delta$ and $\mathcal E$.
If $N$ is a normal, partial tilting module, let us denote
$$
 \mathcal N(N) = \mathcal E(\Delta(N)),
$$
then {\it $\mathcal N(N)$ is the smallest thick subcategory containing $N$.} (Namely, since $N$ has
a filtration with factors in $\Delta(N)$, we see that $N$ belongs to $\mathcal N(N)$.
On the other hand, any object in $\mathcal N(N)$ has a filtration with factors in $\Delta(N)$,
and the modules in $\Delta(N)$ are cokernels of maps in $\add N$.)
Let us add the following obvious assertion: {\it If $N$ is sincere, then also $\mathcal N(T)$ is 
sincere.}
	\medskip 

The construction $\mathcal N(-)$ from the set of normal partial tilting modules to the
exceptional subcategories is the inverse of taking a minimal generator. 
$$
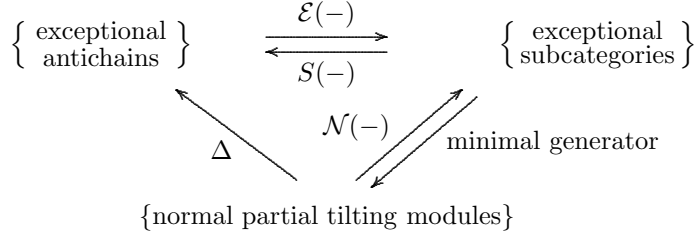
{\beginpicture
\setcoordinatesystem units <2cm,1cm> 
\put{antichains} at -1 -.165
\put{exceptional} at -1 .165
\multiput{$\Big\{$} at -1.55 0  1.7 0 /
\multiput{$\Big\}$} at -.45 0  2.9 0 /
\arr{0.1 0.1}{0.9 0.1}
\arr{0.9 -.1}{0.1 -.1}
\put{$\mathcal E(-)$} at 0.5 .4
\put{$S (-)$} at 0.5 -.4
\put{exceptional} at 2.3 .165 
\put{subcategories} at 2.3 -.165 

\put{\{normal partial tilting modules\}} at .5 -2.3 
\arr{.7 -1.8}{1.4 -.6}
\arr{1.5 -.7}{.8 -1.9}
\arr{.3 -1.8}{-.5 -0.6}
\put{$\Delta$} at -.2 -1.4
\put{$\mathcal N(-)$} at .7 -1.1
\put{minimal generator} at 2. -1.3
\endpicture}
$$

If $T$ is a support-tilting module, we define
$$
 \mathcal N(T) = \mathcal N(\nu(T)).
$$
The next section is devoted to a detailed study of the subcategories of the form $\mathcal N(T)$. 
	\bigskip

\subsection{The category $\mathcal N(T)$}\label{N(T)}
Given a tilting module $T$, it is the category $\mathcal N(T)$ which is of interest. 
In order to understand 
$\mathcal N(T)$, we need some prerequisites.
Let $\Gamma(T) = \End(T)^{\text{op}}$. 
	\medskip

\subsubsection{\bf The simple $\Gamma(T)$-modules.}
First, let us show that {\it the decomposition
$T = \nu(T)\oplus\nu'(T)$ corresponds to the distribution of the simple $\Gamma(T)$-modules
$S'(i)$ into $\mathcal Y(T)$ and $\mathcal X(T)$.}
	\medskip
 
We note that the indecomposable projective $\Gamma(T)$-modules are the modules
$\Hom(T,T_i)$, where $T_i$ is any indecomposable direct summand of $T$, thus
the simple $\Gamma(T)$-modules are the modules $S'(i) = \topp\Hom(T,T_i)$.
	\medskip

For any indecomposable direct summand $T_i$ of $T = \bigoplus T_j$, let $T^{(i)} =
T/T_i = \bigoplus_{j\neq i} T_j$ and let $g_i\!:T^i\to T_i$ be a minimal right 
$T^{(i)}$-approximation of $T_i$.
	\medskip 

\begin{prop}\label{prop1}  Let $T_i$ be an indecomposable direct summand of the tilting module
$T$. Then there are two possibilities: 

If $T_i$ is a direct summand of $\nu(T)$, then $g_i\!:T^i\to T_i$ is a monomorphism
say with cokernel $\Delta(i)$. In this case,
$S'(i) = 
\topp\Hom(T,T_i) = \Hom(T,\Delta(i))$ is a simple $\Gamma(T)$-module which belongs to $\mathcal Y(T)$.

If  $T_i$ is a direct summand of $\nu'(T)$, then $g_i\!:T^i\to T_i$ is an epimorphism
say with kernel $U(i)$. In this case,
$S'(i) = \topp\Hom(T,T_i) 
= \Ext^1(T,U(i))$ is a simple $\Gamma(T)$-module which belongs to $\mathcal X(T)$.
\end{prop}
	\medskip 

Also, let us note:
	\medskip

{\it If $S,S'$ are simple $\Gamma(T)$-modules with an arrow $[S] \to 
[S']$ in the quiver of $\Gamma(T)$ (thus with $\Ext^1(S,S') \neq 0$), 
then, if $S$ belongs to 
$\mathcal Y(T)$, also $S'$ belongs to $\mathcal Y(T)$.} 
		\bigskip

\begin{prop}\label{prop2} Let $T$ be a tilting module and $\phi = \Hom(T,-)$ the
corresponding tilting functor. Then $\mathcal N(T)$ is the inverse image under $\phi$ of the
Serre subcategory in $\mo \Gamma(T)$ given by all $\Gamma(T)$-modules whose composition
factors belong to $\mathcal Y(T).$
\end{prop}

	\bigskip 
The proof of the Propositions \ref{prop1} and \ref{prop2} will be outlined in N\,\ref{simple}. 
	\bigskip 

\begin{prop} $T$ be a support tilting module with normal decomposition
$T = \nu(T)\oplus \nu'(T)$. Then
\end{prop} 
	\smallskip
\Rahmen{$\mathcal N(T) = \mathcal G(T)\cap \mathcal F(\nu'(T)).$}

	\medskip
\begin{proof} 
Since $\mathcal N(T)$ is obtained from $\nu(T)$ by forming cokernels and extensions,
we see that $\mathcal N(T)$ is contained in $\mathcal G(T).$

Next, we note that $\Hom(\nu'(T),\nu(T)) = 0.$ 
Namely, consider a map $f\!:T_j\to T_i$ where 
$T_i$ is an indecomposable direct summand of $\nu(T)$ and $T_j$ is
an indecomposable direct summand of $\nu'(T)$. It cannot be surjective, since otherwise
with $T_j$ also $T_i$ would be a direct summand of $\nu'(T)$. It also cannot be injective,
since there is a non-split epimorphism $T' \to T_j$ with $T'$ in $\add T$ and we have
$\Ext^1(T_i,T') = 0.$ Thus $f = 0.$  

As a consequence, we have $\Hom(\nu'(T),\Delta(i)) = 0$ (since 
$\Ext^1(T_i,\nu'(T)) = 0$) and therefore $\Hom(\nu'(T),\mathcal N(T)) = 0$. Thus
$\mathcal N(T) \subseteq \mathcal F(\nu'(T)).$

Conversely, assume that $M$ belongs to $\mathcal G(T)\cap \mathcal F(\nu'(T))$. Since $M$ is in
$\mathcal G(T)$, the minimal right $\add T$-approximation of $M$ provides 
an exact sequence $0 \to T'' \to T' \to M\to 0$ with
$T',T'' \in \add(T)$. Since $\Hom(\nu'(T),M) = 0$, we see that $T'$ belongs to $\add \nu(T)$.
Since $\Hom(\nu'(T),\nu(T) = 0),$ we also have $T'' \in \add \nu(T)$. Thus, we see
that $M$ is the cokernel of a map in $\add\nu(T)$ and therefore belongs to the thick
closure $\mathcal N(T)$ of $\nu(T)$.
\end{proof}

\subsubsection{\bf The linearity of the triple
$(\mathcal F(T),\mathcal N(T),\mathcal Q(T))$.}
 Let $T$ be a  tilting module.
We use the decomposition $T = \nu(T)\oplus \nu'(T)$ in order to define
$\mathcal N(T)$ and $\mathcal Q(T)$. Let $\mathcal N(T)$ be the category of all modules $M$
generated by $T$ such that $\Hom(\nu'(T),M) = 0)$ and let
$\mathcal Q(T)$ be the full subcategory of all modules generated by $\nu'(T)$.
Then {\it $(\mathcal F(T),\mathcal N(T),\mathcal Q(T))$ is a torsion triple.} 
	\medskip

Namely, let us denote by $\mathcal F'$ the full subcategory of all modules $M$
with $\Hom(\nu'(T),M) = 0$. We deal with the two torsion pairs
$(\mathcal F(T),\mathcal G(T))$ and $(\mathcal F',\mathcal Q(T))$, they satisfy $\mathcal F(T)\subseteq
\mathcal F'$. Since $\mathcal N(T) = \mathcal G(T)\cap \mathcal F'$, we see that 
$(\mathcal F(T),\mathcal N(T),\mathcal Q(T))$ is a torsion triple.
	\bigskip	

Given a $\Lambda$-module $M$, we may consider the linear form 
$\langle \bdim M,-\rangle$ on $K_0(\Lambda)$, we will denote it just by
$\langle M,-\rangle.$ Here is a third version of \ref{version1}, the final version
will be Theorem \ref{version4}

\begin{no-text}\label{version3} Let $T$ be a tilting module. 
Then $\alpha_T = \langle \nu'(T),-\rangle$ defines the torsion triple
$(\mathcal F(T),\mathcal N(T), \mathcal Q(T))$. In particular, $\mathcal N(T)$ is a thick subcategory.

{\rm More generally:} If $T'$ is a module which is Morita-equivalent to $\nu'(T)$, then
$\langle T',-\rangle$ defines the torsion triple $(\mathcal F(T),\mathcal N(T), \mathcal Q(T))$.
\end{no-text} 

Let us recall that modules $M, M'$ are said to be {\it Morita equivalent} 
provided $\add(M) = \add(M')$. 
	\medskip 

The final statement of \ref{version3} 
implies the following: {\it if $\nu'(T)$ is the direct sum of $r$
indecomposable modules, then there are $r$ linearly independent linear forms $\alpha_i$
 which define the torsion triple $(\mathcal F(T),\mathcal N(T), \mathcal Q(T))$.} Namely, let
$\nu'(T) = \bigoplus_{i=1}^r T_i$ with indecomposable direct summands $T_i$. Let
$T'_i = \nu'(T)\oplus T_i$. Then all the modules $T'_i$ with $1\le i \le r$ are
Morita equivalent to $\nu'(T)$, and the linear forms $\alpha_i = \langle T'_i,-\rangle$
 with $1\le i\le r$ are linearly independent.  

	\bigskip 
\begin{proof}[Proof of \ref{version3}.] Let $\alpha = \langle T',-\rangle$ for some module $T'$ 
which is Morita equivalent to
$\nu'(T)$. 
	\smallskip

(1) $\mathcal Q(T) \subseteq \mathcal Q(\alpha).$
	\smallskip 

Proof. Let $M\in \mathcal Q(T)$ be a non-zero module. Since $M$ is generated by $\nu'(T)$,
we must have $\Hom(T',M) \neq 0$. On the other hand, since $M$ is generated by
$\nu'(T)$, it belongs to $\mathcal G(T)$, thus $\Ext^1(T,M) = 0$, therefore also
$\Ext^1(T',M) = 0.$ This shows that $\alpha(M) > 0$. Of course, if $M''$ is a non-zero
factor module of $M$, then also $M''$ is a non-zero module in $\mathcal Q(T)$.
Therefore, as we have seen, $\alpha(M'') = 0.$
	\smallskip 

(2) $\mathcal N(T) \subseteq \mathcal N(\alpha).$
	\smallskip

Proof. Let $M$ belong to $\mathcal N(T)$. Since $M$ belongs to $\mathcal G(T)$, we have 
$\Ext^1(T,M) = 0$, thus $\Ext^1(T',M) = 0.$ Since $\Hom(\nu'(T),M)) = 0,$
we also have $\Hom(T',M) = 0.$ This shows that $\alpha(M) = 0.$ If $M''$ is any
factor module of $M$, then also $M''$ is generated by $T$, thus $\Ext^1(T',M) = 0$
and therefore $\alpha(M'') \ge 0.$
	\smallskip 

(3) $\mathcal F(T) \subseteq \mathcal F(\alpha).$
	\smallskip

Let $M$ be a non-zero module in $\mathcal F(T)$. Since $\Hom(T,M) = 0$, we also
have $\Hom(T',M) = 0$. Thus, it remains to show that $\Ext^1(T',M) \neq 0$,
or, equivalently, that there exists some indecomposable direct
summand $T_j$ of $\nu'(T)$ with $\Ext^1(T_j,M)\neq 0.$

We consider the $\Gamma(T)$-module $\psi(M) = \Ext^1(T,M)$ which
belongs to $\mathcal X(T)$. The support of its dimension vector is the set of all
indices $s$ with $\Ext^1(T_s,M) \neq 0.$ Now assume that $\Ext^1(T_j,M) = 0$
for all indecomposable direct summands of $\nu'(T)$.
Then all the composition factors of $\psi(M)$
are of the form $\Hom(T,\Delta(i))$ with $T_i$ an indecomposable direct summand of
$\nu(T)$. But these modules $\Hom(T,\Delta(i))$ belong to $\mathcal Y(T)$ and
$\mathcal Y(T)$ is closed under extensions. This implies that $\psi(M)$ belongs to
$\mathcal Y(T)$, a contradiction. Thus, we see that $\Ext^1(T',M) \neq 0$, therefore
$\alpha(M) < 0.$ 

Of course, if $M'$ is a non-zero submodule of $M$, then also $M'$ is a non-zero
module in $\mathcal F(T)$ and therefore $\alpha(M') < 0.$ 
\end{proof}

\begin{cor}\label{determines}
The direct summand $\nu'(T)$ of $T$ determines $T$.
\end{cor} 
	
\begin{proof} 
The module $\nu'(T)$ determines  $\alpha_T = \langle \nu'(T),-\rangle$, thus $\mathcal F(T)$
and therefore 
$(\mathcal F(T),\mathcal G(T))$ and this in turn  determines the module $T$. 
\end{proof}

\subsubsection{\bf Remark}
For any tilting module $T$, we have $\mathcal G(T) = \mathcal N(T)\above
\mathcal Q(T)$: any module $M$ in $\mathcal G(T)$ has a (uniquely determined) submodule $M'$
which belongs to $\mathcal Q(T)$ such that $M/M'$ belongs to $\mathcal N(T)$. We should stress
that this submodule $M'$ of $M$ usually will not be a direct summand (thus, the
torsion pair $(\mathcal N(T),\mathcal Q(T))$ in the additive category $\mathcal G(T)$ may not be
split). As an example, consider the linearly oriented quiver of type $\mathbb A_3$ with
path algebra $\Lambda$,
with indecomposable projective modules $P(1) \subset P(2) \subset P(3)$ and choose
the tilting module $T = \bigoplus T_i$ with $T_1 = P(2),\ T_2 = P(3),\ T_3 = S(2)$
(see the left picture below, the bullets mark the modules $T_i$),
 here $\nu(T) = T_1\oplus T_2$ and $\nu'(T) = T_3.$
$$
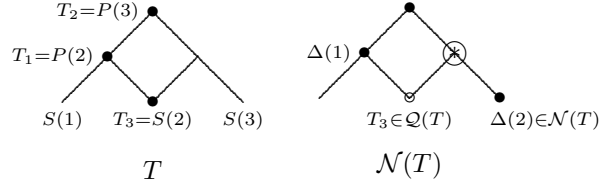
{\beginpicture
\setcoordinatesystem units <.6cm,.6cm>
\put{\beginpicture
\multiput{} at 0 0  4 2 /
\multiput{$\bullet$} at 1 1  2 2  2 0 /
\plot 0 0  2 2  4 0 /
\plot 1 1  2 0  3 1 /
\put{$\ssize S(1)$} at 0 -.4  
\put{$\ssize T_3 = S(2)$} at 2 -.4
\put{$\ssize S(3)$} at 4 -.4
\put{$\ssize T_1 = P(2)$} [r] at .7 1
\put{$\ssize T_2= P(3)$} [r] at 1.7 2
\put{$T$} at 2 -1.5
\endpicture} at 0 0
\put{\beginpicture
\multiput{} at 0 0  4 2 /
\multiput{$\bullet$} at 1 1  2 2  4 0 /
\plot 0 0  2 2  4 0 /
\plot 1 1  2 0  3 1 /
\put{$*$} at 3 1 
\put{$\circ$} at 2 0 
\multiput{$\bigcirc$} at 3 1 /
\put{$\ssize \Delta(1)$} [r] at 0.7 1  
\put{$\ssize \Delta(2)\in \mathcal N(T)$} at 5 -.5  
\put{$\ssize T_3 \in \mathcal Q(T)$} at 2 -.5
\put{$\mathcal N(T)$} at 2 -1.5
\endpicture} at 7 0
\endpicture}
$$
On the right, the bullets mark the indecomposable modules which belong to 
$\mathcal N(T)$. Note that 
$S(3) = \Delta(2)$ belongs to $\mathcal N(T)$, whereas 
$T_3 = S(2)$ belongs to $\mathcal Q(T)$.
The module $P(3)/P(1)$ marked by a
big circle is a non-split extension of $T_3\in  \mathcal Q(T)$ by $\Delta(2)\in \mathcal N(T)$.
	\bigskip

Let us have another look at the 
torsion triple $(\mathcal F(T),\mathcal N(T),\mathcal Q(T))$ defined by a tilting module $T$.
The following picture 
$$
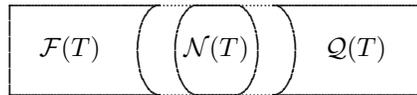
{\beginpicture
\setcoordinatesystem units <1cm,.6cm>
\multiput{} at 0 0   7 2 /
\plot 2 0  0 0  0 2  2 2 /
\ellipticalarc axes ratio 1:2  -180 degrees from 2 0 center at 2 1    

\ellipticalarc axes ratio 1:2  -180 degrees from 2.5 0 center at 2.5 1    
\plot 2.5 0  3 0  /
\plot 2.5 2  3 2  /
\ellipticalarc axes ratio 1:2   180 degrees from 3 0 center at 3 1    

\ellipticalarc axes ratio 1:2   170 degrees from 3.5 0 center at 3.5 1    
\put{$\mathcal F(T)$} at 0.8 1
\put{$\mathcal N(T)$} at 2.75 1
\put{$\mathcal Q(T)$} at 4.6 1
\plot 3.5 0  5.5 0  5.5 2  3.5 2 /
\setdots <.5mm>
\plot 2 0  2.5 0 / 
\plot 2 2  2.5 2 /
\plot 3.5 0  2.5 0 / 
\plot 3.5 2  2.5 2 /
\endpicture}
$$
suggests a complete symmetry between the roles of $\mathcal F(T)$ and $\mathcal Q(T)$,
but there is some {\bf asymmetry} involved which we should mention.
Whereas the subcategory $\mathcal Q(T)$ is generated by $\mathcal N(T)$,
the subcategory $\mathcal F(T)$ is not necessarily cogenerated by  $\mathcal N(T)$
(but it is always cogenerated by $\mathcal G(T)$).

As an example, we 
consider again the linearly oriented quiver of type $\mathbb A_3$,
with indecomposable projective modules $P(1) \subset P(2) \subset P(3)$. This time, we 
take as tilting module $T$ the minimal cogenerator
(see the left picture below),
 here $\nu(T) = P(3)$, and $\nu'(T) = P(3)/S(1)\oplus S(3).$
$$
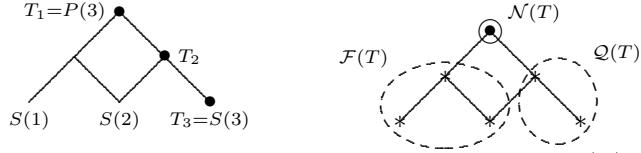
{\beginpicture
\setcoordinatesystem units <.6cm,.6cm>
\put{\beginpicture
\multiput{} at 0 0  4 2 /
\multiput{$\bullet$} at   2 2  3 1  4 0  /
\plot 0 0  2 2  4 0 /
\plot 1 1  2 0  3 1 /
\put{$\ssize S(1)$} at 0 -.4  
\put{$\ssize S(2)$} at 2 -.4
\put{$\ssize T_3 = S(3)$} at 4 -.4
\put{$\ssize T_2$} [l] at 3.3 1
\put{$\ssize T_1= P(3)$} [r] at 1.7 2
\endpicture} at 0 0
\put{\beginpicture
\multiput{} at 0 0  4 2 /
\multiput{$\bullet$} at   2 2  /
\plot 0 0  2 2  4 0 /
\plot 1 1  2 0  3 1 /

\multiput{$\bigcirc$} at 2 2 /
\put{$\ssize \mathcal F(T)$} at -.8 1.4
\put{$\ssize \mathcal N(T)$} [l] at 2.4 2.4
\put{$\ssize \mathcal Q(T)$} at 4.8 1.5
\multiput{$*$} at 0 0  1 1  2 0  3 1  4 0 /
\setdashes <1mm>
\ellipticalarc axes ratio 1.5:1 360 degrees from 1 -.6 center at 1 0.35
\ellipticalarc axes ratio 1:1.1 360 degrees from 3.5 -.5 center at 3.5 0.45
\endpicture} at 8 0
\endpicture}
$$
On the right, we show the corresponding torsion triple.
Here, $\mathcal F(T)$ is not cogenerated by $\mathcal N(T)$.
	\medskip 

Also, $\mathcal Q(T)$ does not determine the torsion triple
(in contrast to $\mathcal F(T)$). We use the same algebra $\Lambda$ as before
and present two non-isomorphic tilting modules $T,T'$ with $Q(T) = Q(T')$.
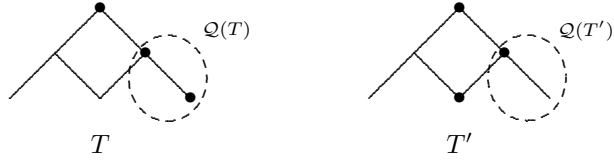
$$
{\beginpicture
\setcoordinatesystem units <.6cm,.6cm>
\put{\beginpicture
\multiput{} at 0 0  4 2 /
\multiput{$\bullet$} at   2 2  3 1  4 0  /
\plot 0 0  2 2  4 0 /
\plot 1 1  2 0  3 1 /
\setdashes <1mm>
\ellipticalarc axes ratio 1:1.1 360 degrees from 3.5 -.5 center at 3.5 0.45
\put{$\ssize \mathcal Q(T)$} at 4.8 1.5
\put{$T$} at 2 -1
\endpicture} at 0 0
\put{\beginpicture
\multiput{} at 0 0  4 2 /
\multiput{$\bullet$} at   2 2  3 1  2 0  /
\plot 0 0  2 2  4 0 /
\plot 1 1  2 0  3 1 /
\put{$T'$} at 2 -1
\setdashes <1mm>
\ellipticalarc axes ratio 1:1.1 360 degrees from 3.5 -.5 center at 3.5 0.45
\put{$\ssize \mathcal Q(T')$} at 4.8 1.5
\endpicture} at 8 0
\endpicture}
$$
Let us repeat: given a tilting module $T$, the partial tilting module $\nu'(T)$ 
determines $T$ uniquely (see the Corollary \ref{determines}), but the class $\mathcal Q(T)$
of modules generated by $\nu'(T)$ does not determine $T$.
	\medskip 

\subsubsection{\bf Remark.}
As we see, any tilting torsion pair $(\mathcal F(T),\mathcal G(T))$
is given by a linear form $\alpha$ on $K_0(\Lambda)$
and in this way, we obtain a sincere exceptional subcategory $\mathcal N(\alpha)$
contained in $\mathcal G(T)$.
But we should remark that there are linear forms $\alpha$ on $K_0(\Lambda)$ with
$\mathcal N(\alpha)$ being a sincere exceptional  subcategory such that 
$\mathcal Q(\alpha)$ is not generated by $\mathcal N(\alpha)$ (thus they do not come from a tilting module),
and also dually, $\mathcal F(\alpha)$ is not cogenerated by $\mathcal N(\alpha)$
(thus they do not come from a cotilting module).
An example is presented in the note N\,\ref{example2}.
	\bigskip 

\subsubsection{\bf The general case of support-tilting modules.} Let $T$ be a support-tilting module.
As we have mentioned in Section \ref{support-tilting-def},
$\mathcal F(T)$ denotes the full subcategory of all
modules $M$ with $\Hom(T,M) = 0.$ As for a tilting module, let 
$\mathcal N(T)$ be the thick subcategory generated
by $\nu(T)$ and $\mathcal Q(T)$ the full subcategory of all modules generated by $\nu'(T).$ 
	
\begin{theorem}\label{version4} {\bf (Ingalls-Thomas).}  Let $T$ be a support-tilting module. Let $P$ be 
a projective module such that the top of $P$ is the direct sum of all simple
modules outside of the support of $T$.
Then  the linear form
$$
 \alpha = \langle \nu'(T),-\rangle - 
 \langle P,-\rangle  
$$
defines the torsion triple 
$(\mathcal F(T),\mathcal N(T),\mathcal Q(T))$. 
\end{theorem} 

\begin{proof}
For the proof, see the note N\,\ref{support-tilting}.
\end{proof} 
	\medskip 

\subsection{The Ingalls-Thomas bijections} We are going to present
a large number of sets which correspond bijectively to the 
set of exceptional antichains.
	\medskip 

\subsubsection{\bf Normal partial-tilting modules and support-tilting modules.} 
We show a bijection between normal partial tilting  modules 
and support-tilting modules. It is sufficient to look at the support of the
modules, thus to show a bijection between sincere normal modules
and tilting modules.
	\medskip 
 
One direction is given by normalization
$$
{\beginpicture
\setcoordinatesystem units <2cm,1cm> 
\put{\{sincere normal modules\}} at -1.2 0
\put{\{tilting modules\}} at 2.2 0
\setdashes <1mm>
\arr{0.1 0.1}{0.9 0.1}
\setsolid
\arr{0.9 -.1}{0.1 -.1}
\put{$\nu(-)$} at 0.5 -.4
\endpicture}
$$

The uniqueness of the normalization 
shows that the map $\nu$ going from right to left is well-defined. 

The map $\nu$ is injective when applied to tilting modules:  
if $T, T'$ are tilting modules with $\nu(T) =
\nu(T'),$ then $T$ and $T'$ are isomorphic. Namely,
$T'$ is generated by $\nu(T') = \nu(T)$, thus
by $T$. But $T'\in \mathcal G(T)$ means that $\Ext^1(T,T') = 0.$
Similarly, we see that $T$ belongs to $\mathcal G(T'),$ thus $\Ext^1(T',T) = 0.$
This shows that $T\oplus T'$ is without self-extensions, thus the number
of isomorphism classes of indecomposable direct summand of $T\oplus T'$
is equal to the number of simple modules, therefore $\add (T\oplus T') =
\add T = \add T'$. Since we assume that $T, T'$ are multiplicity-free,
$T$ and $T'$ are isomorphic.
	
In order to see that $\nu$ is
surjective, we need to find for any sincere normal 
partial tilting module $N$ a
tilting module $T$
with $\nu(T) = N,$  thus we need 
a module $N'$, such that $N'$ is generated by $N$, and 
second, $N\oplus N'$ is a tilting module. The existence (and
unicity) of $N'$ is well-known, see the note N\,\ref{Bongartz}.  We call $N'$ 
a {\it factor complement} for  $N$. 
	\medskip 

\subsubsection{} We need two further definitions:
The antichains $A = \{A_1,\dots,A_t\}$ and $A' = \{A'_1,\dots,A'_{t'}\}$ are said to
be {\it isomorphic,} provided the modules $\bigoplus_i A_i$ and $\bigoplus_j A'_j$ are isomorphic. 
	
If $\mathcal C$ is a subcategory of a module category 
and $C\in \mathcal C$, then $C$ is said to be a {\it cover} of $\mathcal C$ provided any
module in $\mathcal C$ is generated by $C$.
	\medskip 

\begin{theorem}\label{IT-bijections} {\bf (Ingalls-Thomas).} Let $\Lambda$ be a hereditary artin algebra. 
There are bijections between the following data:
\begin{enumerate}
\item[\rm(1)] Exceptional antichains.
\item[\rm(1$'$)] Exceptional subcategories.
\item[\rm(1$''$)] Normal partial tilting modules.
\item[\rm(1$'''$)] Conormal partial tilting modules.
\item[\rm(2)] Support-tilting modules.
\item[\rm(3)] Torsion classes with a cover.
\item[\rm(3$'$)] Torsionfree classes with a cocover.
\end{enumerate}  
 	\smallskip

  If $\Lambda$ is in addition representation-finite, then
all antichains are exceptional, all thick subcategories have a cover and a cocover,
all torsion classes have a cover. all torsionfree classes have a cocover.
\end{theorem}
		 \bigskip

Several of these bijections have been mentioned already. 
We have added here the torsion classes (3), since classical tilting theory asserts
the bijection between (2) and (3):
$$
{\beginpicture
\setcoordinatesystem units <2cm,1cm> 
\put{\{support tilting modules\}} at -1.2 0
\put{\{torsion classes with a cover\}} at 2.2 0
\arr{0.1 0.1}{0.9 0.1}
\arr{0.9 -.1}{0.1 -.1}
\put{$\mathcal G(-)$} at 0.5 .4
\put{$\Ext$-projectives} at 0.5 -.4
\endpicture}
$$
attaching to any support tilting module $T$ the torsion class $\mathcal G(T)$ 
of all the modules generated by $T$, and to any torsion class $\mathcal C$ with a cover
the direct sum of all indecomposable $\Ext^1$-projective modules, one from each
isomorphism class. 
	\bigskip

The impressive list of bijections provided by this theorem means of course,
that in case $\Lambda$ is representation finite, say of Dynkin type $\Delta_n$, all the counting problems 
(as discussed in Chapter 1) have the same answer, namely  
$\a(\Delta_n)$.
	\bigskip 

Here are some additional arguments for establishing the asserted bijections:
	\smallskip 
 
First, we may show the following: If $T$ is a support-tilting module and $\mathcal G = \mathcal G(T)$, then
$\add T$ is the class of the $\Ext$-projective modules in $\mathcal G$. 
Tilting theory asserts that $\mathcal G$ is the class of $\Lambda(T)$-modules $M$ such that $\Ext^1(T,M) = 0$.
Let $M$ be in $\mathcal G$ and $g\!:T' \to M$ be a right $T$-approximation of $M$. Then $g$ is
surjective and the kernel $M'$ of $g$ satisfies $\Ext^1(T,M') = 0$, thus belongs to $\mathcal G$. 
If $M$ is $\Ext$-projective, then the exact sequence $0 \to M' \to T' \to M \to 0$ splits, thus $M$ is in
$\add T.$ This shows that the $\Ext$-projective modules in $\mathcal G$ are just the modules in $\add T.$

From (2) to (3); If $T$ is a partial tilting module, let $\mathcal G(T)$ be the class of modules
generated by $T$.  
Then it is well-known (and easy to see) 
that $T$ is a torsion class. Of course, $T$ is a cover for $\mathcal G(T)$.

From (3) to (2): If $\mathcal C$ is a torsion class with a cover $C$, then we attach to it a module $T$
such that $\add T$ is the class of $\Ext$-projective modules in $\mathcal G$. In order to do so, we need to know
that the class $\mathcal E$ of $\Ext$-projective modules
in $\mathcal C$ is finite, say $\mathcal G = \add T$ for some module $T$. We also have to show that $T$
is support-tilting. 

With $C$ also its normalization $\nu(C)$ is a cover. A normal cover of a torsion class has no self-extension
(see Proposition 1 of \cite{[R8]}). 
Let $B$ be a factor complement for $\nu(C).$ As we have seen,
$T = \nu(C)\oplus B$ is a support-tilting module. Since $B$ is generated by $\nu(C)$, we have $\mathcal G(T) =
\mathcal G(\nu(C)) = \mathcal G(C) = \mathcal C.$ But we have shown already that $\add T$ is the class of 
$\Ext$-projective modules in $\mathcal G(T)$. 

From (2) to (3) to (2): Let us start with a support-tilting module $T$ and attach to it $\mathcal G = \mathcal G(T)$. 
As we have seen, the class of $\Ext$-projectives in $\mathcal G$ is $\add T$. We choose $T'$ with $\add T' = \add T$.
But this just means that $T, T'$ are Morita equivalent.

From (3) to (2) to (3). We start with a torsion class $\mathcal C$ with a cover, we choose a support-tilting module $T$
with $\mathcal C = \mathcal G(T),$  thus we are back at $\mathcal C$. 
       \medskip 

We have used duality, in order to add some further conditions.
	\bigskip

\subsubsection{\bf Remark} 
The bijections between the set (1$'$) of thick subcategories $\mathcal C$ 
and the sets (1), (1$''$) and (1$'''$) 
can be reformulated as follows:
In an abelian category we may look at the semi-simple, the
projective and the injective objects: the set of simple objects in $\mathcal C$ is
an antichain in $\mo \Lambda$, a minimal projective generator in $\mathcal C$
is a normal partial tilting module, a minimal injective cogenerator
is a conormal partial tilting module.  

Conversely, let us start with (1), (1$''$) or (1$'''$). It has been
mentioned already that starting with an antichain $A$, we take the full subcategory $\mathcal E(A)$.
Starting with a normal partial tilting module $P$, 
the corresponding thick subcategory $\mathcal C$ consists of all modules
which arise as the cokernel of a map in $\add P$ (in this way, we specify projective
presentations of the objects in $\mathcal C).$ Dually, starting with a conormal 
partial tilting module $I$, 
the corresponding thick subcategory $\mathcal C$ consists of all modules
which arise as the kernel of a map in $\add I$ (in this way, we specify injective
presentations of the objects in $\mathcal C).$ 
	\bigskip

\subsection{Perpendicular pairs and exceptional sequences}
	\medskip 

If $M$ is a module, the number of isomorphism classes of indecomposable
direct summands of $M$ will be called the {\it rank} of $M$ and denoted by $\rank M$.
Of course, the rank of the regular representation ${}_\Lambda\Lambda$ is just the
rank of the Grothendieck group $K_0(\Lambda)$. In general, if $\mathcal A$ is a length category,
we call the rank of the Grothendieck group $K_0(\mathcal A)$ the rank of $\mathcal A.$
Thus, if $\Lambda$ is an artin algebra, the rank of $\mo\Lambda$ is  
the rank of any projective generator and of any injective cogenerator, and this is
the number of simple modules.
	\medskip 

We assume in this section that $\Lambda$ is a hereditary artin algebra of rank $n$.
	\medskip

\subsubsection{} Let $\mathcal U, \mathcal V$ be classes of modules. Let
$\mathcal V^\perp$ be the full subcategory of all modules $U$ with $\Hom(\mathcal V,U) = 0$ and 
$\Ext^1(\mathcal V,U) = 0$.
Dually, let ${}^\perp\mathcal U$ be the 
full subcategory of all modules $V$ with $\Hom(V,\mathcal U) = 0$ and $\Ext^1(V,\mathcal U) = 0$.
It is easy to see that subcategories of the form $\mathcal V^\perp$  or ${}^\perp\mathcal U$ 
are thick subcategories. 

For example, if $S$ is a simple module and $P(S)$ is a projective cover of $S$, then
$P(S)^\perp = \mo  \Lambda/\langle e_S\rangle,$ where $e_S$ is a primitive idempotent of $\Lambda$
with $e_SS \neq 0$ and $\langle e_S\rangle$ is the two-sided ideal generated by $e_S$.
	\medskip

We call $(\mathcal U,\mathcal V)$ a {\it perpendicular pair} provided $\mathcal U = \mathcal V^\perp$
and $\mathcal V = {}^\perp \mathcal U$ (thus, provided $\Hom(V,U) = 0 = \Ext^1(V,U)$ for all $U\in 
\mathcal U$ and $V\in \mathcal V$, and such that the classes $\mathcal U$ and $\mathcal V$ are maximal with
this property); we should point out that writing first $\mathcal U$, then $\mathcal V$, corresponds
to our preference of writing first $\mathcal F$, then $\mathcal G$ when dealing with a torsion pair
$(\mathcal F,\mathcal G)$.

In these lecture, we only will be interested in perpendicular pairs
$(\mathcal U,\mathcal V)$ such that $\mathcal V$ is exceptional (as we will see, 
$\mathcal V$ is exceptional if and only if $\mathcal U$ is exceptional).
But we should note that also other perpendicular pairs play
a role in the representation theory of artin algebras, see the note N2.8.

\begin{lemma}\label{perp-gen}
Let $\mathcal V$ be an exceptional subcategory, let $N$ be a projective generator
of $\mathcal V$. Then $\mathcal V^\perp = N^\perp.$
\end{lemma} 

\begin{proof} We only have to show that 
$N^\perp \subseteq \mathcal V^\perp.$ Let $U$ belong to $N^\perp$. We have to show that
$\Hom(\mathcal V,U) = 0$ and $\Ext^1(\mathcal V,U) = 0$. 
Let $V\in \mathcal V$. Since $V$ is generated by
$N$, the condition $\Hom(N,U) = 0$ implies that $\Hom(V,U) = 0.$ Next, we show
that $\Ext^1(V,U) = 0.$ Since $N$ is a projective generator of $\mathcal V$,
there is an exact sequence $0 \to N' \to N'' \to V \to 0$ with $N',N''$ in $\add N.$
It follows that there is an exact sequence $\Hom(N',U) \to \Ext^1(V,U) \to \Ext^1(N'',U).$
Since $U$ belongs to $N^\perp$, it follows that $\Ext^1(V,U) = 0.$  
\end{proof}

\subsubsection{}
Let $N$ be a partial tilting module. 
Given any module $X$, let us define $\gamma_N(X)$
as follows: Choose an embedding $X \to X'$ such that $X'/X$ is isomorphic to a direct sum
of copies of $N$ and such that $\Ext^1(N,X') = 0.$ Such an embedding exists: We assume that 
$k$ is a commutative artinian ring and that
$\Lambda$ is a $k$-algebra which is module-finite as a $k$-module. We take a 
finite generating set $(\epsilon_i)_{i=1}^t$ 
of $\Ext^1(N,X)$ say as a $k$-module and form the ``universal extension''
$0 \to X \to X' \to N^t \to 0$ (so that the $t$ canonical embeddings $N \to N^t$
induce the extensions $\epsilon_1,\dots,\epsilon_t$). 

Let $\gamma_N(X) = X'/X'',$
where $X''$ is the $N$-trace of $X'$ (this is the sum of all images of maps $N \to X'$). 
Note that this construction is functorial: any map $f\!: X_1 \to X_2$ can be lifted to a map
$f'\;X'_1 \to X'_2$, but this lifting is not necessarily unique. Of course, $f'$ maps
$X''_1$ into $X''_2$, thus it induces a map $\gamma_N(X_1) \to \gamma_N(X_2)$ and this map
is uniquely determined by $f$. Namely, if $f'\!:X'_1 \to X'_2$ is a map with zero restriction
$f'|X_1$, then $f'$ factors through $X'_1/X_1$. Since $X'_1/X_1$ is a direct sum of copies of 
$N$, we see that the image of $f'$ is contained in $X''_2$, thus the induced map
$\gamma_N(X_1) \to \gamma_N(X_2)$ is the zero map. 
	\medskip

\begin{lemma}\label{perp-gen2}
Let $N$ be a partial tilting module. Then the functor 
$\gamma_N$ maps $\mo\Lambda$ onto $N^\perp$ and its restriction to $N^\perp$ is the identity
functor. The module $\gamma_N({}_\Lambda \Lambda)$ is a projective generator for $N^\perp.$
\end{lemma}

\begin{proof} 
Let $X$ be any $\Lambda$-module. The epimorphism $X' \to X'/X''$ induces
a surjection $\Ext^1(N,X') \to \Ext^1(N,X'/X'')$. Since
$\Ext^1(N,X') = 0$, we see that $\Ext^1(N,X'/X'') = 0$. 
Since $N$ is a partial tilting module, it is easy to see that for any module $X'$ with
$N$-trace $X''$, the $N$-trace of $X'/X''$ is zero. This shows that $\Hom(N,X'/X'') = 0$. 
Altogether, we see that $\gamma_N(X) = X'/X''$ belongs to $N^\perp.$ 

If $X$ belongs to $N^\perp$, then $\Ext^1(N,X)= 0$ shows that we can take $X' = X$ with 
the identity map $X \to X$ as the required embedding. Since $\Hom(N,X) = 0$, the $N$-trace
$X''$ of $X' = X$ is zero, thus $\gamma_N(X) = X'/X'' = X/0 = X.$

Let $M = \gamma_N({}_\Lambda \Lambda)$ and $X\in N^\perp$. There is an epimorphism
$f\!:F \to X$, where $F$ is a free $\Lambda$-module, say $F = ({}_\Lambda \Lambda)^s$ for some
$s\ge 0$. The map $f$ can be extended to a map
$f'\!:F' \to X'=X$. Since the restriction $f$ of $f'$ is surjective, also $f'$ is surjective.
$f'$ maps $F''$ to $X'' = 0$, thus $f'$ induces a map $F'/F'' \to X'/X'' = X$ which again
has to be surjective. But this is the map $\gamma_N(f)$, and $F'/F'' = M^s$.
This shows that there is a surjective map $M^s \to X$ for some $s\ge 0$, therefore $M$
is a cogenerator for $N^\perp$. 

It remains to be seen that $M$ is projective in $N^\perp$. Let $X$ be a module in $N^\perp$.
There is an exact sequence
$0 \to \Lambda \to \Lambda' \to N^t \to 0$. Since $\Ext^1(\Lambda,X) = 0$ and $\Ext^1(N,X) = 0$,
we see that $\Ext^1(\Lambda',X) = 0$. The exact sequence $0 \to \Lambda'' \to \Lambda' \to M \to 0$
yields an exact sequence $\Hom(\Lambda'',X) \to \Ext^1(M,X) \to \Ext^1(\Lambda',X)$.
We know already that the last term $\Ext^1(\Lambda',X) = 0.$ But we also have
$\Hom(\Lambda'',X) = 0$, since $\Lambda''$ is generated by $N$ and $\Hom(N,X) = 0.$
These two assertions imply that $\Ext^1_\Lambda(M,X) = 0.$ 
Since $N^\perp$ is an exact abelian subcategory of $\mo\Lambda$, the vanishing 
$\Ext^1_\Lambda(M,X) = 0$ for all $X$ in $N^\perp$ shows that $M$ is projective in $N^\perp$. 
\end{proof} 

\begin{theorem}\label{perp-perp}
 Let $(\mathcal U,\mathcal V)$ be a perpendicular pair. Then $\mathcal U$ is an exceptional
subcategory if and only if $\mathcal V$ is an exceptional subcategory.
\end{theorem} 

\begin{proof} Assume that $\mathcal V$ is an exceptional subcategory. 
Let $N$ be a projective generator of $\mathcal V$.
As we have seen in Lemma \ref{perp-gen}
we have $\mathcal V^\perp = N^\perp.$ Lemma \ref{perp-gen2} asserts that
$\gamma_N({}_\Lambda \Lambda)$ is a projective generator for $N^\perp = \mathcal V^\perp.$
This shows that $\mathcal U = \mathcal V^\perp$ is also an exceptional subcategory.

The reverse implication follows by duality.
\end{proof} 

\subsubsection{\bf Exceptional sequences.}
Recall that a sequence
$(E_1,\dots,E_t)$ of indecomposable $\Lambda$-modules is said to be {\it exceptional}
provided $\Ext^1(E_i,E_j) = 0$ for $i \ge j$ and $\Hom(E_i,E_j) = 0$ for $i > j$
(in case $t = 2$, one calls it an {\it exceptional pair}). 

\begin{prop}  Let $(E_1,\dots,E_t)$ be an exceptional sequence in $\mo\Lambda$ with thick
closure $\mathcal E$. Then  $\mathcal E$ is an exceptional subcategory of rank $t$.
\end{prop} 
	
Let us stress that for arbitrary exceptional modules 
$E_1,E_2$, the thick closure of $E_1,E_2$ may have 
arbitrarily large rank, see the note N\,\ref{large-rank}.

For the proof of the proposition, we refer to \cite{[R4]}. Actually, the main arguments will be outlined
in Section 3.5, where we discuss the braid group action on the set of exceptional sequences
of fixed length. This braid group action replaces successively an exceptional pair $(E,E')$ by
an exceptional pair $(E',E'')$ or $(E'',E)$ with $E''$ being contained in the thick closure
of $E,E'$. Using the braid group action, one obtains from
the given exceptional sequence $(E_1,\dots,E_t)$ an exceptional sequence $(A_1,\dots,A_t)$
with the same thick closure, such that the modules $A_1,\dots,A_t$ are pairwise orthogonal
(thus, they form an exceptional antichain). Of course, the thick closure of an 
exceptional antichain of cardinality $t$ is an exceptional subcategory of rank $t$.
	
\begin{cor}\label{thick-closure}
Let $(E_1,\dots,E_t)$ be an exceptional sequence in $\mo\Lambda$ with thick
closure $\mathcal E$. Then $t\le n$. If $t = n$, then $\mathcal E = \mo\Lambda$.
\end{cor} 
	
\begin{proof}\label{thick-closure2} Let $N$ be a projective generator of $\mathcal E$. Then $N$ is a partial tilting module of
rank $t$. This shows that $t\le n$. Now assume that $t = n$. Then $N$ is a tilting module.
But the thick closure of a tilting module is $\mo \Lambda$. Namely, there is an exact
sequence of the form $0\to {}_\Lambda\Lambda \to N' \to N'' \to 0$ with $N',N'' \in \add(T)$. This shows that
${}_\Lambda\Lambda$ belongs to the thick closure of $N$. Of course, any module belongs to the thick closure
of ${}_\Lambda\Lambda$.
\end{proof} 

\begin{cor} Let $\mathcal C' \subseteq \mathcal C$ be exceptional subcategories of
$\mo\Lambda$. Let $r$ be the rank of $\mathcal C$ and $r'$ the rank of $\mathcal C'$. Then $r'\le r$
and if $r' = r$, then $\mathcal C' = \mathcal C.$
\end{cor}
	
\begin{proof} 
Since $\mathcal C$ is an exceptional subcategory, it is equivalent to the module category
of a hereditary artin algebra of rank $r$. The assertions follow directly from Corollary
\ref{thick-closure}.
\end{proof} 

\begin{theorem}\label{perp-rank}
 If $\mathcal V$ is an exceptional subcategory of $\mo\Lambda$ and has rank $t$, then
$\mathcal V^\perp$ has rank $n-t$.
\end{theorem}

\begin{proof} As we have seen in Theorem \ref{perp-perp},
the category $\mathcal U = \mathcal V^\perp$ is again
an exceptional subcategory, thus it has finite rank, say rank $s$. 
Take an exceptional sequence  $(M_1,\dots,M_s)$  in $\mathcal C^\perp$ and an
exceptional sequence  $(N_1,\dots,N_t)$  in $\mathcal C.$ 
Then $(M_1,\dots,M_s,N_1,\dots,N_t)$ is an exceptional sequence in $\mo\Lambda$, thus
$s+t \le n$ by Corollary \ref{thick-closure}.
As a consequence, we only have to show that $s \ge n-t$.

We use induction on $n$. 
If $n = 1$, then the only exceptional subcategories are $0$ and $\mo\Lambda$,
and the assertion is clear.

Thus, let us assume that $n\ge 2$. We consider first the case $t = 1,$ thus $\mathcal C = \add N$
for some exceptional module $N$. There are two different cases: $N$ may be projective or not.
If $N = \Lambda e$ is projective, with $e$ a primitive idempotent of $\Lambda$, then 
$N^\perp = \mo \Lambda/\langle e\rangle.$ Of course, $\mo \Lambda/\langle e\rangle$
has rank $n-1$, as required.

Second, assume that $N$ is not projective. Let $\beta(N)$ be the Bongartz complement for $N$,
see N\,\ref{Bongartz}. Then $\beta(N)$ belongs to $N^\perp$ and has rank $n-1$. 
This shows that $N^\perp$ has rank at least $n-1$. 
	\medskip 

Now, let $t \ge 2,$ and take an exceptional sequence $(N_1,\dots,N_t)$  in $\mathcal C.$ 
We look at the exceptional subcategory $N_t^\perp$. As we have seen, $N_t^\perp$ has rank $n-1$. 
The modules 
$N_1,\dots,N_{t-1}$ belong to $N_t^\perp$, also the subcategory $\mathcal C^\perp$ belongs to $N_t^\perp$.
We denote by $\mathcal C'$
the thick closure of $N_1,\dots,N_{t-1}$. The subcategory $\mathcal C'$ has rank $t-1$
and its perpendicular subcategory inside $N_t^\perp$ is just $\mathcal C^\perp.$
By induction, we know that the rank of $\mathcal C^\perp$ is $(n-1)-(t-1) = n-t$.
\end{proof}

\begin{cor} Let $(E_1,\dots,E_t)$ be an exceptional sequence in $\mo\Lambda$.
Then there is an exceptional sequence $(M_1,\dots,M_s,E_1,\dots,E_t)$ with $s = n-t.$
\end{cor} 

\begin{proof} Let $\mathcal E$ be the thick closure of $E_1,\dots,E_t$ and 
take for $(M_1,\dots,M_s)$ any exceptional sequence in $\mathcal E^\perp$ of cardinality $s$.
Such a sequence exists, since $\mathcal E$ is exceptional and has rank $s$.
\end{proof}

An exceptional sequence $(E_1,\dots,E_t)$ will be said to be {\it complete} provided
$t = \rank(\Lambda)$. The corollary asserts that any exceptional sequence can be completed.
	\medskip 

\subsection{Torsion pairs and perpendicular pairs}
We assume again that $\Lambda$ is a hereditary artin algebra of rank $n$.
We are going to look at torsion pairs $(\mathcal F, \mathcal G)$ such that $\mathcal G$ has a cover
and at perpendicular pairs $(\mathcal U, \mathcal V)$ such that $\mathcal V$ is an exceptional
subcategory. We will see that there is a one-to-one correspondence between these
torsion pairs and these perpendicular pairs.
	\medskip 

\subsubsection{}
We recall the following. Let $(\mathcal U, \mathcal V)$ be a perpendicular pair. According to 
Theorem \ref{perp-perp},
  $\mathcal V$ is an exceptional subcategory if and only if $\mathcal U$ is an
exceptional subcategory. If this holds, then Theorem \ref{perp-rank}
asserts that $\rank \mathcal U + \rank\mathcal V = n$. 
Let us now deal with torsion pairs. 
	\medskip

\begin{lemma}\label{cogen}
Let $N$ be a normal partial tilting module. 
Then the modules $F$ with $\Hom(N,F) = 0$ are the modules cogenerated by $N^\perp.$
\end{lemma}

\begin{proof}
First, assume that $F$ is cogenerated by $N^\perp,$ thus, there is an embedding
$u\!:F \to X$ with $X\in N^\perp$. If $f\!:N \to F$, then $uf\!:N \to X$ is zero, thus also $f = 0.$

Conversely, let $\Delta(N) = \{\Delta(1),\dots,\Delta(t)\}$ 
be the antichain corresponding to $N$, thus this is an
antichain consisting of modules which are cokernels of maps in $\add N$ such that $N$ belongs
to $\mathcal E(\Delta(N))$. As we know, $\Delta(N)$ is an exceptional antichain, thus we can assume that
$\Ext^1(\Delta(j),\Delta(i)) = 0$ for $j\le i$ (note that often one deals with the opposite
ordering, but here we deviate from the usual ordering).
Since any $\Delta(i)$ is generated by $N$,
any module $F'$ with $\Hom(N,F') = 0$ satisfies $\Hom(\Delta(i),F') = 0$
for all $i.$ 

Let $F_0 = F$. 
Using induction, we construct for $1\le i \le t$ exact sequences
$$
\hbox{\beginpicture
  \setcoordinatesystem units <1cm,.7cm>
\put{$0 \to F_{i-1} \to F_i \overset{p_i}\longrightarrow \Delta(i)^{m_i} \to 0$} [l] at 0 0
\put{$(*)$} [r] at -3 0
\put{} at 9 0  
\endpicture}
$$
such that $\Hom(N,F_i) = 0$ and $\Ext^1(\Delta(j),F_i) = 0$ for $1\le j \le i$
and such that $m_i$ is chosen minimal. 
Note that the conditions $\Hom(N,F_i) = 0$ and $\Ext^1(\Delta(j),F_i) = 0$ for $1\le j \le i$,
are satisfied for $i = 0.$ Assume that the module $F_{i-1}$ has
been constructed already. We take for $(*)$ the universal extension of $F_{i-1}$ from
above, using copies of $\Delta(i)$. Thus, by construction, $\Ext^1(\Delta(i),F_i) = 0.$
Also, $\Hom(\Delta(i),F_i) = 0.$ Namely, a non-zero homomorphism $f\!:\Delta(i) \to F_i$
cannot map into $F_{i-1}$, since $\Hom(\Delta(i),F_{i-1}) = 0$, therefore $p_if \neq 0.$
However, $\Delta(i)$ is a brick, thus $p_if$ has to be a split monomorphism. But this
contradicts the minimality of $m_i$. 

We also have $\Hom(\Delta(j),F_i) = 0$ for all $j\neq i$, since $\Hom(\Delta(j),\Delta(i)) = 0$
and, by induction $\Hom(\Delta(j),F_{i-1}) = 0$. 
Finally, we assert that $\Ext^1(\Delta(j),F_i) = 0$ for all $j <i$. This follows directly from
the induction hypothesis that $\Ext^1(\Delta(j),F_{i-1}) = 0$ for all $j < i$
and the fact that $\Ext^1(\Delta(j),\Delta(i)) = 0$ for $j < i.$ This completes the inductive
construction.

We have obtained in this way an embedding $F \to F_t$ and $F_t$ satisfies both
$\Hom(N,F_i) = 0$ and $\Ext^1(\Delta(j),F_t) = 0$ for $1\le j \le t$, therefore also
$\Ext^1(N,F_t) = 0$. This shows that $F_t$ belongs to $N^\perp.$ 
\end{proof} 

As a consequence, there is the following result is due to Smal\o{} \cite{[Sm]}.
	
\begin{theorem} {\bf (Smal\o{})} 
Let $(\mathcal F,\mathcal G)$ be a torsion pair. Then $\mathcal G$ has
a cover if and only if $\mathcal F$ has a cocover. If $N$ is a minimal cover of $\mathcal G$ and
$M$ is a minimal cover of $\mathcal F$, then $\rank M + \rank N = n$.
\end{theorem}

\begin{proof}
Let $N$ be a minimal generator of $\mathcal G$ and $\mathcal N$ the thick closure of $N$.
Then $N$ is a normal
partial tilting module. A module $F$ belongs to $\mathcal F$ if and only if 
$\Hom(N,F) = 0$, thus, according to Lemma \ref{cogen}
if and only if $F$ is cogenerated by $N^\perp.$
Since $\mathcal M = N^\perp$ is an exceptional subcategory, it has a minimal cogenerator, say $M$.
Then $M \in N^\perp \subseteq \mathcal F$ and any module $F\in \mathcal F$ is cogenerated by $M$.
This shows that $M$ is a minimal cocover of $\mathcal F$. 

We have $\rank N = \rank \mathcal N$ and 
$\rank M = \rank N^\perp$. According to Theorem \ref{perp-rank}, we have
$$
   \rank N + \rank M = \rank \mathcal N + \rank \mathcal M = n.
$$
\end{proof}

\subsubsection{}
Let us summarize the previous considerations. 

\begin{theorem}\label{tors-perp}
Let $(\mathcal F, \mathcal G)$ be a torsion pair. 
Assume that $N$ is a minimal cover of $\mathcal G$ and $M$ a minimal cocover of $\mathcal F$. 
If $\mathcal N$ is the thick closure of $N$ and $\mathcal M$ the thick closure of $M$, then 
$(\mathcal M,\mathcal N)$ is a perpendicular pair and the subcategories $\mathcal N$ and $\mathcal M$ are
exceptional.
	\smallskip 

Conversely, assume that $(\mathcal U,\mathcal V)$ is a perpendicular pair such that $\mathcal V$
is exceptional. Let $\mathcal G$ be the modules
generated by modules in $\mathcal V$, let
$\mathcal F$ be the modules cogenerated by modules in $\mathcal U$. Then $(\mathcal F,\mathcal G)$ is a torsion pair
such that $\mathcal G$ has a cover and $\mathcal F$ a cocover. 
	\smallskip 

We obtain in this way a bijection between the torsion pairs
$(\mathcal F,\mathcal G)$ such that $\mathcal G$ has a cover, and the perpendicular pairs 
$(\mathcal U,\mathcal V)$ such that $\mathcal V$ is exceptional. 
\end{theorem}
	\bigskip

Let us assume that $(\mathcal F, \mathcal G)$ is a torsion pair and that $N$ is a minimal cover of $\mathcal G$.
Then $N$ is a normal partial tilting module.
Let $N\oplus N'$ be a support tilting module with $N'$ generated by $N$.
We denote by $\mathcal Q$ the full subcategory of all modules generated by $N'$.
Dually, let $M$ be a minimal cocover of $\mathcal F$, thus $M$ is a conormal partial tilting module.
Let $M\oplus M'$ be a support tilting module with $M'$ cogenerated by $M$ and let
$\mathcal P$ be the full subcategory of all modules cogenerated by $M'.$
The four subcategories $\mathcal P, \mathcal M, \mathcal N, \mathcal Q$ should be seen as being 
arranged as follows:
$$
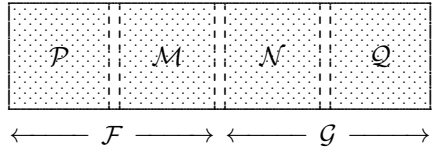
\hbox{\beginpicture
  \setcoordinatesystem units <.7cm,.7cm>
\multiput{} at 0 0  8 2 /
\put{$\mathcal P$} at .95 1
\put{$\mathcal M$} at 3 1
\put{$\mathcal N$} at 5 1
\put{$\mathcal Q$} at 7.05 1
\plot 0 0  8 0  8 2  0 2  0 0 /
\setdashes <.95mm>
\plot 1.9 0  1.9 2 /
\plot 2.1 0  2.1 2 /
\plot 3.9 0  3.9 2 /
\plot 4.1 0  4.1 2 /
\plot 5.9 0  5.9 2 /
\plot 6.1 0  6.1 2 /
\setshadegrid span <.7mm>
\vshade 0 0 2  1.9 0 2 /
\vshade 2.1 0 2  3.9 0 2 /
\vshade 4.1 0 2  5.9 0 2 /
\vshade 6.1 0 2  8 0 2 /

\betweenarrows {$\mathcal F$} from 0 -0.5  to 3.9 -0.5 
\betweenarrows {$\mathcal G$} from 4.1 -0.5  to 8 -0.5 
\endpicture}
$$
with no maps backwards, and such that any module $Z$ has a filtration
$$
 Z = Z_0 \supseteq Z_1 \supseteq Z_2 \supseteq Z_3 \supseteq Z_4  = 0
$$ 
such that
$$
  Z_0/Z_1 \in \mathcal P,\quad 
  Z_1/Z_2 \in \mathcal M,\quad 
  Z_2/Z_3 \in \mathcal N,\quad 
  Z_3/Z_4 \in \mathcal Q,
$$
(thus we deal with what one may call a {\it torsion quadruple} $(\mathcal P,\mathcal M,\mathcal N,\mathcal Q)$; it
refines the given torsion pair $(\mathcal F,\mathcal G)$, since 
$\mathcal F = \mathcal P\above \mathcal M$ and $\mathcal G = \mathcal N\above Q$). 

The bijection between torsion pairs $(\mathcal F,\mathcal G)$ such that $\mathcal G$ has a cover and
perpendicular pairs $(\mathcal M,\mathcal N)$ of exceptional subcategories may be remembered
as follows: 
 $$
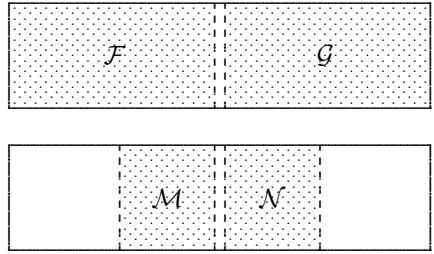
\hbox{\beginpicture
  \setcoordinatesystem units <.7cm,.7cm>
\put{\beginpicture
\multiput{} at 0 0  8 2 /
\put{$\mathcal F$} at 2 1
\put{$\mathcal G$} at 6 1
\plot 0 0  8 0  8 2  0 2  0 0 /
\setdashes <.95mm>
\plot 3.9 0  3.9 2 /
\plot 4.1 0  4.1 2 /
\setshadegrid span <.7mm>
\vshade 0 0 2  3.9 0 2 /
\vshade 4.1 0 2  8  0 2 /
\endpicture} at 0 0
\put{\beginpicture
\multiput{} at 0 0  8 2 /
\put{$\mathcal M$} at 3 1
\put{$\mathcal N$} at 5 1
\plot 0 0  8 0  8 2  0 2  0 0 /
\setdashes <.95mm>
\plot 2.1 0  2.1 2 /
\plot 3.9 0  3.9 2 /
\plot 4.1 0  4.1 2 /
\plot 5.9 0  5.9 2 /
\setshadegrid span <.7mm>
\vshade 2.1 0 2  3.9 0 2 /
\vshade 4.1 0 2  5.9 0 2 /
\endpicture} at 0 -2.7
\endpicture}
$$
There is given a normal partial tilting module $N$ such that $\mathcal G$ are the modules
generated by $N$, whereas $\mathcal N$ is the thick closure of $N$. And, 
there is given a conormal partial tilting module $M$ such that $\mathcal F$ are the modules
generated by $M$, whereas $\mathcal M$ is the thick closure of $M$. 
It remains to mention that the bijectivity assertion is, of course,
part of the Ingalls-Thomas Theorem.
For the relationship between a minimal cover $N$ for $\mathcal G$ and a minimal
cocover $M$ for $\mathcal F$, see the note N\,\ref{relationship}.
	\medskip 

\subsection{Partial orderings on the set of antichains}
For any  hereditary artin algebra $\Lambda$, one may consider 
the set of antichains in $\mo \Lambda$ and its subset of
exceptional antichains. There are several partial orderings on these sets which should not
be confused. 
	\smallskip

\begin{itemize}
\item {$\le_s\ $} the set-inclusion of antichains,
\item {$\le{\phantom{{}_s}}$} the set-inclusion of the corresponding thick subcategories, thus
 $A\le A'$ means that every element of $A$ has a filtration by objects in 
 in $A',$ thus every element of $A$ belongs to $\mathcal E(A')$.
\item {$\le_t\ $} the set-inclusion of the corresponding torsion classes, 
 thus
 $A\le_t A'$ means that every element of $A$ is generated by a module in $\mathcal E(A'),$
\item {$\le_f\, $} the set-inclusion of the corresponding torsionfree classes,
  thus
 $A\le_f A'$ means that every element of $A$  is cogenerated by a module in 
 in $\mathcal E(A').$
\end{itemize}
	
Already the case $\mathbb A_2$ shows that all these partial orderings may be different.
$$
 
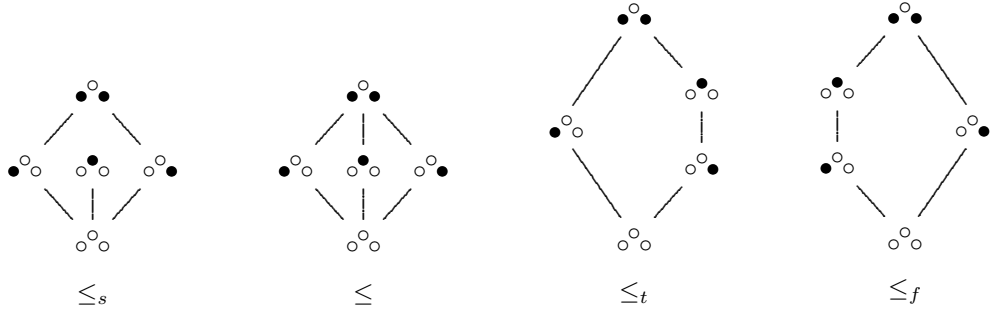
\hbox{\beginpicture
 \setcoordinatesystem units <.9cm,1cm>
\put
 {\beginpicture
\put{} at 0 3.2
\put
 {\beginpicture
 \setcoordinatesystem units <.15cm,.15cm>
  \multiput{$\circ$} at 0 0  1 1  2 0 /
  \multiput{$\bullet$} at 0 0  2 0 / 
 \endpicture} at  0 2 
\put
 {\beginpicture
 \setcoordinatesystem units <.15cm,.15cm>
  \multiput{$\circ$} at 0 0  1 1  2 0 /
  \multiput{$\bullet$} at 0 0  / 
 \endpicture} at  -1 1 
\put
 {\beginpicture
 \setcoordinatesystem units <.15cm,.15cm>
  \multiput{$\circ$} at 0 0  1 1  2 0 /
  \multiput{$\bullet$} at 1 1 / 
 \endpicture} at  0 1
\put
 {\beginpicture
 \setcoordinatesystem units <.15cm,.15cm>
  \multiput{$\circ$} at 0 0  1 1  2 0 /
  \multiput{$\bullet$} at  2 0 / 
 \endpicture} at  1 1 
\put
 {\beginpicture
 \setcoordinatesystem units <.15cm,.15cm>
  \multiput{$\circ$} at 0 0  1 1  2 0 /
  \multiput{$\bullet$} at / 
 \endpicture} at  0 0 
\plot -.3 0.3  -.7 .7 /
\plot 0 .3  0 .7 /
\plot .3 0.3  .7 .7 /
\plot -.7 1.3  -.3 1.7 /
\plot  .7 1.3   .3 1.7 /
\put{$\le_s$} at 0 -.7
\endpicture} at 0 0

\put
 {\beginpicture
\put{} at 0 3.2
\put
 {\beginpicture
 \setcoordinatesystem units <.15cm,.15cm>
  \multiput{$\circ$} at 0 0  1 1  2 0 /
  \multiput{$\bullet$} at 0 0  2 0 / 
 \endpicture} at  0 2 
\put
 {\beginpicture
 \setcoordinatesystem units <.15cm,.15cm>
  \multiput{$\circ$} at 0 0  1 1  2 0 /
  \multiput{$\bullet$} at 0 0  / 
 \endpicture} at  -1 1 
\put
 {\beginpicture
 \setcoordinatesystem units <.15cm,.15cm>
  \multiput{$\circ$} at 0 0  1 1  2 0 /
  \multiput{$\bullet$} at 1 1 / 
 \endpicture} at  0 1
\put
 {\beginpicture
 \setcoordinatesystem units <.15cm,.15cm>
  \multiput{$\circ$} at 0 0  1 1  2 0 /
  \multiput{$\bullet$} at  2 0 / 
 \endpicture} at  1 1 
\put
 {\beginpicture
 \setcoordinatesystem units <.15cm,.15cm>
  \multiput{$\circ$} at 0 0  1 1  2 0 /
  \multiput{$\bullet$} at / 
 \endpicture} at  0 0 
\plot -.3 0.3  -.7 .7 /
\plot 0 .3  0 .7 /
\plot 0 1.3  0 1.7 /
\plot .3 0.3  .7 .7 /
\plot -.7 1.3  -.3 1.7 /
\plot  .7 1.3   .3 1.7 /
\put{$\le$} at 0 -.7
\endpicture} at 4 0

\put
 {\beginpicture

\put
 {\beginpicture
 \setcoordinatesystem units <.15cm,.15cm>
  \multiput{$\circ$} at 0 0  1 1  2 0 /
  \multiput{$\bullet$} at 0 0  2 0 / 
 \endpicture} at  0 3 
\put
 {\beginpicture
 \setcoordinatesystem units <.15cm,.15cm>
  \multiput{$\circ$} at 0 0  1 1  2 0 /
  \multiput{$\bullet$} at 0 0  / 
 \endpicture} at  -1 1.5 
\put
 {\beginpicture
 \setcoordinatesystem units <.15cm,.15cm>
  \multiput{$\circ$} at 0 0  1 1  2 0 /
  \multiput{$\bullet$} at 1 1 / 
 \endpicture} at  1 2
\put
 {\beginpicture
 \setcoordinatesystem units <.15cm,.15cm>
  \multiput{$\circ$} at 0 0  1 1  2 0 /
  \multiput{$\bullet$} at  2 0 / 
 \endpicture} at  1 1 
\put
 {\beginpicture
 \setcoordinatesystem units <.15cm,.15cm>
  \multiput{$\circ$} at 0 0  1 1  2 0 /
  \multiput{$\bullet$} at / 
 \endpicture} at  0 0 
\plot -.2 0.3  -.9 1.2 /
\plot .3 0.3  .7 .7 /
\plot -.9 1.8  -.2 2.7 /
\plot 1 1.3  1 1.7 /
\plot 0.7 2.3  0.3 2.7 /
\put{$\le_t$} at 0 -.7
\endpicture} at 8 0

\put
 {\beginpicture

\put
 {\beginpicture
 \setcoordinatesystem units <.15cm,.15cm>
  \multiput{$\circ$} at 0 0  1 1  2 0 /
  \multiput{$\bullet$} at 0 0  2 0 / 
 \endpicture} at  0 3 
\put
 {\beginpicture
 \setcoordinatesystem units <.15cm,.15cm>
  \multiput{$\circ$} at 0 0  1 1  2 0 /
  \multiput{$\bullet$} at 0 0  / 
 \endpicture} at  -1 1 
\put
 {\beginpicture
 \setcoordinatesystem units <.15cm,.15cm>
  \multiput{$\circ$} at 0 0  1 1  2 0 /
  \multiput{$\bullet$} at 1 1 / 
 \endpicture} at  -1 2
\put
 {\beginpicture
 \setcoordinatesystem units <.15cm,.15cm>
  \multiput{$\circ$} at 0 0  1 1  2 0 /
  \multiput{$\bullet$} at  2 0 / 
 \endpicture} at  1 1.5 
\put
 {\beginpicture
 \setcoordinatesystem units <.15cm,.15cm>
  \multiput{$\circ$} at 0 0  1 1  2 0 /
  \multiput{$\bullet$} at / 
 \endpicture} at  0 0 
\plot .2 0.3  .9 1.2 /
\plot -.3 0.3  -.7 .7 /
\plot .9 1.8  .2 2.7 /
\plot -1 1.3  -1 1.7 /
\plot -.7 2.3  -.3 2.7 /
\put{$\le_f$} at 0 -.7
\endpicture} at 12 0

\endpicture}
$$
  
We see in this case: the partial ordering $\le_s$ does not yield a lattice,
the three other partial orderings yield lattices 
(the definition of a lattice will be recalled in N\,\ref{lattice}, 
And this is a general fact
in case $\Lambda$ is representation-finite:

The partial 
ordering $\le$ yields a lattice, namely the lattice of thick subcategories.
The partial 
ordering $\le_t$ yields a lattice, namely the lattice of torsion subcategories.
The partial 
ordering $\le_f$ yields a lattice, namely the lattice of torsionfree subcategories.
	\medskip

The partial ordering $\le$ seems to be the most import one. This is the 
structure on $A(\mo \Lambda)$ which we will discuss in the next chapter. (But let us add
that also
the lattice of torsion subcategories, as well as dually, the lattice of
torsionfree subcategories are relevant, see the note N\,\ref{torsion}.)
	\bigskip\bigskip 

{\bf Notes.}
	\medskip 

\begin{note}\label{linearity}
{\bf Proof of Lemma \ref{lin}.}
Assume that $P$ belongs to $\mathcal F = 
\mathcal F(\alpha)$, that $R, R'$ belong to $\mathcal N = \mathcal N(\alpha)$ and 
that $Q$ belongs to $\mathcal Q = \mathcal Q(\alpha)$. 
	
Let $f\!:R \to P$ be a non-zero homomorphism. The image $X$
is a non-zero submodule of $P$, thus $\alpha(X) < 0$ and a factor module of $R$, thus 
$\alpha(X) \ge 0$, a contradiction. 
	
Let $f\!:Q \to P$ be a non-zero homomorphism. The image $X$
is a non-zero submodule of $P$, thus $\alpha(X) < 0$ and a non-zero factor module of $Q$, thus 
$\alpha(X) \ge 0$, a contradiction. 
	
Let $f\!:Q \to R$ be a non-zero homomorphism. The image $X$
is a submodule of $R$, thus $\alpha(X) \le 0$ and a non-zero factor module of $Q$, thus 
$\alpha(X) \ge 0$, a contradiction. 
	
Let $f\!:R \to R'$ be a homomorphism. The image $X$ of $f$ is a submodule of $R'$, thus $\alpha(X) \le 0$,
and a factor module of $R$, thus $\alpha(X) \ge 0$. This shows that $\alpha(X) = 0.$
Any submodule $X'$ of $X$ is a submodule of $R'$, thus $\alpha(X') \le 0.$ Thus
$X$ belongs to $\mathcal N.$ Let us show that $\mathcal N$ is closed under kernels.
Since $\mathcal N$ is closed under images,  we may assume that $f\!:R \to R'$ is an epimorphism and
consider its kernel $Y$. We have $\alpha(Y) = \alpha(R) - \alpha(R') = 0.$ Also, a submodule $Y'$ of
$Y$ is a submodule of $R$, thus $\alpha(Y') \le 0$. This shows that $Y$ belongs to $\mathcal N$.
In order to show that $\mathcal N$ is closed under cokernels, it is sufficient to look at the
cokernel $Z$ of a monomorphisms $f\!:R \to R'$. We have $\alpha(Z) = \alpha(R) - \alpha(R') = 0.$
Also, any factor module $Z'$ of $Z$ is a factor module of $R'$, thus $\alpha(Z') \ge 0$. This shows
that $Z$ belongs to $\mathcal N.$ 
	
In order to show that $\mathcal N$ is closed under 
extensions, let $N$ be a module with a submodule $N'$ such that both $N'$ and $N/N'$ belong to 
$\mathcal N$. We have $\alpha(N) = \alpha(N') + \alpha(N/N') = 0.$ Let $U$ be a submodule of
$N$ and $U' = U\cap N'$. Then $U'$ is a submodule of $N'$, thus $\alpha(U') \le 0.$ 
Since $U/U' = U/(U\cap N') \simeq (U+N')/N'$ is isomorphic to a submodule of $N/N'$, we 
$N$ belongs to $\mathcal N.$
	
It remains to be shown that every module $M$ has a filtration 
$M'' \subseteq M' \subseteq M$ with $M''\in \mathcal Q,\ 
M'/M''\in \mathcal N,\  M/M'\in \mathcal F.$ As a preparation, we show that 
$\mathcal F$ and $\mathcal Q$ are closed under extensions.
	
Let $N$ be a module with a submodule $N'$ such that both $N'$ and $N/N'$ belong to 
$\mathcal F$. We have to look at non-zero submodules $U$ of $N$.
Consider the submodule $U\cap N'$ of $U$ with factor module
$U/(U\cap N') \simeq (U+N')/N'$, thus, we have $\alpha(U) = \alpha(U\cap N') + \alpha((U+N')/N').$
Both summands are non-positive, since we evaluate $\alpha$ at submodules of modules 
in $\mathcal F$. Since $U$ is non-zero, at least one of the modules 
$U\cap N',\ (U+N')/N'$ is non-zero and therefore
at least one of the values $\alpha(U\cap N'),\ \alpha((U+N')/N')$ is negative,
thus $\alpha(U) < 0.$
	
Similarly, take a module $N$ with a submodule $N'$ such that both $N'$ and $N/N'$ belong to 
$\mathcal Q$. To show that $N$ is in $\mathcal Q$, we have 
to look at non-zero factor modules of $N$. Thus, let $U$ be a proper submodule of $N$.
The module $N/U$ has the submodule $(U+N')/U \simeq N'/(U\cap N')$,
with factor module isomorphic to $N/(U+N')$, thus
$\alpha(N/U) = \alpha(N'/(U\cap N')) + \alpha(N/(U+N'))$.
Both summands are non-negative, since we deal with factor modules of modules 
in $\mathcal Q$. Since $N/U$ is non-zero, at least one of $N'/(U\cap N'),\ N/(U+N')$ 
is non-zero, thus
at least one $\alpha(N'/(U\cap N')),\ \alpha(N/(U+N'))$ is positive,
therefore $\alpha(N/U) > 0.$
	
Now we show that every module has a filtration with factors in 
$\mathcal F,\ \mathcal N,\ \mathcal Q.$
First, we show that a module $X$ which has no non-zero factor module
in $\mathcal F$ and no non-zero submodule in $\mathcal Q$ belongs to  
$\mathcal N.$ 
Assume that there is a submodule $X'$ of $X$ with $\alpha(X') > 0$. Choose such a 
submodule $X'$  which is minimal with this property. We claim that 
$X'$ belongs to $\mathcal Q$. Namely, if $X'/U$ is a non-zero
factor module of $X'$ (here $U$ is a submodule of $X'$), then the minimality of $X'$
implies that $\alpha(U) \le 0$, thus $\alpha(X'/U) = \alpha(X')-\alpha(U) > 0$. 
But we assume that $X$ has no non-zero submodule in $\mathcal Q.$ This
contradiction shows that $\alpha(X') \le 0$ for any submodule $X'$ of $X$.
By duality, we have $\alpha(X'') \le 0$ for all factor modules $X''$ of $X$.
In particular, we have both $\alpha(X) \ge 0$ and $\alpha(X) \le 0$, thus $\alpha(X) = 0.$
Altogether we see that $X$ belongs to $\mathcal N.$
	
Second case: Let $N$ be a module without any non-zero submodule in $\mathcal Q$. 
Let $N'$ be a submodule of $N$ such that $N/N'$ belongs to $\mathcal F$ and
which is minimal with this property. Since $\mathcal F$ is closed under extensions, 
we know that no proper factor module of $N'$ belongs to $\mathcal F.$
According to the first case, we see that $N'$ belongs to
$\mathcal N.$
	
Now we  look at the general case. Thus, let $M$ be an arbitrary module.
Let $M''$ be a submodule of $M$ which belongs to $\mathcal Q$ and
which is maximal with this property. Since we know that $\mathcal Q$ is closed
under extensions, we know that $M/M''$ has no non-zero submodule which belongs to
$\mathcal Q.$ By the second case, the module $N = M/M''$ has a submodule $N'$
in $\mathcal N$ such that $N/N'$ belongs to $\mathcal F.$ Let $N' = M'/M''$ for
some $M'' \subseteq M' \subseteq M.$ Then $M'/M''$ is in $\mathcal N$, whereas
$M/M' \simeq N/N'$ belongs to
$\mathcal F.$ This completes the proof.
\end{note}
	\medskip

\begin{note}\label{examples}
{\bf Examples of torsion pairs.} Let us stress that the tilting torsion 
pairs, even the support-tilting torsion pairs, are quite special torsion pairs.
These are linear torsion pairs, but {\it torsion pairs are usually not linear.}

For example, take the path algebra $\Lambda$ of the Kronecker quiver 
$$
{\beginpicture
\setcoordinatesystem units <1.5cm,1cm>
\put{} at 2 -.3
\multiput{$\circ$} at 2 0  3 0 /
\put{$\ssize 1$} at 2 0.2
\put{$\ssize 2$} at 3 0.2
\arr{2.8 0.1}{2.2 0.1}
\arr{2.8 -.1}{2.2 -.1}
\endpicture}\ ,
$$
take the simple regular representation $R(0) = (k,k;1,0)$ and let $\mathcal G$ be the
torsion class generated by $R(0)$: it consists of one homogeneous tube and all the
preinjective modules, the corresponding torsionfree class consists of the remaining
tubes and all the preprojective modules. 
In particular, the simple regular representations and
 $R(1) = (k,k;1,1)$ and  $R(\infty) = (k,k;0,1)$ both belong to $\mathcal F$, but it is
not possible to distinguish say $R(0)$ and $R(1)$, using linear forms on $K_0(\Lambda)$.
	
Also: {\it Not all linear torsion pairs are tilting torsion pairs or at least 
support-tilting torsion pairs.} Again, let $\Lambda$ be the path algebra of the
Kronecker quiver and consider the defect function: it takes negative values on the
preprojective Kronecker modules, positive values on the preinjective Kronecker modules
and vanishes on the regular Kronecker modules.
	\medskip 

It is easy to write down all the linear forms of $K_0(\Lambda)$ for the Kronecker algebra
$\Lambda$, since $K_0(\Lambda) = \mathbb Z^2$. We work with the basis
 $e(1) = (1,0) = \bdim S(1),\ e(2) = (0,1) = \bdim S(2).$
In general, given a pair $a,b$ of integers, we define the linear form 
$\alpha_{(a,b)}\!:K_0(\Lambda ) \to \mathbb Z$ by
$\alpha_{(a,b)}(d_1,d_2) = ad_1+bd_2.$ 

Of importance is the linear form $\alpha = \alpha_{(-1,1)}$, thus  
$\alpha(d_1,d_2) = -d_1+d_2$. This is the defect function. 
The classification of the indecomposable Kronecker modules asserts: Any indecomposable Kronecker module belongs to
one of the three classes $\mathcal F(\alpha), \mathcal N(\alpha), \mathcal Q(\alpha)$, the modules in $\mathcal F(\alpha)$ are the preprojective
modules, those in $\mathcal N(\alpha)$ the regular modules, those in $\mathcal Q(\alpha)$ the preinjective modules. 
	
Let us mention some other
special cases in order to outline the variety of possibilities. 
If $a = b = 0$, then $(\mathcal F(\alpha),\mathcal N(\alpha),\mathcal Q(\alpha)) =
(0,\mo  \Lambda , 0).$ If both numbers $a,b$ are positive, then we have
$(\mathcal F(\alpha),\mathcal N(\alpha),\mathcal Q(\alpha)) = (0,0,\mo  \Lambda )$. 
If both numbers $a,b$ are negative, then we have
$(\mathcal F(\alpha),\mathcal N(\alpha),\mathcal Q(\alpha)) = (\mo  \Lambda ,0,0)$. 
	 
For $(a,b) = (-4,3)$, the category $\mathcal N(\alpha_{(-4,3)})$ is $\add P(3)$. The indecomposable
modules in $\mathcal F(\alpha_{(-4,3)})$ are the Kronecker modules $P(0),P(1),P(2)$, and the modules in
$\mathcal Q(\alpha_{(-4,3)})$ are the Kronecker modules without a direct summand of the form $P(0),\dots,P(3).$ 
	 
For $(a,b) = (-5,3)$, the category $\mathcal N(\alpha_{(-5,3)})$ is the zero category. The indecomposable
modules in $\mathcal F(\alpha_{(-5,3)})$ are the Kronecker modules $P(0),P(1)$, and the modules in
$\mathcal Q(\alpha_{(-5,3)})$ are the Kronecker modules without a direct summand of the form $P(0),P(1).$
	\medskip

Finally, let us remark that {\it not every linear torsion pair is defined by
a linear form with values in $\mathbb Z$.} As an example, we can take any wild hereditary
algebra. If $\Lambda$ is the $n$-Kronecker algebra with $n\ge 3$, say with
simple projective module $S$ and simple injective module $T$, define $\alpha$
by $\alpha(S) = -1$ and $\alpha(T) = \sqrt 2$. It is not difficult to show that
the torsion pair $(\mathcal F(\alpha),
\mathcal G(\alpha))$ cannot be defined by a rational linear form. Other obvious examples 
are provided by tubular algebras.
\end{note}
	\medskip

\begin{note}\label{Roiter}
{\bf Proof of Roiter's Normalization Lemma}  (following \cite{[R5]}). 
We use the following facts, here $X,Y$ are modules over an artin algebra.

(a) {\it Let $(f_1,\dots,f_t,g)\!:X \to X^t\oplus Y$ be an injective map for some $t$,
with all the maps $f_i$ in the radical of $\End(X)$. Then $X$ is cogenerated by $Y$.}

(b) {\it Let $(f_1,\dots,f_t,g)\!:X^t \oplus Y \to X$ be a surjective map for some natural number $t$,
with all the maps $f_i$ in the radical of $\End(X)$, then $Y$ generates $X$.}

\begin{proof} (a) Assume that the radical $J$ of $\End(X)$ satisfies $J^m = 0$.
Let $W$ be the set of all compositions $w$ of at most $m-1$ maps of the form $f_i$ with $1\le i \le t$
(including $w = 1_X$).  
We claim that $(gw)_{w\in W}\!:X \to Y^{|W|}$ is injective. Take a non-zero element $x$
in $X$. Then there is $w\in W$ such that $w(x)\neq 0$ and $f_iw(x) = 0$ for $1\le i\le t$.
Since $(f_1,\dots,f_t,g)$ in injective and $w(x) \neq 0$, we have $(f_1,\dots,f_t,g)(w(x)) \neq 0.$ 
But $f_iw(x) = 0$ for $1\le i \le t,$ thus $g(w(x)) \neq 0.$ This completes the proof.
    
(b) This follows by duality.   
\end{proof} 
     
{\bf Normalization lemma.}
{\it Let $M = M_0\oplus M_1 = M_0'\oplus M_1'$ be direct decompositions of a module $M$
over an artin algebra and assume that both modules 
$M_0$ and $M'_0$ generate $M$. 
Then there is a module $N$ which generates $M$ and which is isomorphic to a direct summand of $M_0$ and $M'_0$.}
	\medskip 
     
\begin{proof} 
We may assume that $M$ is multiplicity free. Write 
$M_0 \simeq N\oplus C,\ M_0' \simeq N \oplus C',$ such that $C, C'$
have no indecomposable direct summand in common. 
Now, $N\oplus C$ generates $N\oplus C'$ generates $N\oplus C$ generates $C$.
We see that $N\oplus C$ generates $C$, such that the maps $C \to C$ used 
belong to the radical of $\End(C)$ (since they factor through $\add(N\oplus C')$
and no indecomposable direct summand of $C$ belongs to $\add(N\oplus C')$).
According to (b), $N$ generates $C$, thus it generates $M$.
\end{proof}
\end{note}
	\medskip

\begin{note}\label{simple}
{\bf The simple $\Gamma(T)$-modules.} Let $T$ be a tilting module and 
$\Gamma(T) = \End(T)^\op$. A simple $\Gamma(T)$
module belongs either to $\mathcal Y(T)$ or to $\mathcal X(T)$, since we deal with 
a torsion pair $(\mathcal Y(T),\mathcal X(T)$ (actually, this torsion 
pair is even a split torsion pair, 
thus any indecomposable module belongs to one of the two subcategories).
The presentation of the classical tilting theory in the paper \cite{[HR2]}
used this fact quite prominently, but the actual structure 
of the simple $\Gamma(T)$-modules has not been discussed
there.
The new approach to the classical tilting theory stresses the relevance
of thick subcategories, and, as it turns out, it is the structure of the simple
$\Gamma(T)$-modules which plays a decisive role!
	\medskip 

Let $T = \bigoplus_{i=1}^n T_i$  be a tilting module with indecomposable direct summands $T_i$.
The indecomposable projective $\Gamma(T)$-modules
are of the form $\Hom(T,T_i)$, let us denote the top of $\Hom(T,T_i)$ by 
$S'(i)$.
	\medskip

For any $i$, let $T^{(i)} = T/T_i = \bigoplus_{j\neq i} T_j$ and
let $g_i\!:T^i\to T_i$ be a minimal right $T^{(i)}$-approximation of $T_i.$
	\smallskip 

First case.
{\it If $T_i$ is a direct summand of $\nu(T)$, then $g_i$ is injective, say with
cokernel $\Delta(i).$ Then $\Hom(T,\Delta(i))$
is a simple $\Gamma(T)$-module, 
it belongs to $\mathcal Y(T)$ and its projective cover
is $\Hom(T,T_i),$ thus $S'(i) = \Hom(T,\Delta(i))$.}

Proof: The map $g_i$ cannot be surjective, since otherwise $T_i$ would belong to $\nu'(T)$.
But then $g_i$ has to be injective since $\Ext^1(T_i,T^i) = 0.$ Thus,  
there is an exact sequence of the form
$$
\hbox{\beginpicture
  \setcoordinatesystem units <1cm,.7cm>
\put{$ 0 \to T^i \overset{g_i}\longrightarrow T_i \overset{p_i}\longrightarrow \Delta(i) \to 0$} [l] at 0 0
\endpicture}
$$
(just as the sequence $(*)$). We apply the functor $\Hom(T,-)$ and obtain the
exact sequence
$$
\hbox{\beginpicture
  \setcoordinatesystem units <1cm,.7cm>
\put{$0 \to \Hom(T,T^i)  \overset{\Hom(T,g_i)}\longrightarrow \Hom(T,T_i) \overset{\Hom(T,p_i)}\longrightarrow \Hom(T,\Delta(i)) \to 0$} [l] at 0 0
\put{$(**)$} [r] at -1 0
\put{} at 12 0 
\endpicture}
$$
(the right exactness follows from $\Hom(T,T^i) = 0$). 
Now, $\Hom(T,T_i)$ is indecomposable projective and $Y(i) = \Hom(T,\Delta(i))$ is 
a factor module. In order to see that $Y(i)$ is simple, we only have to show that 
any proper submodule $U$ of $\Hom(T,T_i)$ is contained in the image of $\Hom(T,g_i)$. 
It is sufficient to show this for any proper submodule $U$ which is local. Such a submodule
is a factor module of some $\Hom(T,T_j)$. Thus there is given a map $h\!:T_j \to T_i$
which is not an isomorphism such that $\Hom(T,h)$ maps into $U$. Since the endomorphism ring
of $T_i$ is a division ring, we must have $j\neq i$ and therefore $h$ factors through
$g_i$. But this is what we wanted to show. 

Thus we see: The module  $Y(i) = \Hom(T,\Delta(i))$ is simple. Obviously, $(**)$
is a minimal projective resolution of $Y(i)$ as a $\Gamma(T)$-module. What is of importance
for us is the fact that  $Y(i)$ belongs to $\mathcal Y(T)$, but this is clear, since
$Y(i)  = \Hom(T,\Delta(i))$. 
	\medskip

Second case.
{\it If $T_i$ is a direct summand of $\nu'(T)$, 
then $g_i$ is surjective, say with
kernel $U(i).$ Then $\Ext^1(T,U(i))$
is a simple $\Gamma(T)$-module, it belongs to $\mathcal X(T)$ and its projective cover
is $\Hom(T,T_i),$ thus $S'(i) = \Ext^1(T,U(i)).$}
	\smallskip 
Proof. To see that $g_i$ is surjective, we only have to note that $T_i$ is generated by
$\nu(T)$, thus by $T^{(i)}$. We apply $\Hom(T,-)$ to the exact sequence
$$
 0 \to U(i) \overset{u_i}\longrightarrow T^i \overset{g_i}\longrightarrow T_i \to 0
$$ 
and obtain the exact sequence
$$
\hbox{\beginpicture
  \setcoordinatesystem units <1cm,.7cm>
\put{$0 \to \Hom(T,U(i))  \overset{\Hom(T,u_i)}\longrightarrow \Hom(T,T^i) 
\overset{\Hom(T,g_i)}\longrightarrow \Hom(T,T_i) \overset{d_i}\longrightarrow \Ext^1(T,U(i)) \to 0.$} [l] at 0 0
\put{$(***)$} [l] at 0 -1
\put{} at 9 0 
\endpicture}
$$
The map $g_i$ is not a split epimorphism, since $T_i$ does not belong to $\add T^{(i)}$.
It follows that $\Hom(T,g_i)$ is not surjective, thus $\Ext^1(T,U(i)) \neq 0.$
Since $\Hom(T,T_i)$ is indecomposable projective, the map $d_i$ is a projective cover.
In particular, $\Ext^1(T,U(i))$ is a local module. with top $S'(i)$.

We look at the dimension vector of $\Ext^1(T,U(i)).$ For any index $j$,
the simple module $S'(j)$ occurs as a composition factor of $\Ext^1(T,U(i))$ if
and only if $\Ext^1(T_j,U(i))\neq 0.$  

Consider an index $j\neq i$ and apply $\Hom(T_j,-)$ to $g_i$. We obtain the exact sequence
$$
\hbox{\beginpicture
  \setcoordinatesystem units <1cm,.7cm>
\put{$ \Hom(T_j,T^i) \overset{\Hom(T_j,g_i)}\longrightarrow \Hom(T_j,T_i) \to \Ext^1(T_j,U(i)) \to \Ext^i(T_j,T^i).
$} [l] at 0 0
\endpicture}
$$
The last term is zero since $\Ext^1(T,T) = 0$. The map $\Hom(T_j,g_i)$ is surjective,
since $g_i$ is a $T^{(i)}$-approximation and $T_j$ is a direct summand of 
$T^{(i)}$. This shows that $\Ext^1(T_j,U(i)) = 0.$ 

As a consequence, all composition factors of $\Ext^1(T,U(i))$ are equal to
$S'(i)$. But the quiver of $\Gamma(T)$ has no loops, thus $\Ext^1(T,U(i)) =
S'(i)$. In particular, $S'(i)$ belongs to $\mathcal X(T)$. 
	\medskip 

{\bf The projective dimension of the simple $\Gamma(T)$-modules $S'(i)$.} If
$T_i$ is a direct summand of $\nu(T)$, then we have seen that $S'(i)$ belongs to 
$\mathcal Y(T)$. All the modules in $\mathcal Y(T)$ have projective dimension at most 1, thus
this holds true for $S'(i)$. A minimal projective resolution of $S'(i)$ is given
by the sequence $(**)$. On the other hand, in case $T_i$ is a direct summand of 
$\nu'(T)$, then $S'(i)$ belongs to $\mathcal X(T)$, thus the projective dimension of
$S'(i)$ may be $2$. In case $U(i)$ belongs to $\mathcal F(T)$, the exact sequence
$(***)$ shows that the projective dimension of $S'(i)$ is equal to $1$ (it cannot
be zero, since $T^i\neq 0$; of course, we know that the indecomposable projective
$\Gamma(T)$-modules belong to $\mathcal Y(T)$, thus not to $\mathcal X(T)$). On the other hand,
if $U(i)$ does not belong to $\mathcal F(T)$, then $\Hom(T,U(i)\neq 0$ and the projective
dimension of $S'(i)$ is equal to 2.
	\medskip

{\bf The subcategory $\phi (\mathcal N(T))$.} As we know, the functor 
$\phi = \Hom(T,-)$ yields an equivalence $\mathcal G(T) \to \mathcal Y(T)$. Since $\Ext^1(T,-)$
vanishes on the category $\mathcal Y(T)$, the functor $\phi$ is exact on exact sequences
in $\mathcal Y(T)$ (we mean: exact sequences of $\mo \Lambda$ with all terms lying in
$\mathcal Y(T)$). Now the subcategory $\mathcal N(T)$ consists of all $\Lambda$-modules with
a filtration with factors of the from $\Delta(i)$, thus $\phi(\mathcal N(T))$ consists
of all the $\Gamma(T)$-modules with a composition series with all factors in $\mathcal Y(T)$.
This can be reformulated as follows: 
{\it $\phi(\mathcal N(T))$ is the largest Serre subcategory
of $\mo \Gamma(T)$ which lies inside $\mathcal Y(T)$.}
	\medskip

{\bf The set of simple $\Gamma(T)$-modules which belong to $\mathcal Y(T)$ is closed
under successors in the quiver of $\Gamma(T).$} For the proof, assume that
$S'(i),S'(j)$ are simple $\Gamma(T)$-modules with $\Ext^1(S'(i),S'(j)) \neq 0$, and that
$S'(i)$ belongs to $\mathcal Y(T)$. We have to show that also $S'(j)$ belongs to $\mathcal Y(T)$. 
Since $S'(i)$ belongs to $\mathcal Y(T)$, the module $T_i$ is a direct summand of $\nu(T)$.
But then also $T^i$ belongs to $\add \nu(T)$. On the other hand, $\Ext^1(S'(i),S'(j)) \neq 0$
means that $T_j$ is a direct summand of $T^i$, thus in $\add \nu(T).$ It follows
that $S'(j)$ belongs to $\mathcal Y(T)$.
\end{note}
	\medskip 

\begin{note}\label{example2}{\bf An example.}
Here is an example {\it of a linear form $\alpha$ on $K_0(\Lambda)$ 
with $\mathcal N(\alpha)$ a sincere exceptional subcategory such that
$\mathcal Q(\alpha)$ is not generated by $\mathcal N(\alpha)$ and 
and such that 
$\mathcal F(\alpha)$ is not cogenerated by $\mathcal N(\alpha)$.}
	\medskip

We take as $\Lambda$ the path algebra of a quiver of type $\mathbb A_5$, as shown on the left.
We draw the
Auslander-Reiten quiver of $\Lambda$, but replace the indecomposable modules $M$
by the corresponding values $\alpha(M)$. It is easy to check that these are indeed the
values of an additive function on $K_0(\Lambda)$.
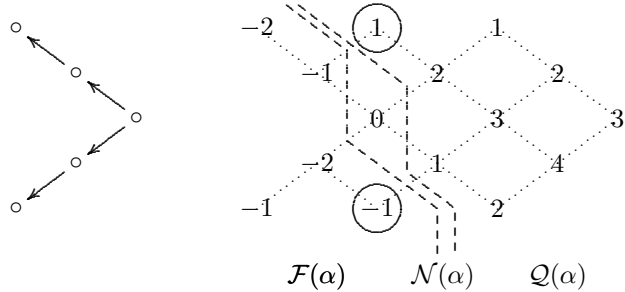
$$
{\beginpicture
\setcoordinatesystem units <.8cm,.6cm>
\multiput{} at 0 4.7 0 -2 /
\put{$-2$} at 0 4
\put{$1$} at 2 4
\put{$1$} at 4 4
\put{$-1$} at 1 3
\put{$2$} at 3 3
\put{$2$} at 5 3
\put{$0$} at 2 2
\put{$3$} at 4 2
\put{$3$} at 6 2
\put{$-2$} at 1 1
\put{$1$} at 3 1
\put{$4$} at 5 1
\put{$-1$} at 0 0
\put{$-1$} at 2 0
\put{$2$} at 4 0
\setdashes <1mm>
\plot  .5 4.5    1.5 3.5  1.5 1.5  3 0  3 -1 /
\plot  .7 4.5    1.7 3.5  2.5 2.7  2.5  0.8  3.3  0  3.3 -1 /
\setdots <1mm>
\put{$\mathcal F(\alpha)$}  at 1 -1.5    
\plot 0 0  4 4  6 2  4 0  0 4 /
\plot 1 1  2 0  5 3 /
\plot 1 3  2 4  5 1 /
\put{$\mathcal F(\alpha)$}  at 1 -1.5    
\put{$\mathcal N(\alpha)$}  at 3.1 -1.5    
\put{$\mathcal Q(\alpha)$}  at 5 -1.5
\setsolid
\circulararc 360 degrees from 2.4 0  center at 2 0 
\circulararc 360 degrees from 2.4 4  center at 2 4 

\multiput{$\circ$} at -4 4  -3 3  -2 2  -3 1  -4 0 / 
\arr{-3.2 3.2}{-3.8 3.8}
\arr{-2.2 2.2}{-2.8 2.8}
\arr{-3.2 .8}{-3.8 .2}
\arr{-2.2 1.8}{-2.8 1.2}
   
\endpicture}
$$
 The separation between the subcategories
$\mathcal F(\alpha),\mathcal N(\alpha), \mathcal Q(\alpha)$ is indicated by dashed lines.
In particular, we see that $\mathcal N(\alpha) = \add(M)$, where $M$ is 
the unique sincere indecomposable module. Two modules (or better, their values under
$\alpha$) have been encircled: the upper one belongs to $\mathcal Q(\alpha)$, it is not
generated by $M$, the lower one belongs to $\mathcal F(\alpha)$, it is not
cogenerated by $M$.
\end{note}
	\bigskip

\begin{note}\label{support-tilting}
{\bf Proof of Theorem \ref{version4}.} 
It is sufficient to show the inclusions 
$$ 
 \mathcal F(T)\subseteq \mathcal F(\alpha),\  
 \mathcal N(T)\subseteq \mathcal N(\alpha),\  
 \mathcal Q(T)\subseteq \mathcal Q(\alpha).
$$
The last two assertions concern modules $M$ with support in $\bold S$, where $\bold S$
is the support of $T$. If the
support of $M$ is in $\bold S$, then $\Hom(P,M) = 0$, thus $\alpha$ and $\langle \nu'(T),-\rangle$
coincide on all submodules and all factor modules of $M$. 
Thus, we only have to show the first inclusion. Assume that $M$ belongs to $\mathcal F(T)$,
let $M'$ be a non-zero submodule of $M$. If the support of $M'$ is not contained in $\bold S$,
then $\Hom(P,M') \neq 0$ (and $\Hom(\nu'(T).M) = 0$, thus $\alpha(M') < 0.$
If  the support of $M'$ is contained in $\bold S$, then $M'$ is a non-zero 
$\Lambda_T$-module with $\Hom(T,M') = 0$, thus we are in the setting of dealing with tilting
modules and, as we have seen, $\langle \nu'(T),M'\rangle < 0.$ 
\end{note}
	\bigskip 

\begin{note}\label{Bongartz}
{\bf The Bongartz complement and the dual construction.}
Let $T$ be a partial tilting module. 
Choose a universal extension of $\Lambda$ by copies of $T$,
say $0 \to \Lambda \to \Lambda' \to T^t \to 0$ with $t\in \mathbb N$, such that $\Ext^1(T,\Lambda') = 0$,
or, equivalently, such that the connecting homomorphism $\Hom(T,T^t) \to \Ext^1(T,\Lambda)$
is surjective. It is easy to see that $\Ext^1(T\oplus \Lambda',T\oplus \Lambda') = 0.$
Namely, first of all, we have $\Ext^1(T,T) = 0$ and $\Ext^1(T,\Lambda') = 0.$ Next,
$\Ext^1(\Lambda,T) = 0$ and $\Ext^1(T,T) = 0$ imply that $\Ext^1(\Lambda',T) = 0.$
Finally, $\Ext^1(\Lambda,\Lambda') = 0$ and $\Ext^1(T,\Lambda') = 0$ imply that
$\Ext^1(\Lambda',\Lambda') = 0.$

The defining sequence for $\Lambda'$ shows that $\Lambda'\oplus T$ is 
Morita equivalent to a tilting module, say to $\beta(T)\oplus T$, where $\beta(T)$
is a direct summand of $\Lambda'$. The module $\beta(T)$ is called a {\it Bongartz complement}
 $\beta(T)$ for $T$, it is uniquely determined by $T$ up to isomorphism. 
	\medskip

{\it If $T$ is an exceptional module and not projective, then $\Hom(T,\beta(T)) = 0$, thus
$\beta(T)$ belongs to $T^\perp.$}
Proof. Denote the inclusion map $\Lambda \to \Lambda'$ by $u$, the projection map 
$\Lambda' \to T^t$ by $q$, thus $qu = 0.$ We can assume that $t$ is minimal. Let $f\!:T \to \Lambda'$
be a homomorphism. We claim that $qf = 0$. Otherwise, there is a split epimorphism $p:T^t \to T$ 
such that $pqf$ is non-zero. But a non-zero endomorphism of $T$ is invertible, therefore $qf$ is
a split epimorphism. But this implies that $t$ is not minimal, a contradiction. It follows from $qf = 0$
that the image of $f$ is projective, thus a direct summand of $T$. Since $T$ is indecomposable and
not projective, it follows that $f = 0.$
	\medskip

{\it If $T$ is a partial tilting module and sincere, then $\beta(T)$ is
cogenerated by $T$.} Proof, following \cite{[R5]}.
As we know, a sincere partial tilting module is faithful, thus 
$\Lambda$ is cogenerated by $T$. Any embedding $v\!:\Lambda \to T^s$ gives rise to a
commutative diagram with exact rows
$$
{\beginpicture
\setcoordinatesystem units <1.6cm,1.2cm>
\put{$0$} at 0 0
\arr{0.3 0}{0.6 0} 
\put{$\Lambda$} at 1 0
\arr{1.3 0}{1.6 0} 
\put{$\Lambda'$} at 2 0
\arr{2.3 0}{2.6 0} 
\put{$T^t$} at 3 0
\arr{3.3 0}{3.6 0} 
\put{$0$} at 4 0

\arr{1  -.3}{1  -.7} 
\arr{2  -.3}{2  -.7} 
\plot 3.02 -.3  3.02 -.7 /
\plot 2.98 -.3  2.98 -.7 /

\put{$0$} at 0 -1
\arr{0.3 -1}{0.6 -1} 
\put{$T^a$} at 1 -1
\arr{1.3 -1}{1.6 -1} 
\put{$M$} at 2 -1
\arr{2.3 -1}{2.6 -1} 
\put{$T^t$} at 3 -1
\arr{3.3 -1}{3.6 -1} 
\put{$0$} at 4 -1

\put{$v$} at 1.2 -.5
\put{$v'$} at 2.2 -.5
\put{$g$} at 2.45 0.2
\endpicture}
$$

Since $\Ext^1(T,T) = 0$, the lower sequence splits, thus $\Lambda'$ is cogenerated by $T$.
As a consequence, also $\beta(T)$ is cogenerated by $T$.
	\medskip

{\bf The sub complement for a partial tilting module.}
Let $T$ be a partial tilting module, say with support algebra
$\Lambda(T)$. We call the Bongartz complement $\beta({}_{\Lambda(T)}T)$
for $T$ (considered as a $\Lambda(T)$-module!), 
the {\it sub complement} for $T$. Since ${}_{\Lambda(T)}T$ is a faithful
$\Lambda(T)$-module, the sub complement for $T$ is cogenerated by $T$.
	\medskip 

{\bf The dual construction.} Again, we start with a
 partial tilting module $T$. We denote by $Z$ a minimal cogenerator
for $\mo \Lambda$. We choose 
a universal ``foundation'' of $Z$ by copies of $T$, that means an exact sequence
$0 \to T^t \to Z' \to Z \to 0$ with $t\in \mathbb N$ such that $\Ext^1(Z',T) = 0$
(such a foundation is often called a universal extension from below).  
Then $Z\oplus Z'$ is Morita equivalent to 
a tilting module, say $T\oplus \beta'(T)$ with 
$\beta'(T)$ is a direct summand of $Z'$. The module $\beta'(T)$ is called a 
{\it co-Bongartz complement} for $T$, it is 
uniquely determined by $T$ up to isomorphism. There are the dual assertions:
	\medskip

{\it If $T$ is an exceptional module and not injective, then $\Hom(\beta'(T),T) = 0$, thus
$\beta'(T)$ belongs to ${}^\perp T.$}
	\medskip

{\it If $T$ is a  partial tilting module and sincere, then $\beta'(T)$ is
generated by $T$.} 
	\medskip

{\bf The factor complement for a partial tilting module.}
Let $T$ be a  partial tilting module, say with support algebra
$\Lambda(T)$. We call the co-Bongartz complement $\beta'({}_{\Lambda(T)}T)$
of $T$ (considered as a $\Lambda(T)$-module!) 
the {\it factor complement} for $T$. Since ${}_{\Lambda(T)}T$ is a faithful
$\Lambda(T)$-module, we see that the factor complement for $T$ is generated by $T$.
\end{note}
	\medskip 
	
\begin{note}\label{large-rank} {\bf Thick closures of large rank.}
{\it Let $E_1,E_2$ be exceptional modules and $\mathcal E$ the thick closure of $E_1,E_2$.
The rank of $\mathcal E$ may be arbitrarily large.}
	\smallskip 

Example: Let $Q$ be the $n$-subspace quiver with sink $0$. Let $E_1 = S(0)$ and $E_2 = I(0)$.
Both are exceptional modules. There is an embedding of $E_1$ into $E_2$, its cokernel is the
direct sum of the simple modules $S(i)$ with $i\neq 0.$ This shows that the thick closure is the
category of all representations of $Q$, this category has rank $n+1$. Note that for $n\ge 2$, neither
$(E_1,E_2)$  
nor $(E_2,E_1)$ is an exceptional sequence, since $\Hom(E_1,E_2) \neq 0$ and $\Ext^1(E_2,E_1) \neq 0.$
\end{note}
	\medskip 

\begin{note}\label{relationship}
{\bf The relationship between a minimal cover $N$ for $\mathcal G$
and a minimal cocover $M$ for $\mathcal F$, where $(\mathcal F,\mathcal G)$ is a torsion pair.}
	\medskip 

Let $(\mathcal F,\mathcal G)$ be a torsion pair with minimal cover $N$ for $\mathcal G$
and minimal cocover $M$ for $\mathcal F$. The corresponding perpendicular pair is 
$(\mathcal M,\mathcal N)$, where $\mathcal N$ is the thick closure of $N$ and $\mathcal M$ the thick closure
of $M$. And we know that $N$ is a minimal generator for the abelian category $\mathcal N$,
whereas $M$ is a minimal cogenerator for the abelian category $\mathcal M$. Let $N'$ be
the factor complement for $N$ and $M'$ the sub complement for $M$. 
There is a 
direct procedure to obtain $M\oplus M'$ from $N\oplus N'$ and vice versa:
	\medskip 

{\it Let $(\mathcal F,\mathcal G)$ be a torsion pair with minimal cover $N$ for $\mathcal G$
and minimal cocover $M$ for $\mathcal F$. Let $N'$ be
the factor complement for $N$ and $M'$ the sub complement for $M$. 
Write $N = N^P\oplus N^C$, where $N^P$ is projective and $N^C$ has no indecomposable
projective direct summand. 
Write $M = M^I\oplus M^C$, where $M^I$ is injective and $M^C$ has no indecomposable
injective direct summand. 

{\rm(a)} $\tau N^C = M'$ and $\tau N' = M^C$.

{\rm(b)} The module $N^P$ is the direct sum of all indecomposable projective modules in $\mathcal G$,
and $M^I$ is the direct sum of all indecomposable injective modules in $\mathcal F$.

{\rm(c)} The module $N\oplus N'$ is a tilting $\Lambda_{\mathcal G}$-module, the 
module $M\oplus M'$ is a tilting $\Lambda_{\mathcal F}$-module. 

(Here, $\Lambda_{\mathcal G}$ is the factor algebra of $\Lambda$ modulo the annihilator of 
$\mathcal G$, and $\Lambda_{\mathcal F}$ is the factor algebra of $\Lambda$ modulo the annihilator of 
$\mathcal F$.)}
	\medskip 

Let us sketch the position of the modules involved: 
 $$
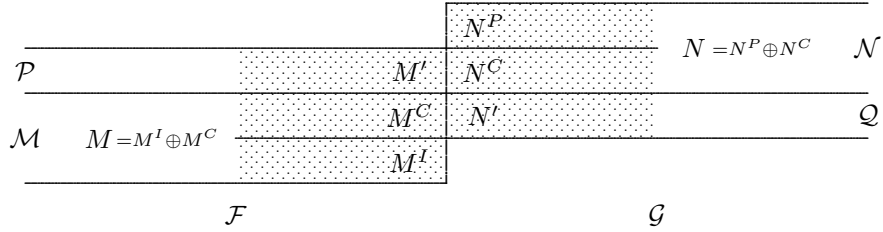
\hbox{\beginpicture
  \setcoordinatesystem units <.7cm,.6cm>
\multiput{} at 0 -1.2  8 4 /
\plot -4 0  4 0  4 4  12 4 /
\plot  0 1  12 1 /
\plot -4 2  12 2 /
\plot -4 3  8 3 /
\put{$N^P$} at 4.7 3.5
\put{$N^C$} at 4.7 2.5
\put{$N'$} at 4.7 1.5

\put{$M'$} at 3.3 2.5
\put{$M^C$} at 3.3 1.5
\put{$M^I$} at 3.3 0.5

\put{$\mathcal F$} at 0 -.7
\put{$\mathcal G$} at 8 -.7

\put{$ N $} at 8.7 3
\put{$\ssize = N^P\oplus N^C$} at 10 3
\put{$\mathcal N$} at 12 3
\put{$\mathcal Q$} at 12 1.5
\put{$ M $} at -2.6 .98
\put{$ \ssize = M^I\oplus M^C$} at -1.3 1
\put{$\mathcal M$} at -4 1
\put{$\mathcal P$} at -4 2.5

\setshadegrid span <.7mm>
\vshade 0 0 3  4 0 3  4.01 1 4  8 1 4 /
\endpicture}
$$
The columns in the middle are related by the Auslander-Reiten translation $\tau$: 
it shifts the right column to the left one. To be precise: $\tau$ sends $N^P$ to zero, 
$N^C$ to $M'$ and
$N'$ to $M^C$ (dually, $\tau^-$ sends $M^I$ to zero, $M^C$ to $N'$ and
$M'$ to $N^C$).
	\medskip

For a proof, we refer to Smal\o{} \cite{[Sm]}.
\end{note}
	\bigskip 

\begin{note}\label{lattice}
{\bf Lattices.} 
Let $P$ be a poset. If $a,b$ are elements of $P$, and the supremum of $a,b$ exists, then we 
denote it by $a\vee b$ and call it the {\it join} of $a$ and $b$ (by definition, $a\vee b$ is an element of $P$ such that $a \le a\vee b,\ b\le a\vee b$ and such that for any element $c$
with $a \le c,\ b \le c$ we have $a\vee b \le c$; of course, if $a\vee b$ exists, it is uniquely
determined by $a$ and $b$). Dually, if the infimum  of $a,b$ exists, then we 
denote it by $a\wedge b$ and call it the {\it meet} of $a$ and $b$. The poset $P$ is said to be
a {\it lattice} provided meet and join exist for any two elements $a,b\in P$.
\end{note}
	\medskip 

\begin{note}\label{torsion}
{\bf The lattice of torsion subcategories.}
Here are some references: 
The restriction of the ordering $\le_t$ to the set of sincere exceptional antichains (or, equivalently,
to the set of tilting modules) has been studied in detail by Happel-Unger \cite{[HU]}, 
for the use of this ordering, see the references in this paper.

The ordering $\le_t$ for the exceptional antichains
has been investigated for  path algebras of quivers
by Iyama-Reiten-Thomas-Todorov \cite{[IRTT]}, 
for general hereditary artin algebra $\Lambda$ see \cite{[R5]}.
Of course, we obtain in this way just the poset of torsion classes with covers
(and these are the functorially finite torsion classes).
In case $\Lambda$ is connected, 
this poset is a lattice if and only if  $\Lambda$ is representation finite
or has just $2$ simple modules.
\end{note}

\vfill\eject
\section{\bf The poset $\A(\mo \Lambda)$ of exceptional antichains}

Let $\Lambda$ be 
a hereditary artin algebra. We consider
 {\bf antichains} in the category $\mo\Lambda$. 
Let us recall that such an antichain $A$ consists of pairwise orthogonal bricks 
and an antichain $A$ is said to be exceptional provided the
quiver of $A$ (with vertices the elements $A_i$ of $A$ and with an arrow from $A_i$ 
to $A_j$ whenever $\Ext^1(A_i,A_j) \neq 0$) has no oriented cyclic paths.
For any hereditary artin algebra, let $\A(\mo\Lambda)$ be the 
{\bf set of exceptional antichains.} This is a poset with respect to the
ordering $A \le A'$ provided any element of $A$
has a filtration with factors in $A'$ (or, equivalently, provided the
extension closure of $A$ is contained in the extension closure of $A'$).
If $\Lambda$ is representation-finite, then $\A(\mo\Lambda)$ is a
lattice (see Theorem \ref{is-lattice}). 
If $\Lambda$ is representation-infinite, then, in general, $\A(\mo\Lambda)$ 
is not a lattice, but also these posets are of interest! This chapter
is devoted to the study of the posets $\A(\mo\Lambda)$ (for the use of the
wording ``exceptional'' see N\,\ref{exceptional}).

As an important example, we consider in Chapter 4
the path algebra $\Lambda_n$ of the linearly oriented quiver 
of type $\mathbb A_n$:
$$
 1 \leftarrow 2 \leftarrow \cdots \leftarrow n
$$
Thus, for $n = 3$, we deal with the
quiver $Q$ shown left, its Auslander-Reiten quiver is shown on the right:
$$
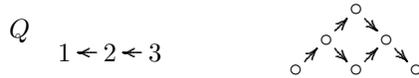
\beginpicture
  \setcoordinatesystem units <0.6cm,0.6cm>
\put{\beginpicture
\put{$1$} at 0 0
\put{$2$} at 1 0
\put{$3$} at 2 0
\arr{0.7 0}{0.3 0}
\arr{1.7 0}{1.3 0}
\put{$Q$} at -1 .5
\endpicture} at 0 0 
\put{\beginpicture
  \setcoordinatesystem units <0.4cm,0.4cm>
\multiput{$\circ$} at 0 0  1 1  2 2  3 1  4 0  2 0 /
\arr{0.3 0.3}{0.7 0.7}
\arr{1.3 0.7}{1.7 0.3}
\arr{2.3 0.3}{2.7 0.7}
\arr{3.3 0.7}{3.7 0.3}
\arr{1.3 1.3}{1.7 1.7}
\arr{2.3 1.7}{2.7 1.3}
\endpicture} at 6 0 
\endpicture
$$
Here is the lattice $\A(\mo\Lambda_3)$. By definition, its
elements are antichains, thus sets of (at most 3) indecomposable $\Lambda_3$-modules;
we use bullets in the Auslander-Reiten quiver to mark such an antichain. 
$$
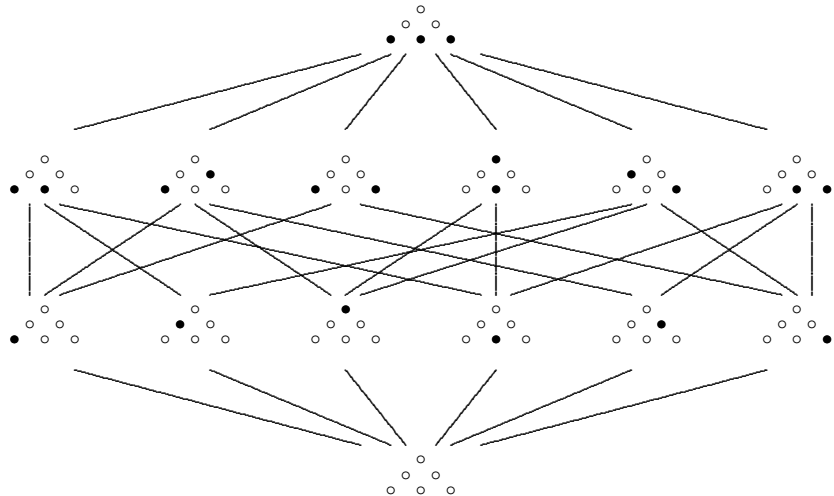
\hbox{\beginpicture
  \setcoordinatesystem units <2cm,2cm>

\put{$\beginpicture
  \setcoordinatesystem units <0.2cm,0.2cm>
  \multiput{$\sssize \circ$} at 0 0  1 1  -1 1  0 2  -2 0  2 0 /
  \multiput{$\ssize \bullet$} at  /
 \endpicture$} at 0 0

\put{$\beginpicture
  \setcoordinatesystem units <0.2cm,0.2cm>
  \multiput{$\sssize \circ$} at 0 0  1 1  -1 1  0 2  -2 0  2 0 /
  \multiput{$\ssize \bullet$} at -2 0 /
 \endpicture$} at -2.5 1
\put{$\beginpicture
  \setcoordinatesystem units <0.2cm,0.2cm>
  \multiput{$\sssize \circ$} at 0 0  1 1  -1 1  0 2  -2 0  2 0 /
  \multiput{$\ssize \bullet$} at  -1 1 /
 \endpicture$} at -1.5 1
\put{$\beginpicture
  \setcoordinatesystem units <0.2cm,0.2cm>
  \multiput{$\sssize \circ$} at 0 0  1 1  -1 1  0 2  -2 0  2 0 /
  \multiput{$\ssize \bullet$} at  0 2 /
 \endpicture$} at -.5 1
\put{$\beginpicture
  \setcoordinatesystem units <0.2cm,0.2cm>
  \multiput{$\sssize \circ$} at 0 0  1 1  -1 1  0 2  -2 0  2 0 /
  \multiput{$\ssize \bullet$} at  0 0 /
 \endpicture$} at 0.5 1
\put{$\beginpicture
  \setcoordinatesystem units <0.2cm,0.2cm>
  \multiput{$\sssize \circ$} at 0 0  1 1  -1 1  0 2  -2 0  2 0 /
  \multiput{$\ssize \bullet$} at  1 1 /
 \endpicture$} at 1.5 1
\put{$\beginpicture
  \setcoordinatesystem units <0.2cm,0.2cm>
  \multiput{$\sssize \circ$} at 0 0  1 1  -1 1  0 2  -2 0  2 0 /
  \multiput{$\ssize \bullet$} at  2 0 /
 \endpicture$} at 2.5  1
\put{$\beginpicture
  \setcoordinatesystem units <0.2cm,0.2cm>
  \multiput{$\sssize \circ$} at 0 0  1 1  -1 1  0 2  -2 0  2 0 /
  \multiput{$\ssize \bullet$} at -2 0  0 0 /
 \endpicture$} at -2.5 2
\put{$\beginpicture
  \setcoordinatesystem units <0.2cm,0.2cm>
  \multiput{$\sssize \circ$} at 0 0  1 1  -1 1  0 2  -2 0  2 0 /
  \multiput{$\ssize \bullet$} at -2 0  1 1  /
 \endpicture$} at -1.5 2
\put{$\beginpicture
  \setcoordinatesystem units <0.2cm,0.2cm>
  \multiput{$\sssize \circ$} at 0 0  1 1  -1 1  0 2  -2 0  2 0 /
  \multiput{$\ssize \bullet$} at  -2 0  2 0 /
 \endpicture$} at -.5 2
\put{$\beginpicture
  \setcoordinatesystem units <0.2cm,0.2cm>
  \multiput{$\sssize \circ$} at 0 0  1 1  -1 1  0 2  -2 0  2 0 /
  \multiput{$\ssize \bullet$} at  0 0  0 2 /
 \endpicture$} at 0.5 2
\put{$\beginpicture
  \setcoordinatesystem units <0.2cm,0.2cm>
  \multiput{$\sssize \circ$} at 0 0  1 1  -1 1  0 2  -2 0  2 0 /
  \multiput{$\ssize \bullet$} at  -1 1  2 0 /
 \endpicture$} at 1.5 2
\put{$\beginpicture
  \setcoordinatesystem units <0.2cm,0.2cm>
  \multiput{$\sssize \circ$} at 0 0  1 1  -1 1  0 2  -2 0  2 0 /
  \multiput{$\ssize \bullet$} at  0 0  2 0 /
 \endpicture$} at 2.5 2

\put{$\beginpicture
  \setcoordinatesystem units <0.2cm,0.2cm>
  \multiput{$\sssize \circ$} at 0 0  1 1  -1 1  0 2  -2 0  2 0 /
  \multiput{$\ssize \bullet$} at  -2 0  0 0  2 0 /
 \endpicture$} at  0 3

\plot -0.4 0.2 -2.3 0.7 /
\plot -0.2 0.2 -1.4 0.7 /
\plot -0.1 0.2 -0.5 0.7 /
\plot .1 0.2 0.5 0.7 /
\plot .2 0.2 1.4 0.7 /
\plot .4 0.2 2.3 0.7 /

\plot -0.4 2.8 -2.3 2.3 /
\plot -0.2 2.8 -1.4 2.3 /
\plot -0.1 2.8 -0.5 2.3 /
\plot .1 2.8 0.5 2.3 /
\plot .2 2.8 1.4 2.3 /
\plot .4 2.8 2.3 2.3 /

\plot -2.6 1.2  -2.6 1.8 /
\plot -2.5 1.2  -1.6 1.8 /
\plot -2.4 1.2  -0.6 1.8 /

\plot -1.6 1.2  -2.5 1.8 /
\plot -1.4 1.2   1.4 1.8 /

\plot -.6 1.2  -1.5 1.8 /
\plot -.5 1.2   0.4 1.8 /
\plot -.4 1.2   1.5 1.8 /

\plot .4 1.2  -2.4 1.8 /
\plot .5 1.2   .5 1.8 /
\plot .6 1.2   2.4 1.8 /

\plot 1.4 1.2  -1.4 1.8 /
\plot 1.6 1.2   2.5 1.8 /

\plot 2.4 1.2  -.4 1.8 /
\plot 2.5 1.2   1.6 1.8 /
\plot 2.6 1.2   2.6 1.8 /

\endpicture}
$$

It turns out that the lattices $\A(\mo\Lambda_n)$ can be identified 
with the lattices of non-crossing
partitions (see Theorem \ref{classical});
these lattices have attracted a lot of interest in recent years
and we will provide in Chapter 4 a short survey on the relevance of these lattices.
The lattices of non-crossing partitions are related to the Coxeter groups
of type $\mathbb A$, corresponding lattices have been defined for any finite
Coxeter group, and corresponding posets for all Coxeter groups, namely the posets
of generalized non-crossing partitions. They will be introduced in Section \ref{main}.

The representation theory of hereditary artin algebras (or, more generally, of hereditary artinian
rings) can be used in order to provide a categorification of the posets of
generalized non-crossing partitions and it is the aims of this chapter
to provide a direct access to this categorification. 
Since the lectures just deal with
artin algebras (and not artinian rings in general), we restrict
to Coxeter groups which arise as Weyl groups for some 
symmetrizable generalized Cartan matrix. We will show in Section \ref{main} that
for any hereditary artin algebra $\Lambda$ of type $\Delta$, the poset $\A(\mo\Lambda)$ is
isomorphic to the poset $\NC(W(\Lambda),c(\Lambda))$, where $W(\Lambda)$ is the
Weyl group of type $\Delta$ and $c(\Lambda)$ is the Coxeter element in $W$ corresponding
to the orientation of $\Delta$ given by $\Lambda$.
	\medskip 

\subsection{The poset $\A(\mo\Lambda)$: Definition and first properties}
Let $\Lambda$ be a hereditary artin algebra. Recall that a full subcategory 
$\mathcal A$ of $\mo\Lambda$ is called 
{\it exceptional} provided it is thick and its quiver has no oriented cyclic paths. We
denote by $\A(\mo\Lambda)$ the set
of exceptional subcategories of $\mo\Lambda$. This is a poset
with respect to set-theoretical inclusion. Equivalently, we may consider $\A(\mo\Lambda)$
as the set of exceptional antichains in $\mo\Lambda$.
In terms of antichains, the
partial ordering has to be formulated as follows: given two exceptional antichains $A,B$, 
we write $A\le B$ provided any element $A_i$ of the antichain $A$ has a filtration with
factors belonging to $B$. 	
	\medskip 

\subsubsection{\bf The layers of $\A(\mo\Lambda)$.}
For any natural number $t$, let
$\A_t(\mo\Lambda)$ be the subset of $\A(\mo\Lambda)$ given by the exceptional subcategories with
rank equal to $t$. Equivalently, we may consider $\A_t(\mo\Lambda)$ as the set of exceptional
antichains of cardinality $t$.

\begin{no-text}
{\rm The poset $\A(\mo\Lambda)$ has a smallest element, namely the zero subcategory $0$, and
a greatest element, namely $\mo\Lambda$ itself. Thus $\A_0(\mo\Lambda) = \{0\}$ and
$\A_n(\mo\Lambda) = \{\mo\Lambda\}$, where $n$ is the rank of $\Lambda$.}
\end{no-text}

\begin{no-text}
The elements of $\A_1(\mo\Lambda)$ are the subcategories $\add X$,
where $X$ is an exceptional module. {\rm Namely, here we deal with antichains of cardinality $1$,
thus with bricks (recall that a brick is a module whose endomorphism ring is a division ring;
in particular, a brick is indecomposable), and to say that the quiver of such an antichain 
has no oriented cyclic paths 
just means that it consists of a single vertex and no arrow, thus we deal with an 
indecomposable module without self-extensions.}
\end{no-text}

If $P$ is a poset, then $p\in P$ is said to have {\it height} $t$
provided any chain $p_0 < p_1 < \cdots < p_s \le p$ has length $s\le t$ and there is such a chain
with $s = t.$

\begin{no-text}
Let $\mathcal A$ belong to $\A(\mo\Lambda)$. Then $\mathcal A$  has height $t$ in $\A(\mo\Lambda)$ 
if and only if $\mathcal A$ has rank $t$ {\rm (thus belongs to
$\A_t(\mo\Lambda)$).}
\end{no-text}

\begin{proof} 
Let 
$$
 \mathcal A_0 \subset \mathcal A_1 \subset \cdots \subset \mathcal A_s \subseteq \mathcal A
$$
be a chain of exceptional subcategories. Then, according to Corollary
\ref{thick-closure2}, we have
$$
 \rank\mathcal A_0 < \rank\mathcal A_1 < \cdots < \rank\mathcal A_s \le \rank\mathcal A = t,
$$
thus $s \le t.$ On the other hand, let $\mathcal A = \mathcal E(A_1,\dots,A_t)$, where 
$A = \{A_1,\dots,A_t\}$ is an antichain of cardinality $t$
whose quiver has no oriented cyclic path.  
Let $\mathcal A(i) = \mathcal E(A_1,\dots,A_i)$. 
Since $\{A_1,\dots,A_i\}$ is an antichain whose quiver has no oriented cyclic path, 
$\mathcal A(i)$ belongs to
$\A_i(\mo\Lambda)$ and there is the inclusion chain
$$
 \mathcal A(0) \subset \mathcal A(1) \subset \cdots \subset \mathcal A(t) = \mathcal A.
$$
Thus shows that $\mathcal A$ has height $t$ as an element of $\A(\mo\Lambda)$.
\end{proof}

\subsubsection{\bf Independence under BGP-reflections.}
As we will see in \ref{change}, changing the orientation of the quiver of $\Lambda$, 
the corresponding posets $\A(\mo\Lambda)$ may be non-isomorphic. But if we use BGP-reflections,
thus changing the orientation of all the arrows at a sink, then we obtain isomorphic posets.
	\medskip

\begin{prop} Let $\Lambda'$ be obtained from $\Lambda$ by changing the orientation
at a sink $x$. Then the posets $\A(\mo\Lambda)$ and $\A(\mo\Lambda')$ are isomorphic.
\end{prop}
	
\begin{proof} 
We have constructed in the proof of Lemma \ref{BGP} a bijection 
$$
 \eta\!:\A(\mo\Lambda) \to \A(\mo\Lambda')
$$ 
as follows: Let $A$ be an antichain in $\mo\Lambda$. If $S(x)$ belongs to $A$, then we may consider
all elements of $A$ also as $\Lambda'$-modules, and we put $\eta(A) = A$. If $S(x)$ does not
belong to $A$, then $\eta(A)$ is obtained from $A$ by applying the BGP-reflection functor $\rho_x$
to all elements of $A$. We have to show that $\eta$ preserves and reflects the partial ordering. 

Assume that $A \le  B$ are antichains in $\mo\Lambda$. If $S(x)$ belongs to $A$, then also to $B$
(namely, $A \le  B$ asserts that $S(x)$ has a filtration with factors in $B$, but $S(x)$ has only
trivial filtrations). Let $\Lambda_0$ be obtained from $\Lambda$ (or $\Lambda'$) by deleting the
vertex $x$. Let $A_i$ be an element of $A$  different from $S(x)$. Since $A \le  B$, the module $A_i$
has a filtration with factors in $B$, but all these factors are $\Lambda_0$-modules, thus $A_i$
considered as a $\Lambda'$-module, has a filtration with factors in $\eta(B)$. Therefore $\eta(A) \le  \eta(B)$.

Next assume that $S(x)$ belongs neither to $A$ nor to $B$. If $A_i$ belongs to $A$, there is a 
filtration with factors $B_{ij}$ in $B$. Since $\sigma_x$ is an exact functor from the
category of $\Lambda$-modules without direct summands of the form $S_\Lambda(x)$ to the 
category of $\Lambda'$-modules without direct summands of the form $S_{\Lambda'}(x)$, we see
that $\sigma_x(A_i)$ has a filtration with factors $\sigma(B_{ij}).$ 

It remains to consider the case that $S(x)$ does not belong to $A$, but belongs to $B$.
As a consequence, $\eta(B) = B.$ 
Let $A_i$ be an element of $A$. There is an exact sequence
$$
 0 \to S(x)^u \to A_i \to M \to 0
$$
such that $M$ is a $\Lambda_0$-module, and an exact sequence
$$
 0 \to M \to \sigma_x(A_i) \to S(x)^v \to 0,
$$
with natural numbers $u,v$.
Let $\mathcal B = \mathcal E(B)$, this is the thick closure of $B$. Since $A \le  B$, we know that
$A_i$ belongs to $\mathcal B$. Since $S(x)$ belongs to $\mathcal B$, also the cokernel $M$
in the first exact sequence belongs to $\mathcal B$. The second sequence now shows that
$\sigma_x(A_i)$ belongs to $\mathcal B$, thus has a filtration with factors in $B$. 
Therefore $\eta(A) \le  \eta(B)$. 
	\smallskip 

This shows that $\eta$ preserves the ordering. A similar argument using the BGP-reflection functors
at sources shows that $\eta^{-1}$ also preserves the ordering. 
\end{proof}
	\medskip 

\subsection{The poset $\A(\mo\Lambda)$: Is it a lattice?} We will show that in case 
$\Lambda$ is representation-finite, $\A(\mo\Lambda)$ is a lattice. If $\Lambda$ is
not representation-finite, then  $\A(\mo\Lambda)$ usually is not a lattice.
	\medskip 

\subsubsection{} First, we consider the representation-finite case.

\begin{theorem}\label{is-lattice}
If $\Lambda$ is representation-finite, then
$\A(\mo\Lambda)$ is a lattice, and in this case, the meet is given by the set-theoretical intersection.
\end{theorem} 
	
\begin{proof}
The set of thick subcategories is closed under (set-theoretical) intersections, 
thus it is a complete lattice. If $\Lambda$ is representation-finite, then any thick
subcategory is exceptional, thus in this case $\A(\mo\Lambda)$ is just the
lattice of thick subcategories. 
\end{proof}

\subsubsection{\bf Example}\label{example-meet}
{\it The set-theoretical intersection of two elements $\mathcal X_1, \mathcal X_2$ 
of
$\A(\mo\Lambda)$ may not belong to $\A(\mo\Lambda)$, whereas the meet $\mathcal X_1\wedge \mathcal X_2$
in $\A(\mo\Lambda)$ still may exist,} as the following example shows:

Let $\Lambda$ be the path algebra of the quiver $\widetilde {\mathbb A}_{2,1}$
$$
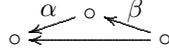
\hbox{\beginpicture
  \setcoordinatesystem units <1cm,1cm>
\multiput{$\circ$} at 0 0  1 0.35  2 0 /
\arr{0.8 0.3}{0.2 0.1}
\arr{1.8 0.1}{1.2 0.3}
\arr{1.8 0}{0.2 0}
\put{$\alpha$} at 0.45 0.4
\put{$\beta$} at 1.6 0.4
\endpicture}
$$
Let $P, P',Q, Q'$ be the indecomposable modules with dimension vectors
$$
 \bdim P = 100,\ \bdim P' = 110,\ \bdim Q = 011,\ \bdim Q' = 001.
$$ 
Then these are exceptional
modules and $\mathcal X = \mathcal E(P,Q)$ and $\mathcal X' = \mathcal E(P',Q')$ are exceptional subcategories. 
The category $\mathcal X$ consists of all representations $M$ of the quiver
such that $M_\alpha$ is bijective, the category 
$\mathcal X'$ consists of all representations $M$ of the quiver
such that $M_\beta$ is bijective. Thus $\mathcal X\cap \mathcal X'$ consists of all representations
$M$ such that both $M_\alpha$ 
and $M_\beta$ are bijective, this is a thick subcategory with infinitely many simple objects,
all having self-extensions (it is just the full subcategory of $\mo\Lambda$ given by the
homogeneous tubes). Since $\mathcal X\cap \mathcal X'$ contains no exceptional module, we have
$\mathcal X\wedge \mathcal X' = 0$ in $\A(\mo\Lambda)$.
	\medskip 

\subsubsection{\bf Example.}\label{example-not-lattice}  
{\it If $\Lambda$ is the path algebra of a quiver of type $\widetilde{\mathbb A}_{2,2}$,
then $\A(\mo\Lambda)$ is not a lattice.} We consider the following quiver:
$$
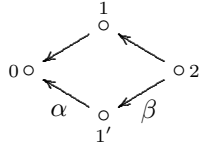
\hbox{\beginpicture
  \setcoordinatesystem units <1cm,.6cm>
\multiput{$\circ$} at 0 0  1 1  1 -1  2 0 /
\arr{0.8 0.8}{0.2 0.2}
\arr{0.8 -.8}{0.2 -.2}
\arr{1.8 0.2}{1.2 0.8}
\arr{1.8 -.2}{1.2 -.8}
\put{$\alpha\strut$} at 0.4 -.86
\put{$\beta\strut$} at 1.6 -.9
\put{$\ssize 0$} at -.2 0
\put{$\ssize 1$} at 1 1.4
\put{$\ssize 1'$} at 1 -1.4
\put{$\ssize 2$} at 2.2 0
\endpicture}
$$
We consider the following elements of $\A(\mo\Lambda)$:
$$
\hbox{\beginpicture
  \setcoordinatesystem units <1cm,1cm>
\put{\beginpicture
\plot 0 0.3  0 1.7 /
\plot 2 0.3  2 1.7 /
\plot 0.3 0.3  1.7 1.7 /
\plot 1.7 0.3  1.2 0.8 /
\plot 0.8 1.2  0.3 1.7 /

\put{$\mathcal E\Bigl(\smallmatrix  &1\cr
                               0 & & 0 \cr
                                 &0 \endsmallmatrix\Bigr) = \mathcal X_1$} at -.9 0
\put{$\mathcal X_2 = \mathcal E\Bigl(\smallmatrix  &0\cr
                               1 & & 1 \cr
                                 &1 \endsmallmatrix\Bigr)$} at 3 0
\put{$\mathcal E\Bigl(\smallmatrix  &0\cr
                               1 & & 0 \cr
                                 &0 \endsmallmatrix,\
                  \smallmatrix  &1\cr
                               0 & & 0 \cr
                                 &0 \endsmallmatrix,\
                  \smallmatrix  &0\cr
                               0 & & 1 \cr
                                 &1 \endsmallmatrix\Bigr) = \mathcal Y_1$} at -1.9 2
\put{$\mathcal Y_2 =  \mathcal E\Bigl(\smallmatrix  &0\cr
                               1 & & 0 \cr
                                 &1 \endsmallmatrix,\
                  \smallmatrix  &1\cr
                               0 & & 0 \cr
                                 &0 \endsmallmatrix,\
                  \smallmatrix  &0\cr
                               0 & & 1 \cr
                                 &0 \endsmallmatrix\Bigr)$} at 3.9 2
\put{} at 0 2.5
\endpicture} at 0 0
\endpicture}
$$
thus, $\mathcal X_1 = \add S(1),\ \mathcal X_2 = \add \tau S(1)$, and $\mathcal Y_2 = S(1')^\perp,\
\mathcal Y_1 = (\tau S(1'))^\perp$.

Let us show that $\mathcal X_1,\mathcal X_2$ have no supremum in $\A(\mo\Lambda)$. Assume, for the contrary, that $\mathcal C$ is the supremum, 
then $\mathcal X_1 \subset \mathcal C \subset \mathcal Y_1$ shows that
$\mathcal C$ has rank $2$. But there is no exceptional subcategory $\mathbb C$ 
which has
rank $2$ and contains $\mathcal X_1$ and $\mathcal X_2.$ Namely, if $C_1,C_2$ are 
the simple objects in $\mathcal C$, then the objects in
$\mathcal X_1$ and in $\mathcal X_2$ must have filtrations with factors of the form $C_1,C_2$, but this
implies that $C_1,C_2$ are the indecomposable modules in $\mathcal X_1$ and $\mathcal X_2$, and thus
the quiver of $\mathcal C$ is the oriented cycle with two vertices and two arrows. 

It follows that $\mathcal Y_1,\mathcal Y_2$ have no infimum in $\A(\mo\Lambda)$. Namely, the infimum 
of $\mathcal Y_1,\mathcal Y_2$ would be an exceptional subcategory of rank $2$ and 
containing $\mathcal X_1,\mathcal X_2$, thus the supremum of $\mathcal X_1,\mathcal X_2.$
	\medskip

One may enlarge $\A(\mo\Lambda)$ by inserting a thick subcategory $\mathcal C$ which is both a supremum
of $\mathcal X_1,\mathcal X_2$ and an infimum of $\mathcal Y_1,\mathcal Y_2$:
$$
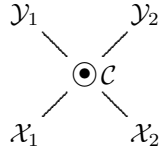
\hbox{\beginpicture
  \setcoordinatesystem units <.8cm,.8cm>
\put{\beginpicture
\plot 0.3 0.3  0.7 0.7 /
\plot 1.7 0.3  1.3 0.7 /
\plot 0.3 1.7  0.7 1.3 /
\plot 1.7 1.7  1.3 1.3 /

\put{$\mathcal X_1$} at 0 0
\put{$\mathcal X_2$} at 2 0
\put{$\mathcal Y_1$} at 0 2
\put{$\mathcal Y_2$} at 2 2
\put{$\bigcirc$} at 1 1
\put{$\bullet$} at 1 1
\put{$\mathcal C$} at 1.4 1
\endpicture} at 0 0
\endpicture}
$$
There are many possibilities to do so: the supremum of $\mathcal X_1,\mathcal X_2$ in the lattice of thick 
subcategories is the extension closure of $\mathcal X_1,\mathcal X_2$, this is just a single tube of rank 2
in $\mo\Lambda$ (namely the tube containing the simple representation $S(1)$).
The infimum of $\mathcal Y_1,\mathcal Y_2$ in the lattice of thick subcategories
consists of the representations $M$ with $M_\alpha$ and $M_\beta$ both being bijective, 
these are all the regular modules without indecomposable direct summands in the rank 2 tube 
containing the simple representation $S(1')$.

\begin{lemma}\label{meet} If  $\mathcal X$ is an exceptional subcategory
of rank at most $2$, then the meet $\mathcal X\wedge \mathcal Y$ exists,
for any exceptional subcategory $\mathcal Y$.
\end{lemma}
	
\begin{proof} 
If $\mathcal X\cap \mathcal Y$ contains no exceptional modules, then clearly
$\mathcal X\wedge \mathcal Y = 0$. If $\mathcal X\cap \mathcal Y$ contains precisely one 
exceptional module $M$, then $\mathcal X\wedge \mathcal Y = \add M$. Finally, if
If $\mathcal X\cap \mathcal Y$ contains at least two  
exceptional modules $M_1,M_2$, then the condition $\rank\mathcal X \le 2$ 
implies that $\rank \mathcal X = 2$ and that $\mathcal X$ has to be the thick closure of $M_1,M_2$
(see N\,\ref{rank-2}), thus $\mathcal X \subseteq \mathcal Y$ and therefore
$\mathcal X\wedge \mathcal Y = \mathcal X$.
\end{proof} 	

By duality we see: {\it If $\Lambda$ has rank $n$, and $\mathcal X$ is an exceptional subcategory
of $\mo\Lambda$ of rank at least $n-2,$ then the join $\mathcal X \vee \mathcal Y$ exists for any
exceptional subcategory $\mathcal Y$.}

\begin{cor} If  $\Lambda$ is a hereditary artin algebra of rank at most 3, 
then the
poset $\A(\mo\Lambda)$ always 
is a lattice.
\end{cor}
	
The Lemma \ref{meet} asserts the existence of meets, by duality,
we also have joins.
	\medskip 

\subsubsection
{\bf Example.}\label{example-lattice}
 {\it If $Q$ is the quiver of type $\widetilde{\mathbb A}_{3,1}$, 
$$
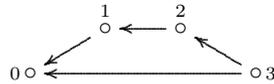
\hbox{\beginpicture
  \setcoordinatesystem units <1cm,.6cm>
\multiput{$\circ$} at 0 0  1 1  2 1  3 0 /
\arr{0.8 0.8}{0.2 0.2}
\arr{1.8 1}{1.2 1}
\arr{2.8 0.2}{2.2 0.8}
\arr{2.8 0}{.2 0}
\put{$\ssize 0$} at -.2 0
\put{$\ssize 1$} at 1 1.4
\put{$\ssize 2$} at 2 1.4
\put{$\ssize 3$} at 3.2 0
\endpicture}
$$
then $\A(\mo kQ)$ is a lattice.}
	\medskip 

Proof. 
We want to show that any two elements $\mathcal X\neq  \mathcal Y$ 
in $\A(\mo\Lambda)$ have a join
(the self duality then shows that any two elements also have a meet). 

Of course, if $\mathcal X$
has rank $3$, then any exceptional subcategory containing $\mathcal X$ and $\mathcal Y$
has to have rank at least $4$, thus must be equal to $\mo\Lambda$. Thus we may
assume that both $\mathcal X, \mathcal Y$ have rank at most $2$. As we also know, 
we can assume that neither $\mathcal X$ not $\mathcal Y$ has rank 2.
Thus it remains to consider the case that 
both subcategories $\mathcal X,\mathcal Y$ have rank 1, thus we deal with two exceptional
modules $X,Y$. Let us denote by $\mathcal Z$ the thick closure of $X,Y$.

If at least one of the modules $X,Y$ is preprojective, then we show that $\mathcal Z$
 is exceptional 
(and therefore the join of $\add X$ and $\add Y$). 
Using reflection functors, we can assume that one of these modules, say $X$,
is simple projective. 
Let $Y'$ be the factor module of $Y$ modulo the trace of $X$ in $Y$. Then $Y'$
is a representation of a quiver of type $\mathbb A_3$, thus its thick closure
$\mathcal Z'$ is exceptional. The simple objects in $\mathcal Z'$ together with $S$ are 
an exceptional antichain and these are the simple objects of $\mathcal Z$,
thus $\mathcal Z$ is exceptional.
By duality, we similarly can assume that none of the modules $X,Y$ is preinjective.

Thus, assume that both $X,Y$ are regular modules. In case they form an
exceptional pair, then again $\mathcal Z$ is exceptional. Thus, at least one of the
modules $X,Y$ has regular rank $2$. If both have regular rank $2$, then
$\mathcal Z$ contains all three non-homogeneous simple regular modules (the modules
with dimension vector $1001, 0100, 0010$), thus $\mathcal Z$ is the tube which contains $X,Y$
(of course, this is not an exceptional subcategory). If $\mathcal C$ is a thick subcategory
which contains $\mathcal Z$ properly, then $\mathcal C$ has to have rank $4$, thus it is
just $\mathcal C = \mo\Lambda$. This shows that $\mo\Lambda$ is the join of
$\add X$ and $\add Y$.

It remains to consider the case that one of the modules $X,Y$ has regular rank 1, 
the other regular rank 2 and that this is an orthogonal pair. Thus, up to duality, there are
two cases: $\{X,Y\} = \{1001,0110\}$ and $\{X,Y\} = \{0100,1011\}$. In the first case,
there is just one exceptional subcategory $\mathcal C$ of rank 3 which contains $X,Y$, namely the
thick closure of $\{1000,0110,0001\}$. Similarly, in the second case,
there is just one exceptional subcategory $\mathcal C$ of rank 3 which contains $X,Y$, namely the
thick closure of $\{1000,0100,0011\}$. In both cases the respective subcategory 
$\mathcal C$ is the join of 
$\add X$ and $\add Y$.
	\medskip 

The examples \ref{example-not-lattice} and \ref{example-lattice} concern 
quivers with the same underlying graph,
they differ only by the choice of the orientation (thus by the choice of a Coxeter element
in the Weyl group). This shows:
	\medskip

\begin{no-text}\label{change}
If $Q, Q'$ are quivers with the same underlying graph, but different orientations,
the posets $\A(\mo kQ)$ and $\A(\mo kQ')$ may be non-isomorphic.
\end{no-text}
	\medskip

\subsection{The poset $\A(\mo\Lambda)$: Intervals}
Given elements $a \le b$ in a poset $P$, we denote by $[a,b] = \{p\in P\mid a \le p \le b \}$ 
the (closed) interval of elements between $a$ and $b$. The aim of this section is to show
that any interval $[\mathcal A,\mathcal B]$
in $\A(\mo\Lambda)$ is isomorphic  to a poset of the form $\A(\mathcal C)$, where
$\mathcal C$ is an element of $\A(\mo\Lambda)$ (thus to a poset which is again of the form $\A(\mo\Lambda')$
for some hereditary artin algebra $\Lambda'$), see Theorem \ref{interval}.
	\medskip 

\subsubsection{} We begin with an analysis of the poset structure of $\A(\mo\Lambda)$.

\begin{lemma}\label{lemma1}
 Let $\mathcal A \subseteq \mathcal V$ be exceptional subcategories of $\mo\Lambda,$
let $a$ be the rank of $\mathcal A$ and $v$ the rank of $\mathcal V$.
Then both subcategories $\mathcal V\cap {}^\perp\!\!\mathcal A$ and $\mathcal V\cap \mathcal A^\perp$ 
are exceptional subcategories of rank $v-a$, they are the meet 
$\mathcal V\wedge {}^\perp\!\!\mathcal A$, and $\mathcal V\wedge \mathcal A^\perp$ in $\A(\mo\Lambda)$, respectively.
\end{lemma}

\begin{proof} We consider $\mathcal A^\perp$. 
In order to show that 
$\mathcal V\cap \mathcal A^\perp$ is an exceptional subcategory of rank $v-a,$ we can assume
that $\mathcal V = \mo\Lambda'$ for some hereditary artin algebra $\Lambda'$.  

Thus, we deal with an exceptional subcategory $\mathcal A$ of
$\mo\Lambda'$, and $\mathcal A^\perp \cap \mathcal V$
is just the right perpendicular
category  for $\mathcal A$ inside $\mo\Lambda'$, thus let us write 
 $\mathcal V  \cap \mathcal A^\perp = \mathcal A^{\perp(\mo \Lambda')}$. 
The claim is that $\mathcal A^{\perp(\mo \Lambda')}$ is exceptional and 
that the rank of $\mathcal A^{\perp(\mo \Lambda')}$
is $\rank(\Lambda') - \rank(\mathcal A) = v-a$. 
But this has been shown in Chapter 2. 

Since $\mathcal V\cap \mathcal A^\perp$ is an exceptional subcategory of $\mo\Lambda$,
it is the infimum $\mathcal V\wedge \mathcal A^\perp$
of $\mathcal V$ and $\mathcal A^\perp$ in $\A(\mo\Lambda)$.

In the same way (or using duality) one deals with ${}^\perp\!\!\mathcal A$.
\end{proof} 
	
\begin{lemma}\label{lemma2}
 Let $\mathcal A$ and $\mathcal B$ be exceptional subcategories of $\mo\Lambda$,
of rank $a$ and $b$, respectively. We assume that $\mathcal A \subseteq \mathcal B^\perp$
or, equivalently, that $\mathcal B \subseteq {}^\perp\!\!\mathcal A$.
Then the join $\mathcal A\vee \mathcal B$ exists in $\A(\mo\Lambda)$, it is an exceptional
subcategory of rank $a+b$, and we have
$$
 \mathcal A\vee \mathcal B = ({}^\perp\!\!\mathcal A\cap {}^\perp \mathcal B)^\perp = 
{}^\perp(\mathcal A^\perp \cap \mathcal B^\perp).
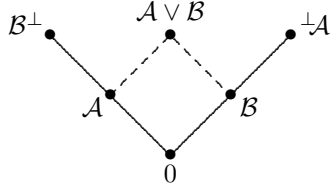
$$
\end{lemma}

 $$
\hbox{\beginpicture
  \setcoordinatesystem units <.8cm,.8cm>
\multiput{$\bullet$} at 0 2  1 1  2 0  3 1  4 2  2 2 /
\plot 0 2  2 0  4 2 /
\setdashes <1mm>
\plot 1 1  2 2  3 1 /
\put{$\mathcal B^\perp$} at -.4 2.2 
\put{$\mathcal A$} at .7 0.8
\put{$\mathcal A\vee \mathcal B$} at 2 2.4 
\put{$0$} at 2 -.3
\put{$\mathcal B$} at 3.3 0.8
\put{${}^\perp\!\!\mathcal A$} at 4.4 2.2 
\endpicture}
$$

\begin{proof}   We assume that $\mathcal A \subseteq \mathcal B^\perp$ (and thus we also have
 ${}^\perp\!\!\mathcal A \supseteq {}^\perp(\mathcal B^\perp) = \mathcal B$).

Note that $\mathcal B^\perp$ is exceptional with rank $n-b$.
According to Lemma \ref{lemma1},
we know that $\mathcal B^\perp \cap \mathcal A^\perp$ is an exceptional subcategory of rank
$n-b-a$. It follows that ${}^\perp(\mathcal B^\perp \cap \mathcal A^\perp)$ is an exceptional
subcategory of rank $n-(n-b-a) = a+b.$ 

We claim that ${}^\perp(\mathcal A^\perp \cap \mathcal B^\perp)$ is the join of $\mathcal A$ and
$\mathcal B$ in $\A(\mo\Lambda)$. First, we have to show that both $\mathcal A$ and $\mathcal B$ are contained in
${}^\perp(\mathcal A^\perp \cap \mathcal B^\perp)$.
Since $\mathcal A^\perp \supseteq \mathcal A^\perp\cap \mathcal B^\perp$,
we see that $\mathcal A = {}^\perp(\mathcal A^\perp) \subseteq {}^\perp(\mathcal A^\perp\cap \mathcal B^\perp)$;
 similarly, we have $\mathcal B \subseteq {}^\perp(\mathcal A^\perp\cap \mathcal B^\perp)$.
Now assume that there is an exceptional subcategory $\mathcal C$ with $\mathcal A \subseteq \mathcal C$
and $\mathcal B \subseteq \mathcal C.$ It follows from $\mathcal A \subseteq \mathcal C$ that
$\mathcal A^\perp \supseteq \mathcal C^\perp$. Similarly, we have $\mathcal B^\perp \supseteq \mathcal C^\perp$.
Thus $\mathcal A^\perp\cap \mathcal B^\perp \supseteq \mathcal C^\perp$,
and therefore ${}^\perp(\mathcal A^\perp\cap \mathcal B^\perp) \subseteq {}^\perp(\mathcal C^\perp) = \mathcal C$.
This shows that ${}^\perp(\mathcal A^\perp \cap \mathcal B^\perp) = \mathcal A\vee \mathcal B$.

In the same way, one shows that $({}^\perp\!\!\mathcal A\cap {}^\perp \mathcal B)^\perp = \mathcal A\vee \mathcal B$.
\end{proof}
	\bigskip

\begin{prop} Let $\mathcal A$ be an exceptional subcategory of $\mo\Lambda$.
Then the poset $[\mathcal A,\mo\Lambda]$ is isomorphic to the poset $\A({}^\perp\!\!\mathcal A)$
using the maps
$$
 \mathcal V \mapsto \mathcal V\cap{}^\perp\!\!\mathcal A, \quad \text{and} \quad
 \mathcal U \mapsto \mathcal U\vee \mathcal A. 
$$
for $\mathcal V$ in $[\mathcal A,\mo\Lambda]$ and $\mathcal U$ in $\A({}^\perp\!\!\mathcal A)$.
\end{prop}
	
 $$
\hbox{\beginpicture
  \setcoordinatesystem units <.8cm,.8cm>
\put{\beginpicture
\multiput{$\bullet$} at 1 0  0 1  1 2  2 1  2 3  3 2 /
\plot 1 0  0 1  2 3  3 2  1 0 /
\setdashes <1mm>
\arr{1 2}{1.8 1.2}
\put{$0$} at 1.2 -0.2  
\put{$\mathcal A$} at -.4 1
\put{$\mathcal V$} at .6 2.1
\put{$\mo\Lambda$} at 2 3.4
\put{${}^\perp\!\!\mathcal A$\strut} at 3.4 2  
\put{$\mathcal V \cap {}^\perp\!\!\mathcal A$} at 2.7 0.8   

\endpicture} at 0 0
\put{\beginpicture
\multiput{$\bullet$} at 1 0  0 1  1 2  2 1  2 3  3 2 /
\plot 1 0  0 1  2 3  3 2  1 0 /
\setdashes <1mm>
\arr{2 1}{1.2 1.8}
\put{$0$} at 1.2 -0.2  
\put{$\mathcal A$} at -.4 1
\put{$\mathcal U \vee \mathcal A$} at .2 2.2
\put{$\mo\Lambda$} at 2 3.4
\put{${}^\perp\!\!\mathcal A$\strut} at 3.4 2  
\put{$\mathcal U $} at 2.3 0.8   
\endpicture} at 5 0
\endpicture}
$$

\begin{proof} Lemma \ref{lemma1}
asserts that given $\mathcal V$ in $[\mathcal A,\mo\Lambda]$, the subcategory
$\mathcal V\cap{}^\perp\!\!\mathcal A$ belongs to $\A(\mathcal {}^\perp A)$. Of course, the map 
$\mathcal V \mapsto \mathcal V\cap{}^\perp\!\!\mathcal A$ is order preserving. Lemma \ref{lemma2}
asserts
that given $\mathcal U$ in $\A(\mathcal {}^\perp A)$, the join $\mathcal U\vee \mathcal A$ exists, and, of course
$\mathcal U\vee\mathcal A$ belongs to the interval $[\mathcal A,\mo\Lambda]$. Also here we see immediately
that the map $\mathcal U \mapsto \mathcal U\vee \mathcal A$ is order preserving. 

It remains to be shown that the maps defined here are inverse to each other, thus
we have to show that
$$
 {}^\perp((\mathcal V\cap{}^\perp\!\!\mathcal A)^\perp\cap\mathcal A^\perp) = \mathcal V\quad
\text{and}\quad 
 {}^\perp(\mathcal U^\perp\cap\mathcal A^\perp)\cap{}^\perp\!\!\mathcal A = \mathcal U
$$
for $\mathcal V$ in $[\mathcal A,\mo\Lambda]$ and $\mathcal U$ in $\A(\mathcal {}^\perp A)$.

First, let us show the inclusion 
${}^\perp((\mathcal V\cap{}^\perp\!\!\mathcal A)^\perp\cap\mathcal A^\perp) \subseteq \mathcal V$.
The inclusion $\mathcal V\cap{}^\perp\!\!\mathcal A \subseteq \mathcal V$ implies that 
$(\mathcal V\cap{}^\perp\!\!\mathcal A)^\perp \supseteq \mathcal V^\perp$, thus
$(\mathcal V\cap{}^\perp\!\!\mathcal A)^\perp\cap\mathcal A^\perp \supseteq \mathcal V^\perp\cap\mathcal A^\perp$
and $\mathcal V^\perp\cap\mathcal A^\perp = \mathcal V^\perp$, since $\mathcal A \subseteq \mathcal V$
so that $\mathcal A^\perp \supseteq \mathcal V^\perp.$
But 
$(\mathcal V\cap{}^\perp\!\!\mathcal A)^\perp\cap\mathcal A^\perp \supseteq \mathcal V^\perp$
implies that
${}^\perp((\mathcal V\cap{}^\perp\!\!\mathcal A)^\perp\cap\mathcal A^\perp) \subseteq 
{}^\perp(\mathcal V^\perp) = \mathcal V.$ In addition, we show that the exceptional subcategories 
${}^\perp((\mathcal V\cap{}^\perp\!\!\mathcal A)^\perp\cap\mathcal A^\perp)$ and $\mathcal V$ have the same
rank (namely $v$), as a consequence, they have to be equal.
 
Similarly, we show the inclusion $\mathcal U \subseteq 
{}^\perp(\mathcal U^\perp\cap\mathcal A^\perp)\cap{}^\perp\!\!\mathcal A$. By assumption, $\mathcal U \subseteq
{}^\perp\!\!\mathcal A$, thus we only have to show that $\mathcal U \subseteq 
{}^\perp(\mathcal U^\perp\cap\mathcal A^\perp)$. But 
$\mathcal U^\perp \supseteq \mathcal U^\perp\cap\mathcal A^\perp,$ 
therefore $\mathcal U = {}^\perp(\mathcal U^\perp) \subseteq {}^\perp(\mathcal U^\perp\cap\mathcal A^\perp)$. 
And again we show that we deal with exceptional subcategories of the same rank:
Since $\mathcal U \subseteq {}^\perp\!\!\mathcal A$, Lemma \ref{lemma2}
asserts that the rank of
$\mathcal A\vee \mathcal U = {}^\perp(\mathcal U^\perp\cap\mathcal A^\perp)$ is $a+u$. Now
$\mathcal A\vee \mathcal U \supseteq \mathcal A$, thus by Lemma
\ref{lemma1} the rank of
$(\mathcal A\vee \mathcal U)\cap{}^\perp\!\!\mathcal A$ is $a+u-a = u.$ This completes the proof. 
\end{proof}

 There is the corresponding assertion invoking $\mathcal A^\perp$:

\begin{prop} Let $\mathcal A$ be an exceptional subcategory of $\mo\Lambda$.
Then the poset $[\mathcal A,\mo\Lambda]$ is isomorphic to the poset $\A(\mathcal A^\perp )$
using the maps
$$
 \mathcal V \mapsto \mathcal V\cap\mathcal A^\perp , \quad \text{and} \quad
 \mathcal U \mapsto \mathcal U\vee \mathcal A. 
$$
for $\mathcal V$ in $[\mathcal A,\mo\Lambda]$ and $\mathcal U$ in $\A(^\perp\mathcal A)$.
\end{prop}
	
 $$
\hbox{\beginpicture
  \setcoordinatesystem units <.8cm,.8cm>
\put{\beginpicture
\multiput{$\bullet$} at 2 0  0 2  1 1  2 2  2 2  3 1  1 3 /
\plot 2 0  0 2  1 3  3 1  2 0 /
\setdashes <1mm>
\arr{2 2}{1.2 1.2}
\put{$0$} at 2.2 -0.2  
\put{$\mathcal A$} at 3.4 1 
\put{$\mathcal V$} at 2.3 2.1
\put{$\mo\Lambda$} at 1 3.4
\put{$\mathcal A^\perp$\strut} at -.5 2 
\put{$\mathcal V\cap \mathcal A^\perp$} at 0.2 0.8

\endpicture} at 0 0
\put{\beginpicture
\multiput{$\bullet$} at 2 0  0 2  1 1  2 2  2 2  3 1  1 3 /
\plot 2 0  0 2  1 3  3 1  2 0 /
\setdashes <1mm>
\arr{1 1}{1.8 1.8}
\put{$0$} at 2.2 -0.2  
\put{$\mathcal A$} at 3.4 1 
\put{$\mathcal U\vee \mathcal A$} at 2.9 2.1
\put{$\mo\Lambda$} at 1 3.4
\put{$\mathcal A^\perp$\strut} at -.5 2 
\put{$\mathcal U$} at 0.7 0.8   

\endpicture} at 5 0
\endpicture}
$$
\begin{theorem}\label{interval}
 Let $\mathcal A\subseteq \mathcal B$  be exceptional subcategories of $\mo\Lambda$.
Then there is a poset isomorphism
$$
 [\mathcal A,\mathcal B] \longrightarrow \A(\mathcal B\cap \mathcal A^\perp)
$$ 
defined by $\mathcal V \mapsto \mathcal V\cap \mathcal A^\perp$ 
for $\mathcal A \subseteq \mathcal V \subseteq \mathcal B$, its inverse sends $\mathcal U$ to $\mathcal U\vee \mathcal A$.

Similarly,  there is a poset isomorphism
$$
 [\mathcal A,\mathcal B] \longrightarrow \A(\mathcal B\cap {}^\perp\!\!\mathcal A)
$$ 
defined by $\mathcal V \mapsto \mathcal V\cap {}^\perp\!\!\mathcal A$ 
for $\mathcal A \subseteq \mathcal V \subseteq \mathcal B$, its inverse sends $\mathcal U$ to $\mathcal U\vee \mathcal A$.
\end{theorem}
	
This shows 
that any interval $[\mathcal A,\mathcal B]$
in $\A(\mo\Lambda)$ is isomorphic  to a poset of the form $\A(\mathcal C)$, where
$\mathcal C$ is an element of $\A(\mo\Lambda)$, namely $\mathcal C = \mathcal B\cap \mathcal A^\perp$ or 
$\mathcal C = \mathcal B\cap {}^\perp\!\!\mathcal A$. 
		\medskip 

Let us discuss two special cases.
	\medskip

\subsubsection{\bf  Neighbors.}\label{neighbors}
Recall that a pair $p,p'$ of elements of a poset $P$ are called {\it neighbors}
provided $p<p'$ and such that the interval $[p,p']$ contains no other element of $P$.
A direct consequence of the proposition is the following assertion:
	
\begin{prop}
Let $\mathcal A \subseteq \mathcal B$ be exceptional subcategories of $\mo\Lambda$. 
Then these are neighbors in the poset $\A(\mo\Lambda)$ 
if and only if
the rank of $\mathcal A$ and $\mathcal B$ differs by $1$.

If $\mathcal A\subset \mathcal B$ are neighbors in $\A(\mo\Lambda)$, then there are
(uniquely determined) exceptional modules $M_{\mathcal A}^{\mathcal B}$ and 
$N_{\mathcal A}^{\mathcal B}$ such that
$$
 \add M_{\mathcal A}^{\mathcal B} = \mathcal B\cap \mathcal A^\perp \quad \text{and}
  \quad
 \add N_{\mathcal A}^{\mathcal B} = \mathcal B\cap {}^\perp\!\!\mathcal A.
$$
\end{prop}

\begin{proof} If $\mathcal A\subset \mathcal B$ are neighbors, any of the subcategories
$\mathcal B\cap \mathcal A^\perp$ and $\mathcal B\cap {}^\perp\!\!\mathcal A$ is  exceptional and of rank $1$,
thus the additive category generated by an exceptional module.
\end{proof} 
	
Section \ref{max-chains} will be devoted to the study of maximal chains in $\A(\mo\Lambda)$. Given a maximal
chain, we will deal with the corresponding sequence of modules 
$M_{\mathcal A}^{\mathcal B},N_{\mathcal A}^{\mathcal B}$ where $\mathcal A\subset
\mathcal B$ belong to the chain. Also, as an example, we will exhibit all the modules
$M_{\mathcal A}^{\mathcal B},N_{\mathcal A}^{\mathcal B}$ for the case of $\Lambda = \Lambda_3$.
	\bigskip 

\subsubsection
{\bf Intervals of height 2.} Such an interval is isomorphic to $\A(\mo\Lambda')$, where
$\Lambda'$ is a hereditary artin algebra with precisely two simple modules. In case $\Lambda'$ 
is representation-finite, it is of type $\mathbb A_1\times\mathbb A_1,\ \mathbb A_2, \mathbb B_2,$ or $\mathbb G_2.$
In case $\Lambda'$ is representation-infinite, it has an infinite preprojective component
and an infinite preinjective component, and the exceptional modules are just the indecomposable
modules which are preprojective or preinjective. 

Here are these possibilities.
 $$
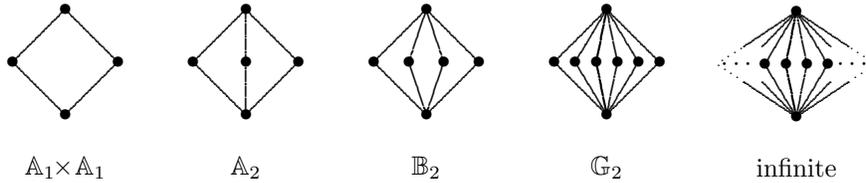
\hbox{\beginpicture
  \setcoordinatesystem units <.6cm,.7cm>
\put{\beginpicture
  \setcoordinatesystem units <.7cm,.7cm>
\multiput{$\bullet$} at 0 0  -1 1  1 1  0 2 /
\plot 0 0  -1 1  0 2  1 1  0 0 /
\put{$\mathbb A_1\!\!\times\! \mathbb A_1$} at 0 -1 
\endpicture} at 0 0
\put{\beginpicture
  \setcoordinatesystem units <.7cm,.7cm>
\multiput{$\bullet$} at 0 0  -1 1  1 1  0 2  0 1 /
\plot 0 0  -1 1  0 2  1 1  0 0 /
\plot 0 0  0 2 /
\put{$\mathbb A_2$} at 0 -1 
\endpicture} at 4 0
\put{\beginpicture
  \setcoordinatesystem units <.7cm,.7cm>
\multiput{$\bullet$} at 0 0  -1 1  1 1  0 2  -.33 1  0.33 1 /
\plot 0 0  -1 1  0 2  1 1  0 0 /
\plot 0 0  -.33 1  0 2  .33 1  0 0 /
\put{$\mathbb B_2$} at 0 -1 
\endpicture} at 8 0
\put{\beginpicture
  \setcoordinatesystem units <.7cm,.7cm>
\multiput{$\bullet$} at 0 0  -1 1  1 1  0 2  -0.6 1  -.2 1  .2 1  .6 1 /
\plot 0 0  -1 1  0 2  1 1  0 0 /
\plot 0 0  -.6 1  0 2  -.2 1  0 0  .2 1  0 2  -.2 1  0 0  .6 1  0 2 /
\put{$\mathbb G_2$} at 0 -1 
\endpicture} at 12 0
\put{\beginpicture
  \setcoordinatesystem units <.7cm,.7cm>
\multiput{$\bullet$} at  0 0   0 2  -0.6 1  -.2 1  .2 1  .6 1  /
\plot   -1 .667  0 0  1 .667  /
\plot   -1 1.333  0 2  1 1.333  /
\plot 0 0  -.6 1  0 2  -.2 1  0 0  .2 1  0 2  -.2 1  0 0  .6 1  0 2 /
\plot -.65 1.333  0 2  .65 1.333 /
\plot -.65  .667  0 0  .65  .667 /
\setdots <1mm>
\plot -1 .667  -1.5 1  -1 1.333 /
\plot  1 .667   1.5 1   1 1.333 /
\multiput{$\cdots$} at -1.1 1  1.1 1 /
\put{infinite} at 0 -.95
\endpicture} at 16.2 0
\endpicture}
$$
In particular, an interval of height 2 is never a chain. 
	\bigskip

{\bf Remark.} Let us stress already here that the posets $\A(\mo\Lambda)$
with $\Lambda$ of rank $2$ (and thus all intervals of height 2) are not only
lattices, but come equipped with a cyclic rotation, in case $\Lambda$ is 
representation-finite, and with a totally ordering of $\A_1(\mo\Lambda)$
in case $\Lambda$ is representation-infinite. The neighbors in $\A_1(\mo\Lambda)$
are just the exceptional pairs of $\Lambda$-modules. We will use this in the proof
of Theorem \ref{braid-exc}.

A direct way to see the cyclic or total ordering of $\A_1(\mo\Lambda)$
is to look at the Auslander-Reiten quiver of $\Lambda$. 
If $\Lambda$ is not connected (the case $\mathbb A_1\times \mathbb A_1$), the
indecomposable modules are the simple modules, say $S$ and $T$ and both
$(S,T)$ and $(T,S)$ are exceptional pairs.
If $\Lambda$ is representation-finite and connected, say with $m$
indecomposable modules (here $m = 3,4,$ or $6$), we may order them
as $(X_1,\dots, X_m)$ such that $\Hom(X_i,X_{i+1})\neq 0$ for $1\le i < m.$
Then the pairs $(X_i,X_{i+1})$ with $1\le i < m$ 
as well as the pair $(X_m,X_1)$ are the exceptional pairs.
Finally, if $\Lambda$ is representation-infinite, the exceptional modules are
just the indecomposable modules which are preprojective or preinjective. We label
the indecomposable preprojective modules as $X_0,X_1,\dots,X_i,\dots,$
the indecomposable preinjective modules as $\dots,X_{i},\dots, X_{-2}, X_{-1}$
such that $\Hom(X_i,X_{i+1})\neq 0$ for $i\le -2$ and for $i\ge 0.$ Then
the pairs $(X_i,X_{i+1})$ with $i\in \mathbb Z$ are the exceptional pairs.
	\medskip

\subsection{The poset $\A(\mo\Lambda)$: Automorphisms and anti-automorphisms}
As we will see, the poset $\A(\mo\Lambda)$ is self-dual and has, in general,
many automorphisms.
	\medskip 

\subsubsection{}\label{delta}
 If $\mathcal A$ is an exceptional subcategory of $\mo\Lambda$, let 
$\delta(\mathcal A) = \mathcal A^\perp$. 

\begin{theorem}\label{anti} The map $\delta$ is a poset anti-automorphism of
$\A(\mo\Lambda)$.
\end{theorem} 
	
\begin{proof} 
Of course, we always have $\mathcal A \subseteq {}^\perp(\mathcal A^\perp)$. In case $\mathcal A$
is an exceptional subcategory, we have the equality 
$\mathcal A = {}^\perp(\mathcal A^\perp)$ (otherwise usually not, see the 
note N\,\ref{thick-perp-perp}). Namely, we have shown in
Chapter 2 that in this case $(\mathcal A^\perp,\mathcal A)$ is a perpendicular pair, but this just
means that $\mathcal A = {}^\perp(\mathcal A^\perp)$. Similarly, for $\mathcal B$ an exceptional subcategory, 
we have $\mathcal B = ({}^\perp\mathcal B)^\perp$, thus $\delta$ is a bijective map
from $\A(\mo\Lambda)$ to itself 
with inverse $\delta^{-1}(\mathcal A) = {}^\perp\mathcal A.$ 

Of course, $\delta$ reverses the partial ordering: $\mathcal A \subseteq \mathcal B$ implies that 
$\mathcal A^\perp \supseteq \mathcal B^\perp$. Similarly, $\delta^{-1}$ reverses the partial ordering.
Thus $\delta$ is a poset anti-automorphism.
\end{proof} 
	
{\bf Remark.} In the case $\Lambda = \Lambda_n$, the anti-automorphism $\delta$ is just
the Kreweras complement (as introduced by Kreweras in 1972), see Theorem \ref{Kreweras}.
	\bigskip

\begin{cor} The map $\delta$ provides a bijection between
$\A_t(\mo\Lambda)$ and $\A_{n-t}(\mo\Lambda)$, for $0 \le t \le n.$
\end{cor} 
	
The special case $t=1$ should be mentioned explicitly:
{\it The map $\delta$ provides a bijection between
the elements in $\A_{n-1}(\mo\Lambda)$ and the exceptional modules.}
	\medskip

Note that the restriction of $\delta$ to $\A_{n-1}(\mo\Lambda)$ can be written as follows:
if $\mathcal A$ belongs to $\A_{n-1}(\mo\Lambda)$, then
$\delta(\mathcal A) = \add N_{\mathcal A}^{\mo\Lambda}.$ 
	\bigskip

We define a permutation $\overline \tau$ on the set of isomorphism classes of
indecomposable $\Lambda$-modules as follows: if $X$ is indecomposable, let
$$
  \overline\tau X = \left\{ \begin{matrix} \ \tau X &\ \text{if $X$ is not projective,} \cr\cr
     \ I(S) &\text{if $X = P(S)$,} \qquad\qquad
  \end{matrix} \right.
$$
where $\tau = D\Tr$ is the Auslander-Reiten translation, $P(S)$ the projective cover of the
simple module $S$ and $I(S)$ the injective envelop of $S$. We call $\overline\tau$ the
{\it extended Auslander-Reiten translation} for $\mo\Lambda$.
	\bigskip 

\begin{theorem}\label{AR} The map $\delta^2$ is a poset automorphism of
$\A(\mo\Lambda)$ and we have 
$$ 
 \delta^2(\mathcal A) = \overline\tau(\mathcal A)
$$ 
for any $\mathcal A \in \A(\mo\Lambda).$
\end{theorem} 
	
\begin{proof} The first assertion follows directly from Theorem \ref{anti}. 
In order to show the
second assertion, it is sufficient to consider the case when $\mathcal A$ has rank one. 
Thus, let $\mathcal A = \add X$ where $X$ is an exceptional module. Let us show that
$\overline\tau(X)$ belongs to $\delta^2(\mathcal A) = \mathcal A^\perp{}^\perp.$
Thus, given $Y\in \mathcal A^\perp$, we claim that $\Hom(Y,\overline\tau X) = 0$
and $\Ext^1(Y,\overline\tau X) = 0$.

First, assume that
$X = P(S)$ for some simple module $S$, thus $\overline\tau(X) = I(S)$. 
Then $P(S)^\perp$ consists of the $\Lambda$-modules
$Y$ which do not have $S$ as a composition factor, thus $\Hom(Y,I(S)) = 0$.
Of course, also $\Ext^1(Y,I(S)) = 0,$ since $I(S)$ is injective.

There is the second case that $X$ is not projective. Let $Y\in\mathcal A^\perp$,
thus $\Hom(X,Y) = 0$ and $\Ext^1(X,Y) = 0.$ Since $X$ is indecomposable and not
projective, $\Hom(Y,\tau X) = D\Ext^1(X,Y) = 0$ and 
$\Ext^1(Y,\tau X) = D\Hom(\tau X, \tau Y) =
D\Hom(X,Y) = 0.$

Now $\overline\tau(X)$ is non-zero and belongs to $\delta^2(\mathcal A)$; since 
$\delta^2(\mathcal A)$ has rank 1, it follows  that 
$\delta^2(\mathcal A) = \add \overline\tau(X) = \overline\tau(\mathcal A)$.
\end{proof} 
	
For example, let $\Lambda_3$ be the path algebra $\Lambda$
of a linearly ordered quiver $Q$ of type $\mathbb A_3$. Then the permutation $\overline\tau$ 
of $A_1(\mo\Lambda_3)$ (or of the indecomposable $\Lambda_3$-modules) has two orbits:
$$
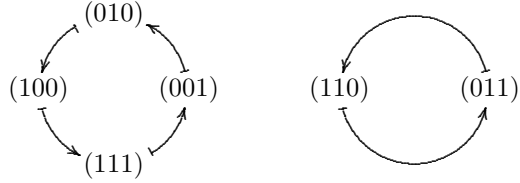
\hbox{\beginpicture
  \setcoordinatesystem units <1cm,1cm>
\put{\beginpicture
\put{$(0 0 1)$} at 1 0 
\put{$(0 1 0)$} at 0 1 
\put{$(1 0 0)$} at -1 0 
\put{$(1 1 1)$} at 0 -1 
\circulararc 45 degrees from 0.95 0.25 center at 0 0
\circulararc -45 degrees from  -.95 .25 center at 0 0
\circulararc -45 degrees from 0.95 -.25 center at 0 0
\circulararc 45 degrees from  -.95 -.25 center at 0 0
\plot 0.9 0.24  1 0.26 /
\plot -.9 -.24  -1 -.26 /
\plot -.52 0.87  -.47 0.77 /
\plot  .52 -.87   .47 -.77 /
\arr{.5 .84}{.45 .87}
\arr{-.5 -.84}{-.45 -.87}
\arr{-.941 .26}{-.947 .24}
\arr{.941 -.26}{.947 -.24}
\endpicture} at 0 0
\put{\beginpicture
\put{$(0 1 1)$} at 1 0 
\put{$(1 1 0)$} at -1 0 
\arr{-.941 .26}{-.947 .24}
\arr{.941 -.26}{.947 -.24}
\circulararc 150 degrees from 0.95 0.25 center at 0 0
\circulararc 150 degrees from  -.95 -.25 center at 0 0
\plot 0.9 0.24  1 0.26 /
\plot -.9 -.24  -1 -.26 /
\endpicture} at 4 0
\endpicture}
$$
Looking at the lattice $\A(\mo\Lambda_3)$, the elements of rank 1 behave differently: 
four of the six elements have 3 upper neighbors, these are the elements in the first
$\overline\tau$-orbit, the remaining two form the second 
$\overline\tau$-orbit, each of them has only 2 upper neighbors.

\begin{cor} The automorphism $\delta^2$ of $\A(\mo\Lambda)$ is the identity if and only if
$\mo\Lambda$ is a semismiple algebra.
\end{cor}
	\medskip

\subsubsection{\bf Remark.} 
{\it If $\mathcal A$ is a thick subcategory of $\mo\Lambda$, then $\delta^2(\mathcal A)$
has the same rank as $\mathcal A$, but may not be equivalent (as a category) to $\mathcal A$.}

Here is an example. Let $\Lambda$ be the path algebra of the $3$-subspace quiver $Q$ as shown below
on the left. To the right, we show twice the Auslander-Reiten quiver of $\Lambda$; first we
mark the antichain $P(1), S(2),S(3)$ by bullets, on the right we mark in the same way the
modules $\overline\tau P(1) = I(1), \overline\tau S(2) = \tau S(2), \overline\tau S(3) = \tau S(3)$

$$
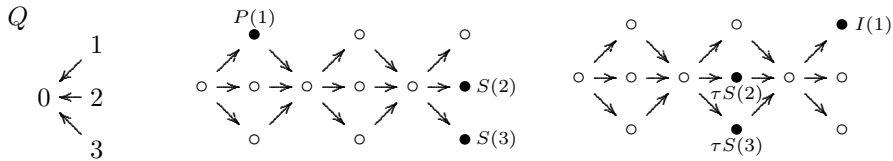
\hbox{\beginpicture
  \setcoordinatesystem units <1cm,1cm>
\put{\beginpicture
  \setcoordinatesystem units <.7cm,.7cm>
\arr{0.7 0.7}{0.3 0.3}
\arr{0.7 0}{0.3 0}
\arr{0.7 -.7}{0.3 -.3}
\put{$ 0$} at 0 0
\put{$ 1$} at 1 1
\put{$ 2$} at 1 0
\put{$ 3$} at 1 -1
\put{$Q$} at -.5 1.5
\endpicture} at 0 0
\put{\beginpicture
  \setcoordinatesystem units <.7cm,.7cm>
\multiput{$\circ$} at 0 0  1 1  1 0  1 -1  2 0  3 1  3 0  3 -1  4 0  5 1  5 0  5 -1  /
\arr{0.3 0.3}{0.7 0.7}
\arr{0.3 0}{0.7 0}
\arr{0.3 -.3}{0.7 -.7}

\arr{2.3 0.3}{2.7 0.7}
\arr{2.3 0}{2.7 0}
\arr{2.3 -.3}{2.7 -.7}

\arr{4.3 0.3}{4.7 0.7}
\arr{4.3 0}{4.7 0}
\arr{4.3 -.3}{4.7 -.7}

\arr{1.3 0.7}{1.7 0.3}
\arr{1.3 0}{1.7 0}
\arr{1.3 -.7}{1.7 -.3}

\arr{3.3 0.7}{3.7 0.3}
\arr{3.3 0}{3.7 0}
\arr{3.3 -.7}{3.7 -.3}
\put{$\ssize P(1)$} at 1 1.3
\put{$\ssize S(2)$} at 5.6 0
\put{$\ssize S(3)$} at 5.6 -1
\multiput{$\bullet$} at 1 1  5 0  5 -1 /
\put{} at 0 -1.3

\endpicture} at 4 0
\put{\beginpicture
  \setcoordinatesystem units <.7cm,.7cm>
\multiput{$\circ$} at 0 0  1 1  1 0  1 -1  2 0  3 1  3 0  3 -1  4 0  5 1  5 0  5 -1  /
\arr{0.3 0.3}{0.7 0.7}
\arr{0.3 0}{0.7 0}
\arr{0.3 -.3}{0.7 -.7}

\arr{2.3 0.3}{2.7 0.7}
\arr{2.3 0}{2.7 0}
\arr{2.3 -.3}{2.7 -.7}

\arr{4.3 0.3}{4.7 0.7}
\arr{4.3 0}{4.7 0}
\arr{4.3 -.3}{4.7 -.7}

\arr{1.3 0.7}{1.7 0.3}
\arr{1.3 0}{1.7 0}
\arr{1.3 -.7}{1.7 -.3}

\arr{3.3 0.7}{3.7 0.3}
\arr{3.3 0}{3.7 0}
\arr{3.3 -.7}{3.7 -.3}
\multiput{$\bullet$} at 5 1  3 0  3 -1 /
\put{} at 0 1.3
\put{$\ssize I(1)$} at 5.6 1
\put{$\ssize \tau S(2)$} at 3  -.3
\put{$\ssize \tau S(3)$} at 3 -1.3
\endpicture} at 9 0
\endpicture}
$$
Let $\mathcal A$ be the thick subcategory generated by the antichain $P(1),S(2),S(3)$,
thus  $\mathcal A' = \delta^2(\mathcal A)$ is the thick subcategory generated by the modules
$I(1) = S(1),\tau S(2),$ $\tau S(3)$, thus by the antichain $P(2), P(3), S(1)$.
Here are the corresponding quivers:
$$
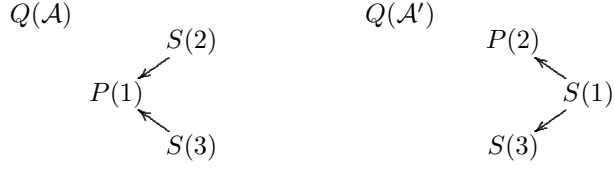
\hbox{\beginpicture
  \setcoordinatesystem units <1cm,1cm>
\put{\beginpicture
  \setcoordinatesystem units <1cm,.7cm>
\put{$P(1)$} at 0 0 
\put{$S(2)$} at 1 1
\put{$S(3)$} at 1 -1
\arr{0.7 0.7}{0.3 0.3}
\arr{0.7 -.7}{0.3 -.3}
\put{$Q(\mathcal A)$} at -1 1.5
\endpicture} at 0 0  
\put{\beginpicture
  \setcoordinatesystem units <1cm,.7cm>
\put{$P(2)$} at 0 1
\put{$S(3)$} at 0 -1
\put{$S(1)$} at 1 0
\arr{0.7 0.3}{0.3 0.7}
\arr{0.7 -.3}{0.3 -.7}
\put{$Q(\mathcal A')$} at -1.5 1.5
\endpicture} at 5 0  

\endpicture}
$$
	\bigskip

As we will see in Proposition \ref{delta-square-a_n}, the algebra $\Lambda = \Lambda_n$: 
has the property that for any thick subcategory $\mathcal A$, the subcategory
$\delta^2(\mathcal A)$ is equivalent to $\mathcal A$.

	\bigskip 
\subsection{The poset $\A(\mo\Lambda)$: Maximal chains, complete 
exceptional sequences}\label{max-chains}
Given a pair $\mathcal A\subset \mathcal B$ of neighbors in $\A(\mo\Lambda)$, we have denoted 
in Section \ref{neighbors} by
$M_{\mathcal A}^{\mathcal B}$ the unique indecomposable module in $\mathcal B\cap \mathcal A^\perp$, and by
$N_{\mathcal A}^{\mathcal B}$ the unique indecomposable module in $\mathcal B\cap {}^\perp\!\!\mathcal A$.
	\medskip 

\subsubsection{}
Recall that a sequence $(M_1,\dots,M_t)$ of $\Lambda$-modules is called an {\it exceptional
sequence} provided all the modules $M_i$ are exceptional and $\Hom(M_j,M_i) = 0 = \Ext(M_j,M_i)$
for all $i < j.$ There is the following equivalent inductive definition: The empty sequence is
exceptional, and for $t\ge 1$, the sequence $(M_1,\dots,M_t)$ is exceptional if and only 
$(M_1,\dots,M_{t-1})$ is exceptional and $M_t$ is an exceptional module in
${}^\perp\!\left(\add\{M_1,\dots,M_{t-1}\}\right),$ or, equivalently, if and only if $(M_2,\dots,M_t)$ 
is exceptional and 
$M_1$ is an exceptional module in $\left(\add\{M_2,\dots,M_t\}\right)^\perp$.
 
Recall that an exceptional sequence $(M_1,\dots,M_t)$ is said to be complete provided
$t = \rank(\Lambda).$ 
Two exceptional sequences $(M_1,\dots,M_t)$ and $(M'_1,\dots,M'_{t'})$
are called {\it isomorphic} provided $t = t'$ and $M_i$ is isomorphic to $M'_i$, for $1\le i\le t.$
We denote by $\E(\mo\Lambda)$ the set of (isomorphism classes of) complete exceptional sequences
in $\mo\Lambda$.

There are some obvious complete exceptional sequences: Assume that the quiver of $\Lambda$ has the
vertex set $\{1,2,\dots,n\}$ and that for any arrow $i \leftarrow j$ we have $i < j$. Given a vertex
$i$, we denote as usual the corresponding simple module by $S(i)$, the projective cover
of $S(i)$ by $P(i)$, the injective envelope of $S(i)$ by $I(i)$. Then the following sequences
$$
 \bigl(P(1),P(2),\dots,P(n)\bigr),\quad \bigl(S(n),S(n-1),\dots,S(1)\bigr),\quad 
 \bigl(I(1),I(2),\dots,I(n)\bigr)
$$
are complete exceptional sequences. 
	\medskip 

\begin{theorem}\label{neighbors1}
The isomorphism classes of complete exceptional sequences of $\Lambda$-modules
correspond bijectively to the maximal chains in the poset $\A(\mo\Lambda).$

A bijection
	\smallskip

\Rahmen{$
 N\!: \M(\A(\mo\Lambda))   \longrightarrow \E(\mo\Lambda))
$}
	\smallskip 

\noindent
is given as follows: Let 
$$
\hbox{\beginpicture
  \setcoordinatesystem units <1cm,1cm>
\put{$0 = \mathcal A(0) \subset \mathcal A(1) \subset \cdots \subset \mathcal A(n) = \mo\Lambda$}
  [l] at 0 0
\put{$(*)$} [r] at -3 0
\put{} at 10 0 
\endpicture}
$$
be a maximal chain in $\A(\mo\Lambda)$, then
$$
 N(\mathcal A(0),\dots,\mathcal A(n)) =  
\bigl(N^{\mathcal A(1)}_{\mathcal A(0)},N^{\mathcal A(2)}_{\mathcal A(1)},\cdots, N^{\mathcal A(n)}_{\mathcal A(n-1)}
 \bigr).
$$
Conversely, if 
$(N_1,N_2,\dots,N_n)$ is a complete
exceptional sequence, let $\mathcal A(i)$ be the thick closure of $N_1,\dots,N_i$.
then we obtain a maximal chain $(*)$ in $\A(\mo\Lambda).$
\end{theorem}

\begin{addendum}\label{neighbors2}
{\it There is a second bijection
	\smallskip

\Rahmen{$
 M\!: \M(\A(\mo\Lambda))   \longrightarrow \E(\mo\Lambda))
$}
	\smallskip 

\noindent
given as follows: Starting with a maximal chain $(*)$, let
$$
 M(\mathcal A(0),\dots,\mathcal A(n)) =  
\bigl(M^{\mathcal A(n)}_{\mathcal A(n-1)}, \cdots,,M^{\mathcal A(2)}_{\mathcal A(1)},M^{\mathcal A(1)}_{\mathcal A(0)} 
 \bigr).
$$
Conversely, if 
$(M_1,M_2,\dots,M_n)$ is a complete
exceptional sequence, let $\mathcal A(i)$ be the thick closure of $M_{n-i+1},\dots,M_n$.
then we obtain a maximal chain $(*)$ in $\A(\mo\Lambda).$}
\end{addendum}
	\bigskip 

\subsubsection{\bf Example: The algebra $\Lambda_2$.}
For the algebra $\Lambda_2$ with Auslander-Reiten quiver
$$
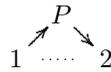
\beginpicture
  \setcoordinatesystem units <0.6cm,0.6cm>
\put{$1$} at 0 0
\put{$P$} at 1 1
\put{$2$} at 2 0
\arr{0.3 0.3}{0.7 0.7}
\arr{1.3 0.7}{1.7 0.3}
\setdots <1mm>
\plot 0.6 0  1.4 0 /
\endpicture
$$
there are precisely three exceptional sequence, namely the sequence
$(1,P)$ of the indecomposable projectives, the sequence
$(P,2)$ of the indecomposable injectives, and the sequence $(2,1)$ of the simple modules.
	\medskip 

Similarly, the lattice $\A(\mo\Lambda_2)$ has precisely three maximal chains.
The map $N$ attaches to these chains in $\A(\mo\Lambda_2)$
the exceptional sequences as shown below (they have to
be read going upwards):

$$
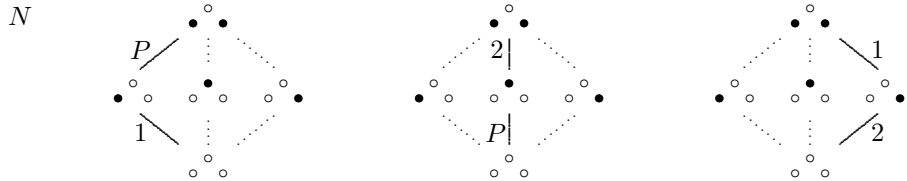
\hbox{\beginpicture
  \setcoordinatesystem units <1cm,1cm>
\put{$N$} at -2.5 1
\put{\beginpicture
  \setcoordinatesystem units <1cm,1cm>
 \put{$\beginpicture
  \setcoordinatesystem units <0.2cm,0.2cm>
  \multiput{$\sssize \circ$} at 0 0  1 1  2 0  /
  \multiput{$\ssize \bullet$} at   /
  \endpicture$} at  0 0
 \put{$\beginpicture
  \setcoordinatesystem units <0.2cm,0.2cm>
  \multiput{$\sssize \circ$} at 0 0  1 1  2 0  /
  \multiput{$\ssize \bullet$} at  0 0 /
  \endpicture$} at  -1 1
 \put{$\beginpicture
  \setcoordinatesystem units <0.2cm,0.2cm>
  \multiput{$\sssize \circ$} at 0 0  1 1  2 0  /
  \multiput{$\ssize \bullet$} at  1 1 /
  \endpicture$} at  0 1
 \put{$\beginpicture
  \setcoordinatesystem units <0.2cm,0.2cm>
  \multiput{$\sssize \circ$} at 0 0  1 1  2 0  /
  \multiput{$\ssize \bullet$} at  2 0 /
  \endpicture$} at  1 1
 \put{$\beginpicture
  \setcoordinatesystem units <0.2cm,0.2cm>
  \multiput{$\sssize \circ$} at 0 0  1 1  2 0  /
  \multiput{$\ssize \bullet$} at  0 0  2 0 /
  \endpicture$} at  0 2
\plot -.4 .3  -.9 .7 /
\plot -.4 1.7  -.9 1.3 /
\setdots <1mm>
  \plot  0 .3   0 .7 /
  \plot  .4 .3  .9 .7 /
  \plot  0 1.7  0 1.3 /
  \plot  .4 1.7   .9 1.3 /
\multiput{$1$} at -.9 0.45  /
\multiput{$P$} at  -.9 1.55 /

\endpicture} at 0 0 
\put{\beginpicture
  \setcoordinatesystem units <1cm,1cm>
 \put{$\beginpicture
  \setcoordinatesystem units <0.2cm,0.2cm>
  \multiput{$\sssize \circ$} at 0 0  1 1  2 0  /
  \multiput{$\ssize \bullet$} at   /
  \endpicture$} at  0 0
 \put{$\beginpicture
  \setcoordinatesystem units <0.2cm,0.2cm>
  \multiput{$\sssize \circ$} at 0 0  1 1  2 0  /
  \multiput{$\ssize \bullet$} at  0 0 /
  \endpicture$} at  -1 1
 \put{$\beginpicture
  \setcoordinatesystem units <0.2cm,0.2cm>
  \multiput{$\sssize \circ$} at 0 0  1 1  2 0  /
  \multiput{$\ssize \bullet$} at  1 1 /
  \endpicture$} at  0 1
 \put{$\beginpicture
  \setcoordinatesystem units <0.2cm,0.2cm>
  \multiput{$\sssize \circ$} at 0 0  1 1  2 0  /
  \multiput{$\ssize \bullet$} at  2 0 /
  \endpicture$} at  1 1
 \put{$\beginpicture
  \setcoordinatesystem units <0.2cm,0.2cm>
  \multiput{$\sssize \circ$} at 0 0  1 1  2 0  /
  \multiput{$\ssize \bullet$} at  0 0  2 0 /
  \endpicture$} at  0 2
  \plot  0 .3   0 .7 /
  \plot  0 1.7  0 1.3 /
\setdots <1mm>
  \plot  .4 .3  .9 .7 /
  \plot  .4 1.7   .9 1.3 /
\plot -.4 .3  -.9 .7 /
\plot -.4 1.7  -.9 1.3 /
\multiput{$P$} at -.18 .45   /
\multiput{$2$} at  -.16 1.55 /

\endpicture} at 4 0 
\put{\beginpicture
  \setcoordinatesystem units <1cm,1cm>
 \put{$\beginpicture
  \setcoordinatesystem units <0.2cm,0.2cm>
  \multiput{$\sssize \circ$} at 0 0  1 1  2 0  /
  \multiput{$\ssize \bullet$} at   /
  \endpicture$} at  0 0
 \put{$\beginpicture
  \setcoordinatesystem units <0.2cm,0.2cm>
  \multiput{$\sssize \circ$} at 0 0  1 1  2 0  /
  \multiput{$\ssize \bullet$} at  0 0 /
  \endpicture$} at  -1 1
 \put{$\beginpicture
  \setcoordinatesystem units <0.2cm,0.2cm>
  \multiput{$\sssize \circ$} at 0 0  1 1  2 0  /
  \multiput{$\ssize \bullet$} at  1 1 /
  \endpicture$} at  0 1
 \put{$\beginpicture
  \setcoordinatesystem units <0.2cm,0.2cm>
  \multiput{$\sssize \circ$} at 0 0  1 1  2 0  /
  \multiput{$\ssize \bullet$} at  2 0 /
  \endpicture$} at  1 1
 \put{$\beginpicture
  \setcoordinatesystem units <0.2cm,0.2cm>
  \multiput{$\sssize \circ$} at 0 0  1 1  2 0  /
  \multiput{$\ssize \bullet$} at  0 0  2 0 /
  \endpicture$} at  0 2
 \plot  .4 .3  .9 .7 /
 \plot  .4 1.7   .9 1.3 /
\setdots <1mm>
\plot -.4 .3  -.9 .7 /
\plot -.4 1.7  -.9 1.3 /
  \plot  0 .3   0 .7 /
  \plot  0 1.7  0 1.3 /
\multiput{$1$} at  .9 1.55 /
\multiput{$2$} at .9 0.45  /

\endpicture} at 8 0 
\endpicture}
$$
	
The map $M$ attaches to the maximal chains in $\A(\mo\Lambda_2)$ the 
exceptional sequences as shown next (now they have to
be read going downwards):
$$
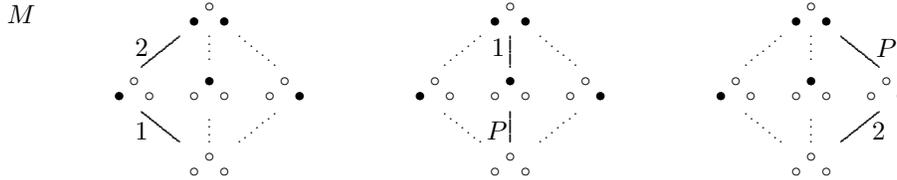
\hbox{\beginpicture
  \setcoordinatesystem units <1cm,1cm>
\put{$M$} at -2.5 1
\put{\beginpicture
  \setcoordinatesystem units <1cm,1cm>
 \put{$\beginpicture
  \setcoordinatesystem units <0.2cm,0.2cm>
  \multiput{$\sssize \circ$} at 0 0  1 1  2 0  /
  \multiput{$\ssize \bullet$} at   /
  \endpicture$} at  0 0
 \put{$\beginpicture
  \setcoordinatesystem units <0.2cm,0.2cm>
  \multiput{$\sssize \circ$} at 0 0  1 1  2 0  /
  \multiput{$\ssize \bullet$} at  0 0 /
  \endpicture$} at  -1 1
 \put{$\beginpicture
  \setcoordinatesystem units <0.2cm,0.2cm>
  \multiput{$\sssize \circ$} at 0 0  1 1  2 0  /
  \multiput{$\ssize \bullet$} at  1 1 /
  \endpicture$} at  0 1
 \put{$\beginpicture
  \setcoordinatesystem units <0.2cm,0.2cm>
  \multiput{$\sssize \circ$} at 0 0  1 1  2 0  /
  \multiput{$\ssize \bullet$} at  2 0 /
  \endpicture$} at  1 1
 \put{$\beginpicture
  \setcoordinatesystem units <0.2cm,0.2cm>
  \multiput{$\sssize \circ$} at 0 0  1 1  2 0  /
  \multiput{$\ssize \bullet$} at  0 0  2 0 /
  \endpicture$} at  0 2
\plot -.4 .3  -.9 .7 /
\plot -.4 1.7  -.9 1.3 /
\setdots <1mm>
  \plot  0 .3   0 .7 /
  \plot  .4 .3  .9 .7 /
  \plot  0 1.7  0 1.3 /
  \plot  .4 1.7   .9 1.3 /
\multiput{$1$} at -.9 0.45  /
\multiput{$2$} at  -.9 1.55 /

\endpicture} at 0 0 
\put{\beginpicture
  \setcoordinatesystem units <1cm,1cm>
 \put{$\beginpicture
  \setcoordinatesystem units <0.2cm,0.2cm>
  \multiput{$\sssize \circ$} at 0 0  1 1  2 0  /
  \multiput{$\ssize \bullet$} at   /
  \endpicture$} at  0 0
 \put{$\beginpicture
  \setcoordinatesystem units <0.2cm,0.2cm>
  \multiput{$\sssize \circ$} at 0 0  1 1  2 0  /
  \multiput{$\ssize \bullet$} at  0 0 /
  \endpicture$} at  -1 1
 \put{$\beginpicture
  \setcoordinatesystem units <0.2cm,0.2cm>
  \multiput{$\sssize \circ$} at 0 0  1 1  2 0  /
  \multiput{$\ssize \bullet$} at  1 1 /
  \endpicture$} at  0 1
 \put{$\beginpicture
  \setcoordinatesystem units <0.2cm,0.2cm>
  \multiput{$\sssize \circ$} at 0 0  1 1  2 0  /
  \multiput{$\ssize \bullet$} at  2 0 /
  \endpicture$} at  1 1
 \put{$\beginpicture
  \setcoordinatesystem units <0.2cm,0.2cm>
  \multiput{$\sssize \circ$} at 0 0  1 1  2 0  /
  \multiput{$\ssize \bullet$} at  0 0  2 0 /
  \endpicture$} at  0 2
  \plot  0 .3   0 .7 /
  \plot  0 1.7  0 1.3 /
\setdots <1mm>
  \plot  .4 .3  .9 .7 /
  \plot  .4 1.7   .9 1.3 /
\plot -.4 .3  -.9 .7 /
\plot -.4 1.7  -.9 1.3 /
\multiput{$P$} at -.18 .45   /
\multiput{$1$} at  -.16 1.55 /

\endpicture} at 4 0 
\put{\beginpicture
  \setcoordinatesystem units <1cm,1cm>
 \put{$\beginpicture
  \setcoordinatesystem units <0.2cm,0.2cm>
  \multiput{$\sssize \circ$} at 0 0  1 1  2 0  /
  \multiput{$\ssize \bullet$} at   /
  \endpicture$} at  0 0
 \put{$\beginpicture
  \setcoordinatesystem units <0.2cm,0.2cm>
  \multiput{$\sssize \circ$} at 0 0  1 1  2 0  /
  \multiput{$\ssize \bullet$} at  0 0 /
  \endpicture$} at  -1 1
 \put{$\beginpicture
  \setcoordinatesystem units <0.2cm,0.2cm>
  \multiput{$\sssize \circ$} at 0 0  1 1  2 0  /
  \multiput{$\ssize \bullet$} at  1 1 /
  \endpicture$} at  0 1
 \put{$\beginpicture
  \setcoordinatesystem units <0.2cm,0.2cm>
  \multiput{$\sssize \circ$} at 0 0  1 1  2 0  /
  \multiput{$\ssize \bullet$} at  2 0 /
  \endpicture$} at  1 1
 \put{$\beginpicture
  \setcoordinatesystem units <0.2cm,0.2cm>
  \multiput{$\sssize \circ$} at 0 0  1 1  2 0  /
  \multiput{$\ssize \bullet$} at  0 0  2 0 /
  \endpicture$} at  0 2
 \plot  .4 .3  .9 .7 /
 \plot  .4 1.7   .9 1.3 /
\setdots <1mm>
\plot -.4 .3  -.9 .7 /
\plot -.4 1.7  -.9 1.3 /
  \plot  0 .3   0 .7 /
  \plot  0 1.7  0 1.3 /
\multiput{$P$} at  1 1.55 /
\multiput{$2$} at .9 0.45  /

\endpicture} at 8 0 
\endpicture}
$$
	\bigskip

\subsubsection{\bf Example: The algebra $\Lambda_3$.}
As a second example, we consider the algebra $\Lambda_3$ (with the following quiver $Q$);
we denote the simple
module $S(i)$ just by $i$, and put $P = P(2),\ W = P(3) = I(1),\  I = I(2).$
$$
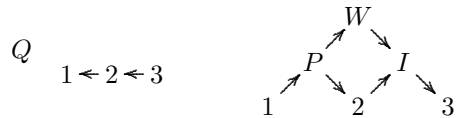
\beginpicture
  \setcoordinatesystem units <0.6cm,0.6cm>
\put{\beginpicture
\put{$1$} at 0 0
\put{$2$} at 1 0
\put{$3$} at 2 0
\arr{0.7 0}{0.3 0}
\arr{1.7 0}{1.3 0}
\put{$Q$} at -1 .5
\endpicture} at 0 0 
\put{\beginpicture
  \setcoordinatesystem units <0.6cm,0.6cm>
\put{$1$} at 0 0
\put{$P$} at 1 1
\put{$W$} at 2 2
\put{$2$} at 2 0
\put{$I$} at 3 1 
\put{$3$} at 4 0
\arr{0.3 0.3}{0.7 0.7}
\arr{1.3 0.7}{1.7 0.3}
\arr{2.3 0.3}{2.7 0.7}
\arr{3.3 0.7}{3.7 0.3}
\arr{1.3 1.3}{1.7 1.7}
\arr{2.3 1.7}{2.7 1.3}
\endpicture} at 6 0 
\endpicture
$$
Here is the lattice $\A(\mo\Lambda_3)$.  
We have added the modules $N^{\mathcal B}_{\mathcal A}$ for every pair $\mathcal A < \mathcal B$
of neighbors to the (dotted) line connecting $\mathcal A$ and $\mathcal B$:

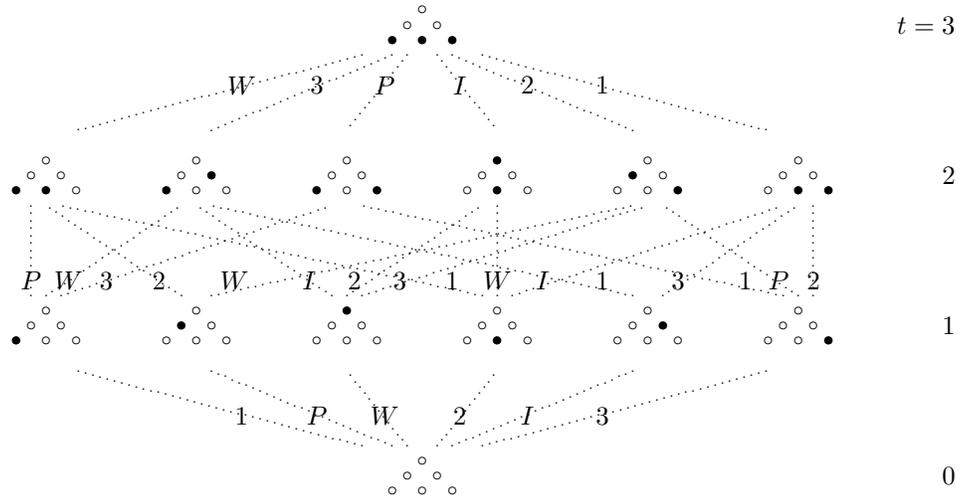
$$
\hbox{\beginpicture
  \setcoordinatesystem units <2cm,2cm>
\put{$t=3$} at 3.35 3
\put{$2$} at 3.5 2
\put{$1$} at 3.5 1
\put{$0$} at 3.5 0

\put{$\beginpicture
  \setcoordinatesystem units <0.2cm,0.2cm>
  \multiput{$\sssize \circ$} at 0 0  1 1  -1 1  0 2  -2 0  2 0 /
  \multiput{$\ssize \bullet$} at  /
 \endpicture$} at 0 0

\put{$\beginpicture
  \setcoordinatesystem units <0.2cm,0.2cm>
  \multiput{$\sssize \circ$} at 0 0  1 1  -1 1  0 2  -2 0  2 0 /
  \multiput{$\ssize \bullet$} at -2 0 /
 \endpicture$} at -2.5 1
\put{$\beginpicture
  \setcoordinatesystem units <0.2cm,0.2cm>
  \multiput{$\sssize \circ$} at 0 0  1 1  -1 1  0 2  -2 0  2 0 /
  \multiput{$\ssize \bullet$} at  -1 1 /
 \endpicture$} at -1.5 1
\put{$\beginpicture
  \setcoordinatesystem units <0.2cm,0.2cm>
  \multiput{$\sssize \circ$} at 0 0  1 1  -1 1  0 2  -2 0  2 0 /
  \multiput{$\ssize \bullet$} at  0 2 /
 \endpicture$} at -.5 1
\put{$\beginpicture
  \setcoordinatesystem units <0.2cm,0.2cm>
  \multiput{$\sssize \circ$} at 0 0  1 1  -1 1  0 2  -2 0  2 0 /
  \multiput{$\ssize \bullet$} at  0 0 /
 \endpicture$} at 0.5 1
\put{$\beginpicture
  \setcoordinatesystem units <0.2cm,0.2cm>
  \multiput{$\sssize \circ$} at 0 0  1 1  -1 1  0 2  -2 0  2 0 /
  \multiput{$\ssize \bullet$} at  1 1 /
 \endpicture$} at 1.5 1
\put{$\beginpicture
  \setcoordinatesystem units <0.2cm,0.2cm>
  \multiput{$\sssize \circ$} at 0 0  1 1  -1 1  0 2  -2 0  2 0 /
  \multiput{$\ssize \bullet$} at  2 0 /
 \endpicture$} at 2.5  1
\put{$\beginpicture
  \setcoordinatesystem units <0.2cm,0.2cm>
  \multiput{$\sssize \circ$} at 0 0  1 1  -1 1  0 2  -2 0  2 0 /
  \multiput{$\ssize \bullet$} at -2 0  0 0 /
 \endpicture$} at -2.5 2
\put{$\beginpicture
  \setcoordinatesystem units <0.2cm,0.2cm>
  \multiput{$\sssize \circ$} at 0 0  1 1  -1 1  0 2  -2 0  2 0 /
  \multiput{$\ssize \bullet$} at -2 0  1 1  /
 \endpicture$} at -1.5 2
\put{$\beginpicture
  \setcoordinatesystem units <0.2cm,0.2cm>
  \multiput{$\sssize \circ$} at 0 0  1 1  -1 1  0 2  -2 0  2 0 /
  \multiput{$\ssize \bullet$} at  -2 0  2 0 /
 \endpicture$} at -.5 2
\put{$\beginpicture
  \setcoordinatesystem units <0.2cm,0.2cm>
  \multiput{$\sssize \circ$} at 0 0  1 1  -1 1  0 2  -2 0  2 0 /
  \multiput{$\ssize \bullet$} at  0 0  0 2 /
 \endpicture$} at 0.5 2
\put{$\beginpicture
  \setcoordinatesystem units <0.2cm,0.2cm>
  \multiput{$\sssize \circ$} at 0 0  1 1  -1 1  0 2  -2 0  2 0 /
  \multiput{$\ssize \bullet$} at  -1 1  2 0 /
 \endpicture$} at 1.5 2
\put{$\beginpicture
  \setcoordinatesystem units <0.2cm,0.2cm>
  \multiput{$\sssize \circ$} at 0 0  1 1  -1 1  0 2  -2 0  2 0 /
  \multiput{$\ssize \bullet$} at  0 0  2 0 /
 \endpicture$} at 2.5 2

\put{$\beginpicture
  \setcoordinatesystem units <0.2cm,0.2cm>
  \multiput{$\sssize \circ$} at 0 0  1 1  -1 1  0 2  -2 0  2 0 /
  \multiput{$\ssize \bullet$} at  -2 0  0 0  2 0 /
 \endpicture$} at  0 3

\setdots <1mm>

\plot -0.4 0.2 -2.3 0.7 /
\plot -0.2 0.2 -1.4 0.7 /
\plot -0.1 0.2 -0.5 0.7 /
\plot .1 0.2 0.5 0.7 /
\plot .2 0.2 1.4 0.7 /
\plot .4 0.2 2.3 0.7 /

\plot -0.4 2.8 -2.3 2.3 /
\plot -0.2 2.8 -1.4 2.3 /
\plot -0.1 2.8 -0.5 2.3 /
\plot .1 2.8 0.5 2.3 /
\plot .2 2.8 1.4 2.3 /
\plot .4 2.8 2.3 2.3 /

\plot -2.6 1.2  -2.6 1.8 /
\plot -2.5 1.2  -1.6 1.8 /
\plot -2.4 1.2  -0.6 1.8 /

\plot -1.6 1.2  -2.5 1.8 /
\plot -1.4 1.2   1.4 1.8 /

\plot -.6 1.2  -1.5 1.8 /
\plot -.5 1.2   0.4 1.8 /
\plot -.4 1.2   1.5 1.8 /

\plot .4 1.2  -2.4 1.8 /
\plot .5 1.2   .5 1.8 /
\plot .6 1.2   2.4 1.8 /

\plot 1.4 1.2  -1.4 1.8 /
\plot 1.6 1.2   2.5 1.8 /

\plot 2.4 1.2  -.4 1.8 /
\plot 2.5 1.2   1.6 1.8 /
\plot 2.6 1.2   2.6 1.8 /

\multiput{$1$} at -1.2 0.4 
   .2 1.3    1.2 1.3       2.15 1.3  
   1.2 2.6  / 

\multiput{$P$} at -0.7 0.4  
   -2.6 1.3    2.37 1.3 
   -.25 2.6   / 

\multiput{$W$} at  -0.25 0.4 
    -2.35 1.3  -1.25 1.3   .5 1.3 
   -1.2 2.6 /

\multiput{$2$} at .25 0.4  
    -1.75 1.3   -.45 1.3   2.6 1.3 
    0.7 2.6   /

\multiput{$I$} at   .7 0.4 
     -.75 1.3      .8 1.3  
    0.25 2.6  /

\multiput{$3$} at 1.2 0.4 
    -2.1 1.3   -.15 1.3    1.7 1.3 
    -.7 2.6 /

\endpicture}
$$
We obtain all the complete exceptional sequences {\bf going up} in the lattice and using the
modules $N_{\mathcal A}^{\mathcal B}$. For example, on the left hand side, we obtain in this way
the sequence $(1,P,W)$ of indecomposable projective modules, on the right hand side one sees 
the sequence $(3,2,1)$ of the simple modules, whereas the exceptional sequence $(W,I,3)$ of the
indecomposable injective modules can be traced somewhere in the middle).
	\bigskip 

Here is again the lattice $\A(\mo\Lambda_3)$, now we show the modules $M^{\mathcal B}_{\mathcal A}$ 
for every pair $\mathcal A < \mathcal B$ of neighbors:
$$
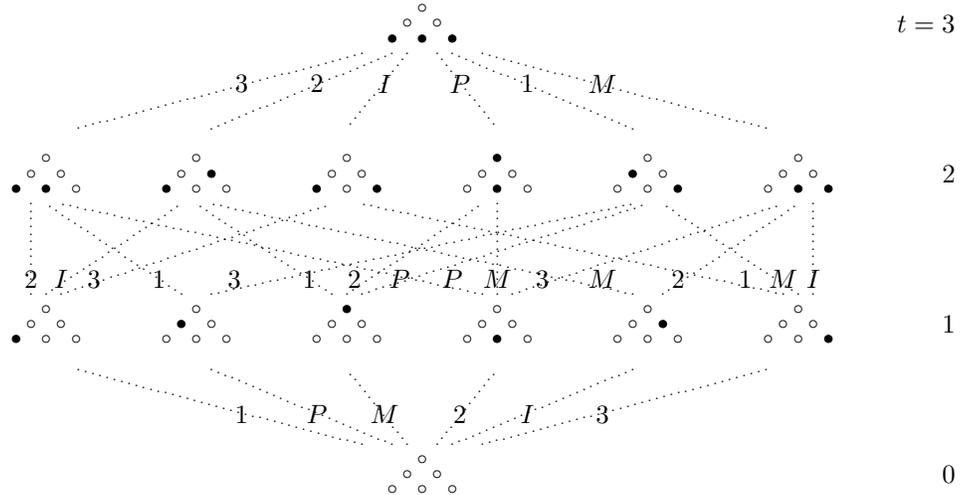
\hbox{\beginpicture
  \setcoordinatesystem units <2cm,2cm>
\put{$t=3$} at 3.35 3
\put{$2$} at 3.5 2
\put{$1$} at 3.5 1
\put{$0$} at 3.5 0

\put{$\beginpicture
  \setcoordinatesystem units <0.2cm,0.2cm>
  \multiput{$\sssize \circ$} at 0 0  1 1  -1 1  0 2  -2 0  2 0 /
  \multiput{$\ssize \bullet$} at  /
 \endpicture$} at 0 0

\put{$\beginpicture
  \setcoordinatesystem units <0.2cm,0.2cm>
  \multiput{$\sssize \circ$} at 0 0  1 1  -1 1  0 2  -2 0  2 0 /
  \multiput{$\ssize \bullet$} at -2 0 /
 \endpicture$} at -2.5 1
\put{$\beginpicture
  \setcoordinatesystem units <0.2cm,0.2cm>
  \multiput{$\sssize \circ$} at 0 0  1 1  -1 1  0 2  -2 0  2 0 /
  \multiput{$\ssize \bullet$} at  -1 1 /
 \endpicture$} at -1.5 1
\put{$\beginpicture
  \setcoordinatesystem units <0.2cm,0.2cm>
  \multiput{$\sssize \circ$} at 0 0  1 1  -1 1  0 2  -2 0  2 0 /
  \multiput{$\ssize \bullet$} at  0 2 /
 \endpicture$} at -.5 1
\put{$\beginpicture
  \setcoordinatesystem units <0.2cm,0.2cm>
  \multiput{$\sssize \circ$} at 0 0  1 1  -1 1  0 2  -2 0  2 0 /
  \multiput{$\ssize \bullet$} at  0 0 /
 \endpicture$} at 0.5 1
\put{$\beginpicture
  \setcoordinatesystem units <0.2cm,0.2cm>
  \multiput{$\sssize \circ$} at 0 0  1 1  -1 1  0 2  -2 0  2 0 /
  \multiput{$\ssize \bullet$} at  1 1 /
 \endpicture$} at 1.5 1
\put{$\beginpicture
  \setcoordinatesystem units <0.2cm,0.2cm>
  \multiput{$\sssize \circ$} at 0 0  1 1  -1 1  0 2  -2 0  2 0 /
  \multiput{$\ssize \bullet$} at  2 0 /
 \endpicture$} at 2.5  1
\put{$\beginpicture
  \setcoordinatesystem units <0.2cm,0.2cm>
  \multiput{$\sssize \circ$} at 0 0  1 1  -1 1  0 2  -2 0  2 0 /
  \multiput{$\ssize \bullet$} at -2 0  0 0 /
 \endpicture$} at -2.5 2
\put{$\beginpicture
  \setcoordinatesystem units <0.2cm,0.2cm>
  \multiput{$\sssize \circ$} at 0 0  1 1  -1 1  0 2  -2 0  2 0 /
  \multiput{$\ssize \bullet$} at -2 0  1 1  /
 \endpicture$} at -1.5 2
\put{$\beginpicture
  \setcoordinatesystem units <0.2cm,0.2cm>
  \multiput{$\sssize \circ$} at 0 0  1 1  -1 1  0 2  -2 0  2 0 /
  \multiput{$\ssize \bullet$} at  -2 0  2 0 /
 \endpicture$} at -.5 2
\put{$\beginpicture
  \setcoordinatesystem units <0.2cm,0.2cm>
  \multiput{$\sssize \circ$} at 0 0  1 1  -1 1  0 2  -2 0  2 0 /
  \multiput{$\ssize \bullet$} at  0 0  0 2 /
 \endpicture$} at 0.5 2
\put{$\beginpicture
  \setcoordinatesystem units <0.2cm,0.2cm>
  \multiput{$\sssize \circ$} at 0 0  1 1  -1 1  0 2  -2 0  2 0 /
  \multiput{$\ssize \bullet$} at  -1 1  2 0 /
 \endpicture$} at 1.5 2
\put{$\beginpicture
  \setcoordinatesystem units <0.2cm,0.2cm>
  \multiput{$\sssize \circ$} at 0 0  1 1  -1 1  0 2  -2 0  2 0 /
  \multiput{$\ssize \bullet$} at  0 0  2 0 /
 \endpicture$} at 2.5 2

\put{$\beginpicture
  \setcoordinatesystem units <0.2cm,0.2cm>
  \multiput{$\sssize \circ$} at 0 0  1 1  -1 1  0 2  -2 0  2 0 /
  \multiput{$\ssize \bullet$} at  -2 0  0 0  2 0 /
 \endpicture$} at  0 3

\setdots <1mm>

\plot -0.4 0.2 -2.3 0.7 /
\plot -0.2 0.2 -1.4 0.7 /
\plot -0.1 0.2 -0.5 0.7 /
\plot .1 0.2 0.5 0.7 /
\plot .2 0.2 1.4 0.7 /
\plot .4 0.2 2.3 0.7 /

\plot -0.4 2.8 -2.3 2.3 /
\plot -0.2 2.8 -1.4 2.3 /
\plot -0.1 2.8 -0.5 2.3 /
\plot .1 2.8 0.5 2.3 /
\plot .2 2.8 1.4 2.3 /
\plot .4 2.8 2.3 2.3 /

\plot -2.6 1.2  -2.6 1.8 /
\plot -2.5 1.2  -1.6 1.8 /
\plot -2.4 1.2  -0.6 1.8 /

\plot -1.6 1.2  -2.5 1.8 /
\plot -1.4 1.2   1.4 1.8 /

\plot -.6 1.2  -1.5 1.8 /
\plot -.5 1.2   0.4 1.8 /
\plot -.4 1.2   1.5 1.8 /

\plot .4 1.2  -2.4 1.8 /
\plot .5 1.2   .5 1.8 /
\plot .6 1.2   2.4 1.8 /

\plot 1.4 1.2  -1.4 1.8 /
\plot 1.6 1.2   2.5 1.8 /

\plot 2.4 1.2  -.4 1.8 /
\plot 2.5 1.2   1.6 1.8 /
\plot 2.6 1.2   2.6 1.8 /

\multiput{$1$} at -1.2 0.4 
   -1.75 1.3  
    -.75 1.3 
    2.15 1.3  
    0.7 2.6  / 

\multiput{$P$} at -0.7 0.4  
   -.15 1.3  
   .2 1.3  
   0.25 2.6   / 

\multiput{$M$} at  -0.25 0.4 
    .5 1.3  1.2 1.3
    2.4 1.3 
    1.2 2.6 /

\multiput{$2$} at .25 0.4  
   -2.6 1.3  
   -.45 1.3  
    1.7 1.3  
   -0.7 2.6   /

\multiput{$I$} at   .7 0.4 
   -2.4 1.3  
    2.6 1.3 
  -.25 2.6 /

\multiput{$3$} at 1.2 0.4 
   -2.18 1.3  -1.25 1.3  .8 1.3 
   -1.2 2.6 /

\endpicture}
$$
We obtain all the complete exceptional sequences {\bf going down} in the lattice, using the
modules $M_{\mathcal A}^{\mathcal B}$. (For example, on the left hand side, we obtain in this way
the sequence $(3,2,1)$ of simple modules, on the right hand side there is the sequence
$(M,I,3)$ of the indecomposable injective modules.) 
	\bigskip

For an inductive formula for determining the number of complete exceptional sequences we refer
to \cite{[ONAFR]}. 
	\medskip 

\subsection{The braid group operation on $\E(\mo\Lambda)$ and on $\M(\A(\mo\Lambda))$}
	\bigskip

Recall that the braid group $B_n$ on $n$ strands is the group with generators
$\sigma_1,\dots \sigma_{n-1}$ and the relations
	\smallskip

\Rahmen{$\begin{matrix} 
 \sigma_i\sigma_{i+1}\sigma_i = \sigma_{i+1}\sigma_{i}\sigma_{i +1} &
 \text{for} & 1 \le i \le n-2 \cr
 \sigma_i\sigma_j = \sigma_j\sigma_i & \text{for} & |i-j| \ge 2 \end{matrix} $.}
	\bigskip 


The braid group $B_n$ operates on $\E= \E(\mo\Lambda)$ as follows:
Let $E = (E_1,\dots,E_n)$ be a complete exceptional sequence in $\mo\Lambda$ (and we may
assume that $n\ge 2$). 
Let $\mathcal E(i)$ be the thick subcategory generated by $E_i$ and $E_{i+1}$ and
$\overline \tau_i$ the extended Auslander-Reiten translation for $\mathcal E(i)$.
Then we define
$$
 \sigma_i(E_1,\dots,E_n) = (E_1,\dots,E_{i-1},\overline\tau_iE_{i+1},E_i,E_{i+2},\dots,E_n).
$$
Note that $\sigma_i(E_1,\dots,E_n)$ is the unique complete exceptional sequence which is of the form
$(E_1,\dots,E_{i-1},?,E_i,E_{i+2},\dots,E_n).$
	\medskip
It is easy to see that the braid group relations
are satisfied (see the note N\,\ref{braid-group-relations}).
	\bigskip

\subsubsection{}
We can use the map 
$$
 M:\M(\A(\mo\Lambda)) \to \E(\mo\Lambda)
$$
defined in Theorem \ref{neighbors1} (or also the map $N$ defined in the Addendum 
\ref{neighbors2}) 
in order to provide 
a corresponding braid group operation on the set  $\M(\A(\mo\Lambda))$.

\begin{theorem}\label{braid-exc}
 The braid group acts transitively on the set $\bold E$ of complete exceptional
sequences.
\end{theorem} 
	
In the case of the path algebra of a quiver, this result is due to Crawley-Boevey \cite{[CB]}, for the
general case, we refer to  \cite{[R5]}.
	
\begin{proof} Let $n =\rank \Lambda$.
Let $E = (E_1,\dots,E_n)$ be an exceptional sequence. We show that using the braid group operation, 
we obtain from $E$ an exceptional antichain. We show this by induction on 
the sum $|E|$ of the length of the modules $E_i$. 

The proof is based on the following observation: 
For any $1\le i \le n-1$, there is some $t = t(i)$ such that 
$$
 \sigma_i^t(E) =
(E_1,\dots,E_{i-1},T,S,E_{i+2},\dots, E_n)
$$ 
with $\Hom(T,S) = 0,$ and either we have already $\Hom(E_i,E_{i+1}) = 0$ or else 
$|\sigma_i^t(E)| < |E|.$ 
	\medskip 

Thus, assume that $\Hom(E_i,E_j) \neq 0$ for some $i\neq j.$ We show that there is a braid
group element $\gamma$ such that $|\gamma(E)| < |E|.$ 
	\medskip 

Since $E$ is exceptional, we must have $i < j$
and we can assume that $j-i$ is minimal. Let $f\!:E_i \to E_j$ be a non-zero map. Such a map is either
injective or surjective. Up to duality, we can assume that $f$ is injective.
We define $E' = \sigma_{j-2}\cdots\sigma_i(E)$ and show that $|E'| = |E|$ and $|\sigma_{j-1}^t(E')| = |E'|$, for some
$t$, thus $|\gamma(E)| < |E|$ for $\gamma = \sigma_{j-1}^t\sigma_{j-2}\cdots\sigma_i$.

Thus, first let us consider any $s$ with  $i < s < j.$ 
By the minimality of $j-i$, we have $\Hom(E_i,E_s) = 0$. We also have 
$\Ext^1(E_i,E_s) = 0$ (namely, the map $f\!:E_i\to E_j$ is injective, thus 
$$
 \Ext^1(f,E_s)\!:\Ext^1(E_j,E_s)\to \Ext^1(E_i,E_s)
$$ 
is surjective, but on the other hand $\Ext^1(E_j,E_s) = 0$, since $j > s$).
Thus, applying the braid group element $\sigma_{j-2}\cdots\sigma_i$ to $E$, we deal just with a
permutation: the entry $E_i$ is shifted from the position $i$ to the position $j-1$, we obtain
$$
 E' = \sigma_{j-2}\cdots\sigma_i(E) = (E_1,\dots,E_{i-1},E_{i+1},\cdots,E_{j-1},E_i,E_j,\cdots,E_n).
$$
and $|E'| = |E|$.

Second,  we apply powers of $\sigma_{j-1}$ to $E'$. We have 
$$
 \sigma_{j-1}^t(E') = (E'_1,\dots,E'_{j-2},T,S,E'_{j+1},\dots, E'_n)
$$ 
for some $t$ with $\Hom(T,S) = 0$. 
Now $E_i,E_j$ both belong to $\mathcal E(S,T)$ and
at most one of these modules is simple in $\mathcal E(S,T)$, it follows that $|\sigma^t(E')| < |E'| = |E|.$
	\medskip

Finally, we note that an element $E = (E_1,\dots, E_n)$ in $\bold E$ which is an antichain consists of
the simple modules. Namely, $\mathcal E$ is an exceptional subcategory of $\mo\Lambda$ of
rank $n$, thus according to Corollary \ref{thick-closure} 
we have $\mo\Lambda = \mathcal E$. Thus, every simple
$\Lambda$-module has a filtration with factors in $E$ and therefore belongs to $E.$

Of course, the antichains in $\bold E$ are obtained from each other by a sequence of transpositions
of simple modules which do not extend each other, thus by elements of the braid group.  
\end{proof} 
	\medskip

\subsection{Generalized non-crossing partitions}\label{main}
We come now to the 
central concept. 
	\medskip

\subsubsection{\bf The absolute length.} 
We start with a Coxeter group $W = (W,S)$,
thus $W$ is a group, $S$ is a finite set of generators of $W$
such that $s^2 = 1$ and $(st)^{m(s,t)} = 1$ where $m(s,t) = m(t,s)\ge 2$ 
are natural numbers or the symbol $\infty$, and these are the defining relations. The
cardinality of $S$ is called the {\it rank} of $(W,S)$. If $(W,S)$ is a Coxeter
group of rank $n$, the elements of the form $s_1s_2\cdots s_n$, where $s_1,\dots,s_n$ 
are the elements of $S$ (in some order) are called the {\it Coxeter elements} of $(W,S).$

The elements which are conjugate to an element in $S$ are called {\it reflections}, we denote
by $T$ the set of reflections. 
Any element $w\in W$ is a product of elements from $S$, thus a product of reflections,
and we write $|w|_a$ for the least number $t$ such that $w$ can be written as a product
of $t$ reflections. If $v,w\in W$, write $v\le_a w$ provided $|v|_a+|v^{-1}w|_a = |w|_a;$
this is a partial ordering on $W$ called the {\it absolute ordering}. If $v \le_a w$, we say
that $v$ is an {\it absolute factor} of $w$. 

If $w$ is of absolute length $n$, say $w = t_1t_2\cdots t_n$
with reflections $t_1,\dots,t_n$, then obviously any prefix word $t_1t_2\cdots t_r$ with $r\le n$
is an absolute factor of $w$
of absolute length $r$, but actually any word obtained from $t_1t_2\cdots t_n$
by deleting $n-r$ letters is an absolute factor of $w$ of absolute length $r$.
Namely, consider the case that we delete the letter $t_i$, thus, we 
consider the word $uv$ where $u = t_1\cdots  t_{i-1}$ and $v = t_{i+1}\cdots t_n$.
Then $(uv)^{-1}w = v^{-1}u^{-1}ut_iv = v^{-1}t_iv$ is in $T$, thus of absolute length 1.
Of course, we see directly that $uv = t_i\cdot t_{i-1}t_{i+1}\cdots t_n$ is of absolute
length at most $n-1.$ 
	\medskip 

\subsubsection{\bf The definition.}
The poset 
$$
 \NC(W,c) = \{w\in W\mid w\le_a c \}
$$
(using $\le_a$ as ordering) is called the poset of {\it generalized non-crossing partitions
of type $(W,c)$.} Of course, if $c,c'$ are conjugate Coxeter elements in $W$, then the posets
$\NC(W,c)$ and $\NC(W,c')$ are isomorphic. 
In case the graph $\Delta(W,S)$ is a tree, in particular, if $W$ is finite, then 
all Coxeter elements are conjugate. But in general, there are
Coxeter elements $c,c'$ which are not conjugate. 

We have seen above that the posets $\A(\mo\Lambda)$, where $\Lambda$ is the path algebra
of the quiver $\widetilde{\mathbb A}_{2,2}$ or $\widetilde{\mathbb A}_{3,1}$ 
behave differently, see the examples \ref{example-not-lattice} and 
\ref{example-lattice}. Note that the quivers $\widetilde{\mathbb A}_{2,2}$, 
 $\widetilde{\mathbb A}_{3,1}$  have the same underlying graph,
but they are distinguished by the orientation, 
thus by the choice of a Coxeter element.
As we have seen, 
the posets $\A(\mo\Lambda)$ are not
isomorphic (one is a lattice, the other one not). As we will see in 
Theorem \ref{IST}, this means that
the corresponding posets $\NC(W(\Lambda),c(\Lambda))$ 
are not isomorphic, thus the Coxeter elements $c(\Lambda)$ cannot be
conjugate.
	\medskip 

For any natural number $r$, let $\NC_r(W,c)$ be the subset of $\NC(W,c)$ of all 
elements of absolute length $r$. Using the sets $\NC_r(W,c)$ with
$0\le r \le n$, the poset $\NC(W,c)$ is layered:
If $v <_a w$ are neighbors in $\NC_r(W,c)$ and $w$ belongs to
$\NC_r(W,c)$, then $v$ belongs to $\NC_{r-1}(W,c).$  
The set $\NC_0(W,c)$ consists just of the identity element of $W$, 
the set $\NC_n(W,c)$ of the single element $c$.
	\bigskip 

The absolute length was introduced Brady-Watt and Bessis in 2002 and 2003 in analogy to the 
usual length function $|-|$ which has been considered in Lie theory much earlier, see the
note N\,\ref{Coxeter-length}. 
	\bigskip

\subsubsection{\bf Example. The case $\mathbb A_n$.} 
The corresponding Weyl group is $W = S_{n+1}$, 
the symmetric group on $n+1$ letters $1,2,\dots,n+1$.
Usually we will consider the Coxeter element $c_{n+1} = (n,n+1)\cdots(23)(12) = (n+1,n,\cdots,2,1)$.

For $n=3$, the poset $(W,\le_a)$ looks as
follows:
$$
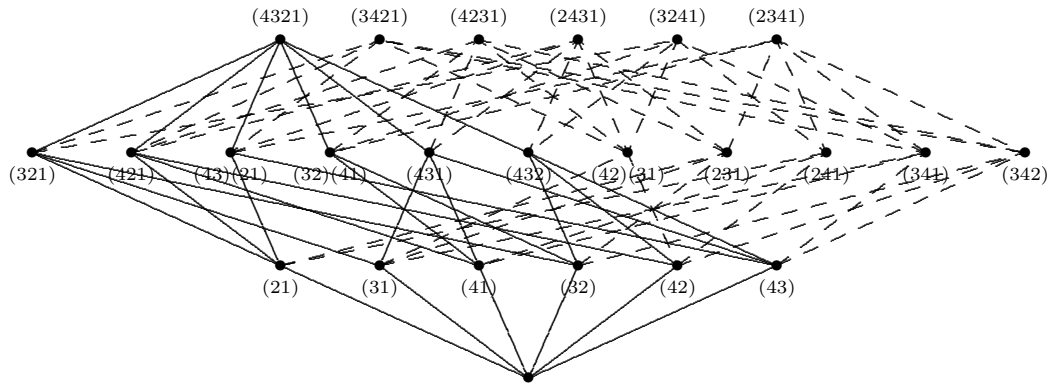
\hbox{\beginpicture
\setcoordinatesystem units <.66cm,1.5cm>

\put{} at 10 0
\multiput{$\bullet$} at 10 0  
  5 1  7 1  9 1  11 1  13 1  15 1
0 2  2 2  4 2  6 2  8 2  10 2  12 2  14 2  16 2  18 2  20 2  
  5 3  7 3  9 3  11 3  13 3  15 3  /
\plot 10 0  5 1 /
\plot 10 0  7 1 /
\plot 10 0  9 1 /
\plot 10 0  11 1 /
\plot 10 0  13 1 /
\plot 10 0  15 1 /

\put{$\ssize (21)$} at 5 .8
\put{$\ssize (31)$} at 7 .8
\put{$\ssize (41)$} at 9 .8
\put{$\ssize (32)$} at 11 .8
\put{$\ssize (42)$} at 13 .8
\put{$\ssize (43)$} at 15 .8

\put{$\ssize (321)$} at  0 1.8
\put{$\ssize (421)$} at 2 1.8
\put{$\ssize (43)(21)$} at 4 1.8
\put{$\ssize (32)(41)$} at 6 1.8
\put{$\ssize (431)$} at 8 1.8
\put{$\ssize (432)$} at 10 1.8

\setsolid
\setlinear
\plot 5 1  0 2  7 1  8 2  15 1  10 2  11 1  6 2  9 1  2 2  5 1 /
\plot 5 1  4 2  15 1 /
\plot 0 2  11 1  /
\plot 2 2  13 1  10 2 /
\plot  9 1  8 2 /
\plot 0 2  5 3  2 2  /
\plot 4 2  5 3  6 2 /
\plot 8 2  5 3  10 2 /

\setdashes <2mm>

\put{$\ssize (42)(31)$} at 12 1.8
\plot 7 1  12 2  13 1 /

\put{$\ssize (231)$} at 14 1.8
\plot 5 1  14 2  7 1 /
\plot 11 1  14 2 /

\put{$\ssize (241)$} at 16 1.8
\plot 5 1  16 2  9 1 /
\plot 13 1  16 2 /

\put{$\ssize (341)$} at 18 1.8
\plot 7 1  18 2  9 1 /
\plot 15 1  18 2 / 

\put{$\ssize (342)$} at 20 1.8
\plot 11 1  20 2  13 1 /
\plot 15 1  20 2 /

\put{$\ssize (4321)$} at  5 3.2

\put{$\ssize (3421)$} at  7 3.2
\plot 12 2  7 3  4 2 /
\plot 0 2  7 3  2 2 /
\plot 20 2  7 3  18 2 /

\put{$\ssize (4231)$} at  9 3.2
\plot 6 2  9 3  12 2  /
\plot 14 2  9 3  2 2   /
\plot 8 2  9 3  20 2  /

\put{$\ssize (2431)$} at  11 3.2
\plot 4 2  11 3  12 2 /
\plot 14 2  11 3  2 2 /
\plot 8 2  11 3  10 2  /

\put{$\ssize (3241)$} at  13 3.2
\plot 6 2  13 3  12 2 /
\plot 18 2  13 3 16 2 /
\plot 0 2  13 3  10 2 /

\put{$\ssize (2341)$} at  15 3.2
\plot 4 2  15 3  6 2 /
\plot 18 2  15 3  16 2 /
\plot 14 2  15 3  20 2 /

\endpicture}
$$

The solid edges describe the subset $\NC(S_4,c_4)$, the dashed edges are the remaining ones. 
The following picture shows just $\NC(S_4,c_4)$:
$$
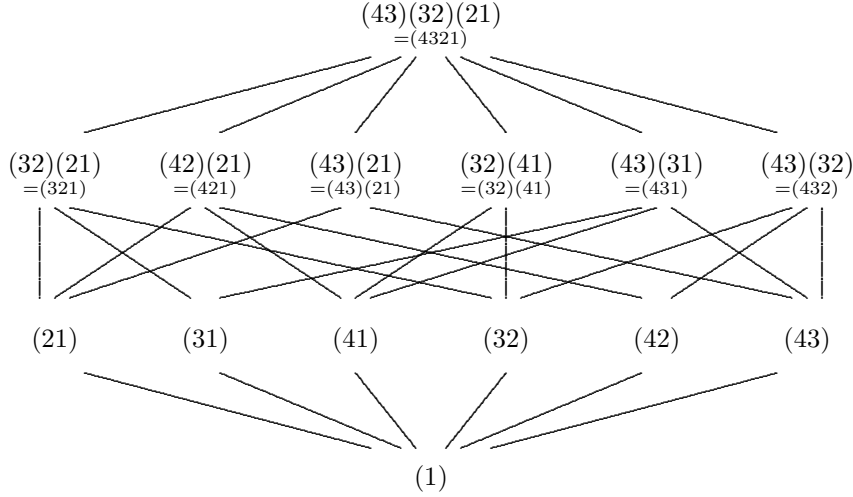
\hbox{\beginpicture
  \setcoordinatesystem units <2cm,2cm>

\put{$(1)$} at 0 0

\put{$(21)$} at -2.5 .92

\put{$(31)$} at -1.5 .92

\put{$(41)$} at -.5 .92

\put{$(32)$} at .5 .92

\put{$(42)$} at 1.5 .92

\put{$(43)$} at 2.5 .92

\put{$(32)(21)$} at -2.5 2.08
\put{$\ssize = (321)$} at -2.5 1.92

\put{$(42)(21)$} at -1.5 2.08
\put{$\ssize = (421)$} at -1.5 1.92

\put{$(43)(21)$} at -.5 2.08
\put{$\ssize = (43)(21)$} at -.5 1.92

\put{$(32)(41)$} at .5 2.08
\put{$\ssize = (32)(41)$} at 0.5 1.92

\put{$(43)(31)$} at 1.5 2.08
\put{$\ssize = (431)$} at 1.5 1.92

\put{$(43)(32)$} at 2.5 2.08
\put{$\ssize = (432)$} at 2.5 1.92

\put{$(43)(32)(21)$} at 0 3.08
\put{$\ssize = (4321)$} at  0 2.92

\plot -0.4 0.2 -2.3 0.7 /
\plot -0.2 0.2 -1.4 0.7 /
\plot -0.1 0.2 -0.5 0.7 /
\plot .1 0.2 0.5 0.7 /
\plot .2 0.2 1.4 0.7 /
\plot .4 0.2 2.3 0.7 /

\plot -0.4 2.8 -2.3 2.3 /
\plot -0.2 2.8 -1.4 2.3 /
\plot -0.1 2.8 -0.5 2.3 /
\plot .1 2.8 0.5 2.3 /
\plot .2 2.8 1.4 2.3 /
\plot .4 2.8 2.3 2.3 /

\plot -2.6 1.2  -2.6 1.8 /
\plot -2.5 1.2  -1.6 1.8 /
\plot -2.4 1.2  -0.6 1.8 /

\plot -1.6 1.2  -2.5 1.8 /
\plot -1.4 1.2   1.4 1.8 /

\plot -.6 1.2  -1.5 1.8 /
\plot -.5 1.2   0.4 1.8 /
\plot -.4 1.2   1.5 1.8 /

\plot .4 1.2  -2.4 1.8 /
\plot .5 1.2   .5 1.8 /
\plot .6 1.2   2.4 1.8 /

\plot 1.4 1.2  -1.4 1.8 /
\plot 1.6 1.2   2.5 1.8 /

\plot 2.4 1.2  -.4 1.8 /
\plot 2.5 1.2   1.6 1.8 /
\plot 2.6 1.2   2.6 1.8 /

\endpicture}
$$
Here the permutations $\pi$ are written as products of reflections as well
as in cycle notation, but we note the following inconsistency: The cycle
$(a_1,a_2,\dots,a_t)$ is the permutation which sends $a_i$ to $a_{i+1}$, for
$1\le i < t$ and $a_t$ to $a_1$, thus one uses the shift to the right. On the other hand,
the product $\pi\pi'$ of two permutations $\pi,\pi'$
means that we apply first $\pi'$, then $\pi$. In particular,
the product $(32)(21)$is the permutation 
$(32)(21) =\left( \smallmatrix 1&2&3\cr
                               3&1&2 \endsmallmatrix\right)$,
and this is the cycle $\psi = (321) = (3,\psi(3),\psi^2(3))$.
	\medskip 

We will see in Chapter 4 that the lattice $\NC(S_n,c_n)$ can be identified in
a natural way with the classical lattice $\NC(n)$ of the non-crossing partitions as
introduced by Kreweras in \cite{[Kr]}, see Theorem \ref{classical-identification}.
	\bigskip

\subsubsection{\bf The map $\cox\!:\A(\mo\Lambda)\to \NC(W(\Lambda),c(\Lambda))$.}
	\medskip 

Given an exceptional module $E$, we denote by $\rho_E$ the reflection with respect to
$\bdim E$, this is a reflection in the Coxeter group $W(\Lambda),c(\Lambda))$.

We define a map 
$$
 \cox\!:\A(\mo\Lambda) \to W(\Lambda)
$$
as follows: Let $\mathcal A$ be an exceptional subcategory of $\mo\Lambda$ of rank $t$, then 
$$
 \cox(\mathcal A) = \rho_{A_1}\cdots\rho_{A_t},
$$
where $\{A_1,\dots,A_t\}$ is an exceptional antichain with  $\mathcal A = \mathcal E(A_1,\dots,A_t)$.
Note that the composition $\rho_{A_1}\cdots\rho_{A_t}$ is independent from the choice
of $\{A_1,\dots,A_t\}$: namely a different choice  $\{A'_1,\dots,A'_t\}$ 
with $\mathcal A = \mathcal E(A'_1,\dots,A'_t)$
is obtained from
$\{A_1,\dots,A_t\}$ by just permuting modules $A_i,A_j$ with 
$\rho_{A_i}\rho_{A_j} = \rho_{A_j}\rho_{A_i}$. 
	\medskip

In particular, we have
	\smallskip

\Rahmen{$\cox(\mo\Lambda) = c(\Lambda).$}
	\bigskip 
The main observation is the following: 
	\medskip

\begin{lemma} Let $(E_1,\dots,E_n)$ and $(E'_1,\dots,E'_n)$ be complete
exceptional sequences in $\mo\Lambda$. Then
$$
 \rho_{E_1}\cdots\rho_{E_n} = \rho_{E'_1}\cdots\rho_{E'_n}.
$$
\end{lemma}
	
\begin{proof}  Since the braid group operates transitively on the set of exceptional
sequences, it is sufficient to assume that $(E'_1,\dots,E'_n) = \sigma_i(E_1,\dots,E_n)$,
thus we can assume that $n=2$ and that we deal with exceptional pairs of the form
$(E,E')$ and $(E'',E).$ But in this case $\bdim E'' = \pm\,\rho_E(\bdim E')$ and therefore
$\rho_{E''} = \rho_E\rho_{E'}(\rho_E)^{-1},$ thus $\rho_E\rho_{E'} = \rho_{E''}\rho_E,$
thus $\rho_{E''}\rho_E = \rho_E\rho_{E'}.$
\end{proof} 

\begin{cor}\label{prod-exc-sequence} If $(E_1,\dots,E_n)$ is any complete exceptional sequence in
$\mo\Lambda$, then 
	\smallskip

\Rahmen{$\rho_{E_1}\cdots\rho_{E_n} = c(\Lambda).$}
	\smallskip 

\end{cor}

\begin{prop} 
Let $\mo\Lambda$ be of rank $n$ and 
let $\mathcal A$ be an exceptional subcategory of rank $m$.
Then there is the following product formula in $W(\Lambda)$:
	\smallskip 

\Rahmen{$\cox(\mathcal A)\cdot\cox({}^\perp\mathcal A) = \cox(\mo\Lambda)).$}
\end{prop}

\begin{proof}  Let $(E_1,\dots,E_m)$ be a complete exceptional sequence of $\mathcal A$
and let $(E_{m+1},\dots,E_n)$ be a complete exceptional sequence of $^\perp\mathcal A.$
Then $(E_1,\dots, E_n)$ is a complete exceptional sequence of $\mo\Lambda$.
We have 
$\cox(\mathcal A) = \sigma_{E_1}\cdots\sigma_{E_m}$, 
$\cox(\mathcal A^\perp) = \sigma_{E_{m+1}}\cdots\sigma_{E_m}$, and 
$\cox(\mo\Lambda) = \sigma_{E_1}\cdots\sigma_{E_n}$, thus
$$
 \cox(\mathcal A)\cdot\cox({}^\perp\mathcal A) 
 = \sigma_{E_1}\cdots\sigma_{E_m}\cdot \sigma_{E_{m+1}}\cdots\sigma_{E_m}
 = \sigma_{E_1}\cdots\sigma_{E_n}
 = \cox(\mo\Lambda)).
$$
\end{proof}

In particular, since
$$
 |\cox(\mathcal A)|_a = m, \quad 
 |\cox({}^\perp\mathcal A)|_a = n-m, \quad 
$$
we see in this way that for any exceptional subcategory $\mathcal A$, we have
$$
 \cox(\mathcal A) \le_a \cox(\mo\Lambda) = c(\Lambda),
$$
thus $\cox(\mathcal A)$ belongs to $\NC(W(\Lambda),c(\Lambda))$.

\begin{theorem}\label{IST} {\bf (Igusa-Schiffler-Thomas).}
The map
	\smallskip

\Rahmen{$\cox\!:\A(\mo\Lambda) \longrightarrow \NC(W(\Lambda),c(\Lambda))$.}
	\smallskip

\noindent
is an isomorphism of posets.
\end{theorem}

In case $\Lambda$ is of finite or tame representation type, the result is due to
Ingalls and Thomas \cite{[IT]}, for the general case, see \cite{[IS]} and also \cite{[K2]}.
The proof will be outlined 
in the next section.
	\medskip  

\subsection{The braid group operation on $\F(W,c)$} 
Given any group $G$, the braid group $B_n$ operates on the set $G^n$ of $n$-tuples of
elements of $G$ as follows:
$$
 \sigma_i(g_1,\dots,g_n) = (g_1,\dots,g_{i-1},g_ig_{i+1}g_i^{-1},g_i,g_{i+1},\dots,g_n)
$$
we shift 
the $i$-th entry one place to the right and insert on the $i$-place the conjugate 
$g_ig_{i+1}g_i^{-1}$ of $g_{i+1}$; 
this operation is sometimes called the {\it Hurwitz action}, see for example \cite{[BDSW]}.  
Note that for $\sigma_i(g_1,\dots,g_n) = (g_1',\dots,g_n')$,
the products coincide: 
$$
  g_1g_2\cdots g_n = g_1'g_2'\cdots g_n'.
$$
	\medskip 

\subsubsection{} 
Given a Coxeter group $(W,S)$ of rank $n$ with set of reflections $T$, one may consider the
set $T^n$  of $n$-tuples of reflections. Since $T$ is closed under conjugation, 
the Hurwitz action sends $T^n$ into itself. If $c$ is an element of $W$ (of interest for us is
the case that $c$ is Coxeter element), we denote by $\F(W,c)$ the set of $n$-tuples $(t_1,\dots,t_n)$
of elements of $T$ with product $t_1t_2\cdots t_n = c$.
There is the following theorem parallel to Theorem \ref{braid-exc}
(for the proof, we refer to \cite{[IS]}).
	\bigskip

\begin{theorem} {\bf (Igusa-Schiffler).} Let $(W,S)$ be a Coxeter group of rank $n$ and 
$c$ a Coxeter element. The braid group $B_n$ operates transitively on the set $\F(W,c)$.
\end{theorem} 
	
Given an exceptional sequence $E = (E_1,\dots,E_n)$
we define 
$$
 \rho(E) = (\rho_{E_1},\dots,\rho_{E_n}),
$$ 
this is an $n$-tuple in $T$
and according to Corollary \ref{prod-exc-sequence}, the product is just the Coxeter element $c$, thus $\rho(E)$
belongs to $\F(W(\Lambda),c(\Lambda),$ this shows that we obtain a map
$$
 \rho\!:\E(\mo\Lambda) \to \F(W(\Lambda),c(\Lambda)).
$$
	\medskip

\begin{cor}\label{IS-cor}

{\rm(1)} If $E = (E_1,\dots,E_n)$ is a sequence of exceptional modules and 
$\rho_{E_1}\cdots\rho_{E_n} = c(\Lambda),$
then $E$ is an exceptional sequence.

{\rm (2)} The map $\rho\!:\E(\mo\Lambda) \to \F(W(\Lambda),c(\Lambda))$ is a 
$B_n$-equivariant bijection.
\end{cor}

	\medskip

\begin{proof}  The map which attaches to an exceptional module $X$ the reflection $\rho_X$ is
an injective map from the set of isomorphism classes of exceptional modules to the set $T$,
thus the map $\rho\!:\E(\mo\Lambda) \to T^n$ is injective. As we know, the image of this map
$\rho$ lies in $\F(W(\Lambda),c(\Lambda))$. If $1\le i < n$, and $E$ is an exceptional
sequence, then we know that $\rho(\sigma_i(E)) = \sigma_i\rho(E)$. It follows that
the map $\rho\!:\E(\mo\Lambda) \to \F(W(\Lambda),c(\Lambda))$ is $B_n$-equivariant.
In order to show that this map is surjective, consider an element $(t_1,\dots,t_n)$ in
$\F(W(\Lambda),c(\Lambda))$. 

Let $S_1,\dots,S_n$ be the simple $\Lambda$-modules ordered in such a way that 
$(S_1,\dots,S_n)$ is an exceptional sequence.
Let $\rho_{S_i} = s_i$, thus , thus $\rho(S_1,\dots,S_n)
= s_1\cdots s_n = \cox(\Lambda)$
and therefore $(s_1,\dots,s_n)$ belongs to $\F(W(\Lambda),c(\Lambda))$.
Since $B_n$ operates transitively on $\F(W(\Lambda),c(\Lambda))$,
there is an element $g\in B_n$ such that $(t_1,\dots,t_n) = g(s_1,\dots,s_n)$.
If we apply $g$ to $(S_1,\dots,S_n),$ we obtain an exceptional sequence
$(E_1,\dots,E_n)$, and 
$$
 \rho(E_1,\dots,E_n) = \rho(g(S_1,\dots,S_n)) = g\rho(S_1,\dots,S_n) = g(s_1,\dots,s_n) = 
 (t_1,\dots,t_n).
$$
\end{proof}

For $0\le r \le n$ let us denote by $\F_r(W,c)$ the set of $r$-tuples $(t_1,\dots,t_r)$ in $T$
which can be completed to an $n$-tuple $(t_1,\dots,t_n)\in \F(W,c)$.
And  we denote by $\E_r(\mo\Lambda)$ the set of exceptional sequences in $\mo\Lambda$
of length $r$.
	\medskip

\begin{prop}\label{bijection} The map 
$$
 \rho\!:\E_r(\mo\Lambda) \to \F_r(W(\Lambda),c(\Lambda))
$$
defined by $\rho(E_1,\dots,E_r) = (\rho_{E_1},\dots,\rho_{E_r})$  is a bijection.
\end{prop}

\begin{proof}
Since any exceptional sequence $(E_1,\dots,E_r)$ can be completed to a
complete exceptional sequence $(E_1,\dots,E_n)$, we see that our map $\rho$ sends
$\E_r(\mo\Lambda)$ into $\F_r(W(\Lambda),c(\Lambda))$. 

On the other hand, given an element
$(t_1,\dots,t_r)$ in $\F_r(W(\Lambda),c(\Lambda))$, it 
can be completed to an $n$-tuple $(t_1,\dots,t_n)\in \F(W,c)$. As we have seen,
$(t_1,\dots,t_n) = \rho(E_1,\dots,E_n)$ for some complete exceptional sequence $(E_1,\dots,E_n)$
and therefore $(t_1,\dots,t_r) = \rho(E_1,\dots,E_r)$.
\end{proof} 
	
Let us stress that the injectivity assertion implies the following:
	\medskip 

{\it If $E_1,\dots,E_r$ are exceptional modules such that $\rho_{E_1}\cdots\rho_{E_r} = \cox(\mathcal A)$,
where $\mathcal A$ is exceptional of rank $r$, then all the modules $E_1,\dots,E_r$ belong to
$\mathcal A$.}
	\bigskip

\begin{theorem} The map $\cox\!:  \A_r(\mo\Lambda) \to \NC_r(W(\Lambda),c(\Lambda))$
is bijective, and there is the following diagram
	\smallskip  

\Rahmen
{$$
\hbox{\beginpicture
  \setcoordinatesystem units <4cm,1.2cm>
\put{$\E_r(\mo\Lambda)$} at 0 0
\put{$\F_r(W(\Lambda),c(\Lambda))$} at 1 0
\put{$\A_r(\mo\Lambda)$} at 0 -1
\put{$\NC_r(W(\Lambda),c(\Lambda))$} at 1 -1
\arr{0.4 0}{0.6 0}
\arr{0.4 -1}{0.6 -1}
\arr{0 -.3}{0 -.7}
\arr{1 -.3}{1 -.7}
\put{$\rho$} at 0.5 0.2
\put{$\cox$} at 0.5 -0.8
\put{$\gamma$} at 0.05 -0.5
\put{$\mu$} at 1.06 -0.5
\endpicture}
$$}
	\smallskip

\noindent
where $\gamma$ sends any exceptional sequence $E$ to the exceptional subcategory generated by $E$ 
and $\mu$ is the multiplication map which sends $(t_1,\dots,t_r)$
to $\mu(t_1,\dots,t_r) = t_1\cdots t_r$.
\end{theorem} 

\begin{proof} It is clear that the diagram commutes. 

By the definition of $\NC_r(W(\Lambda),c(\Lambda))$, the map $\mu$ is surjective.
Since $\rho$ is bijective and $\mu$ is surjective, we see that $\cox$ is surjective.

Let us show that $\cox$ is injective. Let $\mathcal A, \mathcal A'$ be exceptional subcategories of rank $r$
of $\mo\Lambda$ such that $\cox(\mathcal A) = \cox(\mathcal A').$ 
Let $E = (E_1,\dots,E_r)$ be a complete exceptional sequence in $\mathcal A$ and
$(E_1',\dots,E_r')$ a complete exceptional sequences in $\mathcal A'$. 
Let $t_i = \rho_{E_i}$ and $t'_i = \rho_{E_i'}$, then  
$\cox(E) = t_1\cdots t_r$ and 
$\cox(E') = t'_1\cdots t'_r$. 
Now $(t_1,\dots,t_r)$ belongs to $\F_r(W(\Lambda),c(\Lambda))$, thus $(t_1,\dots,t_r)$
can be completed to an element $(t_1,\dots,t_n)$ in $\F(W(\Lambda),c(\Lambda))$.
But also the element $(t'_1,\dots,t'_r,t_{r+1},\dots,t_n)$ belongs to 
$\F(W(\Lambda),c(\Lambda))$, since it is an $n$-tuple of elements of $T$ with product
equal to $c(\Lambda)$. For $r+1\le i \le n$, write $t_i = \rho_{E_i}$ for some 
exceptional module $E_i$. 

It follows from part (1) of Corollary \ref{IS-cor} that both $(E_1,\dots,E_r,E_{r+1},\dots,E_n)$
and  $(E'_1,\dots,E'_r,E_{r+1},\dots,E_n)$ are (complete) exceptional sequences of $\mo\Lambda$.
Let $\mathcal B$ be the thick subcategory generated by $E_{r+1},\dots,E_n$.
Since both $E_1,\dots,E_r$ as well as $E'_1,\dots,E'_r$
are complete exceptional sequences in $\mathcal B^\perp$, we see that
$\mathcal A = \gamma(E_1,\dots,E_r) = \mathcal B^\perp = \gamma(E'_1,\dots,E'_r) = \mathcal A'$.
\end{proof}

	\medskip 

\subsubsection{\bf Proof of Theorem \ref{IST}.} We have seen already that
$$
 \cox\!:  \A(\mo\Lambda) \to \NC(W(\Lambda),c(\Lambda))
$$
is bijective. It remains to be shown:
Given 
exceptional subcategories $\mathcal A, \mathcal A'$ of $\mo(\Lambda)$, then $\mathcal A \subseteq \mathcal A'$
if and only if $\cox\mathcal A \le_a \cox\mathcal A'.$

Of course, if $\mathcal A \subseteq \mathcal A'$, then take a complete exceptional sequence 
$(E_1,\dots,E_r)$ of $\mathcal A,$
extend it to a complete exceptional sequence $(E_1,\dots,E_s)$ of $\mathcal A'$. It follows that
$$
 \cox(\mathcal A) = \rho_{E_1}\cdots\rho_{E_r} \le_a \rho_{E_1}\cdots\rho_{E_s} = \cox(\mathcal A').
$$

Conversely, assume that $\cox(\mathcal A) \le_a \cox(\mathcal A').$ Again, let  
$(E_1,\dots,E_r)$ be a complete exceptional of $\mathcal A,$ thus $\cox(\mathcal A) =
\rho_{E_1}\cdots\rho_{E_r}$.
Since $\cox(\mathcal A) \le_a \cox(\mathcal A'),$ we find reflections $t_{r+1},\dots,t_s$
such that $\cox(\mathcal A') = \rho_{E_1}\cdots\rho_{E_r}t_{r+1}\cdots t_s$, where $s$ is the
rank of $\mathcal A'$.
According to Proposition \ref{bijection}, 
it is possible to write $t_i = \rho_{E_i}$ for $r+1\le i \le s$, 
where the $E_i$ are exceptional modules.
Then $(E_1,\dots,E_s)$
is a sequence of $s$ exceptional modules with product $\cox(\mathcal A')$, and $\mathcal A'$ has
rank $s$, thus all
the modules $E_1,\dots,E_s$ belong to $\mathcal A'.$ In particular, $\mathcal A \subseteq \mathcal A'.$	
	\bigskip\bigskip 

{\bf Notes.}
	\medskip 

\begin{note}\label{exceptional}
{\bf Exceptionality.}
We follow here a long and well-established tradition to call an indecomposable module without
self-extensions an exceptional module, as started by Crawley-Boevey \cite{[CB]} in analogy to
the exceptional vector bundles discussed by Rudakov \cite{[Ru1]} and his collaborateurs. But
the reader should be aware that this terminology seems to be misleading in our context.
What is called exceptional has to be considered, in a non-commutative setting, 
not as an extraordinary feature, but as a quite typical behaviour. Of course,
dealing with commutative rings, only few modules will be exceptional.
Thus, from a commutative perspective, the existence
of large numbers of exceptional vector bundles was rated as an exceptional behaviour
(after all, algebraic geometry is often considered as part of commutative algebra).

\end{note}

\begin{note}\label{rank-2} {\bf Artin algebras of rank 2.}
{\it If $\rank\Lambda = 2$, 
and $X_1\neq X_2$ are non-isomorphic exceptional $\Lambda$-modules, 
then the abelian closure of the set $\{X_1,X_2\}$} 
(the closure under kernels,
cokernels) is $\mo\Lambda$. 

\begin{proof}
The assertion is clear if $X_1,X_2$ are orthogonal
(namely, in this case $X_1,X_2$ are just the simple $\Lambda$-modules. Thus we can assume that
$\Hom(X_1,X_2) \neq 0.$ Up to duality, we can assume that $X_1$, is
preprojective and we use induction on the number of predecessors of $X_1.$ 
If $X_1$ is simple, then we obtain the second simple module as
a direct summand of the cokernel of a map $X_1^t \to X_2.$ 
If $X_1$ is not simple, then $X_1$
generates $X_2$, say there is a surjective map $f\!:X_1^t \to X_2.$ If $X_1'$ is an indecomposable
direct summand of $X_1$, then $X_1'$ has less predecessors than $X_1.$
\end{proof} 
\end{note}

\begin{note}\label{thick-perp-perp} {\bf Thick subcategories $\mathcal X$ with proper inclusion 
$\mathcal X \subset {}^\perp(\mathcal X^\perp)$.}
Consider, for example, the Kronecker algebra $\Lambda$ with sink $0$ and source $1$.
Let $\mathcal N$ be the 
subcategory of regular modules (these are the direct sums of indecomposable modules $M$
with $\dim M_0 = \dim M_1$). This is a thick subcategory and $\mathcal N^\perp = 0$, thus 
${}^\perp(\mathcal N^\perp) = \mo\Lambda$. 
\end{note}

\begin{note}\label{braid-group-relations} 
{\bf The braid group relations.} The relations $\sigma_i\sigma_j = \sigma_j\sigma_i$
for $|i-j| \ge 2$ are obviously satisfied. In order to show that 
$\sigma_i\sigma_{i+1}\sigma_i = \sigma_{i+1}\sigma_{i}\sigma_{i +1}$, we can assume that
$i=1$ and $n = 3$. Thus, let $E = (E_1,E_2,E_3)$ be a complete exceptional
sequence and let $\overline{\tau_*}$ be the extended Auslander-Reiten translation for the
thick subcategory generated by $E_1$ and $E_3$. One sees immediately that both
$\sigma_1\sigma_2\sigma_1E$ and $\sigma_2\sigma_1\sigma_2E$ are of the form
$(?,\overline{\tau_*}E_3,E_1)$. But there is just one complete exceptional sequence of this
form. 

\end{note}\begin{note}\label{Coxeter-length} {\bf The length function on a Coxeter group.}
Given a 
Coxeter group $(W,S)$ and an element $w\in W$, one denotes by
$|w|$ the least number $t$ such that $w$ can be written as a product
of $t$ elements of $S$. If $v,w\in W$, then one sets $v\le w$ provided $|v|+|v^{-1}w| = |w|.$
For the properties of this ordering, we refer to \cite{[B]}, and \cite{[H1], [H2]}.
\end{note}

\vfill\eject 
\section{\bf The hereditary artin algebra $\Lambda_n$}

\subsection{The lattice $\NC(n)$ of non-crossing partitions of an $n$-element set}
	\medskip
 
\subsubsection{}
A {\it partition} of a set $S$ 
is a set of pairwise disjoint non-empty subsets $S_i$ of $S$ such that $S = \bigcup_i S_i$;
the subsets $S_i$ are called the {\it parts.} We will denote the set of all partitions of
$S$ by $\Pi(S)$, and we write $\Pi(n)$ instead of $\Pi(\{1,2,\dots,n\}).$ 

In order to visualize such a partition,
one sometimes uses arcs, as in the following example: 
 \ \beginpicture
  \setcoordinatesystem units <0.2cm,0.2cm>
  \multiput{} at 0 2  /
  \multiput{$\sssize \bullet$} at  0 0.3  1 0.3  2 0.3  3 0.3 /
  \circulararc 180 degrees from 1 1 center at 0.5 1 
  \circulararc 180 degrees from 3 1 center at 2 1 
  \plot 0 0.3  0 1 /
  \plot 1 0.3  1 1 /
  \plot 3 0.3  3 1 /  
 \endpicture \ \
depicts the partition $\{\{1,2,4\},\{3\}\}$ of the set $\{1,2,3,4\};$
here, the bullets are the integers $1,2,3,4$ 
(in this order), the arcs indicate that the
corresponding integers belong to the same part, the parts are just the transitive closure.
To be more precise, let us assume that $S = \{1,2,\dots,n\}$.
An {\it arc} of the partition $P$ of $S = \{1,2,\dots,n\}$
is a pair $[a,b]$ of integers $1\le a < b \le n+1$ such that $a,b$ belong to the same part of $P$
whereas no element $c$ with $a < c < b$ belongs to this part. We denote by $A(P)$ the set of arcs
of $P$. 
We call a set of arcs a {\it partition arc set} provided for any
pair of arcs $[a,b], [a',b']$ in the set, we have $a = a'$ iff $b = b'.$ 

\begin{lemma}\label{part-arc} Let $S = \{1,2,\dots,n\}.$ 
The map $A$ provides a bijection between the partitions of $S$
and the partition arc sets in $S$.
\end{lemma}

\begin{proof} See the note 
N\,\ref{partitions-arcs}.  
\end{proof} 
	 
A partition is said to be {\it non-crossing}
provided given elements $i < i' < j < j'$ with $i,j$ in the same part and $i',j'$ in the same
part, then all elements belong to the same part (this means that arcs do not intersect properly.)
For $n = 4$, all partitions but one are non-crossing, the exception is 
the partition $\{\{1,3\},\{2,4\}\}$ with picture
\ \beginpicture
  \setcoordinatesystem units <0.2cm,0.2cm>
  \multiput{} at 0 2  /
  \multiput{$\sssize \bullet$} at  0 0.3  1 0.3  2 0.3  3 0.3 /
  \circulararc 180 degrees from 2 1 center at 1 1 
  \circulararc 180 degrees from 3 1 center at 2 1 
  \plot 0 0.3  0 1 /
  \plot 1 0.3  1 1 /
  \plot 2 0.3  2 1 /
  \plot 3 0.3  3 1 /  
 \endpicture \ . \quad	
There is the following (quite obvious) 
characterization of the set of arcs of a non-crossing partition.
	\medskip 

\begin{lemma}\label{arcs}
 A set $A$ of pairs $[a,b]$ with $1 \le a < b \le n$ is the set of arcs
of a non-crossing partition $P\in \NC(n)$ if and only if the following condition is satisfied:
If $[a,b], [a',b'] \in A$ and $a\le a' < b\le b'$, then $a=a'$ and $b=b'.$
\end{lemma}
	\bigskip

We denote by $\NC(n)$ the poset of non-crossing partitions of the set $\{1,2,\dots,n\}$;
the order which is used is the (opposite of the) {\it refinement} order: 
to write parts as disjoint union of subsets. If $P, P'$ are partitions of the same set, we write $P \le P'$ provided any part of $P$ is contained in a part of $P'$, or,
equivalently, provided any part of $P'$ is the union of parts of $P$.
The unique minimal element of $\NC(n)$ 
is the partition with $n$ parts, the unique maximal element is the partition with just one part. 
The {\it layers} are formed by the partitions with a fixed number of parts.
The poset $\NC(n)$ has $n$ different layers $\NC_t(n),$ with $0\le t \le n\!-\!1$, where
$\NC_t(n)$ denotes the set of non-crossing partitions with $n-t$ parts.
	\medskip 

All the posets $\NC(n)$ are lattices (this is not quite obvious, but it will be shown later), 
they are self-dual, and have further
interesting properties. For example, there are only two possibilities
for the intervals of height 2, namely  
$$
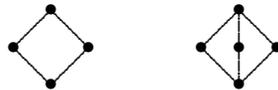
\hbox{\beginpicture
  \setcoordinatesystem units <.5cm,.5cm>
\put{\beginpicture 
\multiput{$\bullet$} at 0 0  -1 1  1 1  0 2 /
\plot 0 0  -1 1  0 2  1 1  0 0  /
\endpicture} at 0 0
\put{\beginpicture 
\multiput{$\bullet$} at 0 0  -1 1  0 1  1 1  0 2 /
\plot 0 0  -1 1  0 2  1 1  0 0  0 1  0 2 /
\endpicture} at 5 0
\endpicture}
$$ 
	\medskip 

Here is the 
lattice $\NC(4)$ of non-crossing partitions of a set with 4 elements.

$$
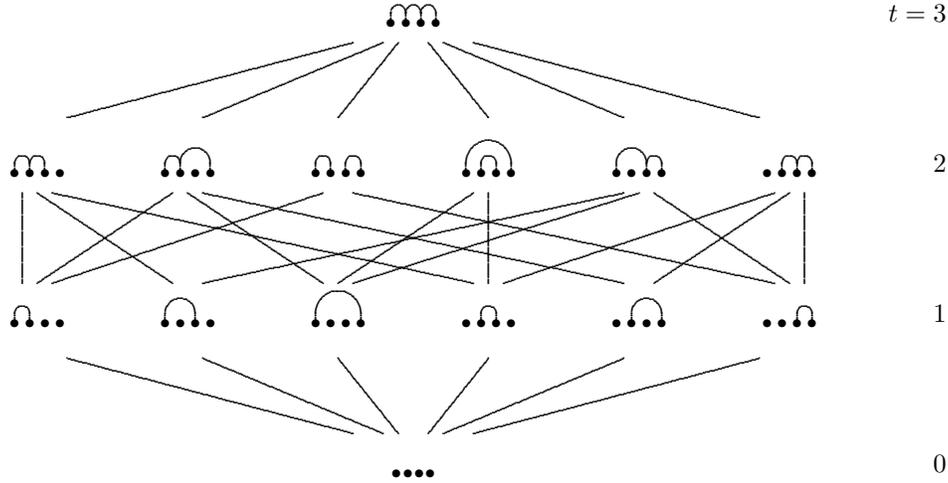
\hbox{\beginpicture
  \setcoordinatesystem units <2cm,2cm>
\put{$t=3$} at 3.35 3
\put{$2$} at 3.5 2
\put{$1$} at 3.5 1
\put{$0$} at 3.5 0
\put{$\beginpicture
  \setcoordinatesystem units <0.15cm,0.2cm>
  \multiput{} at 0 2  /
  \multiput{$\sssize \bullet$} at  0 0.3  1 0.3  2 0.3  3 0.3 /
 \endpicture$} at 0 0

\put{$\beginpicture
  \setcoordinatesystem units <0.2cm,0.2cm>
  \multiput{} at 0 2  /
  \multiput{$\sssize \bullet$} at  0 0.3  1 0.3  2 0.3  3 0.3 /
  \circulararc 180 degrees from 1 1 center at 0.5 1 
  \plot 0 0.3  0 1 /
  \plot 1 0.3  1 1 /
 \endpicture$} at -2.5 1
\put{$\beginpicture  
  \setcoordinatesystem units <0.2cm,0.2cm>
  \multiput{} at 0 2  /
  \multiput{$\sssize \bullet$} at  0 0.3  1 0.3  2 0.3  3 0.3 /
  \circulararc 180 degrees from 2 1 center at 1 1 
  \plot 0 0.3  0 1 /
  \plot 2 0.3  2 1 /
 \endpicture$} at -1.5 1
\put{$\beginpicture
  \setcoordinatesystem units <0.2cm,0.2cm>
  \multiput{} at 0 2  /
  \multiput{$\sssize \bullet$} at  0 0.3  1 0.3  2 0.3  3 0.3 /
  \circulararc 180 degrees from 3 1 center at 1.5 1 
  \plot 0 0.3  0 1 /
  \plot 3 0.3  3 1 /  
 \endpicture$} at -.5 1
\put{$\beginpicture
  \setcoordinatesystem units <0.2cm,0.2cm>
  \multiput{} at 0 2  /
  \multiput{$\sssize \bullet$} at  0 0.3  1 0.3  2 0.3  3 0.3 /
  \circulararc 180 degrees from 2 1 center at 1.5 1 
  \plot 1 0.3  1 1 /
  \plot 2 0.3  2 1 /
 \endpicture$} at 0.5 1
\put{$\beginpicture
  \setcoordinatesystem units <0.2cm,0.2cm>
  \multiput{} at 0 2  /
  \multiput{$\sssize \bullet$} at  0 0.3  1 0.3  2 0.3  3 0.3 /
  \circulararc 180 degrees from 3 1 center at 2 1 
  \plot 1 0.3  1 1 /
  \plot 3 0.3  3 1 /  
 \endpicture$} at 1.5 1
\put{$\beginpicture
  \setcoordinatesystem units <0.2cm,0.2cm>
  \multiput{} at 0 2  /
  \multiput{$\sssize \bullet$} at  0 0.3  1 0.3  2 0.3  3 0.3 /
  \circulararc 180 degrees from 3 1 center at 2.5 1 
  \plot 2 0.3  2 1 /
  \plot 3 0.3  3 1 /  
 \endpicture$} at 2.5  1
\put{$\beginpicture
  \setcoordinatesystem units <0.2cm,0.2cm>
  \multiput{} at 0 2  /
  \multiput{$\sssize \bullet$} at  0 0.3  1 0.3  2 0.3  3 0.3 /
  \circulararc 180 degrees from 1 1 center at 0.5 1 
  \circulararc 180 degrees from 2 1 center at 1.5 1 
  \plot 0 0.3  0 1 /
  \plot 1 0.3  1 1 /
  \plot 2 0.3  2 1 /
 \endpicture$} at -2.5 2
\put{$\beginpicture
  \setcoordinatesystem units <0.2cm,0.2cm>
  \multiput{} at 0 2  /
  \multiput{$\sssize \bullet$} at  0 0.3  1 0.3  2 0.3  3 0.3 /
  \circulararc 180 degrees from 1 1 center at 0.5 1 
  \circulararc 180 degrees from 3 1 center at 2 1 
  \plot 0 0.3  0 1 /
  \plot 1 0.3  1 1 /
  \plot 3 0.3  3 1 /  
 \endpicture$} at -1.5 2
\put{$\beginpicture
  \setcoordinatesystem units <0.2cm,0.2cm>
  \multiput{} at 0 2  /
  \multiput{$\sssize \bullet$} at  0 0.3  1 0.3  2 0.3  3 0.3 /
  \circulararc 180 degrees from 1 1 center at 0.5 1 
  \circulararc 180 degrees from 3 1 center at 2.5 1 
   \plot 0 0.3  0 1 /
  \plot 1 0.3  1 1 /
  \plot 2 0.3  2 1 /
  \plot 3 0.3  3 1 /  
 \endpicture$} at -.5 2
\put{$\beginpicture
  \setcoordinatesystem units <0.2cm,0.2cm>
  \multiput{} at 0 2  /
  \multiput{$\sssize \bullet$} at  0 0.3  1 0.3  2 0.3  3 0.3 /
  \circulararc 180 degrees from 2 1 center at 1.5 1 
  \circulararc 180 degrees from 3 1 center at 1.5 1 
  \plot 0 0.3  0 1 /
  \plot 1 0.3  1 1 /
  \plot 2 0.3  2 1 /
  \plot 3 0.3  3 1 /  
 \endpicture$} at 0.5 2
\put{$\beginpicture
  \setcoordinatesystem units <0.2cm,0.2cm>
  \multiput{} at 0 2  /
  \multiput{$\sssize \bullet$} at  0 0.3  1 0.3  2 0.3  3 0.3 /
  \circulararc 180 degrees from 3 1 center at 2.5 1 
  \circulararc 180 degrees from 2 1 center at 1 1 
  \plot 0 0.3  0 1 /
  \plot 2 0.3  2 1 /
  \plot 3 0.3  3 1 /  
 \endpicture$} at 1.5 2
\put{$\beginpicture
  \setcoordinatesystem units <0.2cm,0.2cm>
  \multiput{} at 0 2  /
  \multiput{$\sssize \bullet$} at  0 0.3  1 0.3  2 0.3  3 0.3 /
  \circulararc 180 degrees from 2 1 center at 1.5 1 
  \circulararc 180 degrees from 3 1 center at 2.5 1 
  \plot 1 0.3  1 1 /
  \plot 2 0.3  2 1 /
  \plot 3 0.3  3 1 /  
 \endpicture$} at 2.5 2

\put{$\beginpicture
  \setcoordinatesystem units <0.2cm,0.2cm>
  \multiput{} at 0 2  /
  \multiput{$\sssize \bullet$} at  0 0.3  1 0.3  2 0.3  3 0.3 /
  \circulararc 180 degrees from 1 1 center at 0.5 1 
  \circulararc 180 degrees from 2 1 center at 1.5 1 
  \circulararc 180 degrees from 3 1 center at 2.5 1 
  \plot 0 0.3  0 1 /
  \plot 1 0.3  1 1 /
  \plot 2 0.3  2 1 /
  \plot 3 0.3  3 1 /  
 \endpicture$} at  0 3

\plot -0.4 0.2 -2.3 0.7 /
\plot -0.2 0.2 -1.4 0.7 /
\plot -0.1 0.2 -0.5 0.7 /
\plot .1 0.2 0.5 0.7 /
\plot .2 0.2 1.4 0.7 /
\plot .4 0.2 2.3 0.7 /

\plot -0.4 2.8 -2.3 2.3 /
\plot -0.2 2.8 -1.4 2.3 /
\plot -0.1 2.8 -0.5 2.3 /
\plot .1 2.8 0.5 2.3 /
\plot .2 2.8 1.4 2.3 /
\plot .4 2.8 2.3 2.3 /

\plot -2.6 1.2  -2.6 1.8 /
\plot -2.5 1.2  -1.6 1.8 /
\plot -2.4 1.2  -0.6 1.8 /

\plot -1.6 1.2  -2.5 1.8 /
\plot -1.4 1.2   1.4 1.8 /

\plot -.6 1.2  -1.5 1.8 /
\plot -.5 1.2   0.4 1.8 /
\plot -.4 1.2   1.5 1.8 /

\plot .4 1.2  -2.4 1.8 /
\plot .5 1.2   .5 1.8 /
\plot .6 1.2   2.4 1.8 /

\plot 1.4 1.2  -1.4 1.8 /
\plot 1.6 1.2   2.5 1.8 /

\plot 2.4 1.2  -.4 1.8 /
\plot 2.5 1.2   1.6 1.8 /
\plot 2.6 1.2   2.6 1.8 /

\endpicture}
$$
	\medskip

The lattices $\NC(n)$ are typical examples of the lattices studied in Chapter 3:
as we will see, they can be identified with the lattices $\A(\mo\Lambda)$, 
where $\Lambda$ is hereditary of type $\Bbb A$ (see Theorem \ref{classical}).
But let us stress already here the index shift: the lattice $\NC(n)$ is 
considered to be of type $\mathbb A_{n-1}$).
	\medskip

Another way to visualize partitions of a set $S$ is to arrange the vertices of $S$ consecutively on
a circle and to consider for every part $\{a_1<a_2<\dots < a_m\}$ the convex polygon with vertices
$a_1,a_2,\dots,a_m$ (this just means to draw for $m\ge 3$ another arc, namely joining $a_m$ and $a_1$). Here are two examples:
$$
\hbox{\beginpicture
  \setcoordinatesystem units <1cm,1cm>
\put{\beginpicture
  \setcoordinatesystem units <.4cm,.4cm>
\multiput{} at 0 0  7 3.5 /
\multiput{$\bullet$} at 0 0  1 0  2 0  3 0  4 0  5 0  6 0  7 0 /
\put{$\ssize 1$} at 0 -.6 
\put{$\ssize 2$} at 1 -.6 
\put{$\ssize 3$} at 2 -.6 
\put{$\ssize 4$} at 3 -.6 
\put{$\ssize 5$} at 4 -.6 
\put{$\ssize 6$} at 5 -.6 
\put{$\ssize 7$} at 6 -.6 
\put{$\ssize 8$} at 7 -.6 
\plot 0 0.3  0 1 /
\plot 1 0.3  1 1 /
\plot 2 0.3  2 1 /
\plot 4 0.3  4 1 /
\plot 5 0.3  5 1 /
\plot 6 0.3  6 1 /
\plot 7 0.3  7 1 /
\circulararc 180 degrees from 2 1 center at 1.5 1
\circulararc 180 degrees from 4 1 center at 3 1
\circulararc 180 degrees from 5 1 center at 2.5 1
\circulararc 180 degrees from 6 1 center at 5.5 1
\circulararc 180 degrees from 7 1 center at 5.5 1
\endpicture} at 0 0
\put{\beginpicture
  \setcoordinatesystem units <1.2cm,1.2cm>
\setdots <1mm>
\multiput{} at -1 -1  1 1 /
\circulararc 360 degrees from 1 0 center at 0 0
\setsolid
\multiput{$\bullet$} at 1 0  0 1  -1 0  0 -1  0.71 .71  -.71 .71  0.71 -.71  -.71 -.71 /
\put{$\ssize 1$} at 1.2 0 
\put{$\ssize 2$} at 0.85 0.85
\put{$\ssize 3$} at 0 1.2  
\put{$\ssize 4$} at -0.85 0.85
\put{$\ssize 5$} at -1.2 0 
\put{$\ssize 6$} at -0.85 -.85
\put{$\ssize 7$} at 0 -1.2  
\put{$\ssize 8$} at 0.85 -.85
\plot 1 0  -.71  -.71  0 -1  1 0 / 
\plot 0.71 0.71  0 1  -1 0  0.71 -0.71  0.71 0.71 /
\endpicture} at 0 -3

\put{\beginpicture
  \setcoordinatesystem units <.4cm,.4cm>
\multiput{} at 0 0  7 3.5 /
\multiput{$\bullet$} at 0 0  1 0  2 0  3 0  4 0  5 0  6 0  7 0 /
\put{$\ssize 1$} at 0 -.6 
\put{$\ssize 2$} at 1 -.6 
\put{$\ssize 3$} at 2 -.6 
\put{$\ssize 4$} at 3 -.6 
\put{$\ssize 5$} at 4 -.6 
\put{$\ssize 6$} at 5 -.6 
\put{$\ssize 7$} at 6 -.6 
\put{$\ssize 8$} at 7 -.6 
\plot 0 0.3  0 1 /
\plot 1 0.3  1 1 /
\plot 2 0.3  2 1 /
\plot 4 0.3  4 1 /
\plot 5 0.3  5 1 /
\plot 6 0.3  6 1 /
\plot 7 0.3  7 1 /
\circulararc 180 degrees from 2 1 center at 1.5 1
\circulararc 180 degrees from 4 1 center at 3 1
\circulararc 180 degrees from 5 1 center at 2.5 1
\circulararc 180 degrees from 6 1 center at 5.5 1
\circulararc 180 degrees from 7 1 center at 6.5 1
\endpicture} at 5 0
\put{\beginpicture
  \setcoordinatesystem units <1.2cm,1.2cm>
\setdots <1mm>
\multiput{} at -1 -1  1 1 /
\circulararc 360 degrees from 1 0 center at 0 0
\setsolid
\multiput{$\bullet$} at 1 0  0 1  -1 0  0 -1  0.71 .71  -.71 .71  0.71 -.71  -.71 -.71 /
\put{$\ssize 1$} at 1.2 0 
\put{$\ssize 2$} at 0.85 0.85
\put{$\ssize 3$} at 0 1.2  
\put{$\ssize 4$} at -0.85 0.85
\put{$\ssize 5$} at -1.2 0 
\put{$\ssize 6$} at -0.85 -.85
\put{$\ssize 7$} at 0 -1.2  
\put{$\ssize 8$} at 0.85 -.85
\plot 1 0  -.71  -.71  0 -1  .71 -.71  1 0 / 
\plot 0.71 0.71  0 1  -1 0  0.71 0.71 /
\endpicture} at 5 -3

\put{$\{\{1,6,7\},\{2,3 ,5,8\},\{4\}\}$} at 0 1.5
\put{$\{\{1,6,7,8\},\{2,3 ,5\},\{4\}\}$} at 5 1.5

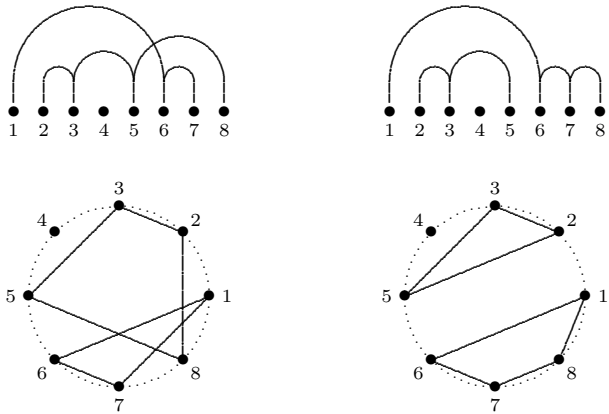
\endpicture}
$$
The partition on the left is crossing, that on the right is non-crossing.
	\bigskip
 
The lattice of non-crossing partitions was introduced and studied by {\bf Kreweras}
(1972). Actually, Becker was considering non-negative partitions
already in 1948, he showed that
the number of elements of $\NC(n)$ is the Catalan number $C_n = \frac1{n+1}\binom{2n}n$. 
For a survey concerning the set of non-crossing partitions we may refer to Simion (2000).

The lattice of non-crossing partitions plays a decisive role in {\bf free probability theory.}
The free probability theory was initiated by Voiculescu (1985), the main
ingredient is the use of non-commuting variables (in contrast to the commuting
ones in ordinary probability theory). Many calculations are rather similar to the
classical case, but surprisingly, some turn out to become easier. 
Namely, it is common in classical probability theory 
that formulas involve the summation over all partitions
of some finite set. Looking at the corresponding
sums in free probability theory it happens quite frequently that several of the summands 
are zero, namely those indexed by crossing partitions, thus, in this case
it is sufficient to sum over the non-crossing partitions. For example,
this happens for the so-called central limit theorem. Speicher (1990)
gave a proof of this theorem using systematically non-crossing partitions.
One should be aware that for $n$ large, 
the number of non-crossing partitions
of an $n$-element set is much smaller than the number of all partitions.
As we saw above, for $n = 4,$ there are 15 partitions and all but one are non-crossing.
But for $n = 20$, there are $51\,724\,158\,235\,372$ partitions, and only 
$6\,564\,120\,420$ are
non-crossing. Also, as soon as one deals just with non-crossing partitions,
one may rely on properties of the lattice of non-crossing partitions ---
and, as we have already mentioned, the lattice of non-crossing partitions has
some quite surprising properties. 
The book by Nica and Speicher (2006) is presently the main reference for the use of
non-crossing partitions in free probability theory. 
	\medskip

\subsubsection{}
Let us compare for a moment the lattice of all partitions with the lattice of
non-crossing partitions, for some fixed $n$. The set of non-crossing partitions is
of course a part of the lattice of all partitions (it is not a sublattice, 
but at least a meet sublattice). Here again the case $n=4$:
$$
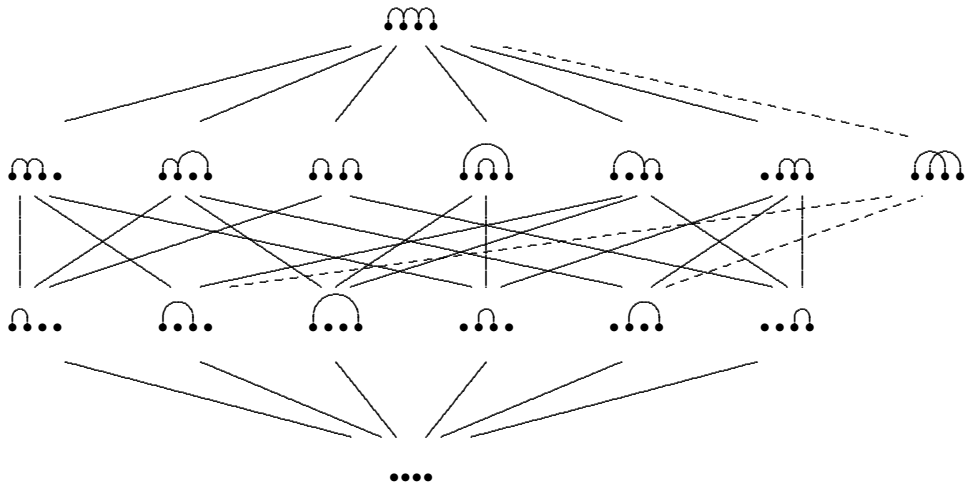
\hbox{\beginpicture
  \setcoordinatesystem units <2cm,2cm>

\put{$\beginpicture
  \setcoordinatesystem units <0.15cm,0.2cm>
  \multiput{} at 0 2  /
  \multiput{$\sssize \bullet$} at  0 0.3  1 0.3  2 0.3  3 0.3 /
 \endpicture$} at 0 0

\put{$\beginpicture
  \setcoordinatesystem units <0.2cm,0.2cm>
  \multiput{} at 0 2  /
  \multiput{$\sssize \bullet$} at  0 0.3  1 0.3  2 0.3  3 0.3 /
  \circulararc 180 degrees from 1 1 center at 0.5 1 
  \plot 0 0.3  0 1 /
  \plot 1 0.3  1 1 /
 \endpicture$} at -2.5 1
\put{$\beginpicture  
  \setcoordinatesystem units <0.2cm,0.2cm>
  \multiput{} at 0 2  /
  \multiput{$\sssize \bullet$} at  0 0.3  1 0.3  2 0.3  3 0.3 /
  \circulararc 180 degrees from 2 1 center at 1 1 
  \plot 0 0.3  0 1 /
  \plot 2 0.3  2 1 /
 \endpicture$} at -1.5 1
\put{$\beginpicture
  \setcoordinatesystem units <0.2cm,0.2cm>
  \multiput{} at 0 2  /
  \multiput{$\sssize \bullet$} at  0 0.3  1 0.3  2 0.3  3 0.3 /
  \circulararc 180 degrees from 3 1 center at 1.5 1 
  \plot 0 0.3  0 1 /
  \plot 3 0.3  3 1 /  
 \endpicture$} at -.5 1
\put{$\beginpicture
  \setcoordinatesystem units <0.2cm,0.2cm>
  \multiput{} at 0 2  /
  \multiput{$\sssize \bullet$} at  0 0.3  1 0.3  2 0.3  3 0.3 /
  \circulararc 180 degrees from 2 1 center at 1.5 1 
  \plot 1 0.3  1 1 /
  \plot 2 0.3  2 1 /
 \endpicture$} at 0.5 1
\put{$\beginpicture
  \setcoordinatesystem units <0.2cm,0.2cm>
  \multiput{} at 0 2  /
  \multiput{$\sssize \bullet$} at  0 0.3  1 0.3  2 0.3  3 0.3 /
  \circulararc 180 degrees from 3 1 center at 2 1 
  \plot 1 0.3  1 1 /
  \plot 3 0.3  3 1 /  
 \endpicture$} at 1.5 1
\put{$\beginpicture
  \setcoordinatesystem units <0.2cm,0.2cm>
  \multiput{} at 0 2  /
  \multiput{$\sssize \bullet$} at  0 0.3  1 0.3  2 0.3  3 0.3 /
  \circulararc 180 degrees from 3 1 center at 2.5 1 
  \plot 2 0.3  2 1 /
  \plot 3 0.3  3 1 /  
 \endpicture$} at 2.5  1
\put{$\beginpicture
  \setcoordinatesystem units <0.2cm,0.2cm>
  \multiput{} at 0 2  /
  \multiput{$\sssize \bullet$} at  0 0.3  1 0.3  2 0.3  3 0.3 /
  \circulararc 180 degrees from 1 1 center at 0.5 1 
  \circulararc 180 degrees from 2 1 center at 1.5 1 
  \plot 0 0.3  0 1 /
  \plot 1 0.3  1 1 /
  \plot 2 0.3  2 1 /
 \endpicture$} at -2.5 2
\put{$\beginpicture
  \setcoordinatesystem units <0.2cm,0.2cm>
  \multiput{} at 0 2  /
  \multiput{$\sssize \bullet$} at  0 0.3  1 0.3  2 0.3  3 0.3 /
  \circulararc 180 degrees from 1 1 center at 0.5 1 
  \circulararc 180 degrees from 3 1 center at 2 1 
  \plot 0 0.3  0 1 /
  \plot 1 0.3  1 1 /
  \plot 3 0.3  3 1 /  
 \endpicture$} at -1.5 2
\put{$\beginpicture
  \setcoordinatesystem units <0.2cm,0.2cm>
  \multiput{} at 0 2  /
  \multiput{$\sssize \bullet$} at  0 0.3  1 0.3  2 0.3  3 0.3 /
  \circulararc 180 degrees from 1 1 center at 0.5 1 
  \circulararc 180 degrees from 3 1 center at 2.5 1 
   \plot 0 0.3  0 1 /
  \plot 1 0.3  1 1 /
  \plot 2 0.3  2 1 /
  \plot 3 0.3  3 1 /  
 \endpicture$} at -.5 2
\put{$\beginpicture
  \setcoordinatesystem units <0.2cm,0.2cm>
  \multiput{} at 0 2  /
  \multiput{$\sssize \bullet$} at  0 0.3  1 0.3  2 0.3  3 0.3 /
  \circulararc 180 degrees from 2 1 center at 1.5 1 
  \circulararc 180 degrees from 3 1 center at 1.5 1 
  \plot 0 0.3  0 1 /
  \plot 1 0.3  1 1 /
  \plot 2 0.3  2 1 /
  \plot 3 0.3  3 1 /  
 \endpicture$} at 0.5 2
\put{$\beginpicture
  \setcoordinatesystem units <0.2cm,0.2cm>
  \multiput{} at 0 2  /
  \multiput{$\sssize \bullet$} at  0 0.3  1 0.3  2 0.3  3 0.3 /
  \circulararc 180 degrees from 3 1 center at 2.5 1 
  \circulararc 180 degrees from 2 1 center at 1 1 
  \plot 0 0.3  0 1 /
  \plot 2 0.3  2 1 /
  \plot 3 0.3  3 1 /  
 \endpicture$} at 1.5 2
\put{$\beginpicture
  \setcoordinatesystem units <0.2cm,0.2cm>
  \multiput{} at 0 2  /
  \multiput{$\sssize \bullet$} at  0 0.3  1 0.3  2 0.3  3 0.3 /
  \circulararc 180 degrees from 2 1 center at 1.5 1 
  \circulararc 180 degrees from 3 1 center at 2.5 1 
  \plot 1 0.3  1 1 /
  \plot 2 0.3  2 1 /
  \plot 3 0.3  3 1 /  
 \endpicture$} at 2.5 2

\put{$\beginpicture
  \setcoordinatesystem units <0.2cm,0.2cm>
  \multiput{} at 0 2  /
  \multiput{$\sssize \bullet$} at  0 0.3  1 0.3  2 0.3  3 0.3 /
  \circulararc 180 degrees from 1 1 center at 0.5 1 
  \circulararc 180 degrees from 2 1 center at 1.5 1 
  \circulararc 180 degrees from 3 1 center at 2.5 1 
  \plot 0 0.3  0 1 /
  \plot 1 0.3  1 1 /
  \plot 2 0.3  2 1 /
  \plot 3 0.3  3 1 /  
 \endpicture$} at  0 3

\put{$\beginpicture
  \setcoordinatesystem units <0.2cm,0.2cm>
  \multiput{} at 0 2  /
  \multiput{$\sssize \bullet$} at  0 0.3  1 0.3  2 0.3  3 0.3 /
  \circulararc 180 degrees from 2 1 center at 1 1 
  \circulararc 180 degrees from 3 1 center at 2 1 
  \plot 0 0.3  0 1 /
  \plot 1 0.3  1 1 /
  \plot 2 0.3  2 1 /
  \plot 3 0.3  3 1 /  
 \endpicture$} at 3.5 2

\setdashes <1mm>
\plot 3.3 2.2  0.6 2.8 /
\plot -1.2 1.2  3.2 1.8 /
\plot  1.7 1.2  3.4 1.8 /

\setsolid
\plot -0.4 0.2 -2.3 0.7 /
\plot -0.2 0.2 -1.4 0.7 /
\plot -0.1 0.2 -0.5 0.7 /
\plot .1 0.2 0.5 0.7 /
\plot .2 0.2 1.4 0.7 /
\plot .4 0.2 2.3 0.7 /

\plot -0.4 2.8 -2.3 2.3 /
\plot -0.2 2.8 -1.4 2.3 /
\plot -0.1 2.8 -0.5 2.3 /
\plot .1 2.8 0.5 2.3 /
\plot .2 2.8 1.4 2.3 /
\plot .4 2.8 2.3 2.3 /

\plot -2.6 1.2  -2.6 1.8 /
\plot -2.5 1.2  -1.6 1.8 /
\plot -2.4 1.2  -0.6 1.8 /

\plot -1.6 1.2  -2.5 1.8 /
\plot -1.4 1.2   1.4 1.8 /

\plot -.6 1.2  -1.5 1.8 /
\plot -.5 1.2   0.4 1.8 /
\plot -.4 1.2   1.5 1.8 /

\plot .4 1.2  -2.4 1.8 /
\plot .5 1.2   .5 1.8 /
\plot .6 1.2   2.4 1.8 /

\plot 1.4 1.2  -1.4 1.8 /
\plot 1.6 1.2   2.5 1.8 /

\plot 2.4 1.2  -.4 1.8 /
\plot 2.5 1.2   1.6 1.8 /
\plot 2.6 1.2   2.6 1.8 /

\endpicture}
$$

The number of partitions of an $n$-element set into exactly $k$ non-empty parts is the 
{\it Stirling number of the second kind} $\left\{n \atop t\right\}$
(we use the so-called Karamata notation, as advertised by Knuth \cite{[Kn]}, 
a competing notation would be
$S(n, k)$). The sequence of the Stirling numbers of the second kind form the triangle  A048993
in Sloane's OEIS \cite{[Sl]}.

The partitions with $n-1$ parts are just given by a two-element subset (the remaining parts are
singletons), the number is $\binom n2 = |\Phi_+(\mathbb A_{n-1})|$.
The partitions with $n-1$ parts are of course non-crossing.  

If we look at partitions with $2$ parts, the number is $2^{n-1}-1$. (There are $2^n$ ordered pairs of complementary subsets $S$ and $S'$. In one case, $S$ is empty, and in another $S'$ is empty, so $2^n - 2$ 
ordered pairs of non-empty 
subsets remain. Finally, since we want unordered pairs rather than ordered pairs we have to divide this last number by $2$).
On the other hand, the number of non-crossing partitions with $2$ parts is again 
$\binom n2$: as we have mentioned, the lattice of non-crossing partitions is selfdual, a very remarkable property.
		\medskip

Here is a sketch for $n = 7$, the shaded part indicates the set of non-crossing
partitions inside the set of all partitions. On the right, we have written the 
corresponding numbers
of partitions with $7-t$ parts, first the number $|\NC_t(7)|$ 
of non-crossing partitions with $7-t$ parts, 
then the number $\left\{7\atop 7-t\right\}$ of all partitions with $7-t$ parts; 
always $0 \le t \le 6$.
$$
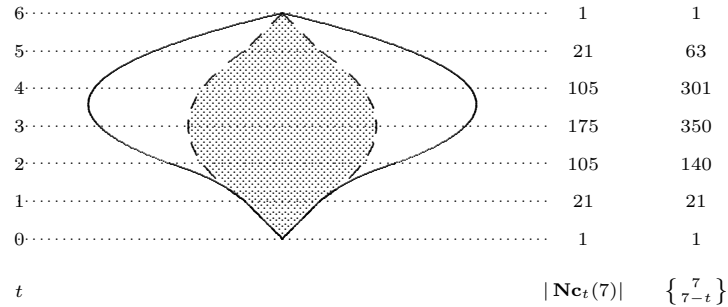
\hbox{\beginpicture
  \setcoordinatesystem units <.5cm,.5cm>
\put{\beginpicture
\multiput{} at 0 0  0 6 /
\plot 0 0  1 1 /
\plot 0 0  -1 1 /
\setquadratic
\plot 1 1  1.85 1.55  3 2 /
\plot 3 2  5 4   0 6 /
\plot -1 1  -1.85 1.55  -3 2 /
\plot -3 2  -5 4   0 6 /
\setdashes <2mm>
\setlinear
\plot 0 0  1 1 /
\plot 0 0  -1 1 /
\plot 0 6  1 5 /
\plot 0 6  -1 5 /
\setquadratic
\plot 1 1   2.5 3    1 5 /
\plot -1 1  -2.5 3  -1 5 /
\setlinear
\setdots <1mm>
\plot -7 0  7 0 /
\plot -7 1  7 1 /
\plot -7 2  7 2 /
\plot -7 3  7 3 /
\plot -7 4  7 4 /
\plot -7 5  7 5 /
\plot -7 6  7 6 /
\setshadegrid span <.4mm>
\hshade 0 0 0  <,,z,z> 1 -1 1  <,,z,z> 2 -2.3 2.3  <,,z,z> 3 -2.6 2.6 
    4 -2.3 2.3  <,,z,z>  5 -1 1  <,,z,z>  6 0 0 /
\put{$\ssize 1 $} at 8 6
\put{$\ssize 21$} at 8 5
\put{$\ssize 105$} at 8 4
\put{$\ssize 175$} at 8 3
\put{$\ssize 105$} at 8 2
\put{$\ssize 21$} at 8 1
\put{$\ssize 1$} at 8 0
\put{$\ssize |\NC_{t}(7)|$} at 8 -1.4

\put{$\ssize 1 $} at 11 6
\put{$\ssize 63$} at 11 5
\put{$\ssize 301$} at 11 4
\put{$\ssize 350$} at 11 3
\put{$\ssize 140$} at 11 2
\put{$\ssize 21$} at 11 1
\put{$\ssize 1$} at 11 0
\put{$\ssize \left\{7\atop 7-t\right\}$} at 11 -1.4

\put{$\ssize t $} at -7 -1.4
\put{$\ssize 6 $} at -7 6
\put{$\ssize 5$} at -7 5
\put{$\ssize 4$} at -7 4
\put{$\ssize 3$} at -7 3
\put{$\ssize 2$} at -7 2
\put{$\ssize 1$} at -7 1
\put{$\ssize 0$} at -7 0

\endpicture} at 0 0
\endpicture}
$$
	\medskip
The partition lattices have the following layers 
$$
 \smallmatrix 1 \endsmallmatrix\qquad 
 \smallmatrix 1 \cr 1\endsmallmatrix\qquad 
 \smallmatrix 1 \cr 3 \cr 1 \endsmallmatrix\qquad 
 \smallmatrix 1 \cr 7 \cr 6 \cr 1\endsmallmatrix\qquad 
 \smallmatrix 1 \cr 15 \cr 25 \cr 10\cr 1\endsmallmatrix\qquad 
 \smallmatrix 1 \cr 31 \cr 90 \cr 65 \cr 15\cr 1 \endsmallmatrix\qquad 
 \smallmatrix 1 \cr 63 \cr 301 \cr 350\cr 140 \cr 21\cr 1\endsmallmatrix\qquad 
 \smallmatrix 1 \cr 127 \cr 966 \cr 1701\cr 1050\cr 266\cr 28 \cr 1\endsmallmatrix\quad \dots 
$$
whereas the lattices of non-crossing partitions have the layers 
$$
 \smallmatrix 1 \endsmallmatrix\qquad 
 \smallmatrix 1 \cr 1\endsmallmatrix\qquad 
 \smallmatrix 1 \cr 3 \cr 1 \endsmallmatrix\qquad 
 \smallmatrix 1 \cr 6 \cr 6 \cr 1\endsmallmatrix\qquad 
 \smallmatrix 1 \cr 10 \cr 20 \cr 10\cr 1\endsmallmatrix\qquad 
 \smallmatrix 1 \cr 15 \cr 50 \cr 50 \cr 15\cr 1 \endsmallmatrix\qquad 
 \smallmatrix 1 \cr 21 \cr 105 \cr 175 \cr  105 \cr 21\cr 1\endsmallmatrix\qquad 
 \smallmatrix 1 \cr 28 \cr 196 \cr 490 \cr 490 \cr 196 \cr 28 \cr 1\endsmallmatrix\quad \dots 
$$

The number of all partitions of a set of cardinality $n$ is
called the {\it Bell numbers} $B_n$ (thus $B_n = |\Pi(n)|$), whereas
the number of non-crossing partition in $\Pi(n)$ is the 
Catalan number $C_n$.
Here is a comparison of the Bell and Catalan numbers for $n\le 10$: 
$$
\hbox{\beginpicture
  \setcoordinatesystem units <.9cm,.6cm>
\multiput{} at 0 1.2  0 -.2 /
\put{$n$} at  -.2 1
\put{$B_n$} at -.2 -1 
\put{$C_n$} at -.2 0
\put{$0$} at .9 1
\put{$1$} at 1.9 1
\put{$2$} at 2.9 1
\put{$3$} at 3.9 1
\put{$4$} at 4.8 1
\put{$5$} at 5.8 1
\put{$6$} at 6.7 1
\put{$7$} at 7.8 1
\put{$8$} at 9.1 1
\put{$9$} at 10.5 1
\put{$10$} at 12 1

\put{$1$} [r] at  1 -1
\put{$1$} [r] at  2 -1
\put{$2$} [r] at  3 -1
\put{$5$} [r] at  4 -1
\put{$15$} [r] at  5 -1
\put{$52$} [r] at  6 -1
\put{$203$} [r] at  7 -1
\put{$877$} [r] at 8.1 -1
\put{$4140$} [r] at  9.5 -1
\put{$21147$} [r] at  11 -1
\put{$115975$} [r] at  12.5 -1

\put{$1$} [r] at  1 0
\put{$1$} [r] at  2 0
\put{$2$} [r] at  3 0
\put{$5$} [r] at  4 0
\put{$14$} [r] at  5 0
\put{$42$} [r] at  6 0
\put{$132$} [r] at  7 0
\put{$429$} [r] at 8.1 0
\put{$1430$} [r] at  9.5 0
\put{$4862$} [r] at  11 0
\put{$16796$} [r] at  12.5 0
\plot 0.55 1.3  0.55 -1.3 /
\plot -1 0.5  12.8 0.5 /
\endpicture}
$$
It is well-known (and easy to see) that $\lim_{n\to \infty} C_n/B_n = 0$.
	\medskip 

\subsection{The categorification of $\NC(n)$} 
Let us return to  the path algebra $\Lambda_n$
of the linearly oriented quiver of type $\mathbb A_n$:
$$
 1 \leftarrow 2 \leftarrow \cdots \leftarrow n
$$
and recall that for any hereditary artin algebra $\Lambda$, we denote by 
$\A(\mo\Lambda)$ the set of (isomorphism classes of)
exceptional antichains in $\mo\Lambda$.
As we know, the exceptional antichains correspond bijectively to the exceptional subcategories,
let us denote here
by $\underline{\mathcal A}(\mo\Lambda)$ the poset of exceptional subcategories
of $\mo\Lambda$ (with ordering being given by the set-theoretical inclusion). 
Usually, we will identify the two sets $\A(\mo\Lambda)$ and $\underline{\mathcal A}(\mo\Lambda)$, 
but for the purpose of the following dictionary, 
we should make a distinction. We denote by $\mathcal E$ and $S$
the canonical poset isomorphisms
\smallskip
 
\Rahmen{$\beginpicture
  \setcoordinatesystem units <2cm,1cm>
\put{$\A(\mo\Lambda_n)$} at 0 0
\put{$\underline{\mathcal A}(\mo\Lambda_n)$} at 2 0
\arr{0.6 0.1}{1.4 0.1}
\arr{1.4 -.1}{0.6 -.1}
\put{$\mathcal E$} at 1 0.3
\put{$S $} at 1 -.3
\endpicture$}
\smallskip

\noindent
Here, $\mathcal E$ maps the antichain $A$ to the full subcategory $\mathcal E(A)$ 
of all modules with a filtration
with factors in $A$, whereas the inverse bijection $S$ 
sends the exceptional subcategory $\mathcal A$ to the 
set $S (\mathcal A)$ of simple objects in $\mathcal A$.
	\bigskip 

\subsubsection{\bf The indecomposable $\Lambda_n$-modules.}

We deviate from the usual notation 
for the indecomposable representation of $\Lambda_n$. We will 
denote the indecomposable projective module of length $j$ by $[1,j\!+\!1]$ and the factor module
$[1,j\!+\!1]/[1,i]$ with $1\le i \le j$ by $[i,j+1]$.
In particular, the simple $\Lambda_n$-modules are denoted by
$S(i) = [i,i+1]$, and the indecomposable module with socle $S[i]$ and
length $t$ by $[i,i+t]$.
(The usual labeling of the indecomposable
module of length $t$ uses its composition factors, thus a sequence of $t$ consecutive numbers.)
Let us repeat: In this chapter, the indecomposable $\Lambda_n$-modules 
are given by the intervals $[a,b]$ with $1\le a < b \le n+1$ and 
the support of the module $[a,b]$ are the simple modules $S(x) = [x,x+1]$
with $a\le x \le b-1$, in particular, the length of $[a,b]$ is $b-a$.

Using this notation, the Auslander-Reiten quiver of $\Lambda_3$
looks as follows:
$$
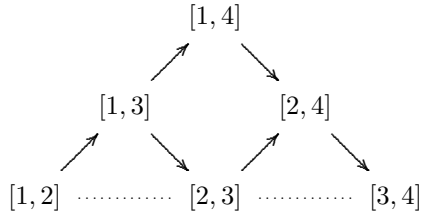
\hbox{\beginpicture
  \setcoordinatesystem units <1.2cm,1.2cm>
\put{$[1,2]$} at 0 0 
\put{$[2,3]$} at 2 0 
\put{$[3,4]$} at 4 0 
\put{$[1,3]$} at 1 1 
\put{$[2,4]$} at 3 1 
\put{$[1,4]$} at 2 2 
\arr{0.3 0.3}{0.7 0.7}
\arr{1.3 1.3}{1.7 1.7}
\arr{2.3 0.3}{2.7 0.7}
\arr{1.3 0.7}{1.7 0.3}
\arr{2.3 1.7}{2.7 1.3}
\arr{3.3 0.7}{3.7 0.3}
\setdots <1mm>
\plot  .5 0  1.5 0 /
\plot 2.5 0  3.5 0 /
 \endpicture}
$$
As we have mentioned, the simple $\Lambda_n$-modules are the modules of the form $S(i) = [i,i+1]$
with $1\le i \le n$. Note that $\Ext^1(S(j),S(i))\neq 0$ if and only if $j = i+1$, thus
$\Ext^1([a,a+1],[c,c+1])\neq 0$ if and only if $c = a+1$.
Also, we should mention that all the indecomposable $\Lambda_n$-modules
are bricks (after all, $\Lambda_n$ is representation directed). 

The essential (and trivial) 
observation which will be used in this chapter
is already suggested by our notation: {\it there is a bijection between
the arcs in $\{1,2,\dots,n+1\}$ and the indecomposable $\Lambda_n$-modules,} both the arcs
and the indecomposable modules are denoted by the intervals $[a,b]$ with $1\le a < b \le n+1.$
	\medskip 

\subsubsection{\bf The map $A\!:\NC(n+1) \to \A(\mo \Lambda_n)$.}
We have defined $A(P)$ for any partition $P$ of $\{1,2,\dots,n+1\}$
as the set of arcs $[a,b]$ for $P$, see Lemma \ref{part-arc}.
Of course, we may interpret such an arc $[a,b]$ also as an indecomposable $\Lambda_n$-module.
Thus, given a partition $P\in \Pi(n+1)$, 
let $A(P)$ be the set of all the arcs of $P$, or as the set of the corresponding 
$\Lambda_n$-modules.
	
\begin{prop}\label{prop-partitions-arc}
 Let $P$ be a partition of $\{1,2,\dots,n+1\}$. Then $P$ is non-crossing
if and only if $A(P)$ is an antichain in $\mo\Lambda_n$.
\end{prop}

For further equivalent properties, see the note N\,\ref{non-crossing-partitions-arcs}.
	 
\begin{proof} First, assume that $P$ is crossing, thus there are two arcs $[a,b], [c,d]$ which are crossing.
We may assume that $a < c$, thus $a < c < b < d.$ But then $[c,b]$ is a non-zero factor module
of $[a,c]$ and also a submodule of $[c,d]$. It follows that $\Hom([a,b],[c,d]) \neq 0,$ thus the modules
$[a,b],[c,d]$ are non-isomorphic, but not orthogonal, and therefore $A(P)$ is not an antichain in
$\mo\Lambda_n$. Conversely, assume that $A(P)$ is not an antichain. Then there are different arcs
$[a,b],[c,d]$ such that $\Hom([a,b],[c,d])\neq 0.$ This implies that 
$a \le c < b \le d$. If $a=c$, then $c,b,d$ belong to the same part, thus
also $b = d$, since $[c,d]$ is an arc. This is a contradiction, since we assume that $[a,b]$ and
$[c,d]$ are different. Therefore $a < c$. Similarly, we see that $b < d$. This shows that
the arcs $[a,b]$ and $[c,d]$ are crossing. 
\end{proof} 
	
It follows that the restriction of the map $A$ defined on $\Pi(n+1)$ 
is a bijection $A\!:\NC(n+1) \to \A(\mo\Lambda_n).$ 

\begin{theorem}\label{classical} The map  $A\!:\NC(n+1) \to \A(\mo\Lambda_n)$ is a poset
isomorphism. 
\end{theorem}

\begin{proof}
It remains to show: If $P,P'$ are non-crossing partitions, then $P \le P'$
if and only if $A(P) \le A(P').$ 

First, assume that $P \le P'$. Let $[a,b] \in A(P)$. Then $a < b$ and $a,b$ belong to the
same part of $P$. Since $P \le P'$, the elements $a,b$ belong to the same part of $P'$,
thus, there are arcs $[a_i,a_{i+1}] \in A(P')$ with $1\le i < t$ such that $a = a_1$ 
and $a_t = b$. But this means that the $\Lambda_n$-module $[a,b]$ has a filtration with
factors $[a_i,a_{i+1}]$ and $1\le i < t$. These modules $[a_i,a_{i+1}]$ belong to $A(P')$.
Thus, we see that $[a,b]$ belongs to $A(P')$.

Conversely, assume that $A(P) \le A(P')$. We want to that $P \le P'$. Let 
$\{a_1 < a_a < \cdots < a_t\}$ be a part of $P$.  The pairs $[a_i,a_{i+1}]$ with $1\le i < t$
are arcs for $P$. Since $A(P) \le A(P')$, we see that any interval $[a_i,a_{i+1}]$
can be refined, say $a_i = a_{i1} < a_{i2} < \cdots < a_{is_{i}} = a_{i+1}$ such that
the intervals $[a_{ij},a_{i,j+1}]$ with $1 \le j < s_i$ belong to $A(P')$.
It follows that the elements  $a_i = a_{i1},a_{i2}, \cdots ,a_{is_{i}} = a_{i+1}$
belong to the same part of $P'$, thus all the $a_i$ with $1\le i < t$ belong to the
same part of $P'$. This shows that $P \le P'$.
\end{proof}

\begin{addendum} 
If we define $\mathcal A = \mathcal E(A),$
there is the following commutative diagram of poset isomorphisms:
\smallskip 

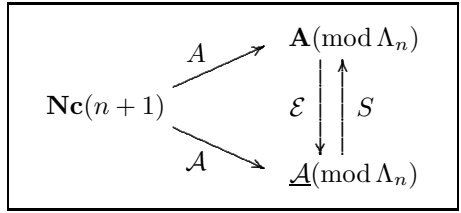
\Rahmen{$\beginpicture
  \setcoordinatesystem units <3cm,.9cm>
\put{$\NC(n+1)$} at 0 0
\put{$\A(\mo\Lambda_n)$} at 1.1 1
\put{$\underline{\mathcal A}(\mo\Lambda_n)$} at 1.1 -1
\arr{0.3 0.3}{0.7 0.9}
\arr{0.3 -.3}{0.7 -.9}
\put{$A$} at 0.4 0.8
\put{$\mathcal A$} at 0.4 -.8
\put{$\mathcal E$} at 0.85 0
\put{$S $} at 1.15 0
\arr{.95 0.7}{.95 -.7}
\arr{1.05 -.7}{1.05 0.7}
\endpicture$}
\smallskip

\noindent
\end{addendum}
	\medskip 

The map $\mathcal A$ provides a bijection between the non-crossing
partitions of the set $S = \{1,2,\dots,n+1\}$ and the thick subcategories
of $\mo\Lambda_n$. 
The structure of the thick subcategories $\mathcal A$ in $\mo\Lambda_n$ 
is very restricted. The essential observation is the following:

\begin{prop}\label{thick-a_n} 
Any connected thick subcategory $\mathcal A$
of $\mo\Lambda_n$ of rank $t$ is equivalent (as a category)
to $\mo\Lambda_t$.
\end{prop}
	
\begin{proof} It is easy to see that we only have to deal with $\mathcal A$ being of rank $3$. Thus, 
let $\mathcal A$ be a connected thick subcategory of $\mo\Lambda_n$ of rank $3$. We show that the quiver
of $\mathcal A$ cannot have two sources.

Let $X$ be a sink in $Q(\mathcal A)$, let $S$ be its top and $S' = \tau^{-1}S$.
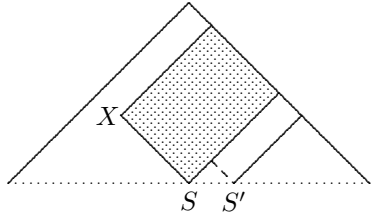
$$
\hbox{\beginpicture
  \setcoordinatesystem units <.3cm,.3cm>
\multiput{} at 0 0  16 8 /
\plot 0 0  8 8  16 0 /
\plot 9 7  5 3  8 0  12 4 /
\plot 10 0  13 3 /
\setdots <1mm>
\plot 0 0  16 0 /
\put{$X$} at 4.4 3
\put{$S$} at 8 -.8
\put{$S'$} at 10 -.8
\setdashes <1mm>
\plot 9 1  10 0 /
\setshadegrid span <.5mm>
\vshade 5 3 3 <,z,,> 8 0 6 <z,z,,> 9 1 7 <z,,,> 12 4 4 /
\endpicture}
$$
The shaded rectangle are the modules $Z$ with $\Hom(X,Z) \neq 0$.
The indecomposable modules $Y$ such that $X,Y$ is an orthogonal pair
and $\Ext^1(Y,X) \neq 0$ are the modules $Y$ with $\soc Y = S'$, these are the modules 
in the ray starting at $S'$. 

There is no pair of orthogonal modules $Y$ with $\soc Y = S'$.
This shows that the quiver of $\mathcal A$ cannot have two sources. 
By duality, it also cannot have two sinks. Thus $\mathcal A$
has to be equivalent to $\mo\Lambda_3.$
\end{proof}

The aim of our further considerations 
is to provide a dictionary between the combinatorial assertions concerning
$\NC(n\!+\!1)$ and the representation theoretical assertions concerning $\A(\mo\Lambda_n)$ and 
$\underline{\mathcal A}(\mo\Lambda_n)$. 
	\medskip 

\subsubsection{\bf Pairs of arcs.} 
For a non-crossing partition $P$, there are three types of pairs $\{[a,b],[c,d]\}$
of different arcs. We may assume that $a\le c.$ As we have seen above, this implies that $a < c$
and similarly, we have $b \neq d$. 
$$
\hbox{\beginpicture
  \setcoordinatesystem units <1cm,1cm>
\put{Case 1: $b=c$} at 0 2.2
\put{\beginpicture
  \setcoordinatesystem units <.4cm,.4cm>
\multiput{$\bullet$} at 0 0  3 0  5 0 /
\plot 0 0  0 1 /
\plot 3 0  3 1 /
\plot 5 0  5 1 /
\circulararc 180 degrees from 3 1 center at 1.5 1
\circulararc 180 degrees from 5 1 center at 4 1
\put{$a$} at 0 -1
\put{$b=c$} at 3 -1
\put{$d$} at 5 -1
\endpicture} at 0 0
\put{Case 2: $b<c$} at 4 2.2
\put{\beginpicture
  \setcoordinatesystem units <.4cm,.4cm>
\multiput{$\bullet$} at 0 0  3 0  4 0  6 0 /
\plot 0 0  0 1 /
\plot 3 0  3 1 /
\plot 4 0  4 1 /
\plot 6 0  6 1 /
\circulararc 180 degrees from 3 1 center at 1.5 1
\circulararc 180 degrees from 6 1 center at 5 1
\put{$a$} at 0 -1
\put{$b$} at 3 -1
\put{$c$} at 4 -1
\put{$d$} at 6 -1
\endpicture} at 4 0
\put{Case 3: $c < b$} at 8 2.2
\put
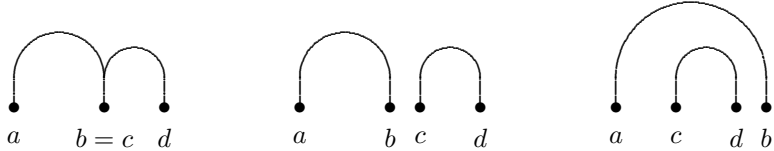
{\beginpicture
  \setcoordinatesystem units <.4cm,.4cm>
\multiput{$\bullet$} at 0 0  2 0  4 0  5 0 /
\plot 0 0  0 1 /
\plot 2 0  2 1 /
\plot 4 0  4 1 /
\plot 5 0  5 1 /
\circulararc 180 degrees from 4 1 center at 3 1
\circulararc 180 degrees from 5 1 center at 2.5 1
\put{$a$} at 0 -1
\put{$c$} at 2 -1
\put{$d$} at 4 -1
\put{$b$} at 5 -1
\endpicture} at 8 0
\endpicture}
$$

In case 3, the support of the module $[c,d]$ is a proper subset of the support of
$[a,b]$; actually, the module $M = [a,b]$ has a filtration $0 \subset M' \subset M'' \subset M$
with non-zero modules $M'$ and $M/M''$, such that $M''/M'$ is isomorphic to $[c,d]$.
In the cases 1 and 2, the modules $[a,b],[c,d]$ have disjoint support. 
In case 1, we have $\Ext^1([c,d],[a,b]) \neq 0$, namely, there is a non-split exact sequence
$$
 0 \to [a,b] \to [a,d] \to [c,d] \to 0, 
$$
whereas 
$\Ext^1([c,d],[a,b]) = 0$ in case 2 (and also in case 3). 
	\medskip 

Let us repeat that we define $\mathcal A = \mathcal F(A)$. Equivalently, we may specify
directly the indecomposable objects of $\mathcal A(P)$, these are the $\Lambda_n$-modules of the
form $[a,b]$ where $a < b$ are elements in the same block of $P$.
The parts of $P$ of cardinality $t\ge 2$ correspond bijectively to the blocks of $\mathcal A$
of rank $t-1$. Here is the recipe: 
Given a part $\{a_1<\dots<a_t\}$  of $P$, one sends it under $A$ to the antichain
$[a_1,a_2],[a_2,a_3],\dots,[a_{t-1},a_t]$ in $\mo\Lambda_n$, and $\mathcal A(P)$ is just the thick
closure of this antichain. The simple objects $[a_1,a_2],[a_2,a_3],\dots,[a_{t-1},a_t]$ 
in $\mathcal A(P)$ have a quiver of the form $1\leftarrow 2 \leftarrow \cdots \leftarrow t\!-\!1$
and the indecomposables in this block are the $\Lambda_n$-modules of the form $[a_i,a_j]$
with $1\le i < j\le t.$ 

The main point to have in mind is the following: the arcs of $P$ are identified with the simple
objects in $\mathcal A(P)$. In particular, the number of arcs of $A$ is the rank of $\mathcal A(A).$ 
For example, the partition $P$ with just one part (and therefore with 
arcs between $i$ and $i+1$, for all $1\le i \le n$) is mapped under $\mathcal A$ 
to the full module category $\mo\Lambda_n$. 
The subcategory $\mathcal A(P)$ is sincere if and only if 
$1$ and $n+1$ belong to the same part of $P$.
	\medskip

\subsubsection{\bf Maximal elements in $\NC(n+1)$.}
Let $P$ be a maximal element in $\NC(n+1)$, thus $P$ consists of precisely two parts $P_1$ and
$P_2$ and we may assume that $1$ belongs to $P_1$.
There are two different cases:

Case 1: The elements $1$ and $n+1$ do not belong to the same part, thus 
$n+1$ belongs to $P_2$. Let $x$ be the maximal element of
$P_1$. Then the arcs are just the pairs $i,i+1$ with $1\le i \le n$ and $i\neq x$.
Here is an  example:
$$
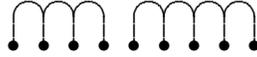
\hbox{\beginpicture
\setcoordinatesystem units <.4cm,.4cm>
\multiput{$\bullet$} at 0 0  1 0  2 0  3 0  4 0  5 0  6 0  7 0  8 0  /
\put{} at 0 2
\plot 0 0  0 1 /
\plot 1 0  1 1 /
\plot 2 0  2 1 /
\plot 3 0  3 1 /
\plot 4 0  4 1 /
\plot 5 0  5 1 /
\plot 6 0  6 1 /
\plot 7 0  7 1 /
\plot 8 0  8 1 /
\circulararc 180 degrees from 1 1 center at 0.5 1
\circulararc 180 degrees from 2 1 center at 1.5 1
\circulararc 180 degrees from 3 1 center at 2.5 1
\circulararc 180 degrees from 5 1 center at 4.5 1
\circulararc 180 degrees from 6 1 center at 5.5 1
\circulararc 180 degrees from 7 1 center at 6.5 1
\circulararc 180 degrees from 8 1 center at 7.5 1
\endpicture}
$$

Case 2: The elements $1$ and $n+1$ are in the same part, thus
now also $n+1$ belongs to $P_1$. Let
$x$ be the largest number with $x < P_2$ and $y$ the largest element of $P_2$. Then
the arcs are the pairs $i,i+1$ with $1\le i \le n$ and $i\notin\{x,y\}$ as well as
$x,y+1$. Here is an example: 
$$
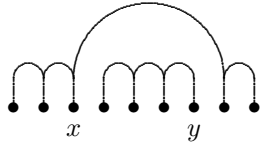
\hbox{\beginpicture
\setcoordinatesystem units <.4cm,.4cm>
\put{} at 0 3.5
\multiput{$\bullet$} at 0 0  1 0  2 0  3 0  4 0  5 0  6 0  7 0  8 0  /
\put{} at 0 3
\plot 0 0  0 1 /
\plot 1 0  1 1 /
\plot 2 0  2 1 /
\plot 3 0  3 1 /
\plot 4 0  4 1 /
\plot 5 0  5 1 /
\plot 6 0  6 1 /
\plot 7 0  7 1 /
\plot 8 0  8 1 /
\circulararc 180 degrees from 7 1 center at 4.5 1

\circulararc 180 degrees from 1 1 center at 0.5 1
\circulararc 180 degrees from 2 1 center at 1.5 1
\circulararc 180 degrees from 4 1 center at 3.5 1
\circulararc 180 degrees from 5 1 center at 4.5 1
\circulararc 180 degrees from 6 1 center at 5.5 1

\circulararc 180 degrees from 8 1 center at 7.5 1
\put{$x\strut$} at 2 -0.7
\put{$y\strut$} at 6 -0.7
\endpicture}
$$
(there is also the special case of $y = x+1$; in this case $P_2$ is a singleton).
	\medskip

Here is the reformulation in terms of $\mathcal A(P)$. In the first case,
$\mathcal A(P)$ is a Serre subcategory of $\mo\Lambda_n$; 
it consists of all $\Lambda_n$-modules such that a fixed
simple module (in the example, it is $S(x)$) is not a composition factor (in particular,
$\mathcal A(P)$ is not sincere).

In the second case, $\mathcal A$ is sincere. The simple objects in $\mathcal A$ are the simple
modules $S(i)$ with $i\notin\{x,y\}$ as well as the module $M = [x,y+1]$ corresponding  
to the arc from $x$ to $y+1$. Note that $M$, considered as a $\Lambda_n$-module,  
has a filtration $M' \subseteq M'' \subset M$
with $M'$ and $M/M''$ being simple $\Lambda_n$-modules (namely $M' = S(x), M/M'' = S(y)$),
such that $M''/M'$ belongs to $\mathcal A$ (it may be zero, otherwise it is indecomposable).
	\bigskip

\subsubsection{\bf The poset isomorphism $\iota_n\!:\NC(n\!+\!1) \to \NC(S_n,c_n)$.}
There are the poset isomorphisms 
$$
 A\!:\NC(n\!+\!1) \to \A(\mo\Lambda_{n})\quad \text{and}\quad 
 \cox\!:  \A(\mo\Lambda_n) \to \NC(W(\Lambda_n),c(\Lambda_n))
$$
as established in Theorem  \ref{classical} and Theorem \ref{IST}, 
respectively. Of course, $W(\Lambda_n) = S_n$
and $c(\Lambda_n) = c_n = (n,n-1,\dots,2,1)$. 
The composition 
yields an isomorphism 
$$
 \iota_n = \cox A \!:\NC(n\!+\!1) \to \NC(S_n,c_n). 
$$
Thus, there is the following theorem:

\begin{theorem}\label{classical-identification}
The posets $\NC(n\!+\!1)$ are $\NC(S_n,c_n)$ isomorphic. There is a commutative diagram
	\smallskip 

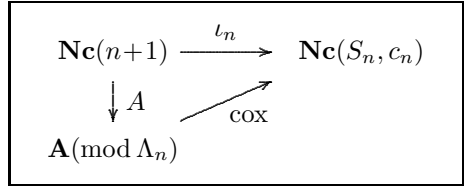
\Rahmen{$
\hbox{\beginpicture
  \setcoordinatesystem units <3cm,1.3cm>
\put{$\NC(n\!+\!1)$} at 0 1
\put{$\A(\mo\Lambda_{n})$} at 0 0 
\put{$\iota_n$} at 0.5 1.2
\arr{0 0.7}{0 0.3}
\arr{0.3 1}{0.7 1}
\put{$\NC(S_n,c_n)$} at 1.1 1 
\put{$A$} at 0.1 0.5
\arr{0.3 0.3}{0.7 0.7}
\put{$\cox$} at 0.6 0.35
\endpicture}
$}
	\smallskip

\noindent
of poset isomorphisms.
\end{theorem}

{\bf Remark.} Here, in order to identify $\NC(n\!+\!1)$ and $\NC(S_n,c_n)$ we have used as intermediate
step the categorification $\A(\mo\Lambda_{n}).$ This seems to be a proper procedure 
in lectures devoted to mathematicians working in representation theory, but it hides
the purely combinatorial nature of the identification $\iota_n$. 
Actually, the identification $\iota_n$ 
has to be considered as the starting point for the theory of
generalized non-crossing partitions. 
	\medskip

\subsection{Perpendicular pairs and the Kreweras complement}
We have defined in section \ref{delta} the anti-automorphism $\delta$ of
$\underline{\mathcal A}(\mo\Lambda_n)$ by $\delta(\mathcal A) = \mathcal A^\perp.$
We are going to show that in terms of non-crossing partitions, this is precisely the 
Kreweras complement as introduced by Kreweras \cite{[Kr]}.
	\medskip 

\subsubsection{\bf Perpendicular pairs.} Let $P$ be a non-crossing partition of $n+1$ elements. Besides the arcs
(which as before we will draw as solid lines) we will consider also coarcs and draw them using dashes lines.
By definition, the {\it coarcs} of $P$ 
are the pairs $[a,z+1]$, where $a \le z\le n$ and $a$ is the minimal
element of a part of $P$, whereas $z$ is its maximal element.  Thus, {\it 
if $p$ is the number of parts of $P$, then
the number of coarcs is $p-1$}. Namely, the coarcs correspond bijectively to the parts which do not
contain the element $n+1$). As a consequence, the number of arcs and coarcs together is $n$ (since
the number of arcs is $n+1-p$ and no coarc is an arc). We denote the set of coarcs by ${}^\perp A(P)$.

As an example, consider the partition $P =\bigl\{ \{1,4,7,\},\{2,3\},\{5\},\{6\},\{8,9,10\}\bigr\}$. 
On the left, we
show the usual arc diagram for $P$, on the right, we have added the coarcs: 
the part $\{1,4,7\}$
gives the coarc $[1,8]$, the part $\{2,3\}$ gives the coarc $[2,4]$, the singletons $\{5\}$ and $\{6\}$
give the coarcs $[5,6]$ and $[6,7]$, respectively (there is no coarc corresponding to the part $\{8,9,10\}$,
since this is the part which contains the element $n+1$): 
$$
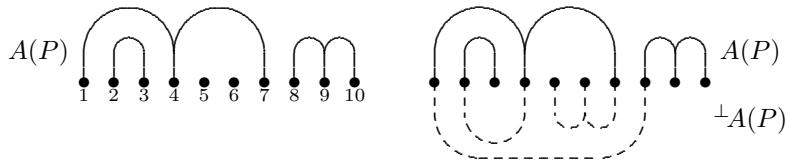
\hbox{\beginpicture
\setcoordinatesystem units <.4cm,.4cm>
\put{\beginpicture 
\multiput{$\bullet$} at 0 0  1 0  2 0  3 0  4 0  5 0  6 0  7 0  8 0  9 0 /
\put{$A(P)$} at -1.5 1 
\put{} at 0 2.5
\plot 0 0  0 1 /
\plot 1 0  1 1 /
\plot 2 0  2 1 /
\plot 3 0  3 1 /
\plot 6 0  6 1 /
\plot 7 0  7 1 /
\plot 8 0  8 1 /
\plot 9 0  9 1 /
\circulararc 180 degrees from 2 1 center at 1.5 1
\circulararc 180 degrees from 3 1 center at 1.5 1
\circulararc 180 degrees from 6 1 center at 4.5 1
\circulararc 180 degrees from 8 1 center at 7.5 1
\circulararc 180 degrees from 9 1 center at 8.5 1
\put{$\ssize 1$} at 0 -.4
\put{$\ssize 2$} at 1 -.4
\put{$\ssize 3$} at 2 -.4
\put{$\ssize 4$} at 3 -.4
\put{$\ssize 5$} at 4 -.4
\put{$\ssize 6$} at 5 -.4
\put{$\ssize 7$} at 6 -.4
\put{$\ssize 8$} at 7 -.4
\put{$\ssize 9$} at 8 -.4
\put{$\ssize 10$} at 9 -.4
\put{} at 0 -2.5
\endpicture} at 0 0

\put{\beginpicture 
\multiput{$\bullet$} at 0 0  1 0  2 0  3 0  4 0  5 0  6 0  7 0  8 0  9 0 /
\put{$A(P)$} at 10.5 1 
\put{${}^\perp\!A(P)$} at 10.5 -1.2 
\put{} at 0 2.5
\plot 0 0  0 1 /
\plot 1 0  1 1 /
\plot 2 0  2 1 /
\plot 3 0  3 1 /
\plot 6 0  6 1 /
\plot 7 0  7 1 /
\plot 8 0  8 1 /
\plot 9 0  9 1 /
\circulararc 180 degrees from 2 1 center at 1.5 1
\circulararc 180 degrees from 3 1 center at 1.5 1
\circulararc 180 degrees from 6 1 center at 4.5 1
\circulararc 180 degrees from 8 1 center at 7.5 1
\circulararc 180 degrees from 9 1 center at 8.5 1

\setdashes <1mm>
\plot 0 0  0 -1 /
\plot 1 0  1 -1 /
\plot 3 0  3 -1 /
\plot 4 0  4 -1 /
\plot 5 0  5 -1 /
\plot 6 0  6 -1 /
\plot 7 0  7 -1 /
\circulararc 90 degrees from 0 -1 center at 1.5 -1
\circulararc -90 degrees from 7 -1 center at 5.5 -1
\circulararc -180 degrees from 3 -1 center at 2 -1
\circulararc -180 degrees from 5 -1 center at 4.5 -1
\circulararc -180 degrees from 6 -1 center at 5.5 -1
\plot 1.5 -2.5  5.5 -2.5 /
\put{} at 0 -2.5
\endpicture} at 14 0

\endpicture}
$$ 
One should observe that the dashed lines (the coarcs) are the arcs of a new partition, namely of
the partition $\bigl\{\{1,8\},\{2,4\},\{3\},\{5,6,7\},\{9\},\{10\}\bigr\}$. Here is the general argument:

\begin{prop}\label{perp}
 The set ${}^\perp A(P)$ is the set of simple objects in ${}^\perp \mathcal A(P)$,
in particular it is an antichain in $\mo\Lambda_n$, thus the arc set of a non-crossing partition. 
\end{prop}
	
\begin{proof}  First, let us use Lemma \ref{arcs}  in order to
show that the coarcs are the arcs of a non-crossing partition.
Thus, assume that $[a,z+1]$ and $[a',z'+1]$ are coarcs and $a\le a' < z+1\le z'+1$. It
follows that $a \le a' \le z \le z'$.
If $a = a'$, or $a' = z$ of $z = z'$, 
then the coarcs are given by the same part of $P$, thus we have $a = a'$ and $z = z'$ (therefore
also $z+1 = z'+1$). Thus it remains to consider the case that $a < a' < z < z'$. But this is impossible,
since we assume that $P$ is non-crossing. 

As a consequence, the set of coarcs is the set of simple objects in some thick subcategory $\mathcal X$
of $\mo\Lambda_n$. We claim that $\mathcal X \subseteq {}^\perp\mathcal A(P)$. It is sufficient to
show the following: if $[x,y]$ is an arc and $[a,z+1]$ is a coarc, then we have
$\Hom([a,z+1],[x,y]) = 0 = \Ext^1([a,z+1],[x,y])$. 

First, assume for the contrary that $\Hom([a,z+1],[x,y]) \neq 0$. Since the $\Lambda_n$-top of
$[a,z+1]$ is $S(z)$, we see that $S(z)$ has to be a composition factor of $[x,y]$, thus $x\le z < y.$
We must have $x < z$, since otherwise $z$ and $y$ belong to the same part, but then $z$
is not maximal in this part. Since $P$ is non-crossing, we must have $x \le a$. Using again that
$z$ and $y$ belong to different parts, we see that $x < a$. The $\Lambda_n$-socle of $[x,y]$ is
$S(x)$. Since we assume that $\Hom([a,z+1],[x,y]) \neq 0$, the simple module $S(x)$ has to be
a composition factor of $[a,z+1]$. But this implies that $a \le x$, a contradiction. This shows
that $\Hom([a,z+1],[x,y]) = 0$.

A similar argument shows that for $a\ge 2$ we have $\Hom([x,y],[a-1,z]) \neq 0$. The Auslander-Reiten
formula implies that in this case $\Ext^1([a,z+1],[x,y]) = 0$. Of course, if $a = 1$, then 
$[a,z+1]$ is projective, thus also in this case $\Ext^1([a,z+1],[x,y]) = 0.$ 

We have shown that $\mathcal X \subseteq {}^\perp\mathcal A(P)$. Now the subcategories 
$\mathcal X$ and ${}^\perp\mathcal A(P)$ have the same rank, namely $p-1$, where $p$ is the number of parts 
of $P$. Thus, $\mathcal X = {}^\perp\mathcal A(P)$.
\end{proof}

\subsubsection{\bf The Kreweras complement.} 
It follows from Proposition \ref{perp} that 
the coarcs determine uniquely the arcs, the rule to obtain the arcs from the
coarcs is just a dual procedure: 

Let $F$ be a partition with arc set $A(F)$. 
For any part of $F$ with 
minimal element $a > 1$ and maximal element $b$, take the arc $[a-1,b]$.
Denote by $F^\perp$ the set of these arcs. The partition generated by $F^\perp$
is called the {\it Kreweras complement} $\kappa(F)$ of $F$.

For example, starting with 
$F = \bigl\{\{1,8\},\{2,4\},\{3\},\{5,6,7\},\{9\},\{10\}\bigl\}$
we obtain the partition
$\kappa(F) =\bigl\{ \{1,4,7,\},\{2,3\},\{5\},\{6\},\{8,9,10\}\bigr\}$. 
On the left, we
show the arc diagram for $F$ using dotted arcs, on the right, we have added 
the arcs of $\kappa(F)$ as solid arcs. The arcs of $F$ are drawn as dotted lines, 
since they are the 
coarcs for $\kappa(F)$. 

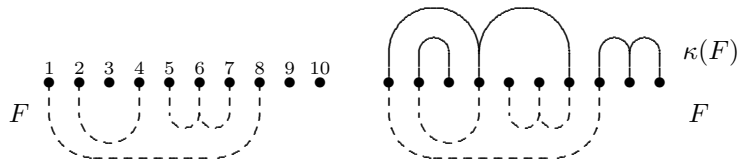
$$
\hbox{\beginpicture
\setcoordinatesystem units <.4cm,.4cm>
\put{\beginpicture 
\put{$F$} at -1 -1 
\multiput{$\bullet$} at 0 0  1 0  2 0  3 0  4 0  5 0  6 0  7 0  8 0  9 0 /
\put{} at 0 2.5
\put{$\ssize 1$} at 0 .5
\put{$\ssize 2$} at 1 .5
\put{$\ssize 3$} at 2 .5
\put{$\ssize 4$} at 3 .5
\put{$\ssize 5$} at 4 .5
\put{$\ssize 6$} at 5 .5
\put{$\ssize 7$} at 6 .5
\put{$\ssize 8$} at 7 .5
\put{$\ssize 9$} at 8 .5
\put{$\ssize 10$} at 9 .5
\setdashes <1mm>
\plot 0 0  0 -1 /
\plot 1 0  1 -1 /
\plot 3 0  3 -1 /
\plot 4 0  4 -1 /
\plot 5 0  5 -1 /
\plot 6 0  6 -1 /
\plot 7 0  7 -1 /
\circulararc 90 degrees from 0 -1 center at 1.5 -1
\circulararc -90 degrees from 7 -1 center at 5.5 -1
\circulararc -180 degrees from 3 -1 center at 2 -1
\circulararc -180 degrees from 5 -1 center at 4.5 -1
\circulararc -180 degrees from 6 -1 center at 5.5 -1
\plot 1.5 -2.5  5.5 -2.5 /
\put{} at 0 -2.5
\endpicture} at 0 0

\put{\beginpicture 
\multiput{$\bullet$} at 0 0  1 0  2 0  3 0  4 0  5 0  6 0  7 0  8 0  9 0 /
\put{$F$} at 10.3 -1 
\put{$\kappa(F)$} at 10.7 1 
\put{} at 0 2.5
\plot 0 0  0 1 /
\plot 1 0  1 1 /
\plot 2 0  2 1 /
\plot 3 0  3 1 /
\plot 6 0  6 1 /
\plot 7 0  7 1 /
\plot 8 0  8 1 /
\plot 9 0  9 1 /
\circulararc 180 degrees from 2 1 center at 1.5 1
\circulararc 180 degrees from 3 1 center at 1.5 1
\circulararc 180 degrees from 6 1 center at 4.5 1
\circulararc 180 degrees from 8 1 center at 7.5 1
\circulararc 180 degrees from 9 1 center at 8.5 1

\setdashes <1mm>
\plot 0 0  0 -1 /
\plot 1 0  1 -1 /
\plot 3 0  3 -1 /
\plot 4 0  4 -1 /
\plot 5 0  5 -1 /
\plot 6 0  6 -1 /
\plot 7 0  7 -1 /
\circulararc 90 degrees from 0 -1 center at 1.5 -1
\circulararc -90 degrees from 7 -1 center at 5.5 -1
\circulararc -180 degrees from 3 -1 center at 2 -1
\circulararc -180 degrees from 5 -1 center at 4.5 -1
\circulararc -180 degrees from 6 -1 center at 5.5 -1
\plot 1.5 -2.5  5.5 -2.5 /
\put{} at 0 -2.5
\endpicture} at 13 0

\endpicture}
$$ 
The part $\{2,4\}$ of $F$ yields the arc $[1,4]$ in $\kappa(F)$, the part $\{3\}$
in $F$ yields the arc $[2,3]$ in $\kappa(F)$, and so on. 
	\medskip

{\bf Remark.} The usual definition of the Kreweras complement of a non-crossing 
partition $F$ of $S = \{1,2,\dots,n\}$ is as follows: One considers the totally
ordered set 
$$
 1 < \overline 1 < 2 < \overline 2 < \cdots < n < \overline n 
$$
and takes as $\kappa(F) \in \NC(\overline 1,\overline 2,\dots,\overline n) 
\simeq \NC(1,2,\dots,n)$ the largest partition such that $F \sqcup \kappa(F)$
is a non-crossing partition of $\{1,\overline 1,\dots,n,\overline n\}.$
It is not difficult to verify that the two definitions of $\kappa$ coincide.

\begin{theorem}\label{Kreweras}
The following diagram commutes
	\smallskip 

\Rahmen{$\beginpicture
\setcoordinatesystem units <1.5cm,1.5cm>
\multiput{$\NC(n\!+\!1)$} at 0 1  2 1 /
\multiput{$\A(\mo\Lambda_n)$} at 0 0  2 0 /
\multiput{$A$} at 0.2 0.55  2.2 0.55 /
\arr{0 0.7}{0 0.3}
\arr{2 0.7}{2 0.3}
\arr{0.6 1}{1.4 1}
\arr{0.7 0}{1.3 0}
\put{$\delta$} at 1 0.2
\put{$\kappa$} at 1 1.2
\endpicture$}
\end{theorem} 

Thus, if we use $A$ in order to identify $\NC(n\!+\!1)$ and $\A(\mo\Lambda_n)$,
then $\delta$ and $\kappa$ coincide.
	\bigskip 

\subsubsection{\bf The automorphism $\delta^2$ of $\A(\mo\Lambda_n)$.}

\begin{prop}\label{delta-square-a_n}
If $\mathcal A$ is a thick subcategory of $\mo \Lambda_n$, then
$\delta^2(\mathcal A)$ is equivalent, as a category, to $\mathcal A.$
\end{prop}

\begin{proof} Let $\mathcal A$ be the product of the connected subcategories
$\mathcal A_1,\dots,\mathcal A_s$. Theorem \ref{AR} shows that $\delta^2(\mathcal A_i)
= \overline{\tau}(\mathcal A_i).$ In particular,  $\delta^2(\mathcal A_i)$ is a connected thick
subcategory of $\mo\Lambda_n$ with the same rank as $\mathcal A_i$, thus. according to Proposition 
\ref{thick-a_n}, the categories $\delta^2(\mathcal A_i)$ and $\mathcal A_i$ are
equivalent.
\end{proof}
	\medskip

\subsection{Non-crossing partitions and binary trees}\label{binary-section}
There is an interesting relationship
between non-crossing partitions and binary trees. 
	\medskip 

\subsubsection{\bf Binary trees.}
Binary trees (sometimes called rooted binary trees) are defined inductively.
The empty set is a binary tree, a non-empty binary tree is 
a triple $(L,s,R)$, where $L,R$ are binary trees and $s$ is a singleton, called the root.
We draw binary trees as graphs with two kinds of edges, solid ones and dashed ones.
The definition is again by induction: Let $B = (L,s,R)$ be a binary tree. We 
take the graphs corresponding to $L$ and $R$ and an additional vertex with
label $s$; if $L$ is non-empty, we
connect $s$ with the root of $L$ by a dashed edge; if $R$ is non-empty, we
connect $s$ with the root of $R$ by a solid edge. We call the roots of $L$ and $R$ 
the {\it successors} of $s$. In this way, we see that 
any vertex of $B$ has at most two successors. 

Here are two examples of binary trees $B$ of cardinality $7$, we usually draw the root (contrary
to its name) at the top of the picture, in both cases we have marked it by $s$.

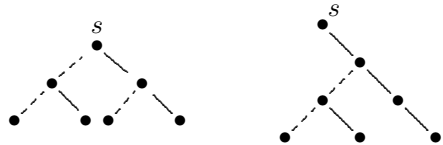
$$
\hbox{\beginpicture
\setcoordinatesystem units <.5cm,.5cm>
\put{\beginpicture
\multiput{$\bullet$} at -.2 0  0.8 1  1.7 0  2 2  3.2 1  2.3 0  4.2 0 /
\plot 1 0.8  1.5 0.3 /
\plot 2.2 1.8  2.8 1.3 /
\plot 3.5 0.7  4 0.2 /
\setdashes <.9mm>
\plot 0 0.2  0.6 0.8 /
\plot 1 1.2  1.6 1.7 /
\plot 2.5 0.2  3 0.8 /
\put{$s$} at 2 2.5 
\endpicture} at 0 0 
\put{\beginpicture
\multiput{$\bullet$} at 0 0  1 1  2 2  1 3  2 0  3 1  4 0 /
\plot 1.2 0.8  1.8 0.2 /
\plot 1.2 2.8  1.8 2.2 /
\plot 2.2 1.8  2.8 1.2 /
\plot 3.2 0.8  3.8 0.2 / 
\setdashes <.9mm>
\plot 0.2 0.2  0.8 0.8 /
\plot 1.2 1.2  1.8 1.8 /
\put{$s$} at 1.3 3.4 
\endpicture} at 7 0 

\endpicture}
$$
For the binary tree $B = (L,s,R)$ on the left, both binary trees $L,R$ have cardinality $3$. On the
right, $L$ is empty and the cardinality of $R$ is $6$.

There is an {\it intrinsic numbering} $\nu_B$ of the vertices of a binary tree $B$. We use again induction.
Of course, if $B$ is empty, nothing has to be done. Now assume $B = (L,s,R)$ and let $L$ be of cardinality $t$. Let $x$ be a vertex of $B$. Then we define
$$
\nu_B(x) = \left\{\begin{matrix} \nu_L(x) &  &\quad x\in L, \cr
                          t+1      & \quad\text{if} &\quad x=s, \cr
                          \nu_R(x)+t+1 & &\quad x\in R.  \end{matrix} \right.
$$
Here is the intrinsic numbering of the binary trees with 3 vertices:
$$
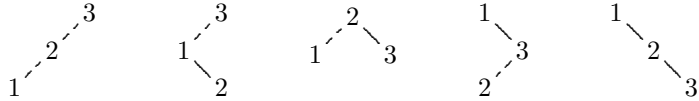
\hbox{\beginpicture
\setcoordinatesystem units <.5cm,.5cm>
\put{\beginpicture
\put{$1$} at 0 0 
\put{$2$} at 1 1
\put{$3$} at 2 2
\setdashes <1mm>
\plot 0.3 0.3  0.7 0.7 / 
\plot 1.3 1.3  1.7 1.7 / 
\endpicture} at 0 0

\put{\beginpicture
\put{$2$} at 2 0 
\put{$1$} at 1 1
\put{$3$} at 2 2
\plot 1.7 0.3  1.3 0.7 / 
\setdashes <1mm>
\plot 1.3 1.3  1.7 1.7 / 
\endpicture} at 4 0

\put{\beginpicture
\put{} at 0 0 
\put{$1$} at 0 1 
\put{$2$} at 1 2
\put{$3$} at 2 1
\plot 1.7 1.3  1.3 1.7 / 
\setdashes <1mm>
\plot 0.3 1.3  0.7 1.7 / 
\endpicture} at 8 0

\put{\beginpicture
\put{$1$} at 0 2 
\put{$3$} at 1 1
\put{$2$} at 0 0
\plot 0.3 1.7  0.7 1.3 / 
\setdashes <1mm>
\plot 0.7 0.7  0.3 0.3 / 
\endpicture} at 12 0

\put{\beginpicture
\put{$1$} at 0 2
\put{$2$} at 1 1
\put{$3$} at 2 0
\plot 1.7 0.3  1.3 0.7 / 
\plot 0.7 1.3  0.3 1.7 / 
\setdashes <1mm>
\endpicture} at 16 0

\endpicture}
$$
	\medskip 

Let $P$ be a non-crossing partition of $n+1$ elements. We define the graph $B(P) = A(P)\cup{}^\perp\! A(P)$
as follows: its vertices are the numbers $\{1,2,\dots,n+1\}$ and we use the arcs and the coarcs of $P$
as edges. It is a labeled graph: any edge is either solid (if it is an arc) or dashed 
(if it is a coarc).
	\medskip

\begin{prop}\label{binary}
 The labeled graph $B(P) = A(P)\cup {}^\perp\! A(P)$ is a binary tree.
\end{prop}
	
\begin{proof} As we know, the number of edges is $n$. Let us show for $x < n+1$ 
that there is a path in $B(P)$ starting at $x$ and ending at a vertex $x'$ with 
$x < x'$ (it follows then by induction that there is a path from $x$ to $n+1$). Let $a$ be the
minimal element and $z$ the maximal element of the part of $P$ which contains $x$. Then $x$ and $a$ are
connected by a sequence of arcs, and $[a,z+1]$ is a coarc. Thus $x$ is connected in $B(P)$ by a path
from $x$ to $x' = z+1$ and $x \le z < z+1$. 

Since $B(P)$ is a connected graph with $n$ edges and $n+1$ vertices, we see that $B(P)$ is 
a tree. If $x$ is a vertex, there is at most one vertex $y$ such that $[x,y]$ is an arc,
and at most one vertex $a$ such that $[a,x]$ is a coarc. This shows that $B(P)$ is a binary
tree using the arcs a solid edges and the coarcs as dashed edges.
\end{proof}
	
\subsubsection{}
Let us return to the example  
 $P =\bigl\{ \{1,4,7,\},\{2,3\},\{5\},\{6\},\{8,9,10\}\bigr\}$.
On the left, we reproduce the picture of the binary tree $B(P)$ as
shown above, on the right we rearrange the
vertices in order to see better the binary tree structure
(in particular, now the edges are straight lines, the root $8$ is the top vertex, 
and the successors of a vertex $x$ are below $x$). 
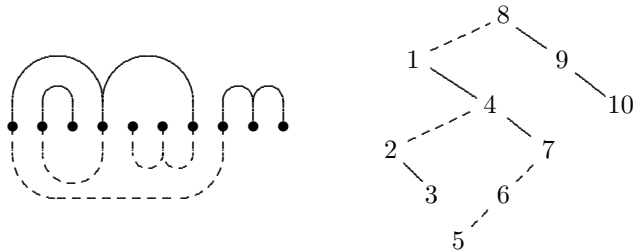
 $$
\hbox{\beginpicture
\setcoordinatesystem units <.4cm,.35cm>
\put{\beginpicture 
\multiput{$\bullet$} at 0 0  1 0  2 0  3 0  4 0  5 0  6 0  7 0  8 0  9 0 /
\put{} at 0 2.5
\plot 0 0  0 1 /
\plot 1 0  1 1 /
\plot 2 0  2 1 /
\plot 3 0  3 1 /
\plot 6 0  6 1 /
\plot 7 0  7 1 /
\plot 8 0  8 1 /
\plot 9 0  9 1 /
\circulararc 180 degrees from 2 1 center at 1.5 1
\circulararc 180 degrees from 3 1 center at 1.5 1
\circulararc 180 degrees from 6 1 center at 4.5 1
\circulararc 180 degrees from 8 1 center at 7.5 1
\circulararc 180 degrees from 9 1 center at 8.5 1

\setdashes <1mm>
\plot 0 0  0 -1 /
\plot 1 0  1 -1 /
\plot 3 0  3 -1 /
\plot 4 0  4 -1 /
\plot 5 0  5 -1 /
\plot 6 0  6 -1 /
\plot 7 0  7 -1 /
\circulararc 90 degrees from 0 -1 center at 1.5 -1
\circulararc -90 degrees from 7 -1 center at 5.5 -1
\circulararc -180 degrees from 3 -1 center at 2 -1
\circulararc -180 degrees from 5 -1 center at 4.5 -1
\circulararc -180 degrees from 6 -1 center at 5.5 -1
\plot 1.5 -2.7  5.5 -2.7 /
\put{} at 0 -2.5
\endpicture} at 0 0
\put{\beginpicture 
\setcoordinatesystem units <.6cm,.6cm>
\put{$1$} at -1 4
\put{$2$} at -1.5 2
\put{$3$} at -.6 1
\put{$4$} at 0.7 3
\put{$5$} at 0 0 
\put{$6$} at 1 1
\put{$7$} at 2 2
\put{$8$} at 1 5
\put{$9$} at 2.3 4
\put{$10$} at 3.6 3
\plot 1.3 4.7  1.95 4.25 /
\plot 2.6 3.7  3.25 3.2 /
\plot -.6 3.8  0.4 3.15 /
\plot 1.1 2.7  1.6 2.3 /
\plot -1.2 1.7  -.8 1.3 /
\setdashes <1mm>
\plot -.6 4.2  0.7 4.8 /
\plot 0.3 0.3  0.7 0.7 /
\plot 1.3 1.3  1.7 1.7 /
\plot -1.1 2.2  0.4 2.9 /
\endpicture} at 12 0
\endpicture}
$$ 
	\medskip 

We note the following: 
{\it The root of $B(P)$ is the minimal element of the part of $P$ which contains 
$n+1$ and the maximal element of the part of ${}^\perp P$ which contains $1$.} In terms of
$\mo\Lambda_n$, the root of $B(P)$ corresponds to cutting the indecomposable sincere $\Lambda$-module
$I = [1,n+1]$ into a submodule $M$ which belongs to ${}^\perp \mathcal A(P)$ and its factor module
$I/M$ which belongs to $\mathcal A(P)$. Namely, there is the following lemma:


\begin{lemma} Let $I/M$ be the largest factor module of $I$ which belongs to $\mathcal A(P)$.
Then $M$ belongs to ${}^\perp \mathcal A(P)$.
\end{lemma} 
	
\begin{proof} Assume that $X$ belongs to $A(P)$ and $\Hom(M,X) \neq 0.$ In particular, we have $M \neq 0$, thus
$I/M$ is not projective and $\tau(I/M) = \rad I/\rad M$. Now $X$ cannot be injective, since it has
$\topp M$ as a composition factor (it would be a factor module of $I$ which belongs to $\mathcal A(P)$,
and of larger length than $I/M$). 
\end{proof}

These considerations indicate an interesting interpretation of the 
base set of $\NC(n+1)$, these are the numbers $1,2,\dots,n\!+\!1$: we may call these elements 
the possible {\it cuts},
since they serve to describe the cuts of the module $[1,n+1]$, or also the cuts of the quiver of $\Lambda_n$.
The Lemma yields the cut $s$,
where $S(s)$ is the socle of $I/M$ provided $I/M\neq 0$, and $s = n+1$ in case $I/M = 0$.
	\bigskip

\subsubsection{} 
Let us denote by $\B(n)$ the set of binary trees with $n$ vertices. The previous considerations
may be summarized as follows:

\begin{theorem} Let $B$ be a binary tree with $n$ vertices and use the intrinsic numbering. 
The set $A(B)$ of solid edges of $B$ is a non-crossing partition, also the set of dashed edges of $B$ is
a non-crossing partitions, it is just ${}^\perp A(B)$. The maps 
	\smallskip

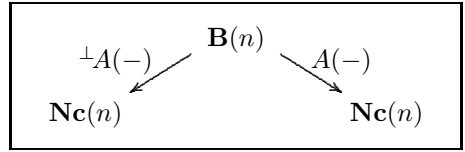
\Rahmen{$\beginpicture
\setcoordinatesystem units <2cm,1cm>
\put{$\B(n)$} at 1 1
\multiput{$\NC(n)$} at 0 0  2 0 /
\arr{0.7 0.8}{0.3 0.3}
\arr{1.3 0.8}{1.7 0.3}
\put{$A(-)\strut$} at 1.7 0.7
\put{${}^\perp\! A(-)\strut$} at 0.2 0.7 
\endpicture$}
	\smallskip

\noindent 
are bijective maps.
\end{theorem}

The composition of the maps from left to right
$$
 \NC(n)  \overset{({}^\perp\! A(-))^{-1}}\longrightarrow \B(n) \overset{A(-)} \longrightarrow \NC(n)
$$
is just the Kreweras complement $\kappa$ defined by 
$\kappa(A) = A^\perp$ for $A\in \NC(n)$,
see Theorem \ref{Kreweras}.
	\bigskip

\subsection{The $n^{n-2}$-problems:
 Maximal chains in $\NC(n)$, parking functions, labeled trees} 
	
We have mentioned in Chapter 1 that $\c(\mathbb A_n) = (n+1)^{n-1}$: the number of
complete exceptional sequences for a quiver of type $\mathbb A_n$ is given by $(n+1)^{n-1}$.
There is a wealth of mathematical counting problems with the answer $(n+1)^{n-1}$ or $n^{n-2}$
and one tries to find canonical bijections.
	\medskip

\subsubsection{} {\bf Here are some of the $n^{n-2}$-problems:}
	\medskip

{\bf (a) Maximal chains in the lattice $\NC(n)$ of non-crossing partitions.} 
	\medskip 

Let us denote by $\M(\NC(n))$ the set of maximal chains in the lattice $\NC(n)$. Then
	\smallskip

\Rahmen{$|\M(\NC(n))| = (n+1)^{n-1}$.}
	\smallskip

\noindent
We hope that the
further discussions in this section will provide a better understanding of this equality. 
	\bigskip 

{\bf (b) The set $[1,n]^{n-2}$.} Of course, $n^{n-2}$ counts the number of (set-theoretical) functions 
$$
 [1,n\!-\!2] \to [1,n];
$$
let us denote the set of these functions by $[1,n]^{n-2}$, we may consider this set
as the set of sequences $c_1,\dots,c_{n-2}$ of integers $c_i$ with $1\le i \le n.$
	\medskip

{\bf (c) Parking functions.} 
We denote by $\P(n)$ the set of  
{\it parking functions} for $n$ cars, 
these are the endofunctions $f$ of the set $[1,n] = \{1,2,\dots,n\}$
such that $\{x\mid f(x) \le i\}$ has cardinality at least $i$ for all $i$, or, equivalently, provided
the non-decreasing rearrangement $b_1,\dots,b_n$ of the value sequence $f(1),\dots,f(n)$
satisfies $b_i\le i$ for all $i$. Instead of looking at the function $f\!:[1,n] \to [1,n]$,
we also may call the corresponding $n$-tuple $(f(1),\dots,f(n))$
a parking function for $n$ cars (the coefficients of such an $n$-tuple are positive integers
and it follows directly from the definition of a parking function that they are bounded by $n$). 
For example, the parking functions for $n=3$ are obtained from 
$$
 (1,1,1),\ (1,1,2),\ (1,1,3),\ (1,2,2),\ (1,2,3)
$$
by taking all rearrangements; thus there are precisely 
$16 = 1+3+3+3+6$ parking functions for $n = 3$.

Parking functions were introduced in 1966 by Konheim and Weiss (and by Pyke in 1959).
Why are these functions called parking functions? Consider $n$ cars $C_1,\dots,C_n$ and $n$
parking lots labeled $1,2,\dots,n$. The cars arrive in arbitrary order at the parking lot, and the 
car $C_i$ is required to park in the lot $f(i)$ provided it is free, otherwise in the 
next free parking lot. 
{\it The function $f$ is a parking function if and only if any car finds a parking lot
if the cars arrive in some fixed order, and in this case any car finds a parking lot if the cars
arrive in any order.}
$$
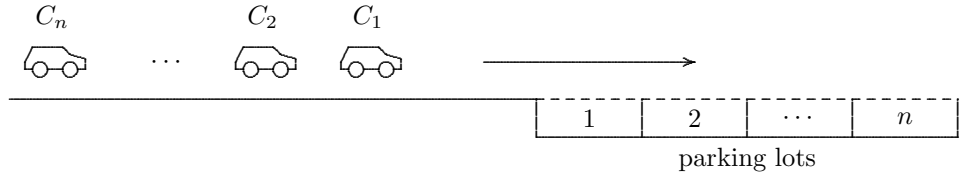
\hbox{\beginpicture
\setcoordinatesystem units <1.4cm,1cm>
\multiput{} at 0 0  9 0 /
\arr{4.5 1}{6.5 1}
\multiput{\car} at 0.5 1  2.5 1  3.5 1 /
\plot 0 0.5  5 0.5 /
\plot 5 0  9 0 /
\plot 5 0  5 0.5 /
\plot 6 0  6 0.5 /
\plot 7 0  7 0.5 /
\plot 8 0  8 0.5 /
\plot 9 0  9 0.5 /
\setdashes <1mm>
\plot 5 0.5  9 0.5 /
\put{$C_1$} at 3.4 1.6
\put{$C_2$} at 2.4 1.6
\put{$\cdots$} at 1.5 1 
\put{$C_n$} at 0.4 1.6
\put{$1$} at 5.5 0.25 
\put{$2$} at 6.5 0.25 
\put{$\cdots$} at 7.5 0.25 
\put{$n$} at 8.5 0.25 
\put{parking lots} at 7 -.3
\endpicture}
$$

We have:
	\medskip
 
\Rahmen{$|\P(n)| = (n+1)^{n-1}$.}
	\bigskip

The proof follows immediately from the following lemma:
	\medskip

\begin{lemma} For any sequence $a = (a_1,\dots,a_n) \in [1,n+1]^n$, there are $n+1$ elements
$a'\in  [1,n+1]^n$ such that $a'-a \mod n+1$ is constant and precisely one of 
these elements $a'$ is a parking function for $n$ cars.
\end{lemma}

\begin{proof}[Proof, using again a parking lot simulation.] This time, $n+1$ parking lots
are arranged in a circle, labeled $1,2,\dots,n+1$ in consecutive order. 
There are $n$ cars $C_i$, where $1\le i \le n$, they come in the order $C_1,\dots,C_n$. 
The car $C_i$ starts at the parking lot
$a_i$, takes it, if it is free, otherwise it takes the next free lot (driving around the circle).
Since there are only $n$ cars, but $n+1$ lots, any car will find a free lot and at the end, 
precisely one of the lots, say $l(a)$ will remain free (we have fixed the order of the cars,
so that we don't have to show that $l(a)$ is well-defined). Clearly, $a$ is a parking function
for $n$ cars if and only if $l(a) = n+1$. On the other hand, 
if we fix some $c$ with $1\le c \le n+1$ and define $a' = (a'_1,\dots,a'_n)$ by $a'_i = a_i+c
\mod n+1$, for all $i$, then $l(a') \equiv l(a)+c \mod n+1.$ 
\end{proof} 

The formula  $|\P(n)| = (n+1)^{n-1}$ follows, since the cardinality of $[1,n+1]^n$ is $(n+1)^n$, 
thus $|\P(n)| = \frac1{n+1}(n+1)^n = (n+1)^{n-1}$.
Here is a reformulation in terms of bijections:
	\medskip

\begin{cor} For any sequence $a = (a_1,\dots,a_{n-1})\in [1,n+1]^{n-1}$, let $u(a) =  
(a_1,\dots,a_{n-1},1) \in [1,n+1]^n.$ If $b\in [1,n+1]^n$, let $p(b)$ be the parking function
for $n$ cars with $p(b)-b \mod n+1$ being constant. Then the map
	\smallskip

\Rahmen{$pu\!:[1,n+1]^{n-1} \longrightarrow \P(n)$}
	\smallskip\noindent
is a bijection.
\end{cor}  
	\medskip 

\subsubsection{\bf From maximal chains in $\M(\NC(n))$ to parking functions.} 
Starting with a maximal chain $A^{(0)},\dots,A^{(n)}$ in $\NC(n)$, Stanley  has defined
 a sequence 
$\lambda(A^{(0)},\dots,A^{(n)}) \in \mathbb N^n$
as follows: 
Assume that $P < P'$ are neighbors in $\NC(n),$ thus $P,P'$ are two 
non-crossing partitions of $\{1,2,\dots,n\}$ and there are two different parts $P_i,P_j$ of $P$
such that $P'$ is obtained from $P$ by replacing the parts $P_i,P_j$ by the disjoint union
$P_i\cup P_j$; we assume that the minimal element of $P_i$ is smaller than the minimal element of $P_j$
and define $\lambda'(P,P')$ as the maximal element belonging to $P_i$ and being smaller 
than all elements of $P_j$
(in particular, always we have $\lambda'(P,P') \le n-1$). Then there is the definition
$$
 \lambda(P^{(0)},\dots,P^{(n)}) = 
 \bigl(\lambda'(P^{(0)},P^{(1)}),\lambda'(P^{(1)},P^{(2)}),\dots,\lambda'(P^{(n-1)},P^{(1)})\bigr).
$$
For example, if we consider the maximal chain
$$
 1|2|3|4|5 < 1|25|3|4 < 125|3|4 < 125|34 < 12345
$$ 
in $\NC(5),$
we have
$$
 \lambda\bigl(1|2|3|4|5,1|25|3|4,125|3|4,125|34,12345\bigr) = (2,1,3,2).
$$
	\medskip 

\begin{theorem} {\bf (Stanley, 1997).} The map 
	\smallskip 

\Rahmen{$\lambda\!:\M(\NC(n+1)) \to \P(n)$.}
	\smallskip

\noindent 
is a bijection.
\end{theorem} 
	

Here is the case $n=3$. For any pair $P<P'$ of neighbors in $\NC(4)$ we have added
the number $\lambda'(P,P')$ to the (dotted) line connecting $P$ and $P'$:
	\vfill\eject

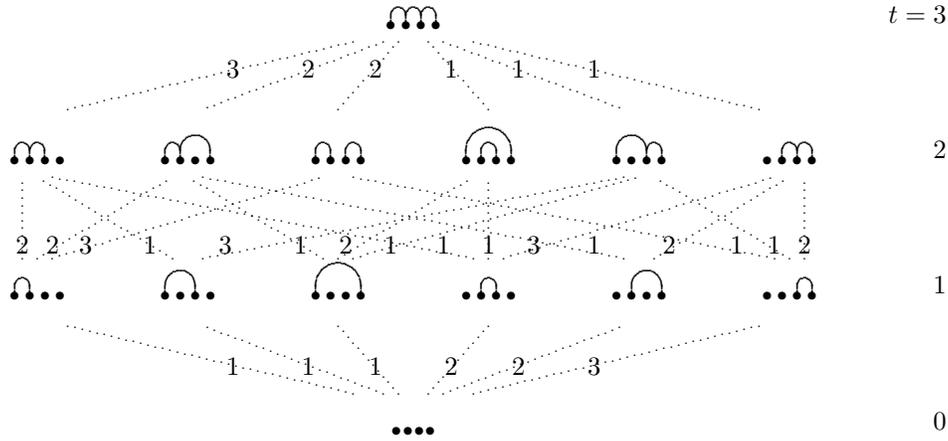
$$
\hbox{\beginpicture
  \setcoordinatesystem units <2cm,1.8cm>
\put{$t=3$} at 3.35 3
\put{$2$} at 3.5 2
\put{$1$} at 3.5 1
\put{$0$} at 3.5 0
\put{$\beginpicture
  \setcoordinatesystem units <0.15cm,0.2cm>
  \multiput{} at 0 2  /
  \multiput{$\sssize \bullet$} at  0 0.3  1 0.3  2 0.3  3 0.3 /
 \endpicture$} at 0 0

\put{$\beginpicture
  \setcoordinatesystem units <0.2cm,0.2cm>
  \multiput{} at 0 2  /
  \multiput{$\sssize \bullet$} at  0 0.3  1 0.3  2 0.3  3 0.3 /
  \circulararc 180 degrees from 1 1 center at 0.5 1 
  \plot 0 0.3  0 1 /
  \plot 1 0.3  1 1 /
 \endpicture$} at -2.5 1
\put{$\beginpicture  
  \setcoordinatesystem units <0.2cm,0.2cm>
  \multiput{} at 0 2  /
  \multiput{$\sssize \bullet$} at  0 0.3  1 0.3  2 0.3  3 0.3 /
  \circulararc 180 degrees from 2 1 center at 1 1 
  \plot 0 0.3  0 1 /
  \plot 2 0.3  2 1 /
 \endpicture$} at -1.5 1
\put{$\beginpicture
  \setcoordinatesystem units <0.2cm,0.2cm>
  \multiput{} at 0 2  /
  \multiput{$\sssize \bullet$} at  0 0.3  1 0.3  2 0.3  3 0.3 /
  \circulararc 180 degrees from 3 1 center at 1.5 1 
  \plot 0 0.3  0 1 /
  \plot 3 0.3  3 1 /  
 \endpicture$} at -.5 1
\put{$\beginpicture
  \setcoordinatesystem units <0.2cm,0.2cm>
  \multiput{} at 0 2  /
  \multiput{$\sssize \bullet$} at  0 0.3  1 0.3  2 0.3  3 0.3 /
  \circulararc 180 degrees from 2 1 center at 1.5 1 
  \plot 1 0.3  1 1 /
  \plot 2 0.3  2 1 /
 \endpicture$} at 0.5 1
\put{$\beginpicture
  \setcoordinatesystem units <0.2cm,0.2cm>
  \multiput{} at 0 2  /
  \multiput{$\sssize \bullet$} at  0 0.3  1 0.3  2 0.3  3 0.3 /
  \circulararc 180 degrees from 3 1 center at 2 1 
  \plot 1 0.3  1 1 /
  \plot 3 0.3  3 1 /  
 \endpicture$} at 1.5 1
\put{$\beginpicture
  \setcoordinatesystem units <0.2cm,0.2cm>
  \multiput{} at 0 2  /
  \multiput{$\sssize \bullet$} at  0 0.3  1 0.3  2 0.3  3 0.3 /
  \circulararc 180 degrees from 3 1 center at 2.5 1 
  \plot 2 0.3  2 1 /
  \plot 3 0.3  3 1 /  
 \endpicture$} at 2.5  1
\put{$\beginpicture
  \setcoordinatesystem units <0.2cm,0.2cm>
  \multiput{} at 0 2  /
  \multiput{$\sssize \bullet$} at  0 0.3  1 0.3  2 0.3  3 0.3 /
  \circulararc 180 degrees from 1 1 center at 0.5 1 
  \circulararc 180 degrees from 2 1 center at 1.5 1 
  \plot 0 0.3  0 1 /
  \plot 1 0.3  1 1 /
  \plot 2 0.3  2 1 /
 \endpicture$} at -2.5 2
\put{$\beginpicture
  \setcoordinatesystem units <0.2cm,0.2cm>
  \multiput{} at 0 2  /
  \multiput{$\sssize \bullet$} at  0 0.3  1 0.3  2 0.3  3 0.3 /
  \circulararc 180 degrees from 1 1 center at 0.5 1 
  \circulararc 180 degrees from 3 1 center at 2 1 
  \plot 0 0.3  0 1 /
  \plot 1 0.3  1 1 /
  \plot 3 0.3  3 1 /  
 \endpicture$} at -1.5 2
\put{$\beginpicture
  \setcoordinatesystem units <0.2cm,0.2cm>
  \multiput{} at 0 2  /
  \multiput{$\sssize \bullet$} at  0 0.3  1 0.3  2 0.3  3 0.3 /
  \circulararc 180 degrees from 1 1 center at 0.5 1 
  \circulararc 180 degrees from 3 1 center at 2.5 1 
   \plot 0 0.3  0 1 /
  \plot 1 0.3  1 1 /
  \plot 2 0.3  2 1 /
  \plot 3 0.3  3 1 /  
 \endpicture$} at -.5 2
\put{$\beginpicture
  \setcoordinatesystem units <0.2cm,0.2cm>
  \multiput{} at 0 2  /
  \multiput{$\sssize \bullet$} at  0 0.3  1 0.3  2 0.3  3 0.3 /
  \circulararc 180 degrees from 2 1 center at 1.5 1 
  \circulararc 180 degrees from 3 1 center at 1.5 1 
  \plot 0 0.3  0 1 /
  \plot 1 0.3  1 1 /
  \plot 2 0.3  2 1 /
  \plot 3 0.3  3 1 /  
 \endpicture$} at 0.5 2
\put{$\beginpicture
  \setcoordinatesystem units <0.2cm,0.2cm>
  \multiput{} at 0 2  /
  \multiput{$\sssize \bullet$} at  0 0.3  1 0.3  2 0.3  3 0.3 /
  \circulararc 180 degrees from 3 1 center at 2.5 1 
  \circulararc 180 degrees from 2 1 center at 1 1 
  \plot 0 0.3  0 1 /
  \plot 2 0.3  2 1 /
  \plot 3 0.3  3 1 /  
 \endpicture$} at 1.5 2
\put{$\beginpicture
  \setcoordinatesystem units <0.2cm,0.2cm>
  \multiput{} at 0 2  /
  \multiput{$\sssize \bullet$} at  0 0.3  1 0.3  2 0.3  3 0.3 /
  \circulararc 180 degrees from 2 1 center at 1.5 1 
  \circulararc 180 degrees from 3 1 center at 2.5 1 
  \plot 1 0.3  1 1 /
  \plot 2 0.3  2 1 /
  \plot 3 0.3  3 1 /  
 \endpicture$} at 2.5 2

\put{$\beginpicture
  \setcoordinatesystem units <0.2cm,0.2cm>
  \multiput{} at 0 2  /
  \multiput{$\sssize \bullet$} at  0 0.3  1 0.3  2 0.3  3 0.3 /
  \circulararc 180 degrees from 1 1 center at 0.5 1 
  \circulararc 180 degrees from 2 1 center at 1.5 1 
  \circulararc 180 degrees from 3 1 center at 2.5 1 
  \plot 0 0.3  0 1 /
  \plot 1 0.3  1 1 /
  \plot 2 0.3  2 1 /
  \plot 3 0.3  3 1 /  
 \endpicture$} at  0 3

\setdots <1mm>

\plot -0.4 0.2 -2.3 0.7 /
\plot -0.2 0.2 -1.4 0.7 /
\plot -0.1 0.2 -0.5 0.7 /
\plot .1 0.2 0.5 0.7 /
\plot .2 0.2 1.4 0.7 /
\plot .4 0.2 2.3 0.7 /

\plot -0.4 2.8 -2.3 2.3 /
\plot -0.2 2.8 -1.4 2.3 /
\plot -0.1 2.8 -0.5 2.3 /
\plot .1 2.8 0.5 2.3 /
\plot .2 2.8 1.4 2.3 /
\plot .4 2.8 2.3 2.3 /

\plot -2.6 1.2  -2.6 1.8 /
\plot -2.5 1.2  -1.6 1.8 /
\plot -2.4 1.2  -0.6 1.8 /

\plot -1.6 1.2  -2.5 1.8 /
\plot -1.4 1.2   1.4 1.8 /

\plot -.6 1.2  -1.5 1.8 /
\plot -.5 1.2   0.4 1.8 /
\plot -.4 1.2   1.5 1.8 /

\plot .4 1.2  -2.4 1.8 /
\plot .5 1.2   .5 1.8 /
\plot .6 1.2   2.4 1.8 /

\plot 1.4 1.2  -1.4 1.8 /
\plot 1.6 1.2   2.5 1.8 /

\plot 2.4 1.2  -.4 1.8 /
\plot 2.5 1.2   1.6 1.8 /
\plot 2.6 1.2   2.6 1.8 /

\multiput{$1$} at -1.2 0.4  -0.7 0.4  -0.25 0.4 
   -1.75 1.3  -.75 1.3  -.15 1.3  .2 1.3  .5 1.3  1.2 1.3
              2.15 1.3  2.4 1.3 
   0.25 2.6  0.7 2.6  1.2 2.6 / 
\multiput{$2$} at .25 0.4  .7 0.4 
   -2.6 1.3  -2.4 1.3  -.45 1.3  1.7 1.3  2.6 1.3 
   -0.7 2.6  -.25 2.6 /
\multiput{$3$} at 1.2 0.4 
   -2.18 1.3  -1.25 1.3  .8 1.3 
   -1.2 2.6 /

\endpicture}
$$
	\medskip 

Let us reformulate the map $\lambda$ in terms of $\M(\A(\mo\Lambda_n))$, where $\Lambda = \Lambda_n$ is the path algebra
of the linearly directed quiver of type $\mathbb A_n$ (so that $\A(\mo\Lambda_n)$ can be identified 
with $\A(\NC(n+1))$). It is sufficient to define the numbers $\lambda'(A,B)$, where 
$A < B$ are antichains in $\mo\Lambda_n$ of cardinality $t-1$ and $t$, respectively, for some $t$.
In case $A$ is a subset of $B$, say $B = A\cup\{B_t\},$ then $\lambda'(A,B) = i$ where $S(i) = \soc_\Lambda B_{t}$. 
Otherwise, there is some element $A_i$ in $A$ such that $A_i$ has a proper filtration with
factors in $B$. In this case, $\lambda'(A,B) = i$ where $S(i) = \soc_\Lambda A_i$. 

We have exhibited in Theorem \ref{neighbors1} and its Addendum \ref{neighbors2}
bijections between the set $\M(\A(\mo\Lambda_n))$ of all maximal chains in the poset
$\A(\mo\Lambda_n)$ and the set 
$\E(\mo\Lambda_n)$ of all complete exceptional sequences in the category $\mo\Lambda_n$.
We are going to describe Stanley's bijection 
in terms of complete exceptional sequences.
	\medskip

\subsubsection{} 
If $(X_1,\dots,X_t)$ is a sequence of indecomposable $\Lambda_n$-modules,  
we write $\soc(X_1,\dots,X_t) =
(s_1,\dots,s_t)$, where $\soc_\Lambda X_i = S(s_i)$, for $1\le i \le t$.
In this way,  $\soc(X_1,\dots,X_t)$
is an element of $[1,n]^t.$
	\medskip

\begin{theorem}\label{soc} If $E$ is complete exceptional sequence in $\mo\Lambda_n$, then
$\soc E$ belongs to $\P(n)$ and the map  
	\smallskip 

\Rahmen{$\soc\!:\E(\mo\Lambda_n) \to \P(n)$}
	\smallskip

\noindent 
is a bijection.

There is the following commutative diagram
	\smallskip 

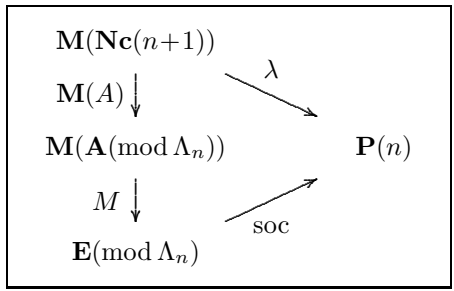
\Rahmen{$
\hbox{\beginpicture
  \setcoordinatesystem units <3cm,1.4cm>
\put{$\M(\NC(n\!+\!1))$} at 0 2
\put{$\M(\A(\mo \Lambda_n))$} at 0 1
\put{$\E(\mo\Lambda_{n})$} at 0 0 
\arr{0 0.7}{0 0.3}
\arr{0 1.7}{0 1.3}
\put{$\P(n)$} at 1.1 1 
\put{$\M(A)$} at -.2 1.5
\put{$M$} at -.13 0.5
\arr{0.4 0.3}{0.8 0.7}
\arr{0.4 1.7}{0.8 1.3}
\put{$\soc$} at 0.6 0.25
\put{$\lambda$} at 0.6 1.75
\endpicture}
$}
\end{theorem}

In order to prove this, it is sufficient to show the following: 
{\it If $A<B$ are neighbors in $\NC(n+1)$ and
$\mathcal A, \mathcal B$ are the corresponding thick subcategories of $\mo\Lambda_n,$ then}
	\smallskip

\Rahmen{$\soc_{\Lambda} M^{\mathcal B}_{\mathcal A}= S(\lambda'(A,B))$ .} 
	\medskip 

\begin{proof}  Since $A < B$ are neighbors, there are two different parts $A_i,A_j$ of $A$
such that $B$ is obtained from $A$ by replacing the parts $A_i,A_j$ by the disjoint union
$A_i\cup A_j$; we assume that the minimal element $a_i$ of $A_i$ is smaller 
than the minimal element $a_j$ of $A_j$.
Then, by definition, $x = \lambda'(A,B)$ is the maximal element which belongs
to $A_i$ and is smaller 
than all elements of $A_j$. We denote by $y$ the maximal element of $A_j.$  
Since $x < a_j \le y$, the module $M = [x,y]$ is indecomposable and its $\Lambda$-socle is $S(x)$.
Since $x,y$ belong to the same part of $B$, we know that $M$ belongs to $\mathcal B$.

Let $U$ be a simple object in $\mathcal A$, say $U = [u,v]$ for some numbers $1\le u < v \le n+1$. 
We want to show that both $\Hom(U,M) = 0 = \Ext^1(U,M)$ so that $M$ belongs to $\mathcal A^\perp.$

Let us assume that $\Hom(U,M) \neq 0.$
Since $S(x)$ is the socle of $M$, we see that $S(x)$ is a composition factor of $U$,
thus $u \le x < v.$ Let us show that we even have $y < v$.
Now $U$ belongs to some
block of $\mathcal A$. First, assume that $U$ belongs to $\mathcal A_i.$ Then $v < a_j$ is impossible by the
maximality of $x$. Also, $v = a_j$ and $v = y$ are impossible, since $A_i\cap A_j = \emptyset$.
Finally, $a_j < v < y$ is impossible, since $A$ is non-crossing. It follows that $y < v.$
Note that $U$ cannot belong to $\mathcal A_j$, since $[x,x+1]$ is a composition factor of $U$, but
$x < a_j$ and $a_j$ is the minimal element $A_j$. It follows that $U$ belongs to a block $A_t$
with $t\notin\{i,j\}$. This block $\mathcal A_t$ is a block of $\mathcal B$, thus, since $B$
is non-crossing and $u < x < v$ it follows that ($u < a_i$ and) $y < v.$
But the condition $y < v$ implies that $\Hom([u,v],[x,y]) = 0,$ a contradiction.
Thus, we have $\Hom(U,M) = 0.$

The assertion $\Ext^1(U,M) = 0$  is clear in case $y = n+1$, since then
$M$ is injective in $\mo \Lambda$. Let us assume that $\Ext^1(U,M) \neq 0$, thus $y < n+1$
and the Auslander-Reiten formula shows that $\Hom([x+1,y+1],U) \neq 0$ (since $\tau_\Lambda^{-1} =
[x+1,y+1]$). It follows that $U$ has $[y,y+1]$ as $\Lambda$-composition factor, thus
$u\le y < v$. We claim that we even have $u \le x$. 

Now $U$ cannot belong to $\mathcal A_j$, since the indecomposable
modules in $\mathcal A_j$ are of the form $[z,z']$ with $z' \le y$. Assume that $U$ belongs to $\mathcal A_i.$
Then we must have $u = x$ (namely, in this case $x$ cannot be maximal in $A_i$, thus there is
$x'\in A_i$ with $x < x'$ and no other element of $A_i$ in-between; then $[x,x']$ is the only simple
object of $\mathcal A_i$ with $\Lambda$-composition factor $[y,y+1]$, therefore $U = [x,x']$).
Finally, if $U$ belongs to $\mathcal A_t$ with $t\notin\{i,j\}$, then $A_t$ is a 
part of $B$. Since $B$ is non-crossing and $U = [u,v]$ has the $\Lambda$-composition factor $[y,y+1]$, 
it follows that $u < a_i$ (and $y < v$), therefore $u < x$. 
Since $u \le x$, it follows that $\Hom([x+1,y+1],[u,v]) = 0$, a contradiction. Thus $\Ext^1(U,M) = 0.$

Altogether, we see that $M \in \mathcal B\cap\mathcal A^\perp,$ therefore $M = M^{\mathcal B}_{\mathcal A}.$
This completes the proof. 
\end{proof}

{\bf Remark.} 
It is easy to see directly: {\it If $M = (M_1,\dots,M_n)$ is a complete exceptional sequence
of $\Lambda_n$-modules, then $\soc M$ is a parking function} 
(see also N\,\ref{another-parking-function}).

\begin{proof} 
We have to show that for any integer $i$ with $1\le i \le n$, the number of modules $M_j$
such that $\soc M_j = S(i')$ for some $i' \le i$ is at most $i$.
Let us delete these modules $M_j$ from the sequence $M$. The remaining modules form an exceptional
sequence $M'$ of $\Lambda'$-modules, where $\Lambda'$ is the path algebra of the quiver
$i\!+\!1 \leftarrow \cdots \leftarrow n$, this quiver has $n-i$ vertices, thus any exceptional
sequence of $\Lambda'$-modules consists of  at most $n-i$ modules. This shows that we have deleted
from $M$ at least $i$ modules. 
\end{proof} 

Let us stress that the Stanley bijection shows, in particular, that {\it for $\Lambda = \Lambda_n$,
the map $M \mapsto \soc M$ defined on the set of complete exceptional sequences is injective}. 
It seems to be of interest to find a direct proof. But one should be aware that for here we deal
with a special feature of the $\Lambda_n$-modules which does not hold for an arbitrary
artin algebra (see the note N\,\ref{socle-of-exceptional-sequences}).
which looks at the quiver $\Bbb A_3$ with 2 sources. The note N\,\ref{algorithm-for-tilting-modules} 
presents an algorithm how to recover a tilting module $T$ from $\topp T$.
	\medskip

\subsubsection{} Using the $k$-duality $M\mapsto M^*$ from the
category $\mo\Lambda_n$ to the category $\mo\Lambda_n^\op$, 
we see that there is also the following bijection: If $(X_1,\dots,X_t)$ is an exceptional sequence in $\mo\Lambda_n$, then $(X_t^*,\dots,X_1^*)$ is an exceptional
sequence for $\Lambda' = \Lambda_n^{\text{op}}$. Let $S_{\Lambda'}(i) = S_\Lambda(\pi i),$ where
$\pi$ is the permutation of $\{1,2,\dots,n\}$ which reverses the order. 
	\medskip

\begin{theorem}  If $(X_1,\dots,X_t)$ is an exceptional sequence in $\mo\Lambda_n$, let 
$\topp(X_1,\dots,X_t) =
(t_n,\dots,t_1)$, where $\topp_\Lambda X_i= 
S(\pi t_i)$, for $1\le i \le t$. Then 
	\smallskip
 
\Rahmen{$\topp\!:\E(\mo\Lambda_n) \to \P(n)$}
	\smallskip

\noindent 
is a bijection.
\end{theorem} 

Using the bijections $\soc$ and $\topp$ from $\E(\mo\Lambda_n)$ to $\P(n)$, 
we obtain a non-trivial involution 
$\xi = \topp\circ(\soc)^{-1}$ of $\P(n)$; for example, we have $\xi(1,1,1) = (1,2,3),$ and 
$\xi(3,2,1) = (3,2,1)$, see the note N\,\ref{soc-topp}.
	\bigskip

\subsubsection{} A parking function is said to be {\it primitive} provided it is weakly increasing.
	\smallskip

Given a tilting module $T$ with direct summands $T_i$, we may index the direct summands
in such a way that $E(T) = (T_1,\dots,T_n)$ is exceptional and the socle sequence is
non decreasing. Then {\it $E$ furnishes a bijection from the set of tilting modules to the
set of primitive parking functions.} See also N\,\ref{algorithm-for-tilting-modules}.
	\bigskip

\subsubsection{} We return to our attempt to provide an overview on some of 
the most relevant $n^{n-2}$-problems.  
  \medskip 

{\bf (d) Labeled trees.} The most famous $n^{n-2}$-problem seems to be the problem of counting labeled trees. 
For a solution of this problem one often refers to Cayley (1889), however, 
already Sylvester (1857) and Borchardt (1860) discussed the problem.
	\medskip 

A {\it labeled tree} $T$ with $n$ vertices is a tree with vertex set $\{1,2,\dots,n\}$.
We should stress that counting labeled trees means to count the actual number of such trees, not just the
number of isomorphism classes (usually, throughout of these lectures, counting meant to count 
isomorphism classes) --- but actually, we also could phrase it as follows: looking at labeled
trees $T,T'$, to say that they are isomorphic (as labeled trees) 
means that $T' = T$, and the only automorphism of a labeled
tree is the identity map. 

We denote by $\LT(n)$ the set of labeled trees with $n$ vertices.
Here are the pictures for $3$ vertices (there are $3^1$ 
possibilities):
$$
\hbox{\beginpicture
  \setcoordinatesystem units <.5cm,.5cm>
\put{\beginpicture
\multiput{$\bullet$} at 0 0  1 0  2 0 /
\plot 0 0  2 0 /
\put{$\ssize 1$} at 0 -0.4
\put{$\ssize 2$} at 1 -0.4
\put{$\ssize 3$} at 2 -0.4
\endpicture} at 0 0 
\put{\beginpicture
\multiput{$\bullet$} at 0 0  1 0  2 0 /
\plot 0 0  2 0 /
\put{$\ssize 1$} at 0 -0.4
\put{$\ssize 3$} at 1 -0.4
\put{$\ssize 2$} at 2 -0.4
\endpicture} at 4 0 
\put{\beginpicture
\multiput{$\bullet$} at 0 0  1 0  2 0 /
\plot 0 0  2 0 /
\put{$\ssize 2$} at 0 -0.4
\put{$\ssize 1$} at 1 -0.4
\put{$\ssize 3$} at 2 -0.4
\endpicture} at 8 0 
\endpicture}
$$
And here are the pictures for $4$ vertices (there are $4^2 = 16$
possibilities):
$$
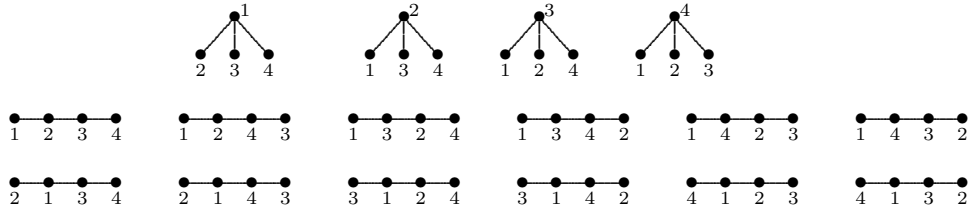
\hbox{\beginpicture
  \setcoordinatesystem units <.45cm,.5cm>
\put{\beginpicture
\multiput{$\bullet$} at 0 0  1 0  2 0  1 1 /
\plot 0 0  1 1  1 0 /
\plot 1 1  2 0 /
\put{$\ssize 1$} at 1.3 1.2
\put{$\ssize 2$} at 0 -0.4
\put{$\ssize 3$} at 1 -0.4
\put{$\ssize 4$} at 2 -0.4
\endpicture} at 0 0 
\put{\beginpicture
\multiput{$\bullet$} at 0 0  1 0  2 0  1 1 /
\plot 0 0  1 1  1 0 /
\plot 1 1  2 0 /
\put{$\ssize 2$} at 1.3 1.2
\put{$\ssize 1$} at 0 -0.4
\put{$\ssize 3$} at 1 -0.4
\put{$\ssize 4$} at 2 -0.4
\endpicture} at 5 0 
\put{\beginpicture
\multiput{$\bullet$} at 0 0  1 0  2 0  1 1 /
\plot 0 0  1 1  1 0 /
\plot 1 1  2 0 /
\put{$\ssize 3$} at 1.3 1.2
\put{$\ssize 1$} at 0 -0.4
\put{$\ssize 2$} at 1 -0.4
\put{$\ssize 4$} at 2 -0.4
\endpicture} at 9 0 
\put{\beginpicture
\multiput{$\bullet$} at 0 0  1 0  2 0  1 1 /
\plot 0 0  1 1  1 0 /
\plot 1 1  2 0 /
\put{$\ssize 4$} at 1.3 1.2
\put{$\ssize 1$} at 0 -0.4
\put{$\ssize 2$} at 1 -0.4
\put{$\ssize 3$} at 2 -0.4
\endpicture} at 13 0 

\put{\beginpicture
\multiput{$\bullet$} at 0 0  1 0  2 0  3 0 /
\plot 0 0  3 0 /
\put{$\ssize 1$} at 0 -0.4
\put{$\ssize 2$} at 1 -0.4
\put{$\ssize 3$} at 2 -0.4
\put{$\ssize 4$} at 3 -0.4
\endpicture} at -5 -2.3 
\put{\beginpicture
\multiput{$\bullet$} at 0 0  1 0  2 0  3 0 /
\plot 0 0  3 0 /
\put{$\ssize 1$} at 0 -0.4
\put{$\ssize 2$} at 1 -0.4
\put{$\ssize 4$} at 2 -0.4
\put{$\ssize 3$} at 3 -0.4
\endpicture} at 0 -2.3
\put{\beginpicture
\multiput{$\bullet$} at 0 0  1 0  2 0  3 0 /
\plot 0 0  3 0 /
\put{$\ssize 1$} at 0 -0.4
\put{$\ssize 3$} at 1 -0.4
\put{$\ssize 2$} at 2 -0.4
\put{$\ssize 4$} at 3 -0.4
\endpicture} at 5 -2.3 
\put{\beginpicture
\multiput{$\bullet$} at 0 0  1 0  2 0  3 0 /
\plot 0 0  3 0 /
\put{$\ssize 1$} at 0 -0.4
\put{$\ssize 3$} at 1 -0.4
\put{$\ssize 4$} at 2 -0.4
\put{$\ssize 2$} at 3 -0.4
\endpicture} at 10 -2.3
\put{\beginpicture
\multiput{$\bullet$} at 0 0  1 0  2 0  3 0 /
\plot 0 0  3 0 /
\put{$\ssize 1$} at 0 -0.4
\put{$\ssize 4$} at 1 -0.4
\put{$\ssize 2$} at 2 -0.4
\put{$\ssize 3$} at 3 -0.4
\endpicture} at 15 -2.3
\put{\beginpicture
\multiput{$\bullet$} at 0 0  1 0  2 0  3 0 /
\plot 0 0  3 0 /
\put{$\ssize 1$} at 0 -0.4
\put{$\ssize 4$} at 1 -0.4
\put{$\ssize 3$} at 2 -0.4
\put{$\ssize 2$} at 3 -0.4
\endpicture} at 20 -2.3

\put{\beginpicture
\multiput{$\bullet$} at 0 0  1 0  2 0  3 0 /
\plot 0 0  3 0 /
\put{$\ssize 2$} at 0 -0.4
\put{$\ssize 1$} at 1 -0.4
\put{$\ssize 3$} at 2 -0.4
\put{$\ssize 4$} at 3 -0.4
\endpicture} at -5 -4 
\put{\beginpicture
\multiput{$\bullet$} at 0 0  1 0  2 0  3 0 /
\plot 0 0  3 0 /
\put{$\ssize 2$} at 0 -0.4
\put{$\ssize 1$} at 1 -0.4
\put{$\ssize 4$} at 2 -0.4
\put{$\ssize 3$} at 3 -0.4
\endpicture} at 0 -4
\put{\beginpicture
\multiput{$\bullet$} at 0 0  1 0  2 0  3 0 /
\plot 0 0  3 0 /
\put{$\ssize 3$} at 0 -0.4
\put{$\ssize 1$} at 1 -0.4
\put{$\ssize 2$} at 2 -0.4
\put{$\ssize 4$} at 3 -0.4
\endpicture} at 5 -4 
\put{\beginpicture
\multiput{$\bullet$} at 0 0  1 0  2 0  3 0 /
\plot 0 0  3 0 /
\put{$\ssize 3$} at 0 -0.4
\put{$\ssize 1$} at 1 -0.4
\put{$\ssize 4$} at 2 -0.4
\put{$\ssize 2$} at 3 -0.4
\endpicture} at 10 -4 
\put{\beginpicture
\multiput{$\bullet$} at 0 0  1 0  2 0  3 0 /
\plot 0 0  3 0 /
\put{$\ssize 4$} at 0 -0.4
\put{$\ssize 1$} at 1 -0.4
\put{$\ssize 2$} at 2 -0.4
\put{$\ssize 3$} at 3 -0.4
\endpicture} at 15 -4 
\put{\beginpicture
\multiput{$\bullet$} at 0 0  1 0  2 0  3 0 /
\plot 0 0  3 0 /
\put{$\ssize 4$} at 0 -0.4
\put{$\ssize 1$} at 1 -0.4
\put{$\ssize 3$} at 2 -0.4
\put{$\ssize 2$} at 3 -0.4
\endpicture} at 20 -4 
\endpicture}
$$

It has been shown by Borchardt (1860), Sylvester (1857) and 
Cayley (1889) that 
	\smallskip 

\Rahmen{$|\LT(n)| = n^{n-2}$.} 
	\medskip

A constructive proof has been given by Pr\"ufer (1918), by attaching to
any labeled tree $T$ a sequence
$c(T)$ of numbers, now called the {\it Pr\"ufer code} of $T$. It is defined as follows:

If $T$ has $n\ge 3$ vertices, let $x_1$ be the leaf with the lowest label
and $c_1$ its (unique) neighbor. Now remove the vertex $x_1$, we get
a tree $T_2$ with labels $\{1,\dots,n\}\setminus\{x_1\}$ and if $n\ge 4$ continue:
Let $x_2$ be the leaf with the lowest label
and $c_2$ its (unique) neighbor. Altogether we get two sequences
$$
\hbox{\beginpicture
  \setcoordinatesystem units <1cm,.8cm>
\put{$x_1,\quad x_2,\quad x_3,\quad ...,\quad x_{n-2},$} at 0 0
\put{$c_1,\quad c_2,\quad c_3,\quad ...,\quad c_{n-2}.$} at 0 -1
\endpicture}
$$
The Pr\"ufer code of $T$ is the sequence $c(T) = (c_1,c_2,\dots,c_{n-2})$.
This is a sequence of $n-2$ integers $c_i$ with $1\le c_i\le n$. The decisive 
fact is that one can recover $T$ from $c(T)$. 
	\medskip 

{\bf Example for calculating $c(T)$.} In any step $i=1,2,\dots,6$, the leaf with lowest label is
encircled; the label $c_i$ of its neighbor is exhibited below the tree. In this
way, the lower row shows the code $c = (c_1,\dots,c_6)$. 
$$
\hbox{\beginpicture
  \setcoordinatesystem units <.45cm,.5cm>
\put{\beginpicture
\multiput{$\bullet$} at 0 1  0 2  1 0  1 1  2 2  2 3  3 1  4 2 /
\plot 0 1  0 2  2 3  4 2 /
\plot 1 0  1 1  2 2  2 3 /
\plot 2 2  3 1 /
\put{$\ssize 1$} at 2.3 3.2 
\put{$\ssize 2$} at -.4 2
\put{$\ssize 3$} at 2.4 2.1
\put{$\ssize 4$} at 4.5 2
\put{$\ssize 5$} at 1.4 1
\put{$\ssize 6$} at -.4 1
\put{$\ssize 7$} at 3.4 1
\put{$\ssize 8$} at 1.4 0
\put{$\bigcirc$} at 4 2
\endpicture} at 0 0
\put{\beginpicture
\multiput{$\bullet$} at 0 1  0 2  1 0  1 1  2 2  2 3  3 1  /
\plot 0 1  0 2  2 3  /
\plot 1 0  1 1  2 2  2 3 /
\plot 2 2  3 1 /
\put{$\ssize 1$} at 2.3 3.2 
\put{$\ssize 2$} at -.4 2
\put{$\ssize 3$} at 2.4 2.1
\put{$\ssize 5$} at 1.4 1
\put{$\ssize 6$} at -.5 1
\put{$\ssize 7$} at 3.4 1
\put{$\ssize 8$} at 1.4 0
\put{$\bigcirc$} at 0 1
\endpicture} at 6 0

\put{\beginpicture
\multiput{$\bullet$} at   0 2  1 0  1 1  2 2  2 3  3 1  /
\plot 0 2  2 3  /
\plot 1 0  1 1  2 2  2 3 /
\plot 2 2  3 1 /
\put{$\ssize 1$} at 2.3 3.2 
\put{$\ssize 2$} at -.5 2
\put{$\ssize 3$} at 2.4 2.1
\put{$\ssize 5$} at 1.4 1
\put{$\ssize 7$} at 3.4 1
\put{$\ssize 8$} at 1.4 0
\put{$\bigcirc$} at 0 2
\endpicture} at 11 0
\put{\beginpicture
\multiput{$\bullet$} at    1 0  1 1  2 2  2 3  3 1  /
\plot 1 0  1 1  2 2  2 3 /
\plot 2 2  3 1 /
\put{$\ssize 1$} at 2.4 3.3 
\put{$\ssize 3$} at 2.4 2.1
\put{$\ssize 5$} at 1.4 1
\put{$\ssize 7$} at 3.4 1
\put{$\ssize 8$} at 1.4 0
\put{$\bigcirc$} at 2 3
\endpicture} at 16 0

\put{\beginpicture
\multiput{$\bullet$} at    1 0  1 1  2 2   3 1  /
\plot 1 0  1 1  2 2  /
\plot 2 2  3 1 /
\put{$\ssize 3$} at 2.4 2.1
\put{$\ssize 5$} at 1.4 1
\put{$\ssize 7$} at 3.5 1
\put{$\ssize 8$} at 1.4 0
\put{$\bigcirc$} at 3 1
\put{} at 2 3
\endpicture} at 20 0

\put{\beginpicture
\multiput{$\bullet$} at    1 0  1 1  2 2   /
\plot 1 0  1 1  2 2  /
\put{$\ssize 3$} at 2.5 2.1
\put{$\ssize 5$} at 1.4 1
\put{$\ssize 8$} at 1.4 0
\put{$\bigcirc$} at 2 2
\put{} at 1 3
\endpicture} at 24 0
\put{$c_i =$} at -2 -3
\put{$1$} at 1 -3 
\put{$2$} at 6 -3 
\put{$1$} at 11 -3 
\put{$3$} at 16 -3 
\put{$3$} at 20 -3 
\put{$5$} at 24 -3 
\endpicture}
$$

\begin{theorem}{\bf (Pr\"ufer, 1918).} The Pr\"ufer code $c(T)$ provides a bijection 
	\smallskip 

\Rahmen{$\LT(n) \to [1,n]^{n-2}$.}
	\smallskip

\end{theorem}

	\bigskip\bigskip 

{\bf Notes.}
	\medskip 

\begin{note}\label{partitions-arcs} {\bf The bijection between the partitions and partition arc sets.} 
We assume that $S = \{1,2,\dots,n\}$. Arcs in $S$ are pairs $[a,b]$
with $1 \le a < b \le n$. A set $U$ of arcs in $S$ is said to be a {\it partition arc set}
provided for $[a,b], [a',b']$ in $U$, we have $a = a'$ iff $b = b'.$ 

Given a partition $P$ of $S$, we have defined $A(P)$ as the set of arcs for $P$, where an arc
for $P$ is a pair $[a,b]$ with $1\le a < b \le n+1$ such that $a,b$ belong to the same part of $P$
whereas no element $c$ with $a < c < b$ belongs to this part. Claim: 
{\it $A(P)$ is a partition arc set.} Namely, assume that $[a,b], [a,b']$ are arcs for $P$. We
may assume that $b\le b'$. Now $a,b,b'$ belong to the same part of $P$. Since $[a,b']$ is an
arc for $P$, it follows that $b = b'$. On the other hand, assume that $[a,b], [a',b]$ are
arcs for $P$. We may assume that $a\le a'$. Thus  $a,b,b'$ belong to the same part of $P$. 
Since $[a,b']$ is an arc for $P$, it follows that $a = a'$. 

On the other hand, assume that $U$ is a partition arc set in $S$. Let $P$ be the smallest
equivalence relation defined by the arcs which belong to $U$. We claim that $U = A(P)$.
Namely, any part of $P$ is of the form $\{a_1 < a_2 < \dots < a_t\}$ such that the pairs
$[a_i,a_{i+1}]$ with $1\le i < t$ are arcs in $U$. 
\end{note}

\begin{note}\label{non-crossing-partitions-arcs}
{\bf Characterization of non-crossing partitions.}
The Proposition \ref{prop-partitions-arc} 
asserts that a partition is non-crossing if and only if $A(P)$ is an antichain. 
There are further ways to characterize non-crossing partitions using properties of $A(P)$:
	\medskip

Lemma. {\it 
 Let $P$ be a partition. Then the following assertions are equivalent:
\begin{enumerate} 
\item[\rm(i)] $P$ is non-crossing.
\item[\rm(ii)] $A(P)\cup\{0\}$ is closed under kernels.
\item[\rm(iii)] $A(P)\cup\{0\}$ is closed under images.
\item[\rm(iv)] $A(P)\cup\{0\}$ is closed under cokernels.
\end{enumerate}}
	\medskip 
	
The proof is left as an exercise. 

\end{note}

\begin{note}\label{another-parking-function} {\bf More parking functions.} 
{\it If $M = (M_1,\dots,M_n)$ 
is a complete exceptional sequence of $\Lambda_n$-modules, then also
the length sequence $|M| = (|M_1,|,\dots,|M_n|)$ is a parking function for $n$ cars.}
	\medskip

\begin{proof} Let $1\le i < n$ be an integer. Let $M'$ be obtained from $M = (M_1,\dots,M_n)$
by deleting all modules $M_j$ with $|M_j|\le i$. 
We have to show that at most $n-i$ modules remain. We consider the factor category
$\mathcal F = \mo\Lambda/\mathcal X$ of $\mo\Lambda$, where $\Lambda = \Lambda_n$ and
$\mathcal X$ is the ideal of all maps in $\mo\Lambda$
which factor through a direct sum of modules of length at most $i$. The Auslander-Reiten
quiver of $\Lambda_n$ shows that $\mo\Lambda_n/\mathcal X$ is equivalent to $\mo\Lambda'$,
with $\Lambda' = \Lambda_{n-i}.$
We claim that $M'$ is an exceptional sequence in $\mo\Lambda_{n-i}$ (therefore it contains at most
$n-i$ elements). Namely, consider two modules $M_s, M_t$ in $M$ of length at least $i+1$ such that
$s > t$. We have to show that $\Hom_{\mathcal F}(M_s,M_t) = 0 = \Ext^1_{\mathcal F}(M_s,M_t) $.
Of course, we have $\Hom_\Lambda(M_s,M_t) = 0$ and also $\Ext^1_\Lambda(M_s,M_t) = 0.$ 
The first equality implies immediately that also $\Hom_{\mathcal F}(M_s,M_t) = 0$. If $M_s$ is 
a projective $\Lambda$-module, then $M_s$ is also projective in $\mathcal F$, thus 
$\Ext^1_{\mathcal F}(M_s,M_t) = 0.$ If $M_s$ is not projective as a $\Lambda$-module, then
$\Ext^1_\Lambda(M_s,M_t) = 0$ implies that $\Hom_\Lambda(M_t,\tau_\Lambda M_s) = 0.$
and therefore $\Hom_{\mathcal F}(M_t,\tau_\Lambda M_s) = 0.$ But the Auslander-Reiten quivers
show that $\tau_\Lambda M_s = \tau_{\mathcal F}M_s,$ therefore we see that $\Ext^1_{\mathcal F}(M_s,M_t)= 0.$ 
\end{proof} 

However, the map $M \mapsto |M|$ from the set of complete exceptional sequences to the
set of parking functions is not bijective! If $n = 3$, the exceptional sequences $(1,3,P)$ and $(3,1,P)$
both are sent to the parking function $(1,1,2)$. On the other hand, there is no exceptional sequence
$M$ with $|M| = (1,2,2).$

\end{note} 

\begin{note}\label{socle-of-exceptional-sequences} {\bf The socle of a complete exceptional
sequence.} For $\Lambda = \Lambda_n$, the Stanley bijection shows that  
the map $M \mapsto \soc M$ defined on the set of complete exceptional sequences is injective. 
For a general hereditary artin algebra $\Lambda$, this assertion is not true, 
already the quiver $\mathbb A_3$
with two sources provides counterexamples: Let $1$ be the sink and $2,2'$ the sources. Then
$$
 \bigl(S(1),S(2),S(2')\bigr),\ \ 
 \bigl(S(1),S(2'),S(2)\bigr),\ \ 
 \bigl(S(2),S(2'),I(1)\bigr),\ \ 
 \bigl(S(2'),S(2),I(1)\bigr)
$$
are complete exceptional sequences $(M_1,M_2,M_3)$ with $\soc(M_1,M_2.M_3) = (1,1,1).$ 
\end{note}

\begin{note}\label{soc-topp} {\bf The functions $\soc$ and $\topp$ on the set of complete 
exceptional sequences.}
Here we show explicitly the bijections between $\M(\NC(\mo\Lambda_n))$ and $\P(n)$
given by the functions $\soc$ and $\topp$:
$$
\hbox{\beginpicture
\setcoordinatesystem units <1cm,.37cm>
\put{$111$} at 0 0 

\put{$112$} at 0 -2
\put{$121$} at 0 -3
\put{$211$} at 0 -4

\put{$113$} at 0 -6
\put{$131$} at 0 -7
\put{$311$} at 0 -8 

\put{$122$} at 0 -10
\put{$212$} at 0 -11
\put{$221$} at 0 -12 

\put{$123$} at 0 -14 
\put{$132$} at 0 -15
\put{$213$} at 0 -16
\put{$231$} at 0 -17
\put{$312$} at 0 -18
\put{$321$} at 0 -19

\put{$\ 1\ P\ M$} at 3 0 

\put{$P\ M\, 2$} at 3 -2
\put{$P\, \, 2\ M$} at 3 -3
\put{$2\ \, 1\ M$} at 3 -4

\put{$1\ M\ 3$} at 3 -6
\put{$1\,\ 3\ \,P$} at 3 -7
\put{$3\,\ 1\ \,P$} at 3 -8 

\put{$M\ 2\ I$} at 3 -10
\put{$2\ M\ I$} at 3 -11
\put{$2\ \,\,I\,\ 1$} at 3 -12 

\put{$M\ I\ 3$} at 3 -14 
\put{$M\ 3\ 2$} at 3 -15
\put{$I\ \,1\,\ 3$} at 3 -16
\put{$I\ \,3\,\ 1$} at 3 -17
\put{$3\ P\ 2$} at 3 -18
\put{$3\ \,2\,\ 1$} at 3 -19

\put{$123$} at 6 0 

\put{$212$} at 6 -2
\put{$122$} at 6 -3
\put{$132$} at 6 -4

\put{$113$} at 6 -6
\put{$213$} at 6 -7
\put{$231$} at 6 -8 

\put{$121$} at 6 -10
\put{$112$} at 6 -11
\put{$312$} at 6 -12 

\put{$111$} at 6 -14 
\put{$211$} at 6 -15
\put{$131$} at 6 -16
\put{$311$} at 6 -17
\put{$221$} at 6 -18
\put{$321$} at 6 -19

\put{$\soc(M_1,M_2,M_3)$} at 0 2
\put{$(M_1,M_2,M_3)$} at 3 2
\put{$\topp(M_1,M_2,M_3)$} at 6 2

\put{$\P(n)$} at 0 3.5
\put{$\M(\NC(\mo\Lambda_n)$} at 3 3.5
\put{$\P(n)$} at 6 3.5
\endpicture}
$$

\end{note} 
	\medskip 

\begin{note}\label{algorithm-for-tilting-modules}
{\bf Algorithm to recover a tilting module $T$ from its top.}
Assume that $T = \bigoplus_r T_r$ is a (multiplicity-free) tilting module. 
For any $i$, let $m_i$ be the number of summands of $T$ with top $i$.
We claim that we can recover $T$ from the numbers $m_i$ with $1\le i \le n.$

By induction on $i$, the position of the modules with top $i$ can be determined
as follows: Let $I$ be the set of integers $1\le i < n$ such that $m_i > 0$.
In case $i\in I$, we denote by $T(i)$ the direct summand of
$T$ with top $i$ which has maximal length. Let us denote by $R_i$ the rectangle
which is the support of the hammock starting at $\tau^-T(i)$. 

Assume we have determined the position of the modules with top $i' < i,$
thus we know the rectangles $R_{i'}$ with $i'\in I$ and $i'<i.$ Let $\mathcal Z(i)$
be the modules with top $i$ which belong to at least one of the rectangles 
$R_{i'}$ with $i'\in I$ and $i'<i.$
Then the direct summands of $T$ with top $i$ are the $m_i$ modules of smallest
possible length which belong to the coray $\mathcal C_i$ and not to $\mathcal Z(i)$
(the coray $\mathcal C_i$ consists of the modules with top $S(i)$, this is
just the southeast diagonal ending in $S(i)$).
(Namely, a direct summand $T'$ of $T$ with top $i$ cannot belong to a rectangle $R_{i'}$,
since otherwise $\Ext^1(T',T(i)) \neq 0.$ On the other hand, by assumption, there are
precisely $m_i$ direct summands of $T$ with top $i$; in particular, there
is one such module. Recall that $T(i)$ is the direct summand of
$T$ with top $i$ of largest length. Since $\Ext^1(T,T(i)) = 0$, also $\Ext^1(T,Y) = 0$
for any factor module $Y$ of $T(i)$. Let $Y$ be a factor module of $T(i)$ which does not
belong to $\mathcal Z(i)$. Then clearly $T \oplus Y$ has no
self-extensions, thus $Y$ has to be a direct summand of $T$. This shows that all the non-zero
factor modules of $T(i)$ which do not belong to $\mathcal Z(i)$ are direct summands of $T$.)

Let us look at the start of the induction, thus let $i$ be minimal with $m_i > 0$. Then
$\mathcal Z(i) = \emptyset$, thus the direct summands of $T$ with top $i$ are just the modules
with top $i$ and length at most $m_i$. 
	\medskip 

{\bf Example.} Let $T$ be a tilting module with top sequence $(3,3,4,4,6,8,8,8)$. Recall that $m_i$ 
denotes the number of indecomposable direct summands of $T$ with top $i$,
thus $(m_1,\dots,m_8) = (0,0,2,2,0,1,0,3)$. The $m_i$ 
direct summands of $T$ with top $i$ lie in the
coray $\mathcal C_i$. 
We have marked by a circle the $m_i$ modules with top $i$ of smallest possible length, an
arrow (and the number $m_i$) points to the coray $\mathcal C_i$. The aim is to shift
the vertices inside the coray $\mathcal C_i$ so that we obtain a tilting module. 
$$
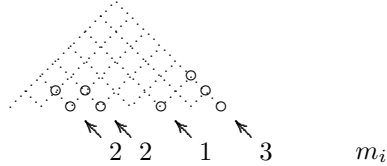
\hbox{\beginpicture
\setcoordinatesystem units <.2cm,.2cm>
\put{\beginpicture
\multiput{} at 0 0 16 6 /
\setdots <1mm>
\plot 0 0  7 7  14 0 /
\plot 1 1  2 0  8 6 /
\plot 2 2  4 0  9 5 /
\plot 3 3  6 0  10 4 /
\plot 4 4  8 0  11 3 /
\plot 5 5  10 0  12 2 /
\plot 6 6  12 0  13 1 /
\multiput{$\circ$} at 3 1  4 0  
              5 1  6 0 
              10  0 
              12 2  13 1  14 0
              /
\setsolid
\put{$m_i$} at 24 -3.2 

\arr{6 -2}{5 -1}
\put{$2$} at 7 -3         
\arr{8 -2}{7 -1}
\put{$2$} at 9 -3         
\arr{12 -2}{11 -1}
\put{$1$} at 13 -3         
\arr{16 -2}{15 -1}
\put{$3$} at 17 -3
\endpicture} at 0 2 
\endpicture}
$$

First, we consider the coray $\mathcal C_3$, since $3$ is the smallest index $i$ with $m_i > 0.$
As we know, the direct summands of $T$ in this coray have to be those of smallest possible
length, thus of length $1$ and $2$.
In particular, $T(3)$ has length $2$. We shade the rectangle $R_3$ starting at $\tau^-T(3)$, see
the first of the following pictures.
Always, the vertices marked by a bullet are the direct summands of $T$ which are already determined
(and arrows point to the corays which have to be considered later).
The vertices marked 
by a star are the modules $\tau^-T(i)$ with $i\in I$, 
the corresponding rectangles $R_i$ starting at such vertices are shaded. 
$$
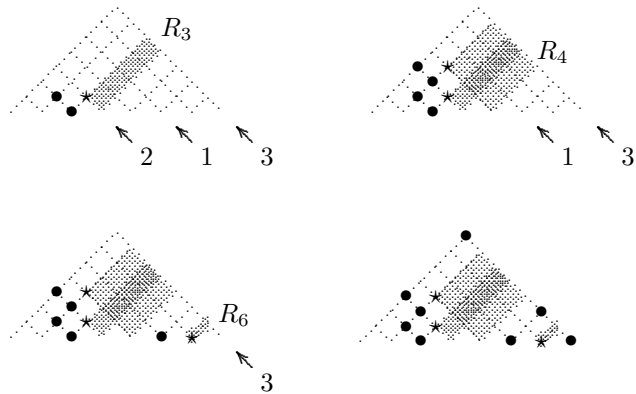
\hbox{\beginpicture
\setcoordinatesystem units <.2cm,.2cm>
\put{\beginpicture
\multiput{} at 0 0 16 8 /
\setdots <1mm>
\plot 0 0  7 7  14 0 /
\plot 1 1  2 0  8 6 /
\plot 2 2  4 0  9 5 /
\plot 3 3  6 0  10 4 /
\plot 4 4  8 0  11 3 /
\plot 5 5  10 0  12 2 /
\plot 6 6  12 0  13 1 /

\multiput{$\bullet$} at 3 1  4 0  / 
\multiput{$\star$} at 5 1 /
\setshadegrid span <.3mm>
\vshade 5 1 1 <z,z,,> 6 0 2 <z,z,,> 9 3 5  <z,z,,> 10 4 4 /
\setsolid
\arr{8 -2}{7 -1}
\put{$2$} at 9 -3         
\arr{12 -2}{11 -1}
\put{$1$} at 13 -3         
\arr{16 -2}{15 -1}
\put{$3$} at 17 -3

\put{$R_3$} at 11 5.5        
\endpicture} at -30 -10
\put{\beginpicture
\multiput{} at 0 0 16 8 /
\setdots <1mm>
\plot 0 0  7 7  14 0 /
\plot 1 1  2 0  8 6 /
\plot 2 2  4 0  9 5 /
\plot 3 3  6 0  10 4 /
\plot 4 4  8 0  11 3 /
\plot 5 5  10 0  12 2 /
\plot 6 6  12 0  13 1 /

\multiput{$\bullet$} at 3 1  4 0  
   3 3  4 2   / 
\multiput{$\star$} at 5 1  5 3 /
\setshadegrid span <.3mm>
\vshade 5 1 1 <z,z,,> 6 0 2 <z,z,,> 9 3 5  <z,z,,> 10 4 4 /
\setshadegrid span <.35mm>
\vshade 5 3 3  <z,z,,> 8 0 6 <z,z,,>  11 3 3 /
\setsolid
\arr{12 -2}{11 -1}
\put{$1$} at 13 -3         
\arr{16 -2}{15 -1}
\put{$3$} at 17 -3

\put{$R_4$} at 12 4
\endpicture} at -6 -10
\put{\beginpicture
\multiput{} at 0 0 16 8 /
\setdots <1mm>
\plot 0 0  7 7  14 0 /
\plot 1 1  2 0  8 6 /
\plot 2 2  4 0  9 5 /
\plot 3 3  6 0  10 4 /
\plot 4 4  8 0  11 3 /
\plot 5 5  10 0  12 2 /
\plot 6 6  12 0  13 1 /

\multiput{$\bullet$} at 3 1  4 0  
   3 3  4 2    
   10 0 / 
\multiput{$\star$} at 5 1  5 3  12 0  /
\setshadegrid span <.3mm>
\vshade 5 1 1 <z,z,,> 6 0 2 <z,z,,> 9 3 5  <z,z,,> 10 4 4 /
\vshade 11.8 -.5 0.5 <z,z,,> 13.2 0.5 1.5 /

\setshadegrid span <.35mm>
\vshade 5 3 3  <z,z,,> 8 0 6 <z,z,,>  11 3 3 /
\setsolid
\arr{16 -2}{15 -1}
\put{$3$} at 17 -3

\put{$R_6$} at 14.8 1.5          
\endpicture} at -30 -25
\put{\beginpicture
\multiput{} at 0 0 16 8 /
\setdots <1mm>
\plot 0 0  7 7  14 0 /
\plot 1 1  2 0  8 6 /
\plot 2 2  4 0  9 5 /
\plot 3 3  6 0  10 4 /
\plot 4 4  8 0  11 3 /
\plot 5 5  10 0  12 2 /
\plot 6 6  12 0  13 1 /

\multiput{$\bullet$} at 3 1  4 0  
   3 3  4 2    
        10 0  
   7 7  12 2  14 0  / 
\multiput{$\star$} at 5 1  5 3   12 0 /
\setshadegrid span <.3mm>
\vshade 5 1 1 <z,z,,> 6 0 2 <z,z,,> 9 3 5  <z,z,,> 10 4 4 /
\vshade 11.8 -.5 0.5 <z,z,,> 13.2 0.5 1.5 /
\setshadegrid span <.35mm>
\vshade 5 3 3  <z,z,,> 8 0 6 <z,z,,>  11 3 3 /
\put{} at 19 -3         
\endpicture} at -6 -25

\endpicture}
$$

In the second step, we have determined the direct summands of $T$ with top $4$.
The modules in $N(4)$
are the modules with top $4$ and length $1$ and $2$. There are $2$ remaining
modules on the coray $\mathcal C_4$. Since $m_4 = 2$, we see that the remaining
modules on the coray $\mathcal C_4$ all are direct summands of $T$. It follows that 
$T(4)$ has length $4$.
We have shaded the rectangle $R_4$ (starting at $\tau^-T(4)$).

The third step concerns the direct summands of $T$ with top $6$. Here,
$N(6)$ consists of the modules with top $6$ and length between $2$ and $5$. It follows
that $T(6)$ has length $1$ and we shade the rectangle $R_6$ (note that it consists just of two modules).

The final step deals with the modules with top $8$. Now $N(8)$ consists of the
modules with top $8$ and length $2,4,5,6,7$; the remaining three modules with top  $8$ have
to be the remaining summands of $T$; they have length $1, 3, 8$. 
	\medskip 

Note: we have considered a tilting module with top sequence $(3,3,4,4,6,8,8,8)$.
If we renumber the simple
modules using the permutation $\pi$, we see that we deal with the 
non-increasing parking function $(6,6,5,5,3,1,1,1)$. 
	\bigskip

{\bf How to find the factor complement for a normal partial tilting module?}
Let $N$ be a normal partial tilting module. We can assume in addition that $N$ is sincere, 
thus that $P(n)$ is a direct summand (namely: a sincere partial tilting module is faithful.
and any faithful module has $P(n)$ as direct summand).

To say that $N$ is normal means that $N$ has 
at most one direct summand, say $N(i)$, from the coray $\mathcal C_i$. Let $I$
be the set of indices $1\le i < n$ such that $i$ occurs in the top of $N$, thus 
$$
 N = P(n)\oplus \bigoplus_{i\in I} N(i).
$$
For $i\in I$, we define $R_i$ as the rectangle starting in $\tau^-N(i),$ and we denote by $R$
the union of the sets $R_i$. 

The factor complement for $N$ consists of all the proper non-zero factor modules of the
modules $N(i)$ which do not belong to $R$.
	\medskip

{\bf Example.} Consider $\Lambda_8$.
Let $N$ be the direct sum of the modules $[2]3,\ [4]4,\ 6,\ [8]8.$
Here is the Auslander-Reiten quiver.
On the left, we mark the modules $N(i)$ by bullets. The arrows below
point to the corresponding corays.
$$
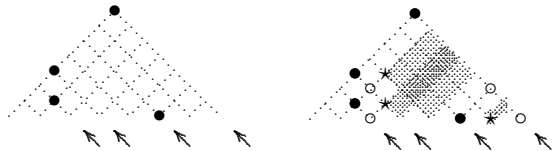
\hbox{\beginpicture
\setcoordinatesystem units <.2cm,.2cm>
\put{\beginpicture
\setdots <1mm>
\multiput{} at 0 -3 16 9 /
\plot 0 0  7 7  14 0 /
\plot 1 1  2 0  8 6 /
\plot 2 2  4 0  9 5 /
\plot 3 3  6 0  10 4 /
\plot 4 4  8 0  11 3 /
\plot 5 5  10 0  12 2 /
\plot 6 6  12 0  13 1 /
\multiput{$\bullet$} at 
   3 1  3 3  10 0  7 7 /
\setsolid

\arr{6 -2}{5 -1}
\arr{8 -2}{7 -1}
\arr{12 -2}{11 -1}
\arr{16 -2}{15 -1}
\endpicture} at 0 0
\put{\beginpicture
\setcoordinatesystem units <.2cm,.2cm>
\multiput{} at 0 -1 16 7 /
\setdots <1mm>
\plot 0 0  7 7  14 0 /
\plot 1 1  2 0  8 6 /
\plot 2 2  4 0  9 5 /
\plot 3 3  6 0  10 4 /
\plot 4 4  8 0  11 3 /
\plot 5 5  10 0  12 2 /
\plot 6 6  12 0  13 1 /

\multiput{$\bullet$} at 3 1   
   3 3  10 0 
         7 7   / 
\multiput{$\circ$} at 
3 1  4 0  
   3 3  4 2    
        10 0  
   7 7  12 2  14 0  / 
\multiput{$\star$} at 5 1  5 3   12 0 /
\setshadegrid span <.3mm>
\vshade 5 1 1 <z,z,,> 6 0 2 <z,z,,> 9 3 5  <z,z,,> 10 4 4 /
\vshade 11.8 -.5 0.5 <z,z,,> 13.2 0.5 1.5 /
\setshadegrid span <.35mm>
\vshade 5 3 3  <z,z,,> 8 0 6 <z,z,,>  11 3 3 /
\setsolid
\arr{6 -2}{5 -1}
\arr{8 -2}{7 -1}
\arr{12 -2}{11 -1}
\arr{16 -2}{15 -1}
\endpicture} at 20 0
\endpicture}
$$
On the right, we have marked the modules $\tau^-N(i)$ with $i\in I$ by
a star $*$ and we have shaded the corresponding rectangles $R_i$. By definition,
$R$ is the union of these rectangles. 

The circles show the position of the direct summands of the factor complement:
these are the proper non-zero factor modules of the modules $N(i)$ which do not belong to
$R$.
Note that the bullets and the circles together provide a tilting module (the unique
tilting module with normalization $N$). 
	\bigskip\bigskip

Let us draw the corresponding non-crossing partition. We start with the module
$$
 N = [2]3\oplus [4]4\oplus 6\oplus [8]8,
$$ 
the corresponding antichain is
$$
\hbox{\beginpicture
\setcoordinatesystem units <1cm,1cm>
\put{$A$} [r] at -0.2 0
\put{$= \left\{ \ \Delta(3),\ \Delta(4),\ \Delta(6),\ \Delta(8)\ \right\}
 = \left\{ \ [2]3,\ [4]4,\ 6,\ [4]8\ \right\}
 = \left\{\ \smallmatrix 3 \cr 2 \endsmallmatrix,\ 
      \smallmatrix 4 \cr 3 \cr 2\cr 1 \endsmallmatrix,\ 
      \smallmatrix 6 \endsmallmatrix,\ 
      \smallmatrix 8 \cr 7 \cr 6 \cr 5  \endsmallmatrix\ \right\}$} [l] at 0 0
\put{$= \left\{ \ [2,4],\  [1,5],\ [6,7],\ [5,9]\ \right\}$} [l] at 0 -1 
\endpicture}
$$
thus, the non-crossing partition is as follows: 
$$
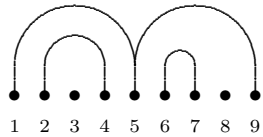
\hbox{\beginpicture
\setcoordinatesystem units <.4cm,.4cm>
\multiput{$\bullet$} at 0 0  1 0  2 0  3 0  4 0  5 0  6 0  7 0  8 0  /
\put{} at 0 3
\plot 0 0  0 1 /
\plot 1 0  1 1 /
\plot 3 0  3 1 /
\plot 4 0  4 1 /
\plot 5 0  5 1 /
\plot 6 0  6 1 /
\plot 8 0  8 1 /
\circulararc 180 degrees from 4 1 center at 2 1
\circulararc 180 degrees from 8 1 center at 6 1
\circulararc 180 degrees from 3 1 center at 2 1
\circulararc 180 degrees from 6 1 center at 5.5 1
\put{$\ssize 1$} at 0 -1
\put{$\ssize 2$} at 1 -1
\put{$\ssize 3$} at 2 -1
\put{$\ssize 4$} at 3 -1
\put{$\ssize 5$} at 4 -1
\put{$\ssize 6$} at 5 -1
\put{$\ssize 7$} at 6 -1
\put{$\ssize 8$} at 7 -1
\put{$\ssize 9$} at 8 -1
\endpicture}
$$
\end{note}

\begin{note}\label{non-nesting}
{\bf Antichains in $\Phi_+$,}  where $\Phi$ is a (finite) root system say of type $\Delta$.
In this last chapter we should have considered also
the set $\A(\Phi_+)$ of {\bf antichains}
in the root poset (the set of ``non-nesting partitions''). 

Let $\Lambda$ be a hereditary artin algebra of type $\Delta$. Then
	\smallskip

\Rahmen{$|\A(\Phi_+)| = |\A(\mo\Lambda)|,$}
	\smallskip

\noindent
thus, there are bijections between $|\A(\Phi_+)|$ and $|\A(\mo\Lambda)|$, but it 
is still an open problem to find a natural one, which takes into account that we even have:
	\smallskip

\Rahmen{$|\A_t(\Phi_+)| = |\A_t(\mo\Lambda)|$} 
	\bigskip 

The relationship between non-nesting and non-crossing is still a mystery.
\end{note}

\begin{note}\label{tilting-support-tilting}
{\bf The equality $\t(\mathbb A_n) = \t_{n+1}(\mathbb A_{n+1})$.} 
{\it There is a natural bijection between the tilting $\Lambda_{n+1}$-modules $T$
and the support-tilting $\Lambda_n$-modules $M$.} It is defined as follows:
	\medskip 

We denote by $\tau_{n+1}$ the Auslander-Reiten translation for $\mo\Lambda_{n+1}$.
Given $M$, let $I = \bigoplus_{i\in V} I(i)$, where $V$ is the set of vertices $1\le i \le n+1$
which do not belong to the support of $\tau_{n+1}^-M$ (and $I(i)$ is the injective module 
corresponding to the vertex $i$).
The Auslander-Reiten
formula for $\mo\Lambda_{n+1}$ asserts that $\Ext^1(I(i),M) = D\Hom(\tau_{n+1}^-M,I(S)) = 0$
for $i\in V$. Since $I$ is injective, we see that $\Ext^1(T\oplus I,T\oplus I) = 0.$ 
Let $s$ be the cardinality of the support of $M$. Then $V$ has cardinality $n+1-s$, thus 
$M\oplus I$ is the direct sum $t$ pairwise non-isomorphic indecomposable modules, thus $M\oplus I$
is a tilting $\Lambda_{n+1}$-module. 

Conversely, assume that $T$ is tilting $\Lambda_{n+1}$-module, let
$T_1,\dots,T_s$ be indecomposable non-injective direct summands of $T$. Then $M = \bigoplus_{i=1}^s$
is a support-tilting $\Lambda_n$-module. 

\end{note} 

\vfill\eject

\section{\bf Appendix}

In the appendix we will try to outline in which way the classical 
Catalan combinatorics could be seen
as the heart of the theory of finite sets, starting with the subsets of
cardinality two. Given a finite set $C$, we are going to find relations between
subsets, quotient sets and automorphisms of $C$. We follow considerations of Knuth \cite{[Kn]},
Biane \cite{[Bi]} and Armstrong \cite{[Ag]}. We start with an overview over some typical 
Catalan problems as discussed by Stanley \cite{[S1],[S2]}. 
	\bigskip

\subsection{What is Catalan combinatorics? A first answer} 
     
Catalan combinatorics is usually considered just as collecting a wealth 
of counting problems which are defined for all natural numbers $n$ such that
the answer $f(n)$ is  the Catalan number $f(n) = \frac1{n+1}\binom{2n}n$, or,
equivalently, such that $f(0) = 1$ and such that for $n\ge 1$ the recursion formula
$$
 f(n+1) = \sum_{t=0}^n f(t)f(n-t)
$$
is satisfied. These problems
concern finite sets with some decoration, with some additional structures. 
There are the famous collections by Stanley \cite{[S1],[S2]} which ask for 
direct enumerations and for providing bijections. The relationship between
any two of these problems is of varied nature: often only small changes are needed
to transform one problem into the other, but sometimes it takes some effort
to provide reasonable bijections. The literature is full of such examples, 
but no systematic treatment of the Catalan problems seems to be available.
The most important examples of Catalan problems are the non-crossing and the
non-nesting partitions. 	
	\medskip

\subsubsection{\bf Partitions of a finite set.}  
Given a partition $P$ of the set $\{1,2,\dots,n\}$, we consider its set $A(P)$ of arcs:
An {\it arc} of $P$ is a pair $(i,j)$ with $i < j$ which belong to the same part such that 
no element $x$ with $a < x < b$ 
belongs to this part (alternatively, we could use the cyclic order on the
set $\{1,2,\dots,n\}$ and convex polygons). We may consider the arcs as elements of the
positive root set $\Phi_+(\mathbb A_{n-1})$, thus the function $A$ provides an
embedding of the set $\Quot(n)$ of all partitions of $\{1,2,\dots,n\}$ into the set 
$\Sub(\Phi_+(\mathbb A_{n-1}))$:
	\smallskip 

\Rahmen{$\Quot(n)\ \subseteq \ \Sub(\Phi_+(\mathbb A_{n-1}))$.}
	\bigskip 

\subsubsection{\bf Non-crossing partitions.}
Let us start with the non-crossing partitions and with counting problems 
which are directly related. 
Recall that a partition of the set $\{1,2,\dots,n\}$
is said to be {\it non-crossing,} 
provided given elements $i < i' < j < j'$ with $i,j$ in the same part 
and $i',j'$ in the same part, then all elements belong to the same part. 
As we have mentioned in Chapter 4, we visualize a partition by drawing 
the vertices of the set $\{1,2,\dots,n\}$ in the natural order, 
as well as all the arcs (alternatively, we could use the cyclic order on the
set $\{1,2,\dots,n\}$ and convex polygons). 

As we have mentioned, the set $A(P)$ of arcs of a partition $P$ can be considered 
as a subset of $\Phi_+$,
thus, we may interpret the arcs as indecomposable $\Lambda_{n-1}$. We saw already
in Proposition \ref{prop-partitions-arc}, that a partition $P$ is non-crossing if and only if 
$A(P)$ is an {\bf antichains in $\mo \Lambda_{n-1}$}.
	\medskip 

In Proposition \ref{binary} 
we have outlined how to attach to a non-crossing partition $P$ a {\bf binary
tree} $B(P)$. 
A binary tree is called a {\bf complete binary trees,} provided any vertex has
either two or none successors. There is an obvious bijection between the
binary trees $B$ with $n\ge 1$ vertices and the complete binary trees $\widetilde B$ with $2n+1$
vertices: if $B$ is a binary tree with $n\ge 1$ vertices, then add to every vertex $x$ with $t \le 1$
successors $2-t$ successors, in order to obtain $\widetilde B$. Conversely, if
$\widetilde B$ is a complete binary tree with more than one vertices, 
delete all its leaves, in order to recover $B$.
	
Starting with a complete binary tree, label the leaves by pairwise different letters $a,b,c,\dots$,
going from left to right. The binary tree structure provides {\bf parentheses} for the word $abc\cdots$.
	\medskip

A permutation $\pi$ in $S_n$ is said to be a 
{\bf $231$-avoiding permutation} provided there are no numbers $a < b < c$
in $\{1,2,\dots,n\}$ with $\pi(b) < \pi(c) < \pi(a)$. These 
permutations are also said to be {\bf stack-sortable,} since such permutations
can be achieved by a computer using a single stack; thus, here, we are in the realm of 
computer science! 
The problem of sorting an input sequence using a single stack was first posed by Knuth (1968).
Given a binary tree $B$ with $n$ vertices, the corresponding stack-sortable permutation $\pi_B$
is defined inductively as follows: endow $B$ with its intrinsic numbering and 
if $B = (L,s,R)$, then $\pi_B(1) = s$, the values of $2,\dots,s$ under $\pi_B$
are $\pi_L(1),\dots,\pi_L(s-1)$, in this order, and the values of $s+1,\dots,n$ under $\pi_B$
are $\pi_R(1)+s,\dots,\pi_R(n-s)+s$, in this order. 
	\medskip

Here are various incarnations of the {\bf non-crossing partitions,} 
starting with the arc diagram, the antichains in the Auslander-Reiten quiver of 
$\mo\Lambda_{n-1}$, as well as the corresponding binary tree. We show the case
$n = 3$, as well as the
only nesting partition for $n = 4$. The leaves of the complete binary trees are 
marked as stars $\star$.

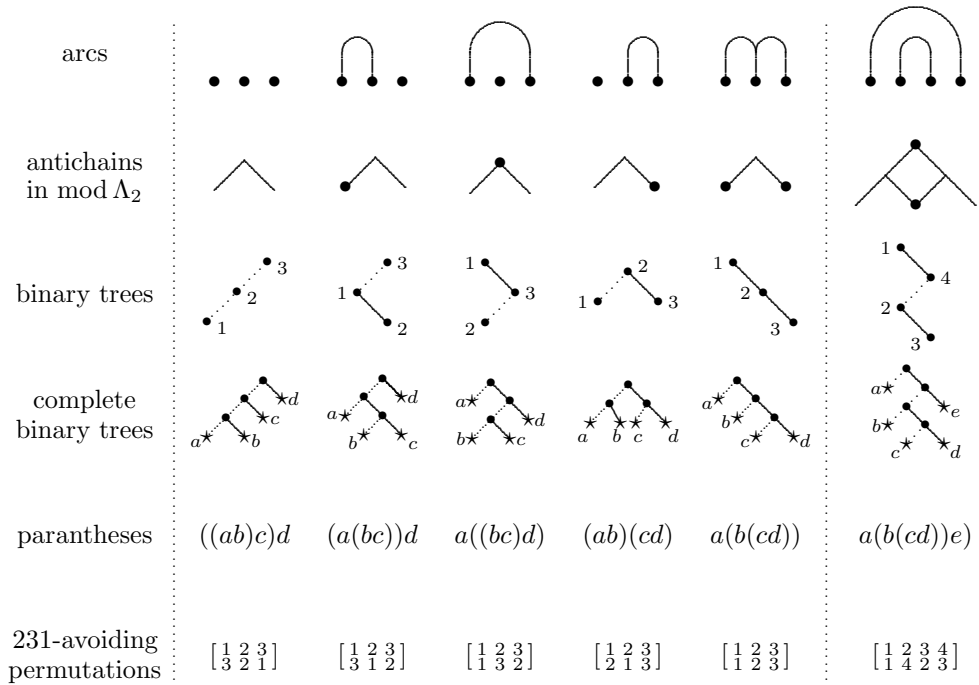
$$
\hbox{\beginpicture
\setcoordinatesystem units <.85cm,.8cm>
\put{arcs} at -2.5 0 
\put{\beginpicture
\setcoordinatesystem units <.4cm,.4cm>
\multiput{} at 0 0  2 2 /
\multiput{$\bullet$} at 0 0  1 0  2 0 /
\endpicture} at 0 0
\put{\beginpicture
\setcoordinatesystem units <.4cm,.4cm>
\multiput{} at 0 0  2 2 /
\multiput{$\bullet$} at 0 0  1 0  2 0 /
\plot 0 0  0 1 /
\plot 1 0  1 1 /
\circulararc 180 degrees from 1 1  center at 0.5 1 
\endpicture} at 2 0
\put{\beginpicture
\setcoordinatesystem units <.4cm,.4cm>
\multiput{} at 0 0  2 2 /
\multiput{$\bullet$} at 0 0  1 0  2 0 /
\plot 0 0  0 1 /
\plot 2 0  2 1 /
\circulararc 180 degrees from 2 1  center at 1 1 
\endpicture} at 4 0
\put{\beginpicture
\setcoordinatesystem units <.4cm,.4cm>
\multiput{} at 0 0  2 2 /
\multiput{$\bullet$} at 0 0  1 0  2 0 /
\plot 1 0  1 1 /
\plot 2 0  2 1 /
\circulararc 180 degrees from 2 1  center at 1.5 1 
\endpicture} at 6 0
\put{\beginpicture
\setcoordinatesystem units <.4cm,.4cm>
\multiput{} at 0 0  2 2 /
\multiput{$\bullet$} at 0 0  1 0  2 0 /
\plot 0 0  0 1 /
\plot 1 0  1 1 /
\plot 2 0  2 1 /
\circulararc 180 degrees from 1 1  center at 0.5 1 
\circulararc 180 degrees from 2 1  center at 1.5 1 
\endpicture} at 8 0
\put{\beginpicture
\setcoordinatesystem units <.4cm,.4cm>
\multiput{} at 0 0  3 2 /
\multiput{$\bullet$} at 0 0  1 0  2 0  3 0 /
\plot 0 0  0 1 /
\plot 1 0  1 1 /
\plot 2 0  2 1 /
\plot 3 0  3 1 /
\circulararc 180 degrees from 2 1  center at 1.5 1 
\circulararc 180 degrees from 3 1  center at 1.5 1 
\endpicture} at 10.5 0
\put{antichains} at -2.5 -1.75
\put{in $\mo\Lambda_2$} at -2.5 -2.25
\put{\beginpicture
\setcoordinatesystem units <.4cm,.4cm>
\multiput{} at 0 0  2 1 /
\plot 0 0  1 1  2 0 /
\multiput{$\bullet$} at   /
\endpicture} at 0 -2

\put{\beginpicture
\setcoordinatesystem units <.4cm,.4cm>
\multiput{} at 0 0  2 1 /
\plot 0 0  1 1  2 0 /
\multiput{$\bullet$} at 0 0  /
\endpicture} at 2 -2

\put{\beginpicture
\setcoordinatesystem units <.4cm,.4cm>
\multiput{} at 0 0  2 1 /
\plot 0 0  1 1  2 0 /
\multiput{$\bullet$} at  1 1 /
\endpicture} at 4 -2

\put{\beginpicture
\setcoordinatesystem units <.4cm,.4cm>
\multiput{} at 0 0  2 1 /
\plot 0 0  1 1  2 0 /
\multiput{$\bullet$} at  2 0 /
\endpicture} at 6 -2

\put{\beginpicture
\setcoordinatesystem units <.4cm,.4cm>
\multiput{} at 0 0  2 1 /
\plot 0 0  1 1  2 0 /
\multiput{$\bullet$} at  0 0  2 0  /
\endpicture} at 8 -2
\put{\beginpicture
\setcoordinatesystem units <.4cm,.4cm>
\multiput{} at 0 0  4 1 /
\plot 0 0  2 2  4 0 /
\plot 1 1  2 0  3 1 /
\multiput{$\bullet$} at  2 0   2 2  /
\endpicture} at 10.5 -2

\put{binary trees} at -2.5 -4
\put{\beginpicture
\setcoordinatesystem units <.4cm,.4cm>
\multiput{} at 0 0  2 2 /
\multiput{$\ssize \bullet$} at 2 2  1 1  0 0  /
\setdots <1mm>
\plot 0 0  2 2 / 
\put{$\ssize 3$} at 2.5 1.8
\put{$\ssize 2$} at 1.5 0.8
\put{$\ssize 1$} at  .5 -.2
\endpicture} at 0 -4
\put{\beginpicture
\setcoordinatesystem units <.4cm,.4cm>
\multiput{} at 0 0  1 2 /
\multiput{$\ssize \bullet$} at 1 0  0 1  1 2  /
\plot 0 1  1 0 /
\setdots <1mm>
\plot 1 2  0 1 / 
\put{$\ssize 3$} at 1.5 2
\put{$\ssize 1$} at -.5 1
\put{$\ssize 2$} at  1.5 -.2
\endpicture} at 2 -4
\put{\beginpicture
\setcoordinatesystem units <.4cm,.4cm>
\multiput{} at 0 0  1 2 /
\multiput{$\ssize \bullet$} at 0 2  1 1  0 0  /
\plot 1 1  0 2 / 
\setdots <1mm>
\plot 1 1  0 0 / 
\put{$\ssize 1$} at -.5 2
\put{$\ssize 3$} at 1.5 1
\put{$\ssize 2$} at  -.5 -.2
\endpicture} at 4 -4
\put{\beginpicture
\setcoordinatesystem units <.4cm,.4cm>
\multiput{} at 0 0  2 2 /
\plot 1 2  2 1 / 

\multiput{$\ssize \bullet$} at 0 1  1 2  2 1  /
\setdots <1mm>
\plot 1 2  0 1 / 
\put{$\ssize 2$} at 1.5 2.2
\put{$\ssize 3$} at 2.5 1
\put{$\ssize 1$} at -.5 1
\endpicture} at 6 -4
\put{\beginpicture
\setcoordinatesystem units <.4cm,.4cm>
\multiput{} at 0 0  2 2 /
\multiput{$\ssize \bullet$} at 2 0  1 1  0 2  /
\plot 0 2  2 0 / 
\put{$\ssize 1$} at -.5 2
\put{$\ssize 2$} at  .4 1
\put{$\ssize 3$} at  1.4 -.2
\endpicture} at 8 -4
\put{\beginpicture
\setcoordinatesystem units <.4cm,.4cm>
\multiput{} at 0 0  1 3 /
\multiput{$\ssize \bullet$} at 1 0  0 1  1 2  0 3  /
\setsolid
\plot 0 3  1 2 /
\plot 0 1  1 0 /
\setdots <1mm>
\plot 0 1  1 2 / 
\put{$\ssize 1$} at -.5 3
\put{$\ssize 4$} at 1.5 2
\put{$\ssize 2$} at  -.5 1
\put{$\ssize 3$} at  .5 -.2
\endpicture} at 10.5 -4

\put{complete} at -2.5 -5.75
\put{binary trees} at -2.5 -6.25
\put{\beginpicture
\setcoordinatesystem units <.25cm,.25cm>
\multiput{} at 0 0  2 2 /
\multiput{$\ssize \bullet$} at 2 2  1 1  0 0 /
\multiput{$\star$} at  -1 -1  1 -1  2 0  3 1  /
\plot 0 0  1 -1 /
\plot 1 1  2 0 /
\plot 2 2  3 1 /
\setdots <.5mm>
\plot -1 -1  2 2 / 
\put{$\ssize a$\strut} at -1.6 -1.2
\put{$\ssize b$\strut} at  1.6 -1.2
\put{$\ssize c$\strut} at 2.6 0
\put{$\ssize d$\strut} at 3.6 1
\endpicture} at 0 -6
\put{\beginpicture
\setcoordinatesystem units <.25cm,.25cm>
\multiput{} at 0 0  1 2 /
\multiput{$\ssize \bullet$} at 1 0  0 1  1 2 /
\multiput{$\star$} at  -1 0  0 -1  2 -1  2 1 /
\plot 0 1  2 -1 /
\plot 1 2  2 1 /
\plot 1 2  2 1 /
\setdots <.5mm>
\plot 1 2  -1 0 /
\plot 0 -1  1 0 /
\put{$\ssize a$\strut} at -1.7 0.2
\put{$\ssize b$\strut} at -.6 -1.4
\put{$\ssize c$\strut} at 2.6 -1.4
\put{$\ssize d$\strut} at 2.6 1
\endpicture} at 2 -6
\put{\beginpicture
\setcoordinatesystem units <.25cm,.25cm>
\multiput{} at 0 0  1 2 /
\multiput{$\ssize \bullet$} at 0 2  1 1  0 0 /
\multiput{$\star$} at  -1 -1  -1 1  1 -1  2 0 /
\plot 2 0  0 2 / 
\plot 0 0  1 -1 /
\setdots <.5mm>
\plot 1 1  -1 -1 / 
\plot 0 2  -1 1 /
\put{$\ssize a$\strut} at -1.6 1
\put{$\ssize b$\strut} at -1.6 -1
\put{$\ssize c$\strut} at  1.6 -1
\put{$\ssize d$\strut} at 2.6  0
\endpicture} at 4 -6
\put{\beginpicture
\setcoordinatesystem units <.25cm,.25cm>
\multiput{} at 0 0  2 2 /
\plot 1 2  3 0 / 
\plot 0 1  0.6 0 / 
\multiput{$\ssize \bullet$} at 0 1  1 2  2 1  /
\multiput{$\star$} at -1 0 .6 0  1.4 0  3 0 /  
\setdots <.5mm>
\plot 1 2  -1 0 / 
\plot 2 1  1.4 0 /
\put{$\ssize a$\strut} at -1.4 -.7
\put{$\ssize b$\strut} at  .4 -.7
\put{$\ssize c$\strut} at 1.5 -.7
\put{$\ssize d$\strut} at 3.4 -.7
\endpicture} at 6 -6
\put{\beginpicture
\setcoordinatesystem units <.25cm,.25cm>
\multiput{} at 0 0  2 2 /
\multiput{$\ssize \bullet$} at 2 0  1 1  0 2  /
\multiput{$\star$} at -1 1  0 0  1 -1  3 -1 /
\plot 0 2  3 -1 / 
\setdots <.5mm>
\plot 0 2  -1 1 /
\plot 1 1  0 0 /
\plot 2 0  1 -1 /
\put{$\ssize a$\strut} at -1.6 .8
\put{$\ssize b$\strut} at -.6 -.2
\put{$\ssize c$\strut} at .4 -1.2
\put{$\ssize d$\strut} at 3.6 -1.2
\endpicture} at 8 -6

\put{\beginpicture
\setcoordinatesystem units <.25cm,.25cm>
\multiput{} at 0 0  1 3 /
\multiput{$\ssize \bullet$} at 1 0  0 1  1 2  0 3  /
\multiput{$\star$} at -1 2  -1 0  0 -1  2 -1 2 1 /
\setsolid
\plot 0 3  2 1 /
\plot 0 1  2 -1 /
\setdots <1mm>
\plot 0 1  -1 0 /
\plot 0 1  1 2 / 
\plot 0 3  -1 2 /
\plot 1 0  0 -1 /
\put{$\ssize a$\strut} at -1.6 2
\put{$\ssize b$\strut} at -1.6 0
\put{$\ssize c$\strut} at -.6 -1.4
\put{$\ssize d$\strut} at 2.6 -1.4
\put{$\ssize e$\strut} at 2.6 0.8 
\endpicture} at 10.5 -6

\put{parantheses} at -2.5 -8
\put{$((ab)c)d$} at 0 -8
\put{$(a(bc))d$} at 2 -8
\put{$a((bc)d)$} at 4 -8
\put{$(ab)(cd)$} at 6 -8
\put{$a(b(cd))$} at 8 -8
\put{$a(b(cd))e)$} at 10.5 -8

\put{$231$-avoiding} at -2.5 -9.75
\put{permutations} at -2.5 -10.25
\put{\beginpicture
\setcoordinatesystem units <.4cm,.4cm>
\put{$\left[\smallmatrix 1 & 2 & 3 \cr
              3 & 2 & 1 \endsmallmatrix\right]$} at 0 0 
\endpicture} at 0 -10
\put{\beginpicture
\setcoordinatesystem units <.4cm,.4cm>
\put{$\left[\smallmatrix 1 & 2 & 3 \cr
              3 & 1 & 2 \endsmallmatrix\right]$} at 0 0 
\endpicture} at 2 -10
\put{\beginpicture
\setcoordinatesystem units <.4cm,.4cm>
\put{$\left[\smallmatrix 1 & 2 & 3 \cr
              1 & 3 & 2 \endsmallmatrix\right]$} at 0 0 
\endpicture} at 4 -10
\put{\beginpicture
\setcoordinatesystem units <.4cm,.4cm>
\put{$\left[\smallmatrix 1 & 2 & 3 \cr
              2 & 1 & 3 \endsmallmatrix\right]$} at 0 0 
\endpicture} at 6 -10
\put{\beginpicture
\setcoordinatesystem units <.4cm,.4cm>
\put{$\left[\smallmatrix 1 & 2 & 3 \cr
              1 & 2 & 3 \endsmallmatrix\right]$} at 0 0 
\endpicture} at 8 -10
\put{\beginpicture
\setcoordinatesystem units <.4cm,.4cm>
\put{$\left[\smallmatrix 1 & 2 & 3 & 4 \cr
              1 & 4 & 2 & 3  \endsmallmatrix\right]$} at 0 0 
\endpicture} at 10.5 -10

\setdots <1mm>
\plot -1.1 0.5  -1.1 -10.4 /
\plot 9.1 0.5  9.1 -10.4 /

\endpicture}
$$

The recursion formula $f(n+1) = \sum_{t=0}^n f(t)f(n-t)$ is most easily seen when looking at
the set of binary trees $B = (L,s,R)$. If $B$ is of cardinality $n+1$, then $L$ has to be
of cardinality $t$ with $0 \le t \le n$, this is the summation index.
If we fix $t$, then there are $f(t)$ possiblities for $L$
and $f(n-t)$ possibilities for $R$.
	\bigskip

\subsubsection{\bf Non-nesting partitions.}
A partition of the set $\{1,2,\dots,n\}$ is said to be {\it non-nesting,} 
provided given elements $i < i' < j < j'$ with $i,j'$ in the same part 
and $i',j$ in the same part, then all elements belong to the same part. 

Given a non-nesting partition, we consider again its  arc diagram. We recall
that the set $A(P)$ of arcs of a partition $P$ can be considered as a subset of $\Phi_+$.
Clearly, a partition $P$ is non-nesting if and only if 
$A(P)$ is an {\bf antichains in $\Phi_+$}.
	\bigskip

There is another model for the set of non-nesting partitions
of $\{1,2,\dots,n\}$, namely the set of {\bf Dyck paths} of length $2n$.
These are paths in the integral lattice $\mathbb Z^2$ from $(0,0)$ to $(2n,0)$, inside the
positive cone $\mathbb N^2$, using just
northeast and southeast arrows. Here is an example for $n = 6$.

$$
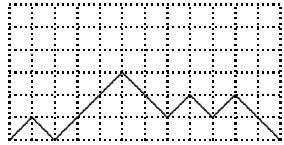
\hbox{\beginpicture
\setcoordinatesystem units <.3cm,.3cm>
\multiput{} at 0 0  12 6 /
\setdots <.7mm>
\setplotarea x from 0 to 12, y from 0 to 6
\grid {12} {6} 
\setsolid 
\plot 0 0  1 1  2 0  5 3   7 1  8 2  9 1  10 2  12 0 /
\endpicture}
$$  
We get a bijection between the Dyck paths of length $2n$ and the set of elements of
$\A(\Phi_+(\mathbb A_{n-1}))$ as follows: Attach to a Dyck path the set of its valleys:
this is an antichain in the root poset $\Phi_+(\mathbb A_{n-1})$.
For example, the valleys of the Dyck path shown above are the three vertices marked
by bullets:
$$
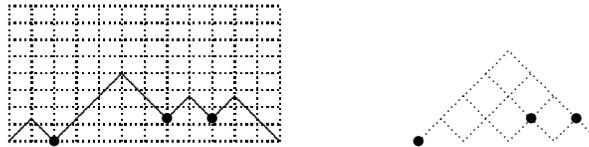
\hbox{\beginpicture
\setcoordinatesystem units <.3cm,.3cm>
\put{\beginpicture 
\multiput{} at 0 0  12 6 /
\setdots <.7mm>
\setplotarea x from 0 to 12, y from 0 to 6
\grid {12} {8} 
\setsolid 
\plot 0 0  1 1  2 0  5 3   7 1  8 2  9 1  10 2  12 0 /
\multiput{$\bullet$} at 2 0  7 1  9 1 /
\endpicture} at 0 0
\put{\beginpicture 
\multiput{} at 0 0  8 6 /
\setdots <.7mm>
\plot 0 0  4 4  8 0 /
\plot 1 1  2 0  5 3 /
\plot 2 2  4 0  6 2 /
\plot 3 3  6 0   7 1 /
\multiput{$\bullet$} at 0 0  5 1  7 1  /
\endpicture} at 16 0
\endpicture}
$$  
Clearly, {\it the set of valleys of a Dyck path is an antichain in $\Phi_+$ and any
antichain is obtained in this way.}
	
Dyck paths often are drawn as {\bf monotonic paths:}
as lattice paths from $(0,0)$ to $(n,n)$ using only east arrows and north
arrows of length $1$ and staying inside the triangle with vertices $(0,0), (n,0), (n,n)$.
The monotonic paths are obtained from the Dyck paths using a rotation by $225^\circ$
(or, alternatively, by a rotation by -45${}^\circ$ degrees followed by a vertical reflection). 
Monotonic paths can be encoded as words using the letters $X$ (for the east arrows) and $Y$
(for the north arrows). Thus, we deal with words using $n$ letters $X$, and $n$
letters $Y$, such that for any initial subword the number of letters $X$ is greater or
equal to the number of letters $Y$. These words are called {\bf Dyck words} of length $2n$.

Replacing in a Dyck word $X$ by an opening bracket ( and $Y$ by a closing bracket ),
we obtain $n$ pairs of {\bf parentheses} which are correctly matched.
Alternatively, we may use {\bf $(\pm 1)$-sequences}, these are sequences with $n$ entries
equal to $1$ and $n$ entries equal to $-1$. such that all (initial) partial sums are
non-negative.   

A permutation $\pi$ in $S_n$ is said to be a
{\bf $321$-avoiding permutation} provided there are no numbers $a < b < c$
in $\{1,2,\dots,n\}$ with $\pi(c) < \pi(b) < \pi(a)$. There is the following
bijection between monotone paths and $321$-avoiding permutations (see for example \cite{[CK], [Ro]}):
Let $\pi(1),\dots,\pi(n)$ be the value sequence of the permutation $\pi$. 
Define $y_1 = 0$. For $i \ge 2,$ let
$y_i = \pi(i)$ in case $i\ge 2$ and $p(i) < p(j)$ for some $j < i;$ otherwise let $y_i = y_{i-1}$.
The edges with endpoints $(i-1,y_i)$ and $(i,y_i)$ are the $X$-steps of the 
monotone path (the corresponding $Y$-steps are determined uniquely by knowing the $X$-steps).
Conversely, given a monotone path, the left-most $X$-steps in any row, say from 
$(i-1,y_i)$ and $(i,y_i)$, show the value $p(i) = y_i$. The remaining values $\pi(j)$ are
determined going downwards as the maximum of the numbers not yet used. 
	\medskip 

Let us exhibit
various incarnations of the {\bf non-nesting partitions,} for $n = 3$, as well as the
only crossing partition for $n = 4$. For the Dyck paths, we have encircled the valleys
(note that they correspond to the elements of the antichain).

$$
\hbox{\beginpicture
\setcoordinatesystem units <.85cm,.9cm>
\put{arcs} at -2.5 0 
\put{\beginpicture
\setcoordinatesystem units <.4cm,.4cm>
\multiput{} at 0 0  2 2 /
\multiput{$\bullet$} at 0 0  1 0  2 0 /
\endpicture} at 0 0
\put{\beginpicture
\setcoordinatesystem units <.4cm,.4cm>
\multiput{} at 0 0  2 2 /
\multiput{$\bullet$} at 0 0  1 0  2 0 /
\plot 0 0  0 1 /
\plot 1 0  1 1 /
\circulararc 180 degrees from 1 1  center at 0.5 1 
\endpicture} at 2 0
\put{\beginpicture
\setcoordinatesystem units <.4cm,.4cm>
\multiput{} at 0 0  2 2 /
\multiput{$\bullet$} at 0 0  1 0  2 0 /
\plot 0 0  0 1 /
\plot 2 0  2 1 /
\circulararc 180 degrees from 2 1  center at 1 1 
\endpicture} at 4 0
\put{\beginpicture
\setcoordinatesystem units <.4cm,.4cm>
\multiput{} at 0 0  2 2 /
\multiput{$\bullet$} at 0 0  1 0  2 0 /
\plot 1 0  1 1 /
\plot 2 0  2 1 /
\circulararc 180 degrees from 2 1  center at 1.5 1 
\endpicture} at 6 0
\put{\beginpicture
\setcoordinatesystem units <.4cm,.4cm>
\multiput{} at 0 0  2 2 /
\multiput{$\bullet$} at 0 0  1 0  2 0 /
\plot 0 0  0 1 /
\plot 1 0  1 1 /
\plot 2 0  2 1 /
\circulararc 180 degrees from 1 1  center at 0.5 1 
\circulararc 180 degrees from 2 1  center at 1.5 1 
\endpicture} at 8 0
\put{\beginpicture
\setcoordinatesystem units <.4cm,.4cm>
\multiput{} at 0 0  3 2 /
\multiput{$\bullet$} at 0 0  1 0  2 0  3 0 /
\plot 0 0  0 1 /
\plot 1 0  1 1 /
\plot 2 0  2 1 /
\plot 3 0  3 1 /
\circulararc 180 degrees from 2 1  center at 1 1 
\circulararc 180 degrees from 3 1  center at 2 1 
\endpicture} at 10.5 0
\put{antichains} at -2.5 -1.75
\put{in $\Phi_+$} at -2.5 -2.25
\put{\beginpicture
\setcoordinatesystem units <.4cm,.4cm>
\multiput{} at 0 0  2 1 /
\plot 0 0  1 1  2 0 /
\multiput{$\bullet$} at   /
\endpicture} at 0 -2

\put{\beginpicture
\setcoordinatesystem units <.4cm,.4cm>
\multiput{} at 0 0  2 1 /
\plot 0 0  1 1  2 0 /
\multiput{$\bullet$} at 0 0  /
\endpicture} at 2 -2

\put{\beginpicture
\setcoordinatesystem units <.4cm,.4cm>
\multiput{} at 0 0  2 1 /
\plot 0 0  1 1  2 0 /
\multiput{$\bullet$} at  1 1 /
\endpicture} at 4 -2

\put{\beginpicture
\setcoordinatesystem units <.4cm,.4cm>
\multiput{} at 0 0  2 1 /
\plot 0 0  1 1  2 0 /
\multiput{$\bullet$} at  2 0 /
\endpicture} at 6 -2

\put{\beginpicture
\setcoordinatesystem units <.4cm,.4cm>
\multiput{} at 0 0  2 1 /
\plot 0 0  1 1  2 0 /
\multiput{$\bullet$} at  0 0  2 0  /
\endpicture} at 8 -2
\put{\beginpicture
\setcoordinatesystem units <.4cm,.4cm>
\multiput{} at 0 0  4 1 /
\plot 0 0  2 2  4 0 /
\plot 1 1  2 0  3 1 /
\multiput{$\bullet$} at  1 1   3 1  /
\endpicture} at 10.5 -2

\put{Dyck paths} at -2.5 -4
\put{\beginpicture
\setcoordinatesystem units <.2cm,.2cm>
\multiput{} at 0 0  6 3 /
\multiput{$\ssize \bullet$} at 0 0  1 1  2 2  3 3  4 2  5 1  6 0   /
\multiput{$\bigcirc$} at /
\plot 0 0  3 3  6 0 /
\endpicture} at 0 -4
\put{\beginpicture
\setcoordinatesystem units <.2cm,.2cm>
\multiput{} at 0 0  6 3 /
\multiput{$\ssize \bullet$} at 0 0  1 1  2 0  3 1  4 2  5 1  6 0   /
\plot 0 0  1 1  2 0  4 2  6 0 /
\multiput{$\bigcirc$} at 2 0 /
\endpicture} at 2 -4
\put{\beginpicture
\setcoordinatesystem units <.2cm,.2cm>
\multiput{} at 0 0  6 3 /
\multiput{$\ssize \bullet$} at 0 0  1 1  2 2  3 1  4 2  5 1  6 0   /
\plot 0 0  2 2  3 1  4 2  6 0 /
\multiput{$\bigcirc$} at 3 1 /
\endpicture} at 4 -4
\put{\beginpicture
\setcoordinatesystem units <.2cm,.2cm>
\multiput{} at 0 0  6 3 /
\multiput{$\ssize \bullet$} at 0 0  1 1  2 2  3 1  4 0  5 1  6 0   /
\plot 0 0  2 2  4 0  5 1  6 0 /
\multiput{$\bigcirc$} at 4 0 /
\endpicture} at 6 -4
\put{\beginpicture
\setcoordinatesystem units <.2cm,.2cm>
\multiput{} at 0 0  6 3 /
\multiput{$\ssize \bullet$} at 0 0  1 1  2 0  3 1  4 0  5 1  6 0   /
\plot 0 0  1 1  2 0  3 1  4 0  5 1   6 0 /
\multiput{$\bigcirc$} at 2 0  4 0 /
\endpicture} at 8 -4
\put{\beginpicture
\setcoordinatesystem units <.2cm,.2cm>
\multiput{} at 0 0  8 3 /
\multiput{$\ssize \bullet$} at 0 0  1 1  2 2  3 1  4 2  5 1  6 2  7 1  8 0   /
\plot 0 0  2 2  3 1  4 2  5 1  6 2   8 0 /
\multiput{$\bigcirc$} at 3 1  5 1  /
\endpicture} at 10.5 -4
\put{monotone} at -2.5 -5.75
\put{paths} at -2.5 -6.25
\put{\beginpicture
\setcoordinatesystem units <.3cm,.3cm>
\multiput{} at 0 0  3 3 /
\multiput{$\ssize \bullet$} at 0 0  1 0  2 0  3 0  3 1  3 2  3 3   /
\plot 0 0  3 0  3 3 /
\endpicture} at 0 -6
\put{\beginpicture
\setcoordinatesystem units <.3cm,.3cm>
\multiput{} at 0 0  3 3 /
\multiput{$\ssize \bullet$} at 0 0  1 0  2 0  2 1  2 2  3 2  3 3   /
\plot 0 0  2 0  2 2  3 2  3 3 /
\endpicture} at 2 -6
\put{\beginpicture
\setcoordinatesystem units <.3cm,.3cm>
\multiput{} at 0 0  3 3 /
\multiput{$\ssize \bullet$} at 0 0  1 0  2 0  2 1  3 1  3 2  3 3   /
\plot 0 0  2 0  2 1  3 1  3 3 /
\endpicture} at 4 -6
\put{\beginpicture
\setcoordinatesystem units <.3cm,.3cm>
\multiput{} at 0 0  3 3 /
\multiput{$\ssize \bullet$} at 0 0  1 0  1 1  2 1  3 1   3 2  3 3   /
\plot 0 0  1 0  1 1  3 1  3 3 /
\endpicture} at 6 -6
\put{\beginpicture
\setcoordinatesystem units <.3cm,.3cm>
\multiput{} at 0 0  3 3 /
\multiput{$\ssize \bullet$} at 0 0  1 0  1 1  2 1  2 2  3 2  3 3   /
\plot 0 0  1 0  1 1  2 1  2 2  3 2  3 3 /
\endpicture} at 8 -6
\put{\beginpicture
\setcoordinatesystem units <.3cm,.3cm>
\multiput{} at 0 0  3 3 /
\multiput{$\ssize \bullet$} at 0 0  1 0  2 0  2 1  3 1  3 2  4 2  4 3  4 4   /
\plot 0 0  2 0  2 1  3 1  3 2  4 2  4 4   /
\endpicture} at 10.5 -6
\put{Dyck words} at -2.5 -8
\put{$\ss XXXYYY$} at 0 -8
\put{$\ss XXYYXY$} at 2 -8
\put{$\ss XXYXYY$} at 4 -8
\put{$\ss XYXXYY$} at 6 -8
\put{$\ss XYXYXY$} at 8 -8
\put{$\ss XXYXYXYY$} at 10.5 -8
\put{parentheses} at -2.5 -9
\put{$(\,(\,(\,)\,)\,)$} at 0 -9
\put{$(\,(\,)\,)\,(\,)$} at 2 -9
\put{$(\,(\,)\,(\,)\,)$} at 4 -9
\put{$(\,)\,(\,(\,)\,)$} at 6 -9
\put{$(\,)\,(\,)\,(\,)$} at 8 -9
\put{$(\,(\,)\,(\,)\,(\,)\,)$} at 10.5 -9

\put{$(\pm1)$-sequences} at -2.5 -9.6
\put{$\ssize(+++---)$} at 0 -9.6
\put{$\ssize(++--+-)$} at 2 -9.6
\put{$\ssize(++-+--)$} at 4 -9.6
\put{$\ssize(+-++--)$} at 6 -9.6
\put{$\ssize(+-+-+-)$} at 8 -9.6
\put{$\ssize(++-+-+--)$} at 10.5 -9.6

\put{$321$-avoiding} at -2.5 -10.25
\put{permutations} at -2.5 -10.75
\put{$\left[\smallmatrix 1 & 2 & 3 \cr
              1 & 2 & 3   \endsmallmatrix\right]$} at 0 -10.5 
\put{$\left[\smallmatrix 1 & 2 & 3 \cr
              1 & 3 & 2  \endsmallmatrix\right]$} at 2 -10.5
\put{$\left[\smallmatrix 1 & 2 & 3 \cr
              2 & 3 & 1  \endsmallmatrix\right]$} at 4 -10.5
\put{$\left[\smallmatrix 1 & 2 & 3 \cr
              2 & 1 & 3  \endsmallmatrix\right]$} at 6 -10.5
\put{$\left[\smallmatrix 1 & 2 & 3 \cr
              3 & 1 & 2  \endsmallmatrix\right]$} at 8 -10.5
\put{$\left[\smallmatrix 1 & 2 & 3 & 4 \cr
              3 & 4 & 1 & 2  \endsmallmatrix\right]$} at 10.5 -10.5

\setdots <1mm>
\plot -1.1 0.5  -1.1 -10.7 /
\plot 9.1 0.5  9.1 -10.7 /
\endpicture}
$$
	\medskip 

In order to verify again the recursion formula $f(n+1) = \sum_{t=0}^n f(t)f(n-t)$, 
consider for example the Dyck words. Any Dyck word $w$
of positive length can be written in a unique way in the form $Xw_1Yw_2$ where
$w_1,w_2$ are (possibly empty) Dyck words. For the Dyck words of length $2(n+1)$
use the summation index $t$, take as $w_1$ any  Dyck word of length $2t$ and
for $w_2$ any Dyck word of length $2(n-t).$
	\bigskip

\vfill\eject
\subsection{What is Catalan combinatorics? A second answer} 
Let us try to outline a general setting behind these Catalan problems.      
\medskip

\subsubsection{} 
The typical examples of Catalan problems concern finite sets with
some additional structure, but it seems to us that Catalan combinatorics 
should  be considered as a kind of heart of finite set theory in itself.
	 
It is the identification
\smallskip

\Rahmen{$\NC(n) \simeq \NC(S_n,c_n)$}
\smallskip

\noindent
which provides the essential hint
(and as we have seen, it also is the starting point for the general theory
when dealing with arbitrary Dynkin diagrams): it shows that the set $\NC(n)$ of
non-crossing partitions (as the typical Catalan object)
is an intrinsic invariant of the automorphism group $S_n$
of the finite set $C = \{1,2,\dots,n\}$. Of course, we need to specify not only
the set $C$ but in addition also a total
(or, at least, a cyclic) ordering of the set $C$ in order
to obtain the Coxeter element $c_n$, but actually, all the Coxeter elements $c$
are conjugate, and conjugations yield isomorphisms of the corresponding
subsets $\NC(S_n,c)$. Thus, we want to stress that
the identification $\NC(n) \simeq \NC(S_n,c_n)$ shows that this typical Catalan
object can be realized both as a subset of the set of all the partitions of $C$
(or, equivalently, of the set $\Quot(C)$ of all quotient sets of $C$), as 
well as as a subset of the automorphism group $\Aut(C)= S_n$ of $C$. Let us
expand these considerations by looking in more detail at the category of finite
sets, taking into account for any finite set $C$ also the set $\Sub(C)$ of all subsets
of $C$.

Modern mathematics considers mathematical structures as objects 
of a corresponding category. As many nice categories, the category of sets has the following two 
important properties: first, 
any morphism can be factored as an epimorphism followed by a monomorphism, and second,
a morphism which is both a monomorphism and an epimorphism is an isomorphism. 
In such a category $\mathcal C$, the basic information about any object is given by 
	\smallskip 

\begin{itemize}
\item the automorphism group $\Aut(C)$ of $C$,
\item the poset $\Sub(C)$ of all subobjects of $C$,
\item the poset $\Quot(C)$ of all quotient objects of $C$. 
\end{itemize}
	\smallskip 

\noindent 
Combining these data, one obtains further specifications about a given object $C$,
as well as about its relation to other objects $C'$, for example about the set of
morphisms $\mathcal C(C,C')$. In particular, we recover in this way also
the set $\End(C) = \mathcal C(C,C)$ of endomorphisms of $C$ (which of course could be
considered as one of the initial data), let us add this to the previous list:
	\smallskip 

\begin{itemize}
\item the semigroup $\End(C)$ of all endomorphisms of $C$.
\end{itemize}

        \smallskip

\noindent
Note that this set $\End(C)$ 
may be thought of as combining the data mentioned before: the group $\Aut(C)$
is a subgroup of $\End(C)$, whereas any quotient object as well as any non-empty subset of $C$
can be realized as the image of an endomorphism. 
Thus, let us concentrate on these four kinds of data in the special case of
the category of finite sets. 
	
In many categories $\mathcal C$ there are obvious relations between some of these
data. For example, if $\mathcal C$ is the module category of a ring, 
the subobjects of any $C$ correspond bijectively to the quotient objects. 

Our concern is the category of finite sets. For any finite set $C$,
all the data mentioned are again finite sets, but usually these sets are quite different.
Here is the table of the corresponding cardinalities for the sets $C$ of cardinality $n$
with $n\le 10$, this concerns the columns 2 to 5. The last three columns present the
numbers which we want to discuss. Here, $\Sub(C)'$ denotes the set of subsets of $C$
of cardinality at least $2$.  
$$
\hbox{\beginpicture
\setcoordinatesystem units <1.7cm,.47cm>
\put{$|C|$} at 0 0
\put{$|\End(C)|$} at 1 0
\put{$|\Aut(C)|$} at 2 0
\put{$|\Sub(C)|$} at 3 0
\put{$|\Quot(C)|$} at 4 0
\put{$|\Phi_+(\mathbb A_{n-1})|$} at 5.05 0
\put{$|\Sub(C)'|$} at 6.1 0
\put{$|\NC(n)|$} at 7 0

\put{$n$} at 0 -1
\put{$n^n$} at 1 -1
\put{$n!$} at 2 -1
\put{$2^n$} at 3 -1
\put{Bell} at 4 -1
\put{$\binom n 2$} at 5 -1         
\put{Euler} at 6 -1
\put{Catalan} at 7 -1

\put{$0$} at 0 -2
\put{$1$} at 1 -2
\put{$1$} at 2 -2
\put{$1$} at 3 -2 
\put{$1$} at 4 -2
\put{$0$} at 5 -2
\put{$0$} at 6 -2
\put{$1$} at 7 -2

\put{$1$} at 0 -3
\put{$1$} at 1 -3
\put{$1$} at 2 -3
\put{$2$} at 3 -3 
\put{$1$} at 4 -3
\put{$0$} at 5 -3
\put{$0$} at 6 -3
\put{$1$} at 7 -3

\put{$2$} at 0 -4
\put{$4$} at 1 -4
\put{$2$} at 2 -4
\put{$4$} at 3 -4 
\put{$2$} at 4 -4
\put{$1$} at 5 -4
\put{$1$} at 6 -4
\put{$2$} at 7 -4
    
\put{$3$} at 0 -5
\put{$27$} at 1 -5
\put{$6$} at 2 -5
\put{$8$} at 3 -5 
\put{$5$} at 4 -5
\put{$3$} at 5 -5
\put{$4$} at 6 -5
\put{$5$} at 7 -5

\put{$4$} at 0 -6
\put{$256$} at 1 -6
\put{$24$} at 2 -6
\put{$16$} at 3 -6 
\put{$15$} at 4 -6
\put{$6$} at 5 -6
\put{$11$} at 6 -6
\put{$14$} at 7 -6

\put{$5$} at 0 -7
\put{$3125$} at 1 -7
\put{$120$} at 2 -7
\put{$32$} at 3 -7
\put{$52$} at 4 -7
\put{$10$} at 5 -7
\put{$26$} at 6 -7
\put{$42$} at 7 -7

\put{$6$} at 0 -8
\put{$720$} at 2 -8
\put{$46656$} at 1 -8
\put{$64$} at 3 -8
\put{$203$} at 4 -8
\put{$15$} at 5 -8
\put{$57$} at 6 -8
\put{$132$} at 7 -8

\plot -.2 -1.5  7.4 -1.5 /
\plot .35 0.5  .35 -8.5 /

\setdots <1mm>
\plot 4.5 .5  4.5 -8.5 /
\plot 1.5 .5  1.5 -8.5 /
\endpicture}
$$
	\bigskip
 
Before we proceed, let us insert in which way $|\End(C)|$ can be obtained
from numbers of subobjects, of quotient object and of
automorphisms. where $C$ is a set of cardinality $n$. We denote by 
$\Sub(t,C)$ the set of subsets of $C$ of cardinality $t$,
of course $|\Sub(t,C)| = \binom nt$ is just a binomial coefficient. 
We denote by $\Quot(C,t)$ the set of partitions of $C$ with $t$
parts, thus $|\Quot(C,t)| = \left\{n \atop t \right\}$ is a
Stirling numbers of the second kind. Here is the decisive formula:
$$
\hbox{\beginpicture
\setcoordinatesystem units <1cm,.7cm>
\put{$|\End(C)|$} [r] at -.2 0
\put{$=\ \sum_{t=0}^n \ |\Quot(C,t)|\cdot|\Aut(t)|\cdot|\Sub(t,C)| $} [l] at 0 0
\put{$=\ \sum_{t=0}^n  \;\qquad \left\{n \atop t \right\}
                    \quad  \cdot\quad\ \ t!\quad\  \cdot \ \ \binom nt$} [l] at 0 -1
\endpicture}
$$
	\medskip 

Altogether, we obtain the following triangle (Sloane A090657), on the
left side, we exhibit the factorization 
$\left\{n \atop t \right\}\cdot t!\cdot \binom nt$,
on the right side, the corresponding product.
$$
\hbox{\beginpicture
\setcoordinatesystem units <.97cm,1cm>
\put{\beginpicture
\setcoordinatesystem units <.5cm,.5cm>
\put{$\ssize n$} at  -7  1
\put{$\ssize 0$} at  -7  0
\put{$\ssize 1$} at  -7  -1
\put{$\ssize 2$} at  -7  -2
\put{$\ssize 3$} at  -7  -3
\put{$\ssize 4$} at  -7  -4
\put{$\ssize 5$} at  -7  -5

\put{$\ssize 1\cdot 1\cdot 1$} at 0 0

\put{$\ssize 0\cdot 1\cdot 1$} at -1 -1
\put{$\ssize 1\cdot 1\cdot 1$} at 1 -1

\put{$\ssize 0\cdot 1\cdot 1$} at -2 -2
\put{$\ssize 1\cdot 1\cdot 2$} at 0 -2
\put{$\ssize 1\cdot 2\cdot 1$} at 2 -2

\put{$\ssize 0\cdot 1\cdot 1$} at -3 -3
\put{$\ssize 1\cdot 1\cdot 3$} at -1 -3
\put{$\ssize 3\cdot 2\cdot 3$} at 1 -3
\put{$\ssize 1\cdot 6\cdot 1$} at 3 -3

\put{$\ssize 0\cdot 1\cdot 1$} at -4 -4
\put{$\ssize 1\cdot 1\cdot 4$} at -2 -4
\put{$\ssize 7\cdot 2\cdot 6$} at 0 -4
\put{$\ssize 6\cdot 6\cdot 4$} at 2 -4
\put{$\ssize 1\cdot 24\cdot 1$} at 4 -4

\put{$\ssize 0\cdot 1\cdot 1$} at -5 -5
\put{$\ssize 1\cdot 1\cdot 5$} at -3 -5
\put{$\ssize 15\cdot 2\cdot 10$} at -1 -5
\put{$\ssize 25\cdot 6\cdot 10$} at 1 -5
\put{$\ssize 10\cdot 24\cdot 5$} at 3 -5
\put{$\ssize 1\cdot 120\cdot 1$} at 5 -5

\endpicture} at 0 0
\put{\beginpicture
\setcoordinatesystem units <.4cm,.5cm>
\put{$\ssize 1$} at 0 0
\put{$\ssize 0$} at -1 -1
\put{$\ssize 1$} at 1 -1

\put{$\ssize 0$} at -2 -2
\put{$\ssize 2$} at 0 -2
\put{$\ssize 2$} at 2 -2

\put{$\ssize 0$} at -3 -3
\put{$\ssize 3$} at -1 -3
\put{$\ssize 18$} at 1 -3
\put{$\ssize 6$} at 3 -3

\put{$\ssize 0$} at -4 -4
\put{$\ssize 4$} at -2 -4
\put{$\ssize 84$} at 0 -4
\put{$\ssize 144$} at 2 -4
\put{$\ssize 24$} at 4 -4

\put{$\ssize 0$} at -5 -5
\put{$\ssize 5$} at -3 -5
\put{$\ssize 300$} at -1 -5
\put{$\ssize 1500$} at 1 -5
\put{$\ssize 1200$} at 3 -5
\put{$\ssize 120$} at 5 -5

\put{sum} at  7  1
\put{$\ssize 1$} at  7.5  0
\put{$\ssize 1$} at  7.5  -1
\put{$\ssize 4$} at  7.5  -2
\put{$\ssize 37$} at  7.5  -3
\put{$\ssize 256$} at  7.5  -4
\put{$\ssize 3125$} at  7.5  -5

\endpicture} at 7 0

\endpicture}
$$

	\bigskip

\subsubsection{} 
Let us draw the attention to the role of the quotient sets $\Quot(C)$ 
in finite set theory. If $C$ is a set of cardinality $n$,  
the cardinality of $\Quot(C)$ is given by the Bell number $B_n$. 
There is an inductive way to determine the Bell numbers, but contrary to the
other cardinalities discussed here, there is no explicit formula. The binomial
coefficients which count $\Sub(C)$ are ubiquitous, they are taught already in school and
are used in all parts of mathematics, whereas the Bell numbers which count 
$\Quot(C)$ really stand in the shadow. 

Algebraists usually appreciate not only subobjects, but also quotient objects --- one of
the reason for the triumphal rise of the theory of abelian categories seems to be that it has put 
quotient objects on an equal footing with subjects. Of course, there are categories such as
the category of groups where the set of subobjects provide much more information than the
set of quotient objects. In contrast, one has to be aware that dealing with sets without any
additional structure, the wealth of quotient sets (thus the number of partitions) exceeds by far
that of the subsets. 

Looking at the wealth of quotient objects $\Quot(C)$ of a given finite set, 
one is tempted to look for proper subsets of $\Quot(C)$ which may be of relevance. 
This is the realm of the Catalan combinatorics: to single out important subsets of $\Quot(C)$
and to relate them to subsets of $\Aut(C)$.
	\medskip 

Some interesting subsets of $\Quot(C)$: 
\begin{itemize}
\item Non-crossing partitions.
\item Non-nesting partitions.
\end{itemize}

	\medskip 

Some interesting subsets of $\Aut(C)$: 
\begin{itemize}
\item The set of $w$ with $w \le_a c$ for some Coxeter element $c$.
\item Pattern avoiding permutations (for example avoiding the pattern $321$ or
   $132$).
\end{itemize}

	\bigskip

\subsubsection{}
We should stress that the three relevant sets  $\Sub(C),\Quot(C),\Aut(C)$ can be (and should
be) considered as posets with a natural level structure (here, $C$ is a set of  cardinality $n$):

\begin{itemize}
\item $\Sub(C)$ is a poset (even a lattice) 
  with respect to inclusion of subsets, and $\Sub(t,C)$ are the subsets of cardinality $t$,   
  with $0 \le t \le n$. Thus there are $n+1$ levels.
\item $\Quot(C)$ is a poset (again even a lattice)
 with respect to the inverse refinement order, the levels are given by the sets
 and $\Quot(C,n-t)$ of the partitions with $n-t$ parts, for $1\le t \le n$. Thus, there are $n$ levels.
\item $\Aut(C) = S_n$ is a poset with respect to the absolute order $\le_a$ 
  (see Chapter 3),
  and $\Aut_t(C)$ is the set of permutation with $n-t$ cycles, for $1\le t \le n$. Again, there
  are $n$ levels. 
\end{itemize}
	\smallskip

\noindent
Here are these posets for $n = 4.$ 
$$
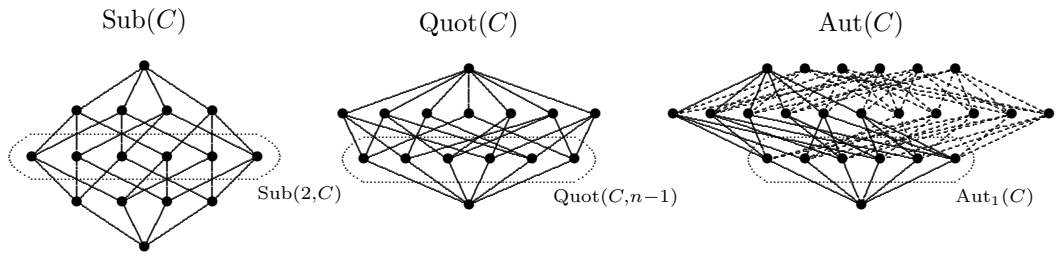
\hbox{\beginpicture
\setcoordinatesystem units <1cm,1cm>
\put{\beginpicture
\setcoordinatesystem units <.3cm,.6cm>
\put{$\Sub(C)$} at 5 5 
\multiput{$\bullet$} at 5 0
   2 1  4 1  6 1  8 1  
   0 2  2 2  4 2  6 2  8 2  10 2
   2 3  4 3  6 3  8 3
   5 4 /
\plot 5 0  2 1  0 2  2 3  5 4 /
\plot 5 0  4 1  0 2  4 3  5 4 /
\plot 5 0  6 1  2 2 /
\plot 5 0  8 1  4 2  4 3 /
\plot 2 1  2 3 /
\plot 2 1  4 2 /
\plot 4 1  6 2  2 3 /
\plot 4 1  8 2  4 3 /
\plot 6 1  6 2  8 3  5 4 /
\plot 6 1  10 2  6 3  5 4 /
\plot 8 1  8 2  8 3  /
\plot 8 1  10 2  8 3  /
\plot 2 2   6 3 /
\plot 4 2  6 3 /
\put{${}^{\Sub(2,C)}$} [l] at 10 1.1 
\setdots <.5mm>
\setquadratic
\plot 0 1.5  5 1.5  10 1.5  11  2  10 2.5  5 2.5  0 2.5  -1 2  0 1.5 /
\endpicture} at 0 0 
\put{\beginpicture
\setcoordinatesystem units <.28cm,.6cm>
\put{} at 6 -1
\put{$\Quot(C)$} at 6 4
\multiput{$\bullet$} at 6 0  6 3
  1 1  3 1  5 1  7 1  9 1  11 1 
  0 2  2 2  4 2  6 2  8 2  10 2  12 2 /
\plot  6 0  1 1  0 2  6 3 /
\plot  6 0  3 1  0 2 /
\plot  6 0  5 1  2 2  6 3 /
\plot  6 0  7 1  0 2 /
\plot  6 0  9 1  2 2 /
\plot  6 0  11 1 6 2  6 3 / 
\plot 1 1  2 2 /
\plot 1 1  6 2 /
\plot 3 1 10 2  6 3 /
\plot 3 1  4 2  6 3 /
\plot 5 1  8 2  6 3 /
\plot 5 1 10 2 /
\plot 7 1  8 2 /
\plot 7 1  12 2  6 3 /
\plot 9 1  4 2 /
\plot 9 1  12 2 /
\plot 11 1  10 2 /
\plot 11 1  12 2 /
\put{${}^{\Quot(C,n-1)}$} [l] at 10 0.1 
\setdots <.5mm>
\setquadratic
\plot 1 .5  6 .5  11 .5  12  1  11 1.5  6 1.5  1 1.5  0 1  1 .5 /
\endpicture} at 4.3 0 
\put{\beginpicture
\setcoordinatesystem units <.25cm,.6cm>

\put{$\Aut(C)$} at 10 4
\put{${}^{\Aut_{1}(C)}$} [l] at 15 0.1 

\put{} at 10 -1
\multiput{$\bullet$} at 10 0  
  5 1  7 1  9 1  11 1  13 1  15 1
0 2  2 2  4 2  6 2  8 2  10 2  12 2  14 2  16 2  18 2  20 2  
  5 3  7 3  9 3  11 3  13 3  15 3  /
\plot 10 0  5 1 /
\plot 10 0  7 1 /
\plot 10 0  9 1 /
\plot 10 0  11 1 /
\plot 10 0  13 1 /
\plot 10 0  15 1 /
\setdots <.5mm>
\setquadratic
\plot 5 .5  10 .5  15 .5  16  1  15 1.5  10 1.5  5 1.5  4 1  5 .5 /

\setsolid
\setlinear
\plot 5 1  0 2  7 1  8 2  15 1  10 2  11 1  6 2  9 1  2 2  5 1 /
\plot 5 1  4 2  15 1 /
\plot 0 2  11 1  /
\plot 2 2  13 1  10 2 /
\plot  9 1  8 2 /
\plot 0 2  5 3  2 2  /
\plot 4 2  5 3  6 2 /
\plot 8 2  5 3  10 2 /

\setdashes <.5mm>

\plot 7 1  12 2  13 1 /
\plot 5 1  14 2  7 1 /
\plot 11 1  14 2 /

\plot 5 1  16 2  9 1 /
\plot 13 1  16 2 /

\plot 7 1  18 2  9 1 /
\plot 15 1  18 2 / 

\plot 11 1  20 2  13 1 /
\plot 15 1  20 2 /

\plot 12 2  7 3  4 2 /
\plot 0 2  7 3  2 2 /
\plot 20 2  7 3  18 2 /

\plot 6 2  9 3  12 2  /
\plot 14 2  9 3  2 2   /
\plot 8 2  9 3  20 2  /

\plot 4 2  11 3  12 2 /
\plot 14 2  11 3  2 2 /
\plot 8 2  11 3  10 2  /

\plot 6 2  13 3  12 2 /
\plot 18 2  13 3 16 2 /
\plot 0 2  13 3  10 2 /

\plot 4 2  15 3  6 2 /
\plot 18 2  15 3  16 2 /
\plot 14 2  15 3  20 2 /
\endpicture} at 9 0 

\endpicture}
$$
	\medskip

{\bf The set $\Phi_+(\mathbb A_{n-1})$.}
We have encircled by a dotted line 
the subsets 
$$
 \Sub(2,C),\ \Quot(C,n-1),\ \Aut_1(C).
$$ 
For any $n$, we claim that there are
canonical bijections between these sets. Note that the set $\Sub(C)_2$ may be considered 
as the set of $\Phi_+$ of positive roots for the Dynkin diagram $\mathbb A_{n-1}$, thus we claim that
	\smallskip

\Rahmen{$\Phi_+(\mathbb A_{n-1}) = \Sub(2,C) \simeq \Quot(C,n-1) \simeq \Aut_1(C).$}
	\smallskip

\noindent
The bijections are given as follows: A positive root in $\Phi_+(\mathbb A_{n-1})$ may be
considered as a pair $(x,y)$ of natural numbers $1\le x < y \le n$, or, equivalently
as the set $\{x,y\}$ in $\Sub(2,C)$. The corresponding partition of $\{1,2,\dots,n\}$ has
$n-1$ parts, namely the part $\{x,y\}$ as well as the remaining singletons; these are the
elements of $\Quot(C,n-1)$. The corresponding permutation exchanges $x$ and $y$ and fixes
the remaining elements, these are the elements of $\Aut_1(C).$ 
	\bigskip

\subsubsection{\bf 
 The canonical map $P\!:\Sub(C) \to \Quot(C)$.}
There is a canonical map
	\smallskip

\Rahmen{$
  P\!:\Sub(C) \to \Quot(C)
$}
	\smallskip 

\noindent 
defined as follows: If $U$ is a non-empty subset of $C$, then $P(U)$ is the
partition with $U$ as one part, the remaining parts being singletons; the image
$P(\emptyset)$ of the empty set is defined as the discrete partition (all parts are singletons).
	
Let us denote by $\Sub(C)'$ the set of subsets of $C$ of cardinality at least $2$.
If $C$ has cardinality $n$, then $\Sub(C)'$ has cardinality $2^n-(n+1)$, these numbers
are called the {\it Euler numbers} (see \cite{[Sl]}, A000295). 
Also, we denote by $\Quot(C)_c$ the set of all 
partitions of $C$ 
with precisely one part of cardinality greater than $1$ 
(the index $c$ stands for {\it connected}).

There is the canonical 
map
$$
 U\!:\Quot_c(C) \to \Sub(C)',
$$
which sends a partition in $\Quot(C)_c$ to the only part of cardinality greater than $1$.
The maps $U,P$ provide inverse poset isomorphisms
	\smallskip

\Rahmen{$ \Quot_c(C) \simeq \Sub(C)'.$}
	\smallskip

	\medskip 
Note that the poset $\Sub(C)'$ becomes a lattice if we add the empty set as zero
element.  Here are these lattices for $n=2,3,4$, and we encircle  
the set $\Sub(2,C)$ of minimal elements of $\Sub(C)'$ by a dotted line:
$$
\hbox{\beginpicture
\setcoordinatesystem units <1cm,1cm>
\put{\beginpicture
\setcoordinatesystem units <.3cm,.47cm>
\put{} at 10 4
\multiput{$\bullet$} at 10 2  10 1 /
\plot 10 1  10 2 /
\put{${}^{\Sub(2,C)}$} [l] at 11 1.5 
\setdots <.5mm>
\setquadratic
\plot 10 1.5  11  2  10 2.5   9 2  10 1.5 /
\put{$n=2$} at 10 0
\endpicture} at 0.2 0  
\put{\beginpicture
\setcoordinatesystem units <.3cm,.47cm>
\put{} at 5 4 
\multiput{$\bullet$} at 5 1
   3 2  5 2  7 2 
   5 3 /
\plot 5 1  3 2  5 3  5 1  7 2  5 3 /
\put{${}^{\Sub(2,C)}$} [l] at 7 1.1 
\setdots <.5mm>
\setquadratic
\plot 3 1.5  5 1.5  7 1.5  8  2  7 2.5  5 2.5  3 2.5  2 2  3 1.5 /
\put{$n=3$} at 5 0
\endpicture} at 3 0 

\put{\beginpicture
\setcoordinatesystem units <.3cm,.47cm>
\multiput{$\bullet$} at 5 1
   0 2  2 2  4 2  6 2  8 2  10 2
   2 3  4 3  6 3  8 3
   5 4 /
\plot 5 1  0 2  2 3  5 4 /
\plot 0 2  4 3  5 4 /
\plot 5 1  2 2 /
\plot 5 1  4 2  4 3 /
\plot 2 2  2 3 /
\plot 5 1  6 2  2 3 /
\plot 5 1  8 2  4 3 /
\plot 6 2  8 3  5 4 /
\plot 5 1  10 2  6 3  5 4 /
\plot 8 2  8 3  /
\plot  10 2  8 3  /
\plot 2 2   6 3 /
\plot 4 2  6 3 /
\put{${}^{\Sub(2,C)}$} [l] at 10 1.1 
\setdots <.5mm>
\setquadratic
\plot 0 1.5  5 1.5  10 1.5  11  2  10 2.5  5 2.5  0 2.5  -1 2  0 1.5 /
\put{$n=4$} at 5 0
\endpicture} at 7 0 
\multiput{} at 0 -1.5 /
\endpicture}
$$

\begin{prop} The partitions in $\Quot_c(C)$ are  both non-crossing and non-nesting.
\end{prop}

\begin{proof} 
Consider a partition $P$ in $\Quot_c(C)$.
Assume that there are elements $a < b < c < d$ with $a,c$ belonging to the same part, and
$b,d$ belonging to the same part. Then we deal with parts of cardinality at least 2. Since 
we assume that $P$ belongs to $\Quot(C)_c$, there is precisely one part of cardinality
at least 2, thus all the elements $a,b,c,d$ belong to the same part. This shows that an
$P$ is non-crossing.

Similarly, one shows that $P$ is non-nesting.
\end{proof}

\subsubsection{\bf The categorification of $\Sub(C)' = \Quot_c(C)$.}
 The identification of $ \Quot_c(C)$ and $\Sub(C)'$ using the maps $U,P$
can be reformulated in terms of the categorification by $\mo\Lambda_{n-1}$ as follows:
Let $\A_c(\mo\Lambda_{n})$ be the set of (non-empty) antichains in $\mo\Lambda_n$ with 
connected quiver (or, equivalently, the set of all exceptional subcategories 
of $\mo\Lambda_n$ which are connected).

Let $2\le t\le n$. There is a bijection between
$\A_c(\mo\Lambda_{n-1})_{t-1}$ and $\Sub(t,C)$ as follows: A subset 
in $\Sub(t,C)$ is of the form $\{a_1,\dots,a_t\}$ with $1\le a_1 < \cdots < a_t\le n$,
it is sent to the antichain $\{(a_1,a_2),\dots,(a_{t-1},a_t)\}$, and this antichain belongs to
$\A_c(\mo\Lambda_{n-1})$. Thus, we see:
The map $A\!:\NC(n) \to \A(\mo\Lambda_{n-1})$ provides a bijection
	\smallskip

\Rahmen{$\Quot_c(C) \to \A_c(\mo\Lambda_{n-1})$,}
	\smallskip

\noindent
where $C$ is a set of cardinality $n$. 
	\bigskip 

\subsubsection{\bf The canonical map $O\!:\End(C) \to \Quot(C)$.} 
There is a canonical map
$$
 O\!:\End(C) \to \Quot(C),
$$
it is defined as follows: If $f$ is an endomorphism of $C$, the parts of $O(f)$ are the orbits
of $f$ in $C$. The orbit of $c\in C$ under $f$ is the smallest subset of $C$ which contains $c$
and which contains with every element $x$ also $f(x)$ and $f^{-1}(x)$ (of course, if $C$ is a
finite set and $f$ is invertible, the orbit of $c$ under $f$ consists of the elements $f^t(c)$
with $t\ge 0$). 
	\medskip 

The map $O$ is surjective, even the
restriction of $O$ to $\Aut(C)$ is surjective: Given a partition $P$ of $C$, we obtain
a permutation $f$ of $C$ with $O(f) = P$ as follows: For any part $P_i$ of $P$,
let $f|P_i$ by a cyclic permutation of $P_i$. 
	
For any total ordering $L$ of $C$, there is a map $\pi_L\!:\Quot(C) \to \Aut(C)$ such that
the composition
	\smallskip

\Rahmen{$\Quot(C) \overset{\pi_L}\longrightarrow \Aut(C)\overset O \longrightarrow \Quot(C)$}
	\smallskip

\noindent
is the identity map. The map $\pi_L$ is defined as follows: Let $P$ be a partition of $C$,
with parts $P_1,\dots,P_t$. If $P_i$ consists of the elements $a_{i,1},a_{i,2},\dots,a_{i,t(i)}$
we can assume that $a_{i,1} < a_{i,2} <\dots < a_{i,t(i)}$ with respect to the ordering $L$,
and we may consider the corresponding cycle $c_i = (a_{i,1},a_{i,2},\dots,a_{i,t(i)})$.
Then $\pi_L(P) = c_1\cdots c_t.$ Of course, this is a permutation of $C$ and $O(\pi_L(P)) = P.$
	\medskip

Note: {\it If $P$ is non-crossing with respect to $L$, then $\pi_L(P) \in W(S_n,c_L).$}
	\medskip
In general, one may ask which ``Catalan subsets'' of $\Aut(C)$ and of $\Quot(C)$
correspond to each other under the map $O$. 
	\bigskip

We have seen in Chapter 4 that the lattice $\NC(n)$ can be identified with $\NC(S_n,c_n)$.
	\medskip 

$$
\hbox{\beginpicture
  \setcoordinatesystem units <3cm,1.3cm>
\put{$\NC(n)$} at 0 1
\put{$\A(\mo\Lambda_{n-1})$} at 0 0 
\put{$\iota_n$} at 0.5 1.2
\arr{0 0.7}{0 0.3}
\arr{0.3 1}{0.7 1}
\put{$\NC(S_n,c_n)$} at 1.1 1 
\put{$\simeq$} at 0.1 0.5
\arr{0.3 0.3}{0.7 0.7}
\put{$\cox$} at 0.6 0.35
\put{} at 2.1 0 
\endpicture}
$$
	\smallskip 

\noindent
with $c_n = (n,n-1,\dots,2,1)$.
This map $\iota_n$ is part of the following larger commutative diagram:
$$
\hbox{\beginpicture
  \setcoordinatesystem units <3cm,1.3cm>
\put{$\Quot(n)$} at 0 2
\put{$\NC(n)$} at 0 1
\put{$\A(\mo\Lambda_{n-1})$} at 0 0 
\put{$\iota_n$} at 0.5 1.2
\arr{0 1.4}{0 1.7}
\arr{0 0.7}{0 0.3}
\arr{0.3 2}{0.7 2}
\arr{0.3 1}{0.7 1}
\put{$S_n$} at 1 2
\put{$\NC(S_n,c_n)$} at 1.1 1 
\arr{1 1.4}{1 1.7}
\put{$\Quot(n)$} at 2 2 
\arr{1.3 2}{1.6 2}
\put{$O$} at 1.45 2.2
\put{$\pi_n$} at 0.5 2.2
\put{$\iota_n$} at 0.1 1.5
\put{$\simeq$} at 0.1 0.5
\arr{0.3 0.3}{0.7 0.7}
\put{$\cox$} at 0.6 0.35
\setquadratic
\plot 0 1.4  -.03 1.3 -.06 1.4 /
\plot 1 1.4  0.97 1.3 0.94 1.4 /
\endpicture}
$$
Here, $\pi_n = \pi_L$ with the natural ordering $L = (1 < 2 < \cdots < n)$.

	\bigskip  

\subsubsection{} 
The interesting numbers are provided by the following chain of inclusion maps.
Here, we denote by $\iota_L\!:\NC(n) \to \Quot(C)$ the canonical inclusion after
we have identified $(C,L)$ with $\{1,2,\dots,n\}$, where $L$ is a total ordering of the
set $C$ of cardinality $n$. 
	\medskip
 
$$
\hbox{\beginpicture
\setcoordinatesystem units <2cm,.6cm>
\put{$C$} at 0 0.4
\put{$\Aut(C)$} at 5.3 0.4
\put{$\Sub(C)'$} at 2.05 0.4
\put{$\Quot(C)$} at 4.3 0.4
\put{$\Sub(2,C)$} at 1 .4
\put{$=\Phi_+(\mathbb A_{n-1})$} at 1 -.2
\put{$\NC(n)$} at 3.15 0.4
\put{$\subseteq$} at 1.54 0.4
\arr{2.45 0.4}{2.7 0.4}
\put{$P$} at 2.55 0.7
\put{$\iota_L$} at 3.7 0.7
\arr{3.55 0.4}{3.8 0.4}
\arr{4.7 0.4}{4.95 0.4}
\put{$\pi_L$} at 4.83 0.7

\put{$|C| = n$} at -.1 -1
\put{$n!$} at 5.3 -1
\put{Euler} at 2 -1
\put{Bell} at 4.3 -1
\put{$\binom n 2$} at 1 -1         
\put{Catalan} at 3.15 -1

\put{$0$} at 0 -2
\put{$1$} at 5.3 -2
\put{$0$} at 2 -2 
\put{$1$} at 4.3 -2
\put{$0$} at 1 -2
\put{$1$} at 3.15 -2

\put{$1$} at 0 -3
\put{$1$} at 5.3 -3
\put{$0$} at 2 -3 
\put{$1$} at 4.3 -3
\put{$0$} at 1 -3
\put{$1$} at 3.15 -3

\put{$2$} at 0 -4
\put{$2$} at 5.3 -4
\put{$1$} at 2 -4 
\put{$2$} at 4.3 -4
\put{$1$} at 1 -4
\put{$2$} at 3.15 -4
    
\put{$3$} at 0 -5
\put{$6$} at 5.3 -5
\put{$4$} at 2 -5 
\put{$5$} at 4.3 -5
\put{$3$} at 1 -5
\put{$5$} at 3.15 -5

\put{$4$} at 0 -6
\put{$24$} at 5.3 -6
\put{$11$} at 2 -6 
\put{$15$} at 4.3 -6
\put{$6$} at 1 -6
\put{$14$} at 3.15 -6

\put{$5$} at 0 -7
\put{$120$} at 5.3 -7
\put{$26$} at 2 -7
\put{$52$} at 4.3 -7
\put{$10$} at 1 -7
\put{$42$} at 3.15 -7

\put{$6$} at 0 -8
\put{$720$} at 5.3 -8
\put{$57$} at 2 -8
\put{$203$} at 4.3 -8
\put{$15$} at 1 -8
\put{$132$} at 3.15 -8

\plot -.2 -1.5  5.6 -1.5 /
\plot .35 0.5  .35 -8.5 /

\endpicture}
$$
	\medskip

All but the last two maps shown in the upper row are canonical, the last two $\iota_L$ and $\pi_L$ 
depend on the choice of $L$, say the choice of a linear
ordering such as $L = \{1 \leftarrow 2 \leftarrow 3 \leftarrow \cdots \leftarrow n\}$.

But note that $\pi_L$ actually only depends on the corresponding cyclic ordering
$$
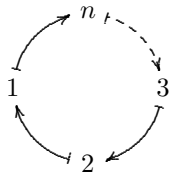
\hbox{\beginpicture
  \setcoordinatesystem units <1cm,1cm>
\put{$3$} at 1 0 
\put{$n$} at 0 1 
\put{$1$} at -1 0 
\put{$2$} at 0 -1 
\circulararc 60 degrees from  -.95 -.25 center at 0 0
\circulararc -60 degrees from 0.95 -.25 center at 0 0
\circulararc -60 degrees from  -.95 .25 center at 0 0
\plot -.9 0.24  -1 0.26 /
\plot .27 1  .23 0.87 /
\plot .9 -.24  1 -.26 /
\plot -.27 -1  -.23 -.87 /

\arr{-.3 .94}{-.25 .97}
\arr{.3 -.94}{.25 -.97}
\arr{.941 .26}{.947 .24}
\arr{-.941 -.26}{-.947 -.24}
\setdashes <1mm>
\circulararc 60 degrees from 0.95 0.25 center at 0 0
\endpicture}
$$
	\bigskip\bigskip

\bibliographystyle{amsplain}

\begin{thebibliography}{99}


\bibitem{[A]}T.~Abe, M.~Barakat, M.~Cuntz, T.~Hoge, H.~Terao:
    The freeness of ideal subarrangements of Weyl arrangements.
   arXiv:1304.8033 
\bibitem{[AZ]}M.~Aigner, G.~M.~Ziegler: Proofs from the BOOK. Springer (2002).
   Proof of Cayley's theorem (chapter 24)
\bibitem{[AC]}E.~Akyildiz, J.~Carrell. 
   Betti numbers of smooth Schubert varieties and the remarkable formula of Kostant, Macdonald, Shapiro, 
  and  Steinberg. Michigan Math.~J.
   61 (3) (2012), 543-553.
\bibitem{[AHK]} L.~Angeleri-H\"ugel, D.~Happel, H.~Krause. Handbook of tilting theory.
   London Math.~Soc.~Lecture Note Series 332. Cambridge University Press (2007).
\bibitem{[Ag]} D.~Armstrong: Generalized Noncrossing Partitions and
      Combinatorics of Coxeter Groups. 
      Memoirs of the Amer.~Math.~Soc.~949 (2009).
\bibitem{[At]} C.~A.~Athanasiadis. On a refinement of the generalized
     Catalan numbers for Weyl groups. Trans.~Amer.~Math.~Soc.~357 (2004), 179-196.
\bibitem{[AR]} C.~A.~Athanasiadis, V.~Reiner: Noncrossing partitions
		     for the group $D_n.$ SIAM J. Discrete Math. 18 (2004), no. 2, 397-417.
\bibitem{[APR]} M.~Auslander, M.~I.~Platzeck, I.~Reiten. Coxeter functors without diagrams. 
    Trans.~Amer.~Math.~Soc.~250 (1979), 1-46.
\bibitem{[ARS]} M.~Auslander, I.~Reiten, S.~S.~Smal\o. Representation Theory of  
    Artin Algebras. Cambridge Studies in Advanced Mathematics 36. Cambridge University 
    Press. 1997.
\bibitem{[AS]} M.~Auslander, S.~S.~Smal\o. Preprojective modules over artin algebras.
  J.Algebra 66 (1980), 61-122.
\bibitem{[Ba]} D.~F.~Bailey: Counting arrangements of 1's and -1's. Math.~Mag. 69
  (1996), 128-131.
\bibitem{[BC]} M.~Barakat, M.~Cuntz.
    Coxeter and crystallographic arrangements are inductively free.
     arXiv:1011.4228 
\bibitem{[BDSW]} B. Baumeister, M. Dyer, C. Stump, P. Wegener: A note on the transitive
  action on decompositions of parabolic Coxeter elements. arXiv 1402.2500. 
\bibitem{[Bn]} A.~Beineke. Oral communication.
\bibitem{[Be]} D.~Bessis: The dual braid monoid. Ann.~Sci.~Ecole Norm.~Sup.~(4) 36
   (2005), nr.~5, 647-683.
\bibitem{[Bi]} P.~Biane. Some properties of crossings and partitions. Discrete Math.~175
    (1997), 41--53.
\bibitem{[Bo]} K.~Bongartz. Tilted algebras. In: Representations of algebras.
   Springer Lecture Notes in Math.~903 (1981), 26--38.
\bibitem{[B]} N.~Bourbaki. Groupes et Algebres de Lie. 
   IV-VI. Hermann (1968).
\bibitem{[BW]} Th.~Brady, C.~Watt: Non-crossing partition lattices in finite real reflection
   groups. Trans.~Amer.~Math.~Soc.~360 (2008), no.~4, 1983-2005. 
\bibitem{[BLR]} O.~Bretscher, Chr.~L\"aser, Chr.\ Riedtmann: Selfinjective and 
   simply connected algebras. Manuscripta Math., 36 (1981), 253 - 307.
\bibitem{[Br]}T.~Bridgeland. Stability conditions on triangulated categories.
   arXiv:math.AG/ 0212237
\bibitem{[B1]}E.~Brieskorn. Die Fundamentalgruppe des Raumes der regul\"aren Orbits einer
   endlichen komplexen Spiegelungsgruppe. Invent.~Math.~12 (1971), 57-61.
\bibitem{[B2]}E.~Brieskorn. Sur les groupes de tresses (d'apres Arnold). Sem.~Bourbaki (1971/72)
   Exp.~401. Lectut.~Notes in Math.~317 (1973) Springer, Berlin.
\bibitem{[Ca]}A.~Cayley. A theorem on trees. Quart.~J.~Appl.~Math.~23 (1889), 376-378.
\bibitem{[CB]} W.~W.~Crawley-Boevey. Exceptional sequences of quivers. In:
    Canadian Math.~Soc.~Proceedings 14 (1993), 117-124.
\bibitem{[CP1]}P.~Cellini, P.~Papi. ad-nilpotent ideals of a Borel subalgebra. J.~Algebra
    225 (2000), 130--141. 
\bibitem{[CP2]}P.~Cellini, P.~Papi. ad-nilpotent ideals of a Borel subalgebra. II. J.~Algebra
    258 (2002), 112--121.
\bibitem{[C1]} F.~Chapoton, 2002. http://igd.univ-lyon1.fr/$\sim$chapoton/clusters.html
\bibitem{[C2]} F.~Chapoton: Enumerative properties of generalized associahedra.
  Sem.~Lothar.~Combin. 51 (2004/5).
\bibitem{[CS]} G.~Chapuy, C.~Stump: Counting factorizations of Coxeter elements
   into products of reflections. arXiv:1211.2789.
\bibitem{[Ch]}C.~Chevalley.
   Invariants of finite groups generated by reflections. Annals of Math. 85 (1955), 778-782.
\bibitem{[Co]}H.~S.~M.~Coxeter. The product of the generators of a finite group generated by
    reflections. Duke Math. J.~18 (1951), 765-782.
\bibitem{[CK]} A.~Claesson, S.~Kitaev: Classification of bijections between $321$- and $132$-avoiding
  permutations. DMTCS proc. AJ 2008, 495-506.
\bibitem{[CB]} W.~Crawley-Boevey: Exceptional sequences of quivers. In:
    Canadian Math.~Soc.~Proceedings 14 (1993), 117-124.
\bibitem{[dH1]}J.~A.~de la Pena, L.~Hille. Distinguished Slopes for Quiver Representations. 
   Bol.~Soc.~Mat.~Mexicana (3) 7, No. 1 (2001), 73-83.
\bibitem{[dH2]}J.~A.~de la Pena, L.~Hille. Stable Representations of Quivers, 
   J.~Pure Appl.~Algebra 172 (2002), no. 2-3, 205-224.
\bibitem{[D]} P.~Deligne: Letter to E.~Looijenga 9.3.1974. Online available:\newline
   http://homepage.univie.ac.at/christian.stump/Deligne\_Looijenga\_Letter\_09-03-1974.pdf
\bibitem{[DR1]} V.~Dlab, C.~M.~Ringel:  On algebras of finite representation type.
    J. Algebra 33 (1975), 306-394
\bibitem{[DR2]} V.~Dlab, C.~M.~Ringel: Indecomposable representations of graphs
   and    algebras.      Mem.~Amer.~Math.~Soc.~173 (1976).
\bibitem{[DR3]} V.~Dlab, C.~M.~Ringel:   A module theoretic interpretation of 
     properties of the root systems. Proc. Conf. Ring Theory Antwerp 1978. 
     Marcel Dekker Lecture Notes Pure Appl. Math 51 (1979), 435-451. 
\bibitem{[DR4]} V.~Dlab, C.~M.~Ringel. 
   The module theoretical approach to quasi-hereditary algebras. In: Representations of Algebras
   and Related Topics (ed.~H.~Tachikawa and S.~Brenner). 
   London Math. Soc. Lecture Note Series 168. Cambridge University Press (1992), 200-224 
\bibitem{[FR]} S.~Fomin, N.~Reading: Root Systems and generalized associahedra.
    Geometric combinatorics. in: \cite{[MRS]}, 63--131.
\bibitem{[FZ]} S.~Fomin, A.~Zelevinsky: Y-systems and generalized associahedra.
    Annals of Math. 158 (2003), 977--1018.
\bibitem{[G1]} P.~Gabriel. Unzerlegbare Darstellungen I. Manuscripta Math. 6 (1972), 71-103.
\bibitem{[G2]} P.~Gabriel. Un jeu? Les nombres de Catalan. 
    Uni Z\"urich. Mitteilungsblatt des Rektorats. 12. Jahrgang, Heft 6 (1981), 4--5.
\bibitem{[GP]} P.~Gabriel and J.~A.~de la Pe\~na, Quotients of representation-finite algebras, Comm. Algebra
     15 (1987), 279–-307
\bibitem{[GL]}W.~Geigle, H.~Lenzing. Perpendicular categories with applications to representations and sheaves.
  J.Algebra 144 (1991), 273-343. 
\bibitem{[HR1]} D.~Happel, C.~M.~Ringel. Tilted algebras.  
   Trans. Amer. Math. Soc. 274 (1982), 399-443. 
\bibitem{[HR2]} D.~Happel, C.~M.~Ringel. Construction of tilted algebras.
    Proceedings ICRA 3. Springer LNM 903 (1981), 125-167. 
\bibitem{[HRS]}D.~Happel, I.~Reiten, and S.~Smal\o{}.
   Tilting in abelian categories and quasitilted algebra, 
   Memoirs of the AMS, vol. 575, Amer. Math. Soc. (1996).
\bibitem{[HU]} 
    D.~Happel; L.~Unger: On the set of tilting objects in hereditary categories, 
    Fields Institute Communications Vol. 45 (2005), 141-159.
\bibitem{[Ho]} M.~Hovey. Classifying subcategories of modules. 
    Trans.~Amer.~Math.~Soc.~353 (2001), 3181-3191.
\bibitem{[HK]} A.~Hubery,
      H.~Krause: A categorification of non-crossing partitions. arXiv:1310.1907.
      to appear in: Journal of the European Mathematical Society.
\bibitem{[H1]} J.~E.~Humphreys. Introduction to Lie Algebras and Representation Theory.
     Springer Graduate Texts in Mathematics (1974,1997).
\bibitem{[H2]} J.~E.~Humphreys. Reflection groups and
   Coxeter groups. Cambridge University Press (1990).
\bibitem{[IS]} K.~Igusa, R.~ Schiffler. Exceptional sequences and clusters. 
     Journal of Algebra 323. no. 8 (2010): 2183--2202.
\bibitem{[IT]} C.~Ingalls, H.~Thomas: Noncrossing partitions and representations of quivers.
            Comp.~Math.~145 (2009), 1533-1562.
\bibitem{[IRTT]} O.~Iyama, I.~Reiten, H.~Thomas, G.~Todorov: Lattice structure of torsion classes for
     path algebras of quivers. arXiv:1312.3659
\bibitem{[Ki]}A.~D.~King. Moduli of representations of finite dimensional algebras, Quart.~J.~Math.
   Oxford (2) 45 (1994), 515-530.
\bibitem{[Kn]} D.~E.~Knuth: 
   Two notes on notation, The American Mathematical Monthly 99, 1992, S. 403–422.
\bibitem{[Ko]}B.~Kostant. The principal three-dimensional subgroup and the Betti numbers
  of a complex simple Lie group. Amer.~J.~Math.~81 (1959), 973-1032.
\bibitem{[Kr]}G.~Kreweras. Sur les partitions non crois\'ees d.un cycle.
    Discrete Math., 1(4):333-350 (1972).
\bibitem{[K1]}H.~Krause. Thick subcategories of modules over commutative noetherian rings 
  (with an appendix by Srikanth Iyengar), Math. Ann. 2007.
\bibitem{[K2]}H.~Krause. 
    Report on locally finite triangulated categories, J.~K-Theory 9 (2012), 421-458.
\bibitem{[La]} S.~Ladkani: Universal derived equivalences of posets of tilting modules.
     arXiv:0708.1287
\bibitem{[Lo]}  E.~Looijenga: The complement of the bifurcation
      variety of a simple singularity. Invent. Math. 23 (1974), 105-116.
\bibitem{[Ma]}I.~G.~Macdonald. The Poincar\'e series of a Coxeter group. Math.~Ann.
   199 (1972)161-174. 
\bibitem{[Mc]}J. McCammond. Noncrossing partitions in surprising locations.
    Amer. Math. Monthly, 113(7):598.610, 2006.
\bibitem{[MRS]} E.~Miller, V.~Reiner, B.~Sturmfels (ed): Geometric Combinatorics.
    IAS/Park City Math. Ser., 13, Amer. Math. Soc.,
    Providence, RI (2007)
\bibitem{[NS]}A. Nica and R. Speicher. Lectures on the combinatorics of free probability,
   volume 335 of London Mathematical Society Lecture Note Series.
   Cambridge University Press, Cambridge, 2006.
\bibitem{[ONAFR]} A.~A.~Obaid, S.~K.~Nauman, W.~S.~Al Shammakh,
     W.~M.~Fakieh and C.~M.~Ringel:
     The number of complete exceptional sequences for a Dynkin algebra.
     Colloq. Math. 133 (2013), 197--210 
\bibitem{[ONFR1]} A.~A.~Obaid, S.~K.~Nauman,
     W.~M.~Fakieh and C.~M.~Ringel:  The Ingalls-Thomas bijections. Preprint. 
\bibitem{[ONFR2]} A.~A.~Obaid, S.~K.~Nauman,
     W.~M.~Fakieh and C.~M.~Ringel:  The numbers of support-tilting modules for Dynkin algebras.
      arXiv:1403.5827.
\bibitem{[OT]}P. Orlik, H. Terao. Arrangements of hyperplanes, 
  Grundlehren der Mathematischen
  Wissenschaften, vol 300. Springer-Verlag, Berlin (1992)
\bibitem{[Pa]} D.~I.~Panyushev,
    Ad-nilpotent ideals of a Borel subalgebra: generators and duality. 
    J. Algebra 274 (2004), 822--846 
\bibitem{[Pd]}H. Prodinger. A correspondence between ordered trees and noncrossing partitions.
   Discrete Math. 46(2):205.206, 1983.
\bibitem{[Pf]}H.~Pr\"ufer. Neuer Beweis eines Satzes \"uber Permutationen. 
   Archiv der Math.~u.~Physik (3) 27 (1918), 142-144. 
\bibitem{[Rd]} N.~Reading: Chains in the noncrossing partition lattice.
      SIAM J. Discrete Math. 22 (2008), no. 3, 875-886.
\bibitem{[Rn]} V.~Reiner, Non-crossing partitions for classical reflection
            groups. Discrete Math. 177 (1997), no. 1-3, 195-222.
\bibitem{[RSW]}V. Reiner, D. Stanton, and D. White. The cyclic sieving phenomenon.
   J. Combin. Theory Ser. A 108(1); 17-50, 2004.
\bibitem{[R1]} C.~M.~Ringel. Representations of K-species and bimodules. J. Algebra 41 (1976),    269-302. 
\bibitem{[R2]} C.~M.~Ringel: Reflection functors for hereditary algebras.
		    J. London Math. Soc. (2) 21 (1980), 465--479.
\bibitem{[R3]} C.~M.~Ringel: Tame algebras and integral quadratic forms. Springer LNM 1099 (1984).
\bibitem{[R4]} C.~M.~Ringel: The braid group action on the set of exceptional sequences of a
      hereditary    algebra.
                In: Abelian Group Theory and Related Topics. Contemp.~Math.~171 (1994), 339-352.
\bibitem{[R5]}C.~M.~Ringel. Exceptional objects in hereditary categories.
    Proceedings Constantza Conference. An. St. Univ. Ovidius Constantza Vol. 4 (1996), f. 2, 150-158. 
\bibitem{[R6]} C.~M.~Ringel. Some remarks concerning tilting modules and tilted algebras.
      Origin. Relevance. Future. (An appendix to the Handbook of Tilting Theory.)
      London Math.~Soc.~Lecture Note Series 332. Cambridge University Press (2007),
      413-472.
\bibitem{[R7]} C.~M.~Ringel. The $(n\!-\!1)$-antichains in a root poset of width $n$.
       arXiv:1306.1593
\bibitem{[R8]} C.~M.~Ringel. Lattice structure of torsion classes for hereditary artin   
    algebras.\newline arXiv:1402.1260.
\bibitem{[Ro]} A.~V.~Roiter. Unbounded dimensionality of indecomposable representations of an 
    algebra with an infinite number of indecomposable representations.
    Izv. Akad. Nauk SSSR Math. USSR - Izvestija
    Ser. Mat. Tom 32 (1968), No. 6 Vol. 2 (1968).
\bibitem{[Ro]} D.~Rotem: On the correspondence between
  binary trees and certain type of permutations. Information processing letters 4, no.3 (1975),
  58-61. 
\bibitem{[Ru1]}A.~Rudakov . Helices and vector bundles. Seminaire Rudakov. 
    London Math. Soc.~Lect.~Note Series 148 (1990).
\bibitem{[Ru2]}A.~Rudakov. Stability for an abelian category. J.~of Algebra 197 (1997), 231-245.
\bibitem{[S1]}K.~Saito. On the uniformization of complements of discriminant loci.
   In: Conference Notes. Amer.~Math.~Soc.~Summer Institute, Williamstown, 1975.
\bibitem{[S2]}K.~Saito. Theory of logarithmic differential forms and logarithmic vector fields.
  J.~Fac.~Sci.~Univ.~Tokyo Sect. IA Math 27 (1981) 265-291.
\bibitem{[Sc]} A.~Schofield. Semi-invariants of quivers. J. London Math. Soc. (2) 43 (1991), 385–395.
\bibitem{[Se]} U.~Seidel: Exceptional sequences for quivers of Dynkin type. Comm.~Algebra
      29 (2001). 1373-1386. 
\bibitem{[ST]} G.~C.~Shephard, J.~A.~Todd. Finite unitary reflection groups, Canad. J. Math.
   6 (1954), 274-304
\bibitem{[Si]}R. Simion. Noncrossing partitions. Discrete Math., 
   217(1-3) (2000), 367-409.
\bibitem{[SU]}R.~Simion and D.~Ullman. On the structure of the lattice of 
  noncrossing partitions. Discrete Math., 98(3) (1991), 193-206,
\bibitem{[Sl]} N.~J.~A.~Sloane: On-Line Encyclopedia of Integer Sequences. http://oeis.org/
\bibitem{[Sm]}S.~O.~Smal\o{}. Torsion theories and tilting modules.
     Bull.~London Math.~Soc.~16 (1984), 518-522.
\bibitem{[So]}L.~Solomon. Invariants of Euclidean reflection groups.
   Math.~Ann.~113 (1964), 274-286.
\bibitem{[Sm]} E.~Sommers. B-stable ideals in the nilradical of a Borel subalgebra.
   Canad.~Math.~Bull 48 (2005), 460--472. 
\bibitem{[SoT]} E.~Sommers, J.~Tymoczko. Exponents for B-stable ideals.  
   Trans.~Amer.~Math.~Soc.~358 (2006), 3493--3509.
\bibitem{[S1]}R. P. Stanley. Enumerative combinatorics. Vol. 2, volume 62 of
   Cambridge Studies in Advanced Mathematics.
   Cambridge University Press, Cambridge, 1999. 
\bibitem{[S2]}R. P. Stanley. Catalan addendum. New problems (and solutions) 
  related to Catalan numbers, continually updated and available online at \newline
  http://www-math.mit.edu/$\sim$rstan/ec/
\bibitem{[S3]}R. P. Stanley. An Introduction to hyperplane arrangements. in [MRS], 389-419.
\bibitem{[T]}H.~Terao. The exponents of a free hypersurface. In
  Singularities, Part 2 (Arcata, Calif., 1981), 
  Proc.~Sympos.~Pure Math.~vol 40, 561-566. Amer.~Math.~Soc., Providence, RI, (1983).




\end{thebibliography}

\vfill\eject 

\setcounter{tocdepth}{2}
\tableofcontents
\end{document}